\documentclass[10pt,oneside]{report} 
\title{The Derived Contraction Algebra}
\author{Matt Booth}
\date{2019}

\pdfoutput=1

\usepackage[phd, hyperref]{edmaths}

\usepackage{mathtools}

\newtheorem{thm}{Theorem}[section]
\newtheorem{prop}[thm]{Proposition}

\newtheorem{cor}[thm]{Corollary}

\newtheorem{lem}[thm]{Lemma}
\theoremstyle{definition}
\newtheorem{defn}[thm]{Definition}
\newtheorem{setup}[thm]{Setup}
\newtheorem*{thma}{Theorem}

\newtheorem*{conji}{Conjecture}
\newtheorem{conj}[thm]{Conjecture}
\theoremstyle{remark}
\newtheorem{ex}[thm]{Example}
\newtheorem{rmk}[thm]{Remark}

\usepackage{amsmath}
\usepackage{amsfonts}
\usepackage{amssymb}
\usepackage{amsthm}
\usepackage{tikz-cd}
\usepackage{tikz}
\usepackage{amsxtra}
\usepackage{wrapfig}
\usepackage{url}
\usepackage[percent]{overpic}
\usepackage{mathrsfs}
\usepackage{stmaryrd}

\usepackage{textcomp}
\DeclareSymbolFont{ugrf@m}{U}{eur}{m}{n}
\DeclareMathSymbol{\upmu}{\mathord}{ugrf@m}{"16}

\newcommand{\cat}[1]{{\mathbf{#1}}}
\newcommand{\p}{\paragraph{}}
\newcommand{\spec}{\operatorname{Spec}}
\newcommand{\onto}{\twoheadrightarrow}
\newcommand{\into}{\hookrightarrow}
\newcommand{\from}{\leftarrow}
\newcommand{\lot}{\otimes^{\mathbb{L}}}
\newcommand{\per}{{\ensuremath{\cat{per}}}\kern 1pt}
\newcommand{\stab}{\underline{\mathrm{CM}}}

\DeclareMathOperator{\id}{id}
\let\im\relax\DeclareMathOperator{\im}{im}
\let\ker\relax\DeclareMathOperator{\ker}{ker}
\DeclareMathOperator{\coker}{coker}

\let\hom\relax\newcommand{\hom}{\mathrm{Hom}}
\newcommand{\enn}{\mathrm{End}}
\DeclareMathOperator{\tor}{Tor}
\DeclareMathOperator{\ext}{Ext}
\newcommand{\cell}{\operatorname{Cell}}

\newcommand{\C}{\mathbb{C}}
\newcommand{\Q}{\mathbb{Q}}
\newcommand{\Z}{\mathbb{Z}}
\newcommand{\N}{\mathbb{N}}
\newcommand{\A}{\mathbb{A}}
\renewcommand{\P}{\mathbb{P}}
\newcommand{\R}{{\mathrm{\normalfont\mathbb{R}}}}
\newcommand{\mm}{{\upmu\upmu}}

\newcommand{\tbtm}[4]{\ensuremath{\begin{pmatrix}#1&#2\\#3&#4\end{pmatrix}}}
\newcommand{\stbtm}[4]{\ensuremath{\left(\begin{smallmatrix}#1&#2\\#3&#4\end{smallmatrix}\right)}}
\newcommand{\sthbthm}[9]{\ensuremath{\left(\begin{smallmatrix}#1&#2&#3\\#4&#5&#6\\#7&#8&#9\end{smallmatrix}\right)}}

\numberwithin{equation}{section}

\newcommand{\con}{\mathrm{con}}
\newcommand{\dca}{{\ensuremath{A^\mathrm{der}_\con}}}
\newcommand{\dq}{\ensuremath{A/^{\mathbb{L}}\kern -2pt AeA} }
\newcommand{\dqb}{\ensuremath{B/^{\mathbb{L}}\kern -2pt BeB} }
\newcommand{\thick}{\ensuremath{\cat{thick} \kern 0.5pt}}

\newcommand{\dgh}{\mathrm{HOM}}
\newcommand{\dge}{\mathrm{END}}
\newcommand{\dgart}{\cat{dgArt}_k^{\leq 0}}
\newcommand{\dga}{\cat{dga}_{k}^{\leq 0}}
\newcommand{\proart}{{\cat{pro}(\cat{dgArt}_k^{\leq 0})}}

\newcommand{\recol}{\mathrel{\substack{\textstyle\leftarrow\\[-0.6ex]
			\textstyle\rightarrow \\[-0.6ex]
			\textstyle\leftarrow}}}

\newcommand{\dloc}{\mathbb{L}_S(A)}

\newcommand{\del}{\text{\raisebox{.15ex}{$\mathscr{D}$}}\mathrm{el}}
\newcommand{\sdel}{\underline{\del}}
\newcommand{\defm}{\text{\raisebox{.15ex}{$\mathscr{D}$}}\mathrm{ef}}
\newcommand{\sdefm}{\underline{\defm}}
\newcommand{\mcs}{\mathrm{MC}}
\newcommand{\mc}{\mathscr{M}\kern -0.7pt C}
\newcommand{\smc}{\underline{\mc}}
\newcommand{\ggr}{\mathscr{G}\kern -1pt g}

\newcommand{\pdf}{{\ensuremath{\Omega(\Delta^\bullet)}}}
\newcommand{\sset}{\cat{sSet}}
\newcommand{\prodef}{\widehat{\sdefm}}
\newcommand{\frmdef}{\sdefm^{\mathrm{fr}}}
\newcommand{\profrmdef}{\prodef^{\text{\raisebox{-1.1ex}{$\mathrm{fr}$}}}}

\makeatletter
\newcommand{\holim@}[2]{%
	\vtop{\m@th\ialign{##\cr
			\hfil$#1\operator@font holim$\hfil\cr
			\noalign{\nointerlineskip\kern1.5\ex@}#2\cr
			\noalign{\nointerlineskip\kern-\ex@}\cr}}%
}
\newcommand{\holim}{%
	\mathop{\mathpalette\holim@{\leftarrowfill@\textstyle}}\nmlimits@
}
\makeatother

\usepackage[utf8]{inputenc}
\usepackage[T1]{fontenc}
\usepackage{textcomp}

\begin{document}

	\maketitle

	\dedication{For Emily.}
	
	\dedication{Why should things be easy to understand?\newline\normalfont --Thomas Pynchon}
	
		\begin{abstract}
			All minimal models of a given variety are linked by special birational maps called flops, which are a type of codimension two surgery. A version of the Bondal--Orlov conjecture, proved by Bridgeland, states that if $X$ and $Y$ are smooth complex projective threefolds linked by a flop, then they are derived equivalent (i.e.\ their bounded derived categories of coherent sheaves are equivalent). Van den Bergh was able to give a new proof of Bridgeland’s theorem using the notion of a noncommutative crepant resolution, which is in particular a ring $A$ together with a derived equivalence between $X$ and $A$. The ring $A$ is constructed as an endomorphism ring of a decomposable module, and hence admits an idempotent $e$. Donovan and Wemyss define the contraction algebra to be the quotient of $ A$ by $e$; it is a finite-dimensional noncommutative algebra that is conjectured to completely recover the geometry of the base of the flop. They show that it represents the noncommutative deformation theory of the flopping curves, as well as controlling the flop-flop autoequivalence of the derived category of $X$ (which, for the algebraic model $A$, is the mutation-mutation autoequivalence). 
			
			\p In this thesis, I construct and prove properties of a new invariant, the derived contraction algebra, which I define to be Braun--Chuang--Lazarev’s derived quotient of $A$ by $e$. A priori, the derived contraction algebra – which is a dga, rather than just an algebra -- is a finer invariant than the classical contraction algebra. I prove (using recent results of Hua and Keller) a derived version of the Donovan--Wemyss conjecture, a suitable phrasing of which is true in all dimensions. I prove that the derived quotient admits a deformation-theoretic interpretation; the proof is purely homotopical algebra and relies at heart on a Koszul duality result. I moreover prove that in an appropriate sense, the derived contraction algebra controls the mutation-mutation autoequivalence. These results both recover and extend Donovan--Wemyss’s. I give concrete applications and computations in the case of partial resolutions of Kleinian singularities, where the classical contraction algebra becomes inadequate.

	\end{abstract}
	
	\chapter*{Lay Summary}
	\addcontentsline{toc}{chapter}{Lay Summary}
	
	Algebraic geometry, one of the oldest parts of mathematics, studies very broadly the interplay between algebraic equations and geometric shapes. For example, consider the double cone in 3-dimensional space given by the equation $$x^2 + y^2 - z^2 = 0$$Geometrically, one can see that away from the cone point $(0,0,0)$, a small part of the cone looks roughly like a flat surface. However, this is not the case at the cone point; in technical terms the cone point is called a singularity. Algebraically, one can deduce the same information by looking at the polynomial $x^2+y^2-z^2$ and using a little calculus\footnote{The singularities of the surface $f(x,y,z)=0$ are those points that also satisfy $\frac{\partial f}{\partial x}=\frac{\partial f}{\partial y}=\frac{\partial f}{\partial z}=0$.}. One hence has a correspondence between the geometry (the singularities of the cone) and the algebra (some properties of its defining polynomial). The goal of algebraic geometry is to build a dictionary between the world of geometry, where one has strong visuospatial intuition, and the world of abstract algebra, a very powerful mathematical theory.
	
	\p Birational geometry is a subfield of algebraic geometry where one considers two geometric objects to be equivalent if they are the same outside of subsets of lower dimension. For example, if one removes the cone point from the cone, then intuitively one is left with two (tapered) cylinders. Thus birational geometry considers a cone and an infinite cylinder to be equivalent: removing a point from the cone, and removing a longitudinal circle from the cylinder, give the same space. The cylinder is smooth, meaning it has no singularities, so that a small part of the cylinder always looks roughly like flat 2-dimensional space. One typically considers smooth spaces to be `nice', as they share many properties with classical Euclidean space; hence we have found a nice space birational to the cone. One goal of birational geometry is to assign to every space $X$ a nice\footnote{In low dimensions, ``nice'' means smooth, but the situation is more complicated in higher dimensions.} space that is birational to $X$, called a minimal model. 
	
	\p A given space may have many different minimal models, but one can pass between them via special geometric operations called flops. For 3-dimensional spaces\footnote{Such as familiar 3D Euclidean space, or the 3-sphere given by the equation $x^2+y^2+z^2+w^2=1$, which lives in 4D Euclidean space.} $X$ , the idea of a flop is as follows: one picks a special type of curve inside $X$, cuts it out, and pastes it back in with the opposite orientation to get a new space $X^+$ called the flop of $X$. Because $X$ and $X^+$ are the same outside of a subset of lower dimension (the curve we cut out and pasted back in!) they must be birational.
	
	\p If $X$ is a 3-dimensional space, and $C$ is a curve inside $X$ that we can flop, then one can actually contract $C$ down to a point: imagine the curve inside the cylinder contracting down to a point, turning the cylinder into the cone (we picture this later in Figure \ref{conefig}). Call the resulting space we get $X_\con$ (so in the previous example, $X$ would be the cylinder and $X_\con$ the cone). Associated to such a contraction, Donovan and Wemyss construct an algebraic object called the contraction algebra, which recovers some geometric information about the contraction: for example, it knows some information about how $C$ sits inside $X$. They conjecture that actually, the contraction algebra determines all of the geometric properties of $X_\con$. As of writing, this conjecture is not known to be true.
	
	\p In this thesis, we enhance the contraction algebra to a new invariant, the derived contraction algebra. It contains all of the geometric information that the Donovan--Wemyss contraction algebra does, but also contains some extra information coming from the world of derived geometry\footnote{One can think of derived geometry in several ways. One intuition is that in the derived world, objects are more flexible. So geometric situations that seem intractable in the classical world are more well-behaved in the derived world; for example two spaces that intersect badly in the classical world will have a `derived intersection' that is much easier to deal with. For spaces that already intersect nicely in the classical world, their derived intersection will simply be the same as their usual intersection.}, a powerful general approach to algebraic geometry that has been developed since the 1990s. We prove that the derived contraction algebra contains information about the derived geometry of $X$ analogous to the information that the usual contraction algebra contains. We prove that the derived version of the Donovan--Wemyss conjecture is true: the derived contraction algebra determines all of the geometric properties of $X_\con$.

	\chapter*{Publications}
	\addcontentsline{toc}{chapter}{Publications}
	The material in this thesis is based on the preprints \cite{dqdefm} and \cite{dcalg}, although much of the material has been reordered to ensure a coherent narrative structure. Some material (in particular Chapter 4) is new, and generalises or improves \textit{op. cit}. We also include additional background material not present in either preprint. Part I of this thesis roughly corresponds to the first half of \cite{dqdefm} and the Appendix of \cite{dcalg}. Part II roughly corresponds to the second half of \cite{dqdefm}, and Part III roughly corresponds to the main body of \cite{dcalg}.

	\declaration\addcontentsline{toc}{chapter}{Declaration}

	\chapter*{Acknowledgements}
	\addcontentsline{toc}{chapter}{Acknowledgements}
	First and foremost, I would like to thank my supervisor Jon Pridham for his many hours of guidance, patience, and conversation. I would also like to thank Michael Wemyss for introducing me to contraction algebras and for his support at every stage; this thesis would not have been possible without either of them. I am grateful to Ivan Cheltsov, Joe Chuang, Ben Davison, Will Donovan, Zheng Hua, Martin Kalck, Bernhard Keller, Brent Pym, Ed Segal, and Dong Yang for their interest in my work and for helpful discussions and comments while it was in progress. I would like to thank the EPSRC for funding my studies, and the mathematics department at Edinburgh for being such a welcoming environment to do research in. I would also like to thank the examiners Andrey Lazarev and Michael Wemyss for reading this thing and for their helpful comments.
	
	\p Edinburgh has been a great place to study for four years, and I have made so many good friends in the department (and outside of it).	I'd like to thank everyone, in particular Jenny, Simon, X{^^c4^^ab}\kern 0.5pt l\'ing, Manuel, Ruth, Veronika, Tim W., Carlos, Juliet, Lukas S., Leonardo, Ai, Marcel, Sjoerd, Soheyla, Graham, Severin, Lukas M., Martina, Dani, Zoe, Tim H., Marco, Karen, Trang, Chris, Ollie, Ross, Fred, and Peter for conversations both mathematical and non-mathematical.
	
	\p
	I would like to thank everyone at EUQS, and more generally in the UK quizbowl community at large, for being so friendly and welcoming, and for keeping me sane during the second half of my doctoral studies. In particular I would like to thank Innis as well as my University Challenge teammates Marco, Max, Robbie, and Zak; I couldn't have asked for a better team. I'd like to thank my family and my non-mathematical friends, in particular Ryan, Jeannette, Calum, Cathy, Tim, Daniel, and Kess, none of whom I see enough of. Finally, I would like to thank Emily for her wonderful love and support; it's been a joy.

	\tableofcontents

	\chapter{Introduction} \pagenumbering{arabic}
	We give a broad introduction to some of the main concepts of this thesis: the birational geometry of threefolds, derived and noncommutative geometry via differential graded (dg) categories, derived deformation theory, and contraction algebras. We then discuss our results, give an outline of the structure of this thesis, and set notation and conventions. The expert reader who wishes to immediately see what is new can skip to \S\ref{intronew}. Throughout, $k$ will denote an algebraically closed field of characteristic zero.

	\section{The minimal model programme}
	In this thesis, the word ``variety'' means an irreducible quasi-projective variety over $k$. Recall that a {rational map} $X \dashrightarrow Y$ is a morphism $U \to Y$, where $U$ is an open subset of $X$, and that a {birational map} is a rational map with rational inverse. So two varieties are birationally equivalent (or just birational) if and only if they are isomorphic away from positive codimensional subvarieties. For example, projective $n$-space $\P^n$ is birational to $\A^n$ via the usual coordinate map $\P^n \dashrightarrow \A^n$ sending $[x_0:x_1:\cdots:x_n]$ to $(\frac{x_1}{x_0},\ldots,\frac{x_n}{x_0})$. The function field $K(X)$ of rational functions on $X$ is a birational invariant, and this shows, for example, that dimension is also a birational invariant. If $f: X \dashrightarrow Y$ is a (dominant) rational map, then pulling back along $f$ gives a map $K(Y)\to K(X)$, and in fact this is an isomorphism if and only if $f$ was a birational equivalence \cite[I.4.4]{hartshorne}. One can use this function field characterisation to show, for example, that every variety is birational to a hypersurface \cite[I.4.9]{hartshorne}.
	
	\p One fundamental example of a birational transformation is a blowup, which is the universal way of turning subvarieties into (effective Cartier) divisors. Let $X$ be a variety and $Z$ a closed subvariety of positive codimension, with ideal sheaf $\mathcal{I}$. Then the {blowup} of $Z$ in $X$ is the relative Proj of the Rees algebra $\pi:\mathrm{Proj}(\oplus_n\mathcal{I}^n) \to X$, which comes with a fibration $\pi$ to $X$ which is an isomorphism away from the centre $Z$, and in particular is a birational map \cite[II.7]{hartshorne}. For example, if $Z=p$ is a smooth point of $X$, then the blowup is morally just $X$ but with $p$ replaced by a copy of the projectivised tangent space $\P T_pX$; one can think of this as separating out all the tangent directions at $p$ \cite[I.4]{hartshorne}. This operation of blowing up a single point is called a monoidal transformation, and one can show that they preserve smoothness. More generally, they are useful for resolving singularities: for example, let $X$ be the ordinary double point $\spec\left(\frac{k[x,y,z]}{xy-z^2}\right)\subseteq \A^3$. It has a unique singular point at the origin, and if we blow it up we get a smooth variety $\tilde X$ with a map $\tilde X \to X$. Away from the origin, this map is an isomorphism, and above the origin there is a copy of $\P^1$; we picture this resolution in Figure \ref{conefig}. Call a proper birational map $\tilde X \to X$ a {resolution} if $\tilde X$ is smooth. In characteristic zero, resolutions always exist: for curves this has been known for a long time, for surfaces and threefolds the first algebraic proofs were given by Zariski \cite{zariskisurf, zariskithreef}, and in all dimensions this is a famous theorem of Hironaka \cite{hironaka}. A resolution $\tilde X \to X$ is called a {minimal resolution} if any other resolution factors through it. 
	
		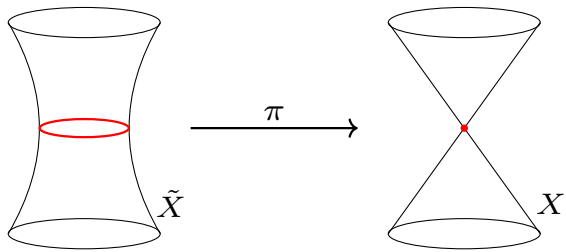
\begin{figure}[h!]
			\vspace{0.2cm}
		\begin{center}
			\tikzset{Bullet/.style={fill=black,draw,color=#1,circle,minimum size=3pt,scale=0.25}}
			\begin{tikzpicture}[looseness=1, scale=0.2]
			
			\draw (0,-7) ellipse (5cm and 1cm);
			\draw (0,7) ellipse (5cm and 1cm);
			\draw (-5,-7) to [bend right=30] (-5,7);
			\draw (5,-7) to [bend left=30] (5,7);
			\draw[line width=0.3mm, red] (0,0) ellipse (2.93cm and 0.6cm);
			\draw (25,-7) ellipse (5cm and 1cm);
			\draw (25,7) ellipse (5cm and 1cm);
			\draw (20,-7) to [bend left=0] (30,7);
			\draw (30,-7) to [bend left=0] (20,7);
			\node[Bullet=red] at (25,0)  {};
			\draw[->, looseness=0, line width=0.3mm] (7,0) to (18,0);
			\node[scale=1.7] at (12.5,1) {$\scriptstyle \pi$};
			\node[scale=1.6] at (5.7,-5) {$\scriptstyle \tilde{X}$};
			\node[scale=1.6] at (30.7,-5) {$\scriptstyle X$};
			\end{tikzpicture}
		\end{center}
		\caption{The minimal resolution $\tilde X$ of the ordinary double point $X$.}
		\label{conefig}
		\vspace{-0.5cm}
	\end{figure} 
	
	\p One can try to classify projective varieties up to birational equivalence. In view of Hironaka's theorem, we may as well start with a smooth projective variety $X$ and ask: what nice varieties are in the birational equivalence class of $X$? For curves, this is easy to answer: two smooth projective curves are birational if and only if they are isomorphic\footnote{Because any birational map from a smooth curve extends to a morphism. See e.g.\ \cite[1.4]{flipsflops}.}. For surfaces, the situation is already more complicated. Any birational map between smooth surfaces factors as a finite zig-zag of monoidal transformations\footnote{A version of this statement -- the {weak factorisation theorem} --  is true in all dimensions \cite[1.11]{flipsflops}.}. So one could hope that, starting with a smooth surface $S$ and repeatedly contracting blown-up curves, one might arrive at a nice birational model for $S$. This is more or less how it goes: say that a curve in a smooth surface is a $(-1)$-curve if its self-intersection number is $-1$. It is not too hard to show that the exceptional locus of a monoidal transformation is a $(-1)$-curve \cite[V.3.1]{hartshorne}, and in fact the converse is true: Castelnuovo's theorem \cite[V.5.7]{hartshorne} says that, given a $(-1)$-curve $C$ in a smooth surface $S$, it can be smoothly blown down; i.e.\ there is a surface $S'$ such that $S$ is obtained from $S'$ by a monoidal transformation with exceptional locus $C$. Given a smooth surface, one can then repeatedly contract $(-1)$-curves: this must eventually stop, since a blowdown drops the Picard number\footnote{The Picard number is the rank of the finitely generated abelian group of divisor classes modulo those classes algebraically equivalent to zero, known as the N\'eron-Severi group \cite[II.6.10.3]{hartshorne}.} of a surface. In other words, every surface $S$ is birational to a surface $S'$ with no $(-1)$-curves (classically, these are called minimal surfaces), and as long as $S$ is not rational or ruled then $S'$ is unique\footnote{This is a theorem of Zariski; see e.g.\ \cite[V.5.8.4]{hartshorne}.}. Note that $\P^2$ and $\P^1\times \P^1$ are two birational minimal surfaces, so that one cannot drop the `not rational or ruled' assumption. A minimal resolution of a surface (that is not rational or ruled) in the earlier sense is precisely a minimal surface in the new sense.
	
	\p Let us look at an example in some more detail. If $G$ is a subgroup of $\mathrm{SL}(2)$, then $G$ acts on the plane $\A^2$ by matrix multiplication. When $G$ is finite, the quotient $\A^2/G$ is Spec of the ring of invariants $k[x,y]^G$, and such quotients -- all isolated surface singularities -- are known as the Kleinian singularities (or the Du Val singularities); see e.g.\ Reid \cite{reidypg}. For example, if $\omega$ is a primitive $n^\text{th}$ root of unity, then the matrix $\stbtm{\omega}{0}{0}{\omega^{-1}}$ generates a copy of $\Z/n$, and the corresponding quotient is the surface $\spec\left(\frac{k[x,y,z]}{xy-z^n}\right)$. The minimal resolution of a Kleinian singularity exists, is an isomorphism away from the singular point, and above the singular point there is a tree of rational curves linked in a Dynkin configuration. The corresponding Kleinian singularity is labelled by its Dynkin diagram; the example above is the $A_{n-1}$ singularity. Note that if $n=2$, we get the ordinary double point, whose minimal resolution we know has just one curve above the origin.

	\p Trying to classify high-dimensional varieties immediately gets much harder. Minimal resolutions may not exist: let $Y$ be the quadric cone $xy-zw=0$ in $\A^4$. Blowing up the origin gives us a birational morphism $X \to Y$ with exceptional divisor $\P^1 \times \P^1$. One can contract either of the copies of $\P^1$ by a projection to end up with two varieties $X^-$ and $X^+$. These are both resolutions, neither of which factors through the other \cite{atiyahflop}. However, $Y$ does admit a crepant resolution -- we call a resolution $\pi: X \to Y$ crepant if $\pi^*K_Y= K_X$, where $K_X$ denotes the canonical divisor \cite{reidpagoda}. In general, for any resolution $\pi: X \to Y$ with irreducible exceptional divisors $E_i$, then one can write $K_X=\pi^*K_Y+\sum_i a_iE_i$; the $a_i$ are known as the discrepancies, so that a crepant resolution is one without discrepancy \cite{reidypg}. Call a singularity canonical if it admits a resolution with nonnegative discrepancies, and terminal if it admits a resolution with positive discrepancies\footnote{See Reid \cite{reidctf, reidpagoda} for some alternate characterisations.}. Kleinian singularities are precisely the two-dimensional canonical singularities, and their minimal resolutions are exactly the crepant resolutions \cite{reidypg}. Three-dimensional terminal singularities are all isolated \cite{moriterm}.
	
	\p The cone $xy-zw=0$ is an example of a compound Du Val (cDV) singularity; these are the three-dimensional terminal singularities whose generic hyperplane cut is a Du Val singularity. They can always be written in the form $f(x,y,z)+tg(x,y,z,t)=0$, where $f$ is a polynomial defining a Du Val singularity and $g$ is any polynomial \cite[\S2]{reidctf}. A crepant resolution of an isolated cDV singularity behaves much like that of a Du Val singularity: away from the isolated singular point, the map is an isomorphism, and above the singular point the resolution has a chain of rational curves\footnote{Note however that this need not be a Dynkin configuration.}. However, even cDV singularities do not always admit crepant resolutions \cite{katzsr}.
	
	\p What to do now? Work of Reid \cite{reidctf} and Mori \cite{mori} suggests that we should allow some singularities in our minimal models. Call a variety $M$ a minimal model if it has terminal\linebreak $\Q$-factorial singularities and the canonical divisor $K_M$ is nef\footnote{A divisor $D$ is nef if $D\cdot C$ is positive for every curve $C$.}. Heuristically, this latter condition means that there are no $(-1)$-curves. A minimal model of a surface $S$ is a minimal surface, and the converse is true as long as $S$ is not rational or ruled \cite[1.10]{glimpse}. 
	
	\p The goal of the minimal model programme for a projective variety $X$ of nonnegative Kodaira dimension is to find a minimal model $Y \to X$. This is conjecturally achieved as follows: as a first approximation, let $Y$ be a resolution of $X$. Now if $K_Y$ is nef then we are done. If not, pick a curve $C$ with negative intersection against $K_Y$; one should then be able to find a birational map $\pi:Y\to Y_\con$ contracting $C$. If $\pi$ is a divisorial contraction then replace $Y$ with $Y_\con$ and keep going; the Picard number has dropped. If not (i.e.\ $\pi$ does not contract divisors; we say in this situation that $\pi$ is small), then one needs to `flip' $C$ by replacing it with a curve whose intersection with $K_Y$ is positive; in this case replace $Y$ with this flip and keep going. This process ought to stop, and eventually we ought to arrive at a minimal model. However, it is far from clear that this process should terminate! See \cite{glimpse} or \cite{flipsflops} for more discussion, the latter especially with regard to the conjecture that one cannot flip curves infinitely many times.
	
	 \p We remark that if $X$ has Kodaira dimension $-\infty$, one looks instead for a Mori fibre space birational to $X$. Note that the surfaces of Kodaira dimension $-\infty$ are precisely the rational or ruled surfaces \cite[V.6.1]{hartshorne}, so that even in the surface case we see that one needs to make Kodaira dimension restrictions.
	 
	 \p In dimension three, it is known that the minimal model programme `works': the above algorithm assigns every projective variety a minimal model \cite{moriflip}. Minimal models are not in general unique: recall that one had to make choices in picking curves to contract\footnote{There may be infinitely many contractible curves \cite[V.5.8.1]{hartshorne}.}. However, one has some control over this nonuniqueness: minimal models are all linked by special birational maps called flops, which, like flips, are isomorphisms in codimension one \cite{kollarflops}. In fact, a theorem of Kawamata says that flops link minimal models in all dimensions \cite{kawconnect}. There are many equivalent definitions of flops in the literature \cite{kollarffm}, and we take ours from Hacon--McKernan \cite{flipsflops}, which works in all dimensions: a flopping contraction is a small projective birational morphism $\pi:X \to X_\con$ of relative Picard number 1 such that $K_X$ is trivial over $K_{X_\con}$. In this situation, there exists a commutative diagram $$\begin{tikzcd}X\ar[rr, dashrightarrow,"\phi"]\ar[rd,"\pi"] & & X^+\ar[ld,"\pi^+"'] \\ & X_\con & \end{tikzcd}$$ where $\pi^+$ is a flopping contraction and $\phi$ is birational and an isomorphism in codimension 1 (but not an isomorphism); we call $\phi$ the flop of $X$. Note that $\pi$ does not contract divisors but contracts some finite number of rational curves; intuitively one can think of the flop $X^+$ as being the result of taking out these curves and sewing them back in with the opposite orientation. We draw a cartoon of a typical threefold flop in Figure \ref{flopfig}. Note the symmetry in the definition, so that one can regard $X$ as the flop of $X^+$.
	 
	 	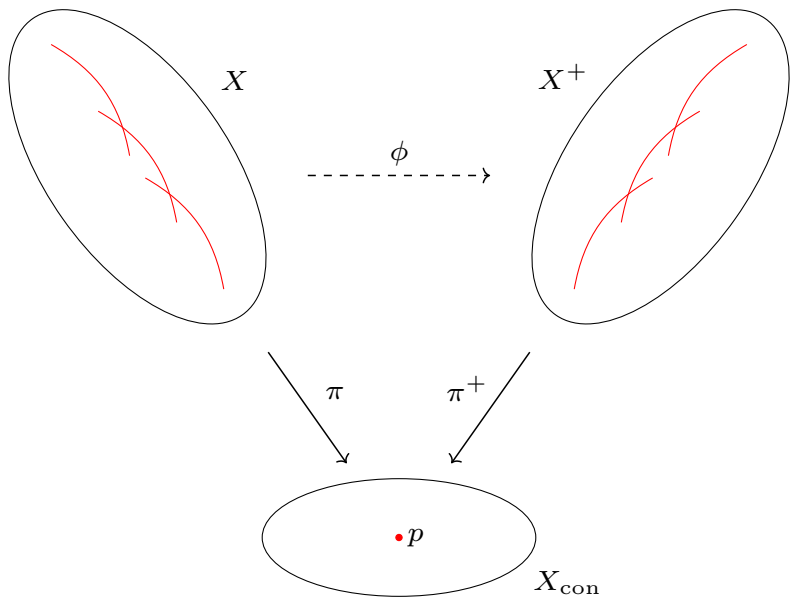
\begin{figure}[h!]
	 	\begin{center}
	
	 		\tikzset{Bullet/.style={fill=black,draw,color=#1,circle,minimum size=3pt,scale=0.25}}
	 		\begin{tikzpicture}[looseness=1, scale=0.6]
	 		\draw (0,0) ellipse (3cm and 1.3cm);
	 		\node[Bullet=red] at (0,0) {};
	 		\draw[rotate=35] (0,10) ellipse (2cm and 4cm);
	 		\draw[rotate=35, red] (0,6.7) to [bend right=25] (0,9.7);
	 		\draw[rotate=35, red] (0,8.5) to [bend right=25] (0,11.5);
	 		\draw[rotate=35, red] (0,10.3) to [bend right=25] (0,13.3);
	 		\draw[rotate=-35] (0,10) ellipse (2cm and 4cm);
	 		\draw[rotate=-35, red] (0,6.7) to [bend left=25] (0,9.7);
	 		\draw[rotate=-35, red] (0,8.5) to [bend left=25] (0,11.5);
	 		\draw[rotate=-35, red] (0,10.3) to [bend left=25] (0,13.3);
	 		\draw[->, looseness=0,rotate=35, line width=0.2mm] (0,5) to (0,2);
	 		\draw[->, looseness=0,rotate=-35, line width=0.2mm] (0,5) to (0,2);
	 		\draw[->, looseness=0, line width=0.2mm, dashed] (-2,8) to (2,8);
	 		\node[scale=1.5] at (-3.6,10.1) {$\scriptstyle X$};
	 		\node[scale=1.5] at (3.6,10.2) {$\scriptstyle X^+$};
	 		\node[scale=1.5] at (-1.4,3.2) {$\scriptstyle \pi$};
	 		\node[scale=1.5] at (1.5,3.3) {$\scriptstyle \pi^+$};
	 		\node[scale=1.5] at (0,8.5) {$\scriptstyle \phi$};
	 		\node[scale=1.5] at (3.7,-1) {$\scriptstyle X_\con$};
	 		\node[scale=1.5] at (0.37,0) {$\scriptstyle p$};
	 		\end{tikzpicture}
	 	\end{center}
 	\caption{A threefold flop.}
 	\label{flopfig}
 	 		\vspace{-0.5cm}
	 \end{figure} 
	 
	 \p Returning to the quadric $Y$ given by the equation $xy-zw=0$, recall the existence of two varieties $X^-$ and $X^+$ resolving $Y$ whose exceptional loci are both copies of $\P^1$. The map $X^- \to Y$ is a flopping contraction, and $X^+$ is its flop; this is known as the Atiyah flop, and is the simplest example of a three-dimensional flopping contraction \cite{atiyahflop}. Reid generalised this example to the {Pagoda flop}, where the base is $xy-(z+w^n)(z-w^n)=0$; taking $n=1$ recovers the Atiyah flop \cite{reidpagoda}.
	  
	 \p If $\pi:X \to X_\con$ is a minimal model of a terminal threefold, then it is a flopping contraction, and its flop is again a minimal model. One can obtain all minimal models of a given threefold this way, via iterated sequences of flops \cite{kollarflops}. Call a flopping contraction simple if the exceptional locus of $\pi$ is a single rational curve; in this case $\pi$ is an isomorphism away from the contracted locus, which is a single point. In this thesis, we will be particularly interested in simple flopping contractions for which the contracted locus is an isolated singularity; in this case the completion of $X_\con$ at this singularity must then be cDV \cite[5.38]{kollarmori}.

\section{Derived categories and geometry}
Originally introduced by Verdier in his thesis \cite{verdier} to unify various aspects of homological algebra, derived categories have since become a central part of modern mathematics. They can be motivated by Thomas' slogan ``complexes good, cohomology bad'' \cite{thomasworking}: the point is that many homological constructions one wants to make really take place at the chain level. For example, if $M,N$ are two modules over a ring $R$, then to compute $\ext^*_R(M,N)$ one takes a resolution $P\to M$, computes the hom-complex $\hom_R(P,N)$, and takes cohomology. But, up to quasi-isomorphism, this hom-complex is well-defined, and so one would like some construction of a `total derived hom' functor $\R\hom(-,N)$ that outputs complexes up to quasi-isomorphism. The derived category was originally defined in order to be the natural home for total derived functors. Hence, the derived category $D(\mathcal{A})$ of an abelian category $\mathcal A$ has as objects chain complexes from $\mathcal A$, but we identify two complexes precisely when they are quasi-isomorphic \cite{weibel}.

\p Why are derived categories useful in algebraic geometry? Given a variety (or a scheme, or a stack...) $X$, its derived category of (quasi)coherent sheaves contains a surprising amount of geometric information: for nice schemes one can recover invariants such as the dimension, Kodaira dimension, pluricanonical ring, and much more \cite{caldararu,huybrechts} from the derived category. Moreover, the freedom given by derived categories means that interesting new objects may exist. For example, let $f:X\to Y$ be a morphism of smooth projective varieties. On the level of sheaves, the pushforward functor $f_*$ admits a left adjoint, the pullback functor $f^*$, but in general $f_*$ does not admit a right adjoint as it is not right exact. However, on the level of derived categories, the total derived functor $\R f_*$ does admit a right adjoint, which is essentially given by Serre duality: more precisely, one takes the derived pullback, tensors with the relative dualising sheaf, and shifts by the relative dimension of $f$. This is known as Grothendieck duality (see e.g.\ \cite{lipman}).

\p Sometimes, the derived category of a scheme knows all of the information about it: for example, let $X$ be  a smooth projective variety with ample or antiample canonical bundle. Bondal and Orlov proved a famous theorem stating that in this setup, the derived category of $X$ recovers $X$ amongst all smooth projective varieties \cite{bosemiorthog, bondalorlov}. The idea is to reconstruct the point sheaves and then the line bundles as complexes satisfying certain cohomological conditions. One can then recover the pluricanonical ring of $X$, and hence $X$ by taking $\mathrm{Proj}$. However, derived categories are not complete invariants for schemes: for example, Mukai proved that an abelian variety is always derived equivalent to its dual variety \cite{mukaifourier}. It is this flexibility that affords derived categories their power. 

\p One can concretely compute with derived categories: Be{^^c4^^ad}linson showed \cite{beilinsonpn} that the derived category of $\P^n_k$ is generated by the $n+1$ line bundles $\mathcal{O}, \mathcal{O}(-1),\ldots, \mathcal{O}(-n)$. We will return to this result later in the context of noncommutative geometry. To prove it, Be{^^c4^^ad}linson uses the concept of a Fourier--Mukai transform, the first example of which appears in the aforementioned paper of Mukai \cite{mukaifourier}. Let $X$ and $Y$ be smooth projective varieties and let $\mathcal K$, which we refer to as the kernel, be an object of the bounded derived category $D^b(X\times Y)$. Then $\mathcal K$ defines a functor $\Phi^{X\to Y}_{\mathcal K}: D^b(X)\to D^b(Y)$ given by the following recipe: given a complex on $X$, take its derived pullback to $X\times Y$, `twist' it by tensoring with the kernel $\mathcal K$, and take the derived pushforward to $Y$. We call $\Phi^{X\to Y}_{\mathcal K}$ the Fourier--Mukai (FM) transform with kernel $\mathcal K$. Many functors of interest are FM transforms: for example, let $f:X\to Y$ be a morphism and let $\Gamma_f$ be its graph, regarded as a sheaf on $X\times Y$. Then $\Phi^{X\to Y}_{\Gamma_f}$ is the derived pushforward along $f$, whereas $\Phi^{Y\to X}_{\Gamma_f}$ is the derived pullback along $f$ \cite{huybrechts}. Orlov proved that every fully faithful triangle functor between derived categories of smooth projective varieties is a FM transform \cite{orlovfm}.

\p Derived categories behave well with respect to birational transformations. For example, let $X$ be a smooth variety and $Y\into X$ a smooth subvariety of codimension $r$. Let $\pi_X:\tilde X \to X$ be the blowup of $X$ at $Y$, let $\tilde Y$ be the exceptional locus with inclusion map $j:\tilde Y \into \tilde X$, and let $\pi_Y:\tilde Y \to Y$ be the projective fibration of the exceptional locus over the centre. Orlov proves \cite{orlovblowup} that the functors $\mathbb{L}\pi_X^*: D^b(X) \to D^b(\tilde X)$ and $\R j_*(\mathcal{O}_{\tilde Y}(m)\otimes \pi_Y^*): D^b(Y)\to D^b(\tilde X)$ are embeddings, and moreover that the images of $D^b(X)$ and $D^b(Y)$ (as $m$ varies from $1-r$ to $1$) generate $D^b(\tilde X)$. In other words, one can piece together the derived category of a blowup from the derived categories of the base and centre.

\p This computation of the derived category of a blowup inspired Bondal and Orlov to conjecture \cite{bosemiorthog, bondalorlov} that any two crepant resolutions of a Gorenstein variety are derived equivalent. In dimension three, proving this is equivalent to showing that a flop $X \dashrightarrow X^+$ between two smooth threefolds induces a derived equivalence, because all crepant resolutions are linked by flops \cite{kollarflops}. The Bondal-Orlov conjecture in dimension three was proved by Bridgeland in a landmark paper \cite{bridgeland} and was later generalised by Chen \cite{chen} to allow Gorenstein terminal singularities. Loosely, Bridgeland's strategy is to construct a flop as a moduli space of perverse sheaves, and then the universal family on this moduli space gives a Fourier--Mukai kernel which induces an equivalence. Note that flops are symmetric: if $X^+$ is the flop of $X$, then $X$ is also the flop of $X^+$. Hence, one can compose the Bridgeland--Chen equivalences $D^b(X)\to D^b(X^+)\to D^b(X)$ to obtain a (nontrivial!) autoequivalence $D^b(X) \to D^b(X)$, the {flop-flop autoequivalence}. We will return to this later in the introduction.

\p Derived categories have a rich structure: namely they are triangulated, meaning they admit a shift functor $[1]$ that satisfies some reasonable compatibility conditions. For example, one can take mapping cones, and these have a limited form of functoriality; see Neeman \cite{neemantxt} for a comprehensive reference. However, the axioms of triangulated categories do not suffice for some geometric purposes. For example, derived categories when viewed as triangulated categories do not satisfy Zariski descent: if a scheme $X$ is covered by opens $U_i$, then knowing the derived categories of the $U_i$ together with the induced morphisms does not in general suffice to determine the derived category of $X$. Locality fails even for the standard affine cover of $\P^1$, discussion of which appears as an extended example in \cite{toendglectures}.

\p Many of the issues with triangulated categories can be fixed by passing to the world of dg categories, which we will view as an enhancement of triangulated categories. The basic idea of a dg category is to remember the morphism complexes inside the derived category $D(X)$ and bundle them together in a homotopy coherent manner; more formally dg categories are exactly those categories enriched over chain complexes of vector spaces \cite{keller}. We review the theory of dg categories in \S\ref{dgcats}; the use of dg categories (as opposed to mere triangulated categories) will be fundamental in this thesis. At an abstract level, one can think of the dg enhancement of a derived category as remembering `higher homotopical information', whereas remembering only the triangulated structure corresponds to `flattening' these higher morphisms down to a 1-categorical structure. Indeed, pretriangulated dg categories are equivalent (in an appropriate sense) to linear stable $\infty$-categories \cite{cohnstable}, and the triangles of a triangulated category can be thought of as flattenings of long homotopy cofibre sequences.

	\section{Noncommutative geometry}
Scheme-theoretic algebraic geometry relies at its core on the $\spec$ construction: given a commutative ring $R$, one augments its set of prime ideals with the Zariski topology. After equipping it with the structure sheaf, this becomes a locally ringed space $\spec R$, and all schemes are glued together out of such affine pieces. Unfortunately, trying to make the same construction work verbatim for noncommutative rings will fail as many noncommutative rings have `too few' prime ideals: for example the Weyl algebra $\frac{k\langle X,Y\rangle}{[X,Y]-1}$ of polynomial differential operators on $\A_k^1$ is a simple ring, and so its prime spectrum contains little geometric information. Moreover, constructions of noncommutative spectra that are both nonempty for all rings and agree with the usual commutative spectrum for commutative rings cannot be functorial \cite{specreyes,specheunen}.

\p However, one would still like to do geometry with noncommutative rings. Following Ginzburg \cite{ginzburgnc}, one should distinguish between two different approaches to noncommutative geometry. In the first approach, `noncommutative geometry in the small', one tries to mimic the constructions of commutative algebraic geometry and extend them to noncommutative rings (which should be thought of as `noncommutative deformations' or `quantisations' of commutative ones). The second approach, `noncommutative geometry in the large', takes a different tack by replacing rings with new objects. One can recover a commutative ring $R$ from its module category; similarly, one can recover a scheme $X$ from its category $\mathrm{Coh}X$ of coherent sheaves\footnote{This is the Gabriel--Rosenberg theorem \cite{gabriel, rosenberg}.}. One would hence like to view a general abelian category as a `category of sheaves on a noncommutative space'. Unfortunately, this notion is not satisfactory: abelian categories of sheaves do not contain enough cohomological information and typically lack enough projectives. The solution is to pass to the derived category.

\p The viewpoint of `noncommutative geometry in the large' is hence to regard triangulated or dg categories as fundamental geometric objects in their own right. From the point of view of more classical geometry this loses information; however we gain a lot of powerful technical machinery and new ways to think. For example, a result of Bondal and Van den Bergh\linebreak \cite[3.18]{bvdb} says that the bounded derived category of quasicoherent sheaves on every (qcqs) scheme $X$ admits a compact generator $P$; hence $D^b(X)$ is `derived affine' in the sense that it is quasi-equivalent to the derived category of the derived endomorphism algebra $\Lambda=\R\enn_X(P)$, which in general is a differential graded algebra (dga). Any approach to derived commutative geometry that uses commutative dgas as affines (e.g.\ in characteristic zero, To\"en and Vezzosi's homotopical algebraic geometry \cite{hag1, hag2}) can hence be viewed through the framework of noncommutative geometry. A more descriptive name for `noncommutative geometry in the large' might be `noncommutative derived geometry'.

\p As an example of the two notions of noncommutative geometry interacting, consider again Be\u\i linson's determination of the derived category of $\P^n$ as generated by the $n+1$ sheaves $\mathcal{O}, \mathcal{O}(-1),\ldots, \mathcal{O}(-n)$. Let $G_n$ denote the direct sum of them all and put $B_n\coloneqq \enn_{\P^n}(G_n)$ the endomorphism ring. A sheaf cohomology computation shows that $B_n$ is quasi-isomorphic to the derived endomorphism ring $\R\enn_{\P^n}(G_n)$. It now follows that the functor $$\R\hom(G_n,-):D^b(\P^n) \to D^b(\cat{mod}\text{-}B_n)$$is a triangle equivalence. When $n=1$, it is easy to compute $B_1$ directly: it is the path algebra of the Kronecker quiver. More generally, one can write down explicit quivers with relations that are derived equivalent to $\P^n$. One can now study algebraic properties of these quivers in order to learn about the geometry of $\P^n$. Although this is a toy example, the idea is powerful: given a (qcqs) scheme $X$, it is `derived affine' in the sense that it is derived equivalent to some noncommutative dga, and the homological algebra of this dga reflects the geometry of $X$. Even better, if $X$ admits a tilting bundle -- essentially, a generator $V$ of $D(X)$ whose higher self-Exts vanish -- then $X$ is derived equivalent to a genuine ring (namely, the endomorphism ring of $V$). Noncommutative geometry provides powerful new ways to study classical commutative geometry.

\p For us, the specific jumping off point into noncommutative geometry will be Van den Bergh's noncommutative proof of Bridgeland's theorem on autoequivalences. Consider a threefold flopping contraction of a chain of rational curves $\pi: X \to \spec R$. Van den Bergh \cite{vdb} uses this data to construct a (noncommutative) ring $A$ with a derived equivalence to $X$. Specifically, he constructs $A$ as the endomorphism ring of a relative tilting bundle $\mathcal V=\mathcal{O}_X\oplus \mathcal{M}$ on $X$. When $X$ is smooth, $A$ is an example of a {noncommutative crepant resolution} ({NCCR}) of $R$ \cite{vdbnccr}. Van den Bergh then uses this noncommutative model for $X$ to give a new proof of Bridgeland's and Chen's theorems that flops induce autoequivalences.

\section{The contraction algebra}
Let $\pi:X \to \spec R$ be a threefold flopping contraction which is an isomorphism away from a single closed point $p$ in the base. Donovan and Wemyss \cite{DWncdf} used Van den Bergh's noncommutative model $A$ of $X$ to produce a new invariant, the contraction algebra. Loosely, the idea is as follows: $A$ is naturally presented as an endomorphism ring $\enn_X(\mathcal{V})$. Because $\mathcal V$ admits the trivial line bundle $\mathcal{O}_X$ as a direct summand, the ring $A$ admits an idempotent $e=\id_{\mathcal{O}_X}$, and one would like to define the contraction algebra to be the quotient $A/AeA$. This is the basic idea; to make the theory work one first needs to pass to a complete local base $\hat{R}_p$ and and then through a Morita equivalence. One key point is that one can compute $A$ as the endomorphism ring $\enn_R(\pi_*\mathcal{V})$; putting $\pi_*\mathcal V=R\oplus M$ one then wants to take the stable endomorphism ring $A/AeA\cong \underline{\enn}_R(M)$ as the definition of the contraction algebra $A_\con$.

\p When the flop is simple, the contraction algebra is a noncommutative Artinian local algebra; i.e.\ it is a finite-dimensional algebra with unique maximal ideal $\mathfrak m$ with residue field $k$. More generally, suppose that the chain of flopping curves is composed of $n$ irreducible rational curves. Then $A_\con$ is an $n$-pointed algebra, meaning that it has an augmentation $A_\con \to k^n$. In fact, if the exceptional locus has irreducible components $C_1,\ldots, C_n$ then the sheaves $\mathcal{O}_{C_i}(-1)$ on $X$ correspond across the derived equivalence $D(X) \xrightarrow{\simeq} D(A) $ to the one-dimensional simple $A$-modules appearing as the irreducible summands of the $n$-dimensional $A$-module $A_\con / \mathrm{rad}(A_\con)$ \cite[\S2]{DWncdf}.

\p One does not really need a threefold flop to define the contraction algebra: the relevant properties of Van den Bergh's tilting bundle for a contraction are axiomatised in \cite{contsdefs}. In \cite{enhancements} it is shown that one can globalise this idea: to any contraction admitting a certain generalisation of a tilting bundle, one can associate a sheaf of algebras on the base which is supported at the non-isomorphism locus. For a flopping contraction, this locus is a finite set of isolated points, and one recovers the usual contraction algebra at a point as the stalk. In fact, the contraction algebra of a threefold contraction $\pi$ is finite-dimensional if and only if $\pi$ is flopping (because $\pi$ is flopping if and only if it contracts curves without contracting a divisor, which is the case if and only if the contracted locus is zero-dimensional).

\p Why is the contraction algebra interesting? It recovers all known invariants of simple threefold flops; for example the width \cite{reidpagoda} of the flop (when defined) is the dimension of the associated contraction algebra. The Dynkin type of the flop and the type of the normal bundle of the flopping curve (for which there are three choices \cite{pinkham}) are both closely related to whether the contraction algebra is commutative \cite{DWncdf}. The dimension of the contraction algebra is a weighted sum of the Gopakumar--Vafa invariants of the flop \cite{todawidth}, and moreover the contraction algebra is a strictly stronger invariant than the G--V invariants \cite{katzgv, bwgv}. Based on this, Donovan and Wemyss conjecture that the contraction algebra determines the base of a smooth simple complete local flop:

 	\begin{conji}[Donovan--Wemyss {\cite[1.4]{DWncdf}}]
	Let $X \to \spec R$ and $X' \to \spec R'$ be flopping contractions of an irreducible rational curve in a smooth projective threefold, with $R$ and $R'$ complete local rings. If the associated contraction algebras are isomorphic, then $R \cong  R'$.
\end{conji}
One should note that when $k=\C$, this is an analytic classification. As of writing, this conjecture is still open. We remark that in the multiple-curve case, the correct generalisation of the conjecture \cite[1.3]{jennyf} stipulates that the contraction algebras should be derived equivalent, not necessarily isomorphic.

\p Donovan and Wemyss also show that the contraction algebra controls the Bridgeland--Chen flop-flop autoequivalence, in the following sense. Algebraically, the flop-flop autoequivalence $D(X)\to D(X)$ can be interpreted as a mutation-mutation autoequivalence $ \mm: D(A) \to D(A)$, analogous to Fomin--Zelevinsky mutation \cite{iyamareiten}. Donovan and Wemyss prove \cite[5.10]{DWncdf} that $\mm$ is a `noncommutative twist' around $A_\con$; more specifically they prove that $\mm$ is represented by the $A$-bimodule ${AeA \simeq \mathrm{cocone}(A \to A_\con)}$, in the sense that $\mm\cong\R\hom_A(AeA,-)$.

	\section{Deformation theory}
	For us, the contraction algebra has one more extremely important property: namely, that it controls the noncommutative deformations of the flopping curves. The idea of deformation theory, done geometrically, is to study extensions of structures along infinitesimal directions. Loosely, think of an infinitesimal $\mathcal I$ as a point together with some nonreduced fuzz infinitely close by. Given some geometric object $X$, a deformation of $X$ is then a family $\mathcal X$ over $\mathcal I$ whose fibre over the closed point is precisely $X$.
	
	\p In order to formalise this, we need first an algebraic description of infinitesimals. Firstly, they should have a unique closed point, and hence be spectra of commutative local rings. Secondly, we do not want arithmetic information; hence the residue field should be $k$. Thirdly, we only want a finite amount of nonreduced information. Hence our infinitesimals will be spectra of Artinian local commutative $k$-algebras $\Gamma$; those finite-dimensional algebras with unique maximal ideal $\mathfrak {m}_\Gamma$ whose residue field is $k$. This gives a canonical augmentation $\Gamma \to k$ and hence a canonical $k$-rational point of $\spec \Gamma$. If $X$ is a scheme then a deformation of $X$ is defined to be a scheme $\mathcal X$ with a flat map $\mathcal X \to \spec \Gamma$ together with an identification of the fibre $\mathcal{X} _ k$ with $X$. Say that two deformations are isomorphic if there is an isomorphism between them reducing to the identity on the copies of $X$. 
	
	\p A particularly important example of an Artinian algebra is the dual numbers $k[\epsilon]\coloneqq \frac{k[x]}{x^2}$, where we think of $\epsilon$ as a square-zero element. Deformations of objects over $k[\epsilon]$ are referred to as first-order deformations. For example, let $R=\frac{k[x,y]}{xy}$ be the coordinate ring of the coordinate axes in $\A_k^2$. For a fixed $t\in k$, it is easy to check that the $k[\epsilon]$-algebra $\frac{k[x,y]}{xy-t\epsilon}$ is a deformation of $R$; one can think of this as the infinitesimal first-order shred of the hyperbolic paraboloid $\frac{k[x,y]}{xy-t}$ fibred over $\A^1_k=\spec k[t]$, whose generic fibre is a hyperbola and whose special fibre is $R$. In fact, one can see fairly easily \cite[1.5]{artinnotes} that for a plane curve with isolated singularities $\frac{k[x,y]}{f}$, the deformations are in one-to-one correspondence with the elements of the Tjurina algebra $\frac{k[x,y]}{(f,f_x,f_y)}$; in particular if $f$ defines a smooth plane curve then it has no nontrivial deformations. Because the Tjurina algebra of $xy$ is one-dimensional, the family above exhausts all of the first-order deformations of $R$. 
	
	\p Based on the above example, one might guess that smooth affine schemes have no non-trivial deformations, and indeed this is true \cite{artindef}. The idea is that smoothness implies formal smoothness, which is a lifting property against square-zero extensions. Every Artinian algebra is a composition of square-zero extensions, which allows us to iteratively lift the identity map $R \to R$ to an isomorphism $\mathcal{R} \to R\otimes \Gamma$ between any deformation and the trivial one. Note that this argument only works in the affine world! Conceptually, one can think of the failure of smooth affines to admit interesting deformations as a kind of failure of compactness: any `kinks' one introduces can be smoothed away to infinity. Smooth non-affine varieties may admit nontrivial deformations: if $X$ is a smooth variety, choose a cover of $X$ by smooth affines. Any deformation of $X$ must be trivial on each piece of the cover, and so we are just deforming the gluing maps. One can check that automorphisms of the trivial first-order deformation of an affine patch $\spec R$ are in bijection with $\mathrm{Der}(R,R)$. These derivations patch together into a \v{C}ech 1-cocycle valued in $\mathrm{Der}(\mathcal{O}_X,\mathcal{O}_X)\cong\mathcal{T}_X$, and isomorphisms of deformations induce \v{C}ech coboundaries. One gets an isomorphism $$\frac{\{\text{first-order deformations of }X\}}{(\text{isomorphism})}\cong H^1(X, \mathcal{T}_X).$$So the tangent sheaf of a smooth variety controls its first-order deformations. See for example Sernesi \cite{sernesi} for the full argument.

	\p One can also think of deformation theory as the infinitesimal study of moduli spaces. Assume that the objects we are interested in fit together into some sort of moduli space $\mathcal M$, so that maps $Y \to \mathcal{M}$ correspond to flat families over $Y$. In particular if $\Gamma$ is an Artinian local algebra then a map $\spec \Gamma \to \mathcal{M}$ is the same as a closed point $p=[X]$ of $\mathcal M$ together with a flat family $\mathcal X$ over $\spec \Gamma$ reducing to $X$ at the unique closed point. But this is exactly the definition of a deformation of $X$. So deformation theory can be interpreted as poking around infinitesimally in a moduli space near a point of interest. Looking at first-order deformations, the set of maps $\spec k[\epsilon] \to \mathcal{M}$ taking the residue field $k$ to a point $p$ is well-known to be the tangent space $T_p\mathcal{M}$ to $\mathcal M$ at $p$. For example, let $X$ be a smooth projective variety and $Y$ a smooth closed subvariety. Then one can consider the first-order embedded deformations of $Y$ in $X$, and it turns out that this set is in bijection with the space of sections over $X$ of the normal bundle $\mathcal{N}_{Y/X}$. When $Y$ is a point, this is just the tangent space $T_YX$. Considering $Y$ as a point of the Hilbert scheme, we hence have $T_Y\mathrm{Hilb(X)}\cong H^0(X,\mathcal{N}_{Y/X})$. See e.g.\ \cite{hartshornedef} for an in-depth discussion. Hence, deformation theory can tell us interesting facts about moduli spaces, such as whether they are smooth.
	
	\p Deformations pull back along maps of commutative Artinian local rings, so fixing an object of interest $X$, the assignment $$\Gamma\mapsto \frac{\{\text{deformations of }X\text{ over }\Gamma\}}{(\text{isomorphism})}$$is a functor $\cat{Art}_k \to \cat{Set}$. We call it the deformation functor associated to $X$. Schlessinger \cite{schlessinger} axiomatised the properties of these functors, and moreover gave criteria for when such a functor is prorepresentable (i.e.\ representable by a pro-object in $\cat{Art}_k$). So at the formal level, one can study deformation functors abstractly as those satisfying Schlessinger's conditions (loosely, that deformations should glue appropriately, and that deformations over $\Gamma=k$ should be trivial). Let us look at some examples of deformation functors:
	\begin{itemize}
\item If $\Gamma$ is an Artinian local algebra, then $\hom(\Gamma,-)$ is a deformation functor. 
\item If $X$ is a scheme, then we obtain a deformation functor $\mathrm{Def}(X)$ as above. 
\item If $\mathcal F$ is a sheaf of modules on a scheme $X$, then we can look for sheaves $\tilde{\mathcal{F}}$ on the trivial deformation $X\otimes \Gamma$ which restrict to $\mathcal F$; this gives us a functor $\mathrm{Def}_X(\mathcal{F})$. 
\item In particular, we can also deform modules over (commutative) rings.
\item Similarly, one can deform dg modules too. 
\item If $A$ is a noncommutative $k$-algebra, then we can deform the multiplication on $A$ by looking for algebras $\tilde A$, flat over $\Gamma$, such that $\tilde{A}\otimes k \cong A$, and we get an associated functor $\mathrm{Def}(A)$. 
\item Similarly, if $A$ is commutative and we deform only to commutative algebras, we obtain a subfunctor $\mathrm{Def}_{\mathrm{Com}}(A)$.
	\end{itemize}
	
	\p Given a deformation functor $F:\cat{Art}_k\to \cat{Set}$, let $T^1F$ denote the set $F(k[\epsilon])$ obtained by evaluating $F$ on the dual numbers. Call it the tangent space of $F$; one can show that it is a vector space. Any map $F \to F'$ of deformation functors induces a map $T^1F \to T^1F'$, the differential. Let us compute the tangent spaces of some of the functors we saw earlier. Firstly, we already know that $T^1\mathrm{Def}(X)\cong H^1(X, \mathcal{T}_X)$ for smooth varieties; when $X$ is proper (and possibly singular) we are at least guaranteed that $T^1\mathrm{Def}(X)$ is finite-dimensional. It is easy to see that $T^1\hom(\Gamma,-)\cong T_{\mathfrak{m}_\Gamma}\spec\Gamma$. If $V$ is a bounded dg vector space, then $T^1\mathrm{Def}_k(V)$ is the space $\ext^1_k(V,V)\cong H^1\enn_k(V)$.
	
	\p Let us look at the example of a noncommutative algebra $A$ in more detail. Deform the multiplication on $A$ over $k[\epsilon]$ by putting $x\odot y \coloneqq  xy+\epsilon f(x,y)$, for some $f\in\hom_k(A,A)$. In order for $\odot$ to be associative, we require that the equation $$xf(y,z) -f(xy,z)+f(x,yz)-f(x,y)z=0$$holds. Two deformations are equivalent if they differ by an automorphism $x\mapsto x+ \epsilon g(x)$, and this sends $f(x,y) \mapsto f(x,y) -xg(y)+g(xy)-g(x)y$. Hence the set of first-order deformations is in bijection with the space of functions $f$ satisfying the equation above, modulo the given equivalence relation. The eagle-eyed will recognise $f$ as precisely a Hochschild 2-cochain, and $-xg(y)+g(xy)-g(x)y$ as precisely a Hochschild 2-coboundary. Hence, the first-order deformations of $A$ are in bijection with the Hochschild cohomology group $HH^2(A,A)$. Similarly, if we deform to commutative algebras, we see that $T^1\mathrm{Def}_{\mathrm{Com}}(A)$ is precisely the Harrison cohomology $\mathrm{HHar}^2(A,A)$, defined as the cohomology of the Harrison complex, the subcomplex of the Hochschild complex on those functions vanishing on signed sums of shuffles \cite{gerstschack}.
	
	\p Observe that first-order deformations tend to be in bijection with cohomology groups, in particular usually first cohomology groups. We will come back to this, and provide a conceptual explanation, later in this introduction when we enter the world of derived deformation theory.
	
	\p So far, we have deformed along commutative algebras. What happens if we try to deform along noncommutative algebras? The definition of a noncommutative Artinian algebra is much the same as that of a commutative one. It is hard to deform schemes, but we can still deform modules, in exactly the same way as before; this theory was originally developed by Laudal \cite{laudalpt}, who goes one step further and considers pointed deformations: roughly, these are deformations of collections $\{S_1,\ldots,S_n\}$ along $n$-pointed Artinian algebras, remembering some information about the Ext groups between pairwise distinct modules. Pointed noncommutative deformations have recently found many applications within algebraic geometry, in particular the geometry of threefolds \cite{todatwists,DWncdf,kawamatapointed}. 
	
	\p Returning to where we started this section, Donovan and Wemyss prove that the contraction algebra $A_\con$ controls the noncommutative deformation theory of the flopping curves. What do we mean by this? Recall that across the derived equivalence $D(X) \to D(A)$, the (twisted) structure sheaves of the flopping curves $C_i$ correspond across the derived equivalence to the 1-dimensional simple modules $S_i$ appearing in the quotient of the contraction algebra by its radical. One can check that the commutative pointed deformations of the simple sheaves $\mathcal{O}_{C_i}(-1)$ are in natural bijection with the commutative pointed deformations of the $A$-modules $S_i$. The contraction algebra, which is an $n$-pointed Artinian algebra, represents the functor of noncommutative pointed deformations of the $S_i$ \cite{DWncdf, contsdefs}. In this sense, we say that $A_\con$ represents the functor of noncommutative deformations of the curves $C_i$. Note that this implies that the abelianisation $A_\con^{\mathrm{ab}}$ represents the functor of commutative deformations of the $C_i$.

\section{The derived contraction algebra}\label{intronew}
Can one extend the definition of the contraction algebra to more general geometric situations than flopping contractions? Recall that the base of a flop is a cDV singularity, and hence a generic hyperplane cut is a partial resolution of a Kleinian singularity. Because tilting bundles behave well under such cuts (see \S\ref{slicingsctn}), partial resolutions obtained in this manner admit contraction algebras. Moreover, although one no longer has flops -- since a flop is an isomorphism in codimension one -- one still has mutation autoequivalences, so in some sense one can still flop curves on the derived level. However, by producing an infinite family of one-curve partial resolutions of type $A_n$ surface singularities, we show explicitly that the contraction algebra does not control the mutation-mutation autoequivalence (see \S\ref{intromut} later). Because the contraction algebra is $k$ for each member of this family, the exceptional locus (a copy of $\P^1$) is rigid and does not deform, even noncommutatively. However, one can show that this curve admits nontrivial derived deformations, indicating that one should study a derived version of the contraction algebra. 
	
\p Noncommutative (partial) resolutions, such as those constructed by Van den Bergh in \cite{vdb}, often yield (noncommutative) rings with idempotents, which motivates the serious homological study of such rings. In particular, if $A$ is a noncommutative ring with an idempotent $e$ then putting $R\coloneqq eAe$, the standard functors $D(A)\recol D(R)$ fit into one half of a recollement, a strong type of short exact sequence of triangulated categories. Kalck and Yang \cite{kalckyang,kalckyang2} show that there exists a nonpositive cohomologically graded dga $B$ with $H^0(B)\cong A/AeA$ fitting into a recollement $D(B)\recol D(A) \recol D(R)$. 

\p Suppose that $A$ is a noncommutative ring with an idempotent $e$, or more generally a dga with an idempotent $e \in H^0(A)$. Starting from this data, Braun, Chuang, and Lazarev \cite{bcl} define the derived quotient $\dq$ of $A$ by $e$ to be the universal dga under $A$ that homotopy annihilates $e$. In fact, the derived quotient is a special case of a more general construction, that of the derived localisation; if $A$ is commutative, or more generally if the localising set (in our case $\{1-e\}$) is an {\O}re set, then this coincides with the usual localisation. In general, the derived localisation contains extra homotopy-theoretic information: for example, the derived quotient can be thought of as a Drinfeld quotient (one takes $A$ and adds a contracting homotopy to kill $e$). In particular, if $A$ is an algebra in degree zero, then $\dq$ is a nonpositively cohomologically graded dga with $H^0(\dq)\cong A/AeA$. Braun--Chuang--Lazarev prove that the derived quotient fits into a standard recollement $D(\dq) \recol D(A)\recol D(eAe)$, which gives an abstract construction of Kalck and Yang's dga $B$.  
	
\p Given a suitably general isolated simple contraction $X \to X_\con$ of an irreducible rational curve to a point $p$, we will define the derived contraction algebra $\dca$ to be the derived quotient $\dq$, where $A$ is Van den Bergh's noncommutative model for $X$ (see Chapter \ref{dercondefns} for the rigorous construction). We construct $\dca$ complete locally; in particular it will only depend on the formal fibre $U \to \spec R$, where we let $R$ denote the completion of the local ring of $X_\con$ at $p$. It follows that $eAe\cong R$ here, so that one can think of $D(\dca)$ as a sort of `derived exceptional locus'. However, it can be computed using only global data, and we give conditions under which one can compute it in a Zariski local neighbourhood of $p$ (see \S\ref{ltgc}). Importantly, one can view the derived contraction algebra as an enhancement or categorification of the Donovan--Wemyss contraction algebra: it is easy to see that there is an algebra isomorphism $H^0(\dca)\cong A_\con$. 

\p The overall aim of this thesis is to use the derived contraction algebra to generalise Donovan--Wemyss's results \cite{DWncdf,contsdefs,enhancements} to non-threefold settings. For technical reasons (see \ref{dimrmk}) we are only able to define the derived contraction algebra for contractions of curves in threefolds and surfaces. However, even in the threefold setting, we obtain new proofs of Donovan--Wemyss's results and our theorems already extend theirs. We work in a general setup and only introduce geometric hypotheses when necessary, in order to obtain a unified theory that works in a variety of situations. The point of our construction is that it behaves well in a general setup: some of our more algebraic results will hold in arbitrary dimension.
	 	
	\p There are three properties one would like a derived contraction algebra to have:
	\begin{itemize}
		\item One would like a derived version of the Donovan--Wemyss conjecture to be true: namely, the derived contraction algebra should classify flops with complete local base.
		\item One would like the derived contraction algebra to control, in an appropriate sense, the derived noncommutative deformations of the exceptional locus.
		\item One would like $\dca$ to have control over an appropriate mutation autoequivalence.
		\end{itemize} 	
In this thesis, we show that the derived quotient possesses all three of these properties.

 \section{Singularity categories and the derived Donovan--Wemyss conjecture}
 \p If $R$ is a right noetherian ring, its singularity category is the triangulated (or dg) category given by the quotient $D_{\mathrm{sg}}(R)\coloneqq D^b(R)/\per(R)$, which can be seen as quantifying the type of singularities of $R$ (at least when $R$ is commutative). Singularity categories were introduced by Buchweitz \cite{buchweitz} who proved that, when $R$ is Gorenstein, $D_{\mathrm{sg}}(R)$ is equivalent to the stable category $\stab R$ of maximal Cohen--Macaulay (MCM for short) $R$-modules. Singularity categories were later studied for schemes by Orlov \cite{orlovtri}, who gave applications to homological mirror symmetry, and dg enhancements of singularity categories have been studied recently by Blanc--Robalo--To\"{e}n--Vezzosi \cite{motivicsingcat} where they are constructed using the dg quotient \cite{kellerdgquot,drinfeldquotient}. 
 	
\p When $R=k\llbracket x_1,\ldots, x_n\rrbracket/\sigma$ is a complete local isolated hypersurface singularity, the two triangulated categories $D_\mathrm{sg}(R)$ and $\stab R$ are triangle equivalent to a third category, the category of matrix factorisations $\mathrm{MF}(\sigma)$, a fact essentially first noticed by Eisenbud \cite{eisenbudper}. This latter dg category has a natural enhancement over $\Z/2$-graded complexes -- and hence becomes a dg category by extending the $\Z/2$-graded morphism complexes periodically -- and the triangle equivalence between $\mathrm{MF}(\sigma)$ and $D_\mathrm{sg}(R)$ lifts to a quasi-equivalence of $\Z$-graded dg categories \cite{motivicsingcat}.
 	
\p Dyckerhoff \cite{dyck} proved that the Hochschild cohomology of the 2-periodic dg category of matrix factorisations is the Milnor algebra of $R$. When $\sigma$ is quasi-homogeneous, this agrees with the Tjurina algebra of $R$, which can then be used to recover $R$ via the formal Mather--Yau theorem, as long as one knows the dimension \cite{gpmather}. Recently, a $\Z$-graded analogue of Dyckerhoff's theorem was proved by Hua and Keller \cite{huakeller}, who showed that the Tjurina algebra of $R$ is the zeroth Hochschild cohomology of the underlying $\Z$-graded dg category of matrix factorisations (which is a different object to the $\Z/2$-graded Hochschild cohomology). Hence, the dg singularity category of a complete local hypersurface singularity $R$ classifies $R$ amongst all hypersurface singularities of the same dimension. 
 	
\p How does this relate to the derived contraction algebra? Let $R$ be a commutative Gorenstein ring and $M$ a MCM $R$-module. Put $A\coloneqq \enn_R(R\oplus M)$; we refer to $A$ as a noncommutative partial resolution of $R$. Note that $A$ comes with an idempotent $e=\id_R$; what information does the derived quotient $\dq$ contain? Following Kalck and Yang \cite{kalckyang, kalckyang2}, we investigate the relationship between $\dq$ and the dg singularity category ${D^\mathrm{dg}_\mathrm{sg}(R)}$ in detail, and we prove a key technical theorem (\ref{qisolem}) stating that $\dq$ is the truncation to nonpositive degrees of the endomorphism dga of $M$ considered as an object of the dg singularity category ${D^\mathrm{dg}_\mathrm{sg}(R)}$. 	 	
 	
\p When $R$ is in addition a complete local hypersurface singularity, Eisenbud's 2-periodicity \cite{eisenbudper} lets us obtain a degree $-2$ element $\eta \in H^{-2}(\dq)$, homotopy unique up to multiplication by units in $H(\dq)$, with the property that $\eta:H^j(\dq) \to H^{j-2}(\dq)$ is an isomorphism for all $j\leq 0$. Morally, we think of $\eta$ as witnessing the 2-periodicity in the singularity category. We prove that the derived localisation of $\dq$ at $\eta$ is the whole of the endomorphism dga of $M \in {D^\mathrm{dg}_\mathrm{sg}(R)}$ (\ref{etaex}, \ref{uniqueeta}). Under some mild finiteness assumptions, $\dq$ hence determines the dg subcategory of ${D^\mathrm{dg}_\mathrm{sg}(R)}$ generated by $M$ (\ref{recovwk}). When $A$ is a resolution, $M$ automatically generates the singularity category (\ref{genrmk}). So, under some finiteness and smoothness conditions, the dga $\dq$ determines ${D^\mathrm{dg}_\mathrm{sg}(R)}$. Combining this with the recovery result of Hua and Keller, one obtains:

	\begin{thma}[\ref{recov}] Fix $n \in \N$ and let $R\coloneqq k\llbracket x_1,\ldots, x_n \rrbracket/\sigma$ be an isolated hypersurface singularity. Let $M$ be a maximal Cohen--Macaulay $R$-module that generates $D_\mathrm{sg}(R)$, and let $A\coloneqq \enn_{R}(R\oplus M)$ be the associated partial resolution of $R$ with $e=\id_{R}$. Assume that $A/^\mathbb{L}AeA$ is cohomologically locally finite and that $A/AeA$ is a local algebra. Then the quasi-isomorphism type of $\dq$ recovers the isomorphism type of $R$ amongst all complete local isolated hypersurface singularities of the same dimension as $R$.
\end{thma}
We conjecture (\ref{perconj}, \ref{perconjrmk}) that in fact $\dq$ determines the $\Z/2$-graded dg category of matrix factorisations.

\p Returning to the geometric setup, a generalised version of the derived Donovan--Wemyss conjecture immediately follows from the above theorem:
	\begin{thma}[derived Donovan--Wemyss conjecture; \ref{ddwconj}]\label{thrma}
	Let $X \to X_\con$ and $X' \to X_\con'$ be isolated contractions of an irreducible rational curve in smooth varieties of the same dimension to points $p,p'$ respectively. If the associated derived contraction algebras are defined and quasi-isomorphic, then the completions $\widehat{(X_\con)}_p$ and $\widehat{(X_\con')}_{p'}$ are isomorphic.
\end{thma}
The original conjecture would follow if one could prove that the usual contraction algebra determines the quasi-isomorphism type of the derived contraction algebra. We remark that in the multiple-curve case, the isomorphism class of the contraction algebra is too fine an invariant, as it distinguishes between different flops with the same base. Hence, one would expect that the quasi-isomorphism type of the derived contraction algebra is also too fine an invariant in this situation, and that one should use something like its derived Morita equivalence class instead.

\p In the setting of threefold flopping contractions, we will show that when $X \to X_\con$ is a minimal model, then $H^*(\dca)\cong A_\con[\eta]$, the classical contraction algebra with the periodicity element freely adjoined in degree $-2$. In general, $\dca$ is not formal: in Chapter \ref{threecomps} we provide explicit computations (as a minimal $A_\infty$-algebra) of the derived contraction algebra associated to the Pagoda and Laufer flops, and we will see that $\dca$ is not necessarily formal.
	
	\section{Derived deformation theory}
	\p We saw earlier that tangent spaces to deformation functors are typically cohomology groups. Moreover, with the exception of deforming algebras, they are first cohomology groups. Is there an explanation for this? The following philosophy is originally due to Deligne, first written down in a letter to Goldman and Millson of April 1986, and first appearing in print in \cite{goldmanmillson}:
	$$\textit{Every characteristic 0 deformation problem is controlled by a differential graded Lie algebra.}$$
	A dg Lie algebra (dgla) is essentially a complex of vector spaces endowed with a Lie bracket $[-,-]$ satisfying graded versions of the usual antisymmetry and Jacobi identities, as well as a Leibniz rule for compatibility with the differential. How does one construct a deformation functor from a dgla? Given a dgla $L$, a Maurer--Cartan element in $L$ is a degree 1 element $x$ that satisfies the equation $dx+\frac{1}{2}[x,x]=0$. If $L$ is nilpotent, then the degree zero part $L^0$ acts via formal exponentiation on the set of Maurer--Cartan elements. Given a general dgla $L$ and a commutative Artinian local algebra $\Gamma$, one can check that the tensor product $L\otimes \mathfrak{m}_\Gamma$ admits the structure of a nilpotent dgla. The deformation functor associated to $L$, denoted $\mathrm{Def}_L$, sends $\Gamma$ to the set of equivalence classes of Maurer--Cartan elements of $L\otimes \mathfrak{m}_\Gamma$. One can check that this is indeed a deformation functor! There is an isomorphism between the tangent space $T^1\mathrm{Def}_L $ and $ H^1(L)$, the first cohomology group of $L$. In fact, the functor $L \mapsto \mathrm{Def}_L$ is invariant under quasi-isomorphism. See e.g.\ \cite{manettidgla, manetti} for proofs of these statements. Let us see some examples of dgla deformation functors:
	\begin{itemize}
		\item If $V$ is a dg vector space, then deformations of $V$ are controlled by the dgla $\enn(V)$ equipped with the commutator bracket.
		\item If $A$ is a noncommutative algebra, then the Hochschild complex $HC(A,A)$ is well known to admit a bracket, the Gerstenhaber bracket. This makes the shifted complex $HC(A,A)[1]$ into a dgla, and this is precisely the dgla controlling algebra deformations of $A$ \cite{gerstschack}. In particular, this explains why the first-order deformations are $HH^2(A,A)$ rather than a first cohomology group.
		\item Similarly, the shifted Harrison complex is a dgla, and it controls commutative deformations.
		\item If $X$ is a complex manifold, then the deformations of $X$ are controlled by the {Kodaira--Spencer} dgla, which is built out of certain sections of the holomorphic tangent bundle; see \cite{manetti} for further information.
	\end{itemize}
	
	\p If $H^1L$ is the tangent space to the deformation functor associated to $L$, what do the other cohomology groups classify? Let us go back to the situation of a smooth projective variety $X$. Automorphisms of the trivial first-order deformation are in bijection with $H^0(X,\mathcal{T}_X)$. First-order deformations themselves are in bijection with $H^1(X,\mathcal{T}_X)$. Furthermore, $H^2(X,\mathcal{T}_X)$ is an obstruction space for $X$: given a deformation $\mathcal X$ over $k[\epsilon]$, one can ask whether it extends to a deformation over $\frac{k[x]}{x^3}$. It turns out that one can assign an obstruction class $\nu_{\mathcal X}\in H^2(X,\mathcal{T}_X)$ to $\mathcal X$ which vanishes if and only if the deformation lifts. This situation is mirrored in the more general theory: if $L$ is a dgla then $H^2(L)$ is an obstruction space for $\mathrm{Def}_L$. However, not all elements in $H^2(L)$ are necessarily in the image of the obstruction map, and one would like a deformation-theoretic interpretation of these classes.
	
	\p It turns out that those `non-geometric' classes in $H^2(L)$ actually measure derived deformations. The basic setup of derived deformation theory is similar to the underived setup: instead of our infinitesimals being modelled by commutative Artinian local rings, they are now modelled by Artinian local cdgas; loosely these are the cdgas of finite total dimension admitting a unique maximal ideal with residue field $k$. The first axiomatisation of deformation functors in this derived setting was done by Manetti \cite{manettiextended}. The Maurer--Cartan functor and deformation functor of a dgla also extend readily to the derived world, and they are deformation functors in the sense of Manetti.
	
	\p Let $R_i$ be the dual numbers, but now considered as an Artinian dga where $\epsilon$ has degree $i-1$; equivalently this is the square-zero extension of $k$ by $k[i-1]$. We recover the usual dual numbers as $R_1$. If $F:\cat{dgArt}_k \to \cat{Set}$ is a derived deformation functor, put $T^iF\coloneqq F(R_i)$ and call it the $i^{\text{th}}$ order tangent space to $F$. We can think of them as `higher obstruction spaces' measuring the failure of derived deformations to lift. In particular, $T^2F$ is precisely the correct derived obstruction space; see the introduction of \cite{luriedagx} for a detailed example. Importantly, one can show that if $L$ is a dgla then there is an isomorphism $T^i\mathrm{Def}_L\cong H^iL$. So we get a deformation-theoretic interpretation of all cohomology groups of $L$. The category of Manetti's deformation functors admits a model structure where the weak equivalences are the maps inducing isomorphisms on all tangent spaces, and one can show that the homotopy category is equivalent to the category of dglas localised at the quasi-isomorphisms. This provides a concrete realisation of Deligne's philosophy: up to homotopy, every derived deformation functor comes from a dgla.
	
	\p However, this is not the end of the story. Recall that deformation functors are defined as quotients of a set by an equivalence relation. A more homotopy-invariant way to talk about this would be to use groupoid-valued rather than set-valued functors, which corresponds to remembering some stacky data about our deformation problems. Even better, one can take functors valued in simplicial sets, which model higher stacks. Deformation functors defined on nonpositively graded Artinian dgas and valued in simplicial sets are known as `formal derived stacks' or `derived moduli problems', and the higher version of the deformation functor--dgla correspondence is known as the Lurie--Pridham correspondence:
	
	\begin{thma}[\cite{unifying, luriedagx}]
		The $\infty$-category of derived moduli problems is equivalent to the $\infty$-category of dglas.
	\end{thma}
We remark that Lurie works with $E_\infty$-algebras whereas Pridham uses simplicial algebras. The equivalence between derived moduli problems and dglas is more or less given by the Koszul duality between the commutative and Lie operads. To see why this is at least plausible, note that Maurer--Cartan elements of a tensor product $L\otimes\Gamma$ are roughly the same thing as twisting morphisms $ \Gamma^* \to L$. The idea that the equivalence between formal derived stacks and dglas is a manifestation of Koszul duality goes back at least to Hinich, who interpreted dg coalgebras as coalgebras of distributions on formal stacks \cite{hinich}. The intuition behind Lurie's proof is loosely that given a formal moduli problem $X$, one can take its loop space $\Omega X$, which is also a formal moduli problem and moreover admits a group structure. Its tangent complex $T_{\Omega X}$ hence behaves like the tangent space to a derived Lie group, i.e.\ a dgla. But because the tangent complex functor commutes with finite limits, we have $T_{\Omega X}\cong T_X[-1]$. Hence, the tangent complex to a formal stack is a shifted dgla. When $X$ is represented by a genuine augmented algebra $A \to k$, then the tangent complex of $X$ is naturally identified with the complex of derived derivations $\R\mathrm{Der}_k(A,k)$ computing Andr\'e--Quillen cohomology, which is known to be a shifted $L_\infty$-algebra \cite{luriedagx, quillender, sstangent}.

\p How, if at all, does the Lurie--Pridham view on derived deformation theory adapt to the noncommutative world? In view of the Koszul self-duality of the associative algebra operad, one would expect noncommutative dgas to control noncommutative derived deformation problems, and Lurie also provides a noncommutative version of the correspondence. One obtains the commutative correspondence by abelianisation, which on the level of representing objects corresponds to viewing a dga as a dgla via the commutator bracket. 

\p We remark that in positive characteristic, Pridham's formulation of the correspondence no longer holds: simplicial commutative algebras are no longer Quillen equivalent to nonpositive cdgas. However, noncommutative deformation theory is more well behaved: simplicial algebras are Quillen equivalent to nonpositive dgas in all characteristics. Hence, if one wants to do deformation theory in positive characteristic, it is natural to consider noncommutative deformations. Indeed, nearly all of our deformation-theoretic results will hold in positive characteristic.

\p Typically, one is interested in deforming objects of derived or homotopy categories, which has been studied in detail by Efimov, Lunts, and Orlov \cite{ELO,ELO2,ELO3}. In the smooth setting, their results are essentially enough to let us prove that, when $X\to X_\con$ is an isolated contraction of an irreducible rational curve in a smooth variety, then its derived contraction algebra controls the noncommutative deformations of the contracted curve.

\p However, we run into an issue: we would like to remove the smoothness hypothesis, whereas the results of Efimov--Lunts--Orlov really seem to require it. We must take a different route. Let us recall what we are trying to do: suppose that $A$ is an algebra, and $e\in A$ an idempotent such that $A/AeA$ is Artinian local. Let $S$ be the unique simple 1-dimensional $A/AeA$-module. The results of Chapter 4 let us identify the prorepresenting object of the functor of derived noncommutative deformations of $S$: it is precisely the continuous Koszul dual $B^\sharp\R\enn_A(S)$, which is a pro-Artinian (a.k.a.\ pseudocompact) dga. Recall that if $E\to k$ is an augmented dga, then the bar construction on $E$ is the tensor coalgebra on the shifted augmentation ideal, equipped with a new differential. The bar construction is a dg coalgebra, and hence its linear dual $E^!$, the Koszul dual, is a dga. Because every coalgebra is the colimit of its finite-dimensional subcoalgebras, one can view $BE$ as an ind-coalgebra; dualising this ind-coalgebra levelwise yields the pro-Artinian algebra $B^\sharp E$ which we refer to as the continuous Koszul dual of $E$. We have $\varprojlim B^\sharp E \simeq E^!$.

\p One can show relatively easily that $\dq^{!!} \simeq \R\enn_{A}(S)^!$. It in fact follows that $\dq^{!!}$ controls the deformations of $S$; although $\dq^{!!}$ does not prorepresent on the nose (as it is not pro-Artinian!) it does at least determine the prorepresenting object. For a rigorous statement of this, see \S5.6. Hence, the main difficulty lies in proving that $\dq$ is quasi-isomorphic to its own double Koszul dual $\dq^{!!}$. In the smooth setting, this is reasonably standard, but when $\ext_A(S,S)$ is unbounded, this is harder. Importantly, this is not automatic: if one remembers that $\dq^!$ is pro-Artinian, and takes the Koszul dual levelwise, then one ends up with a pro-Artinian algebra whose limit is $\dq$. However, we are forgetting the pro-Artinian structure on $\dq^!$ to obtain a plain dga, and then taking the Koszul dual again. It is far from clear that this new dga should be quasi-isomorphic to $\dq$ (see also \ref{parmk}).

\p However, we manage to show that if $B$ is a nonpositively graded cohomologically locally finite dga such that $H^0(B)$ is local, then $B$ is quasi-isomorphic to its double Koszul dual (\ref{kdfin}), recalling that we call a dga $B$ cohomologically locally finite if each $H^j(B)$ is a finite-dimensional $k$-vector space. The meat of this is a strictification result: we prove that certain homotopy pro-Artinian dgas $B$ (those with finite-dimensional cohomology in each degree with $H^0B$ Artinian local; the pro-structure is given by the Postnikov tower) are in fact genuinely pro-Artinian (\ref{goodchar}). This allows us to prove our next main theorem:

	\begin{thma}[\ref{maindefmthm}]Let $A$ be a $k$-algebra and $e \in A$ an idempotent. Suppose that $A/AeA$ is a local algebra and that $\dq$ is cohomologically locally finite. Let $S$ be $A/AeA$ modulo its radical, regarded as a right $A$-module. Then $\dq$ is quasi-isomorphic to a pro-Artinian algebra which prorepresents the functor of framed noncommutative derived deformations of $S$. In particular, $\dq$ determines the deformation functor.
\end{thma}
This can be regarded as generalising some of the prorepresentability results of \cite{ELO2} to the singular setting, or alternately as generalising some of the results of Segal's thesis \cite{segaldefpt}. A framed deformation of $S$ is essentially a deformation of $S$ that respects a fixed choice of isomorphism $S \cong k$. We use framings to rigidify slightly; if one works with just set-valued functors as opposed to $\sset$-valued functors then it is unimportant whether one uses framed or unframed deformations (see e.g.\ the proof of \ref{subdefmthm}). We note that cohomological local finiteness of $\dq$ can be checked explicitly when $A$ is presented as the path algebra of a quiver with relations (\ref{derquotcohom}, \ref{h1finite}) and is in general a quite weak finiteness condition. We also remark that the condition that $A/AeA$ is local can probably be dropped if one uses pointed deformations, as in Laudal \cite{laudalpt} or Kawamata \cite{kawamatapointed}.

\p Once again, our theorem about the derived quotient specialises to the geometric setting, and we obtain the following:
	\begin{thma}[\ref{dcaprorep}]\label{thrmb}
	Let $X \to X_\con$ be an isolated contraction of an irreducible rational curve $C$ in a surface or threefold. Then $\dca$ is quasi-isomorphic to a pro-Artinian algebra that prorepresents the functor of derived noncommutative deformations of $C$. In particular, $\dca$ determines the deformation functor.
\end{thma}
We note that this provides a new proof, via the inclusion-truncation adjunction, that $A_\con$ represents the underived noncommutative deformations of $C$. We also remark that, as above, a similar theorem ought to hold in the non-simple (equivalently, pointed) case, when $C$ is a not necessarily irreducible chain of curves.
	
\p The deformation-theoretic interpretation of $\dca$ allows us to establish local-to-global arguments on computing the derived contraction algebra of a contraction with non-affine base, and also allows us to compute $\dca$ as a minimal $A_\infty$-algebra via Koszul duality. We provide some explicit computations of derived contraction algebras for Pagoda flops, as well as one-curve partial resolutions of $A_n$ singularities.

\section{The mutation-mutation autoequivalence}\label{intromut}
Let $X \dashrightarrow X^+$ be a simple flop between threefolds. This induces a Bridgeland--Chen flop functor $D^b(X)\xrightarrow{\cong} D^b(X^+)$. On the algebraic side, this corresponds to a mutation equivalence $D^b(A) \to D^b(A^+)$, where $A^+$ is the mutation of $A$. The basic idea of mutation is as follows: given an endomorphism ring of the form $A=\enn_R(R\oplus M)$, one replaces the summand $M$ by its syzygy $\Omega M$ to obtain a new ring $A^+\coloneqq \enn_R(R\oplus \Omega M)$. Under some hypotheses on $R$ and $M$, mutation induces a derived equivalence between $A$ and $A^+$, and this is precisely the noncommutative analogue of the Bridgeland--Chen functor, in the sense that flopping then tilting is naturally isomorphic to tilting then mutating \cite{DWncdf}.

\p Because $X^{++}$ is isomorphic (over the base $X_\con$) to $X$, composing flop functors gives an autoequivalence of $D^b(X)$. It turns out that this is nontrivial -- for example, if the flop is simple with irreducible flopping curve $C$, then this autoequivalence shifts the sheaf $\mathcal{O}_{C}(-1)$ by 2. On the algebraic side, the flop-flop autoequivalence corresponds to a mutation-mutation autoequivalence $\mm$ of $D^b(A)$. In the threefold setting, Donovan and Wemyss prove that $\mm$ is a `noncommutative twist' around the contraction algebra, in the sense that the cocone of the natural map $A \to A_\con$ represents $\mm$. Note that because $A\to A_\con$ is a surjection, one has a quasi-isomorphism $\mathrm{cocone}(A\to A_\con)\simeq \ker(A\to A_\con)\cong AeA$.

\p In certain non-threefold settings -- in particular, when $X \to \spec R$ is a partial resolution of a Kleinian singularity -- one can still define mutation. So although one cannot flop curves because flops are isomorphisms in codimension one, mutation gives us a derived analogue of flops. Moreover, the mutation-mutation autoequivalence is still nontrivial. Does the contraction algebra still control the mutation-mutation equivalence via noncommutative twists in this more general setting?

\p The answer is no. In fact, we show something stronger: for all partial resolutions of Kleinian singularities, the ideal $AeA$ never represents $\mm$ (\ref{mtnrmk}). However, our third main theorem shows that $\dca$ does control $\mm$, in the following analogous sense:
	\begin{thma}[\ref{mutncontrol}]\label{thrmc}
		Let $X \to \spec R$ be either a threefold simple flopping contraction to a complete local base, or a cut of such a contraction to a one-curve partial resolution of a Kleinian singularity. Let $A_\mm\coloneqq  \tau_{\geq -1}(\dca)$ be the two-term truncation of the associated derived contraction algebra. Then $\mm$ is a `noncommutative twist' around $A_\mm$, in the sense that $\mm$ is represented by the $A$-bimodule $\mathrm{cocone}(A \to A_\mm)$.
	\end{thma}
	The twist interpretation comes from the fact that one has an exact triangle $$\R\hom_A(A_\mm,-) \to \id \to \mm \to$$ of endofunctors of $D(A)$ (\ref{twistrmk}). The crux of the proof is the statement that $\mm$ restricts to the shift functor $[-2]$ on $D(\dca)$, and the proof of this second fact makes crucial use of the recollement $D(\dca)\recol D(A) \recol D(R)$ to reduce to a calculation in the singularity category of $R$.
	
	\p One can show that $\mathrm{cocone}(A \to A_\mm)$ is still a module in this setting, and in fact is an extension of $AeA$ by $H^{-1}(\dca)\cong \ext_R^1(M,M)$. In the threefold setting, $\dca$ has no cohomology in degree $-1$, and hence $A_\mm\simeq A_\con$, which gives a new proof of Donovan and Wemyss's result. In the surface setting, Auslander--Reiten duality allows one to conclude that $A_\mm$ always has cohomology in degree -1, and hence is never $A_\con$. Note that, by the inclusion-truncation adjunction, $A_\mm$ represents the functor of $[-1,0]$-truncated derived noncommutative deformations of $S$. One can informally think of $A_\mm$ as $\dca/\eta$, the quotient of the derived contraction algebra by the periodicity element $\eta \in H^{-2}(\dca)$.

	\section{Organisation of the thesis}
	Part I of this thesis is primarily concerned with ideas from homotopical algebra. In Chapter 2 we provide some background on dg categories and $A_\infty$-algebras. We also take the opportunity to set conventions about dg objects. In Chapter 3 we introduce the category of noncommutative pro-Artinian dgas, and prove a Koszul duality result for suitably finite dgas (\ref{kdfin}), which can be viewed as a strictification result. In Chapter 4 we recall some noncommutative derived deformation theory, and prove that the functor of framed deformations of a simple module is prorepresentable by the Koszul dual of the controlling dga (\ref{prorepfrm}). We also investigate prodeformations, and in particular the universal prodeformation.
	
	\p Part II concerns the interplay between Braun--Chuang--Lazarev's derived quotient \cite{bcl} and dg singularity categories. In Chapter 5 we recall some facts about the derived quotient. Using our results from earlier, we show that, under suitable finiteness conditions, the derived quotient admits a deformation-theoretic interpretation. In Chapter 6 we introduce some key concepts from singularity theory, with a focus on singularity categories. We review Hua and Keller's recent result on recovering a hypersurface singularity from its dg singularity category \cite{huakeller}. In Chapter 7, we combine these results to prove some theorems about derived quotients of partial resolutions of commutative Gorenstein rings. We prove some structure theorems before proving an algebraic version of the derived Donovan--Wemyss conjecture.
	
	\p Part III applies the ideas of Part II to geometric situations. In Chapter 8, we introduce the derived contraction algebra associated to a suitably general isolated contraction of an irreducible rational curve $C$. We can immediately deduce that the derived contraction algebra controls the noncommutative derived deformations of $C$, which we use to prove some local-to-global results on computing the derived contraction algebra. We can also immediately deduce the derived Donovan--Wemyss conjecture. In Chapter 9, we use $A_\infty$ methods to compute the derived contraction algebras of both Pagoda flops and of a certain family of one-curve partial resolutions of $A_n$ singularities, and we also sketch a computation for the Laufer flop. In Chapter 10, we consider the mutation-mutation autoequivalence in a general setup and prove that $\mm$ is a noncommutative twist about $A_\mm$.

			\section{Notation and conventions}
	Throughout this thesis, $k$ will denote an algebraically closed field of characteristic zero. Many of our theorems are true in positive characteristic, or even if one drops the algebraically closed assumption, and we will try to indicate where this holds. Modules are right modules, unless stated otherwise. Consequently, noetherian means right noetherian, global dimension means right global dimension, et cetera. Unadorned tensor products are by default over $k$. We denote isomorphisms (of modules, functors, \ldots) with $\cong$ and weak equivalences with $\simeq$.
	
	\p We freely use terminology and results from the theory of model categories; see \cite{quillenHA, hovey, dwyerspalinski, riehlCHT} for references. We will in particular assume that the reader knows the basics of the homotopy theory of simplicial sets, and that a model category admits derived mapping complexes which are (weak equivalence classes of) simplicial sets \cite[5.4.9]{hovey}. See \cite{goerssjardine} for a comprehensive textbook account of simplicial homotopy theory.
	
	\p We use cohomological grading conventions, so that the differential of a complex has degree $1$. 	If we refer to an object as just \textbf{graded}, then by convention we mean that it is $\Z$-graded. We will be explicit about any gradings by different groups, and in particular by $\Z/2$. If $X$ is a complex, we will denote its cohomology complex by $H(X)$ or just $HX$. If $X$ is a complex, let $X[i]$ denote `$X$ shifted left $i$ times': the complex with $X[i]^j=X^{i+j}$ and the same differential as $X$, but twisted by a sign of $(-1)^i$. This sign flip can be worked out using the \textbf{Koszul sign rule}: when an object of degree $p$ moves past an object of degree $q$, one should introduce a factor of $(-1)^{pq}$. If $x$ is a homogeneous element of a complex of modules, we denote its degree by $|x|$.

\p Let $V$ be a complex of vector spaces. The \textbf{total dimension} or just \textbf{dimension} of $V$ is $\sum_{n\in \Z}\mathrm{dim}_k V^n$. Say that $V$ is \textbf{finite-dimensional} or just \textbf{finite} if its total dimension is finite. Say that $V$ is \textbf{locally finite} if each $\mathrm{dim}_kV^n$ is finite. Say that $V$ is \textbf{cohomologically locally finite} if the cohomology dg vector space $HV$ is locally finite. Say that $V$ is \textbf{bounded} if $V^n$ vanishes for all but finitely many $n$, \textbf{bounded above} if $V^n$ vanishes for all $n\gg 0$, and \textbf{bounded below} if $V^n$ vanishes for all $n \ll 0$. There are obvious implications $$\text{finite }\implies\text{ locally finite }\implies\text{cohomologically locally finite}.$$We use the same terminology in the case that $V$ admits extra structure.
	
	\p Recall that a $k$-algebra is a $k$-vector space with an associative unital $k$-bilinear multiplication. In other words, this is a monoid inside the monoidal category $(\cat{Vect}_k, \otimes)$. Similarly, a \textbf{differential graded algebra} (\textbf{dga} for short) over $k$ is defined as a monoid inside the category of chain complexes of vector spaces. More concretely, a dga is a complex of $k$-vector spaces $A$ with an associative unital chain map $\mu:A\otimes A \to A$, which we refer to as the multiplication. Note that the condition that $\mu$ be a chain map forces the differential to be a derivation for $\mu$. Recall that the category of dgas is a model category with weak equivalences the quasi-isomorphisms and fibrations the levelwise surjections \cite{hinichhom}.
	
	\p A $k$-algebra is equivalently a dga concentrated in degree zero, and a graded $k$-algebra is equivalently a dga with zero differential. We will sometimes refer to $k$-algebras as \textbf{ungraded algebras} to emphasise that they should be considered as dgas concentrated in degree zero. A dga is \textbf{graded-commutative} or just \textbf{commutative} or a \textbf{cdga} if all graded commutator brackets $[x,y]=xy-(-1)^{|x||y|}yx$ vanish. Commutative polynomial algebras are denoted with square brackets $k[x_1,\ldots, x_n]$ whereas noncommutative polynomial algebras are denoted with angle brackets $k\langle x_1,\ldots, x_n\rangle$. Note that in a cdga, even degree elements behave like elements of a symmetric algebra, whereas odd degree elements behave like elements of an exterior algebra: in particular, odd degree elements are square-zero since they must commute with themselves.
	
	\p A \textbf{dg module} (or just a \textbf{module}) over a dga $A$ is a complex of vector spaces $M$ together with an action map $M \otimes A \to M$ satisfying the obvious identities (equivalently, a dga map $A \to \enn_k(M)$). Note that a dg module over an ungraded ring is exactly a complex of modules. Just as for modules over a ring, the category of dg modules is a closed monoidal abelian category. If $A$ is an algebra, write $\cat{Mod}\text{-}A$ for its category of right modules and $\cat{mod}\text{-}A\subseteq\cat{Mod}\text{-}A$ for its category of finitely generated modules.
	
		\p We will assume that the reader has a good familiarity with the theory of triangulated and derived categories; see \cite{neemanloc} and \cite{weibel, huybrechts} respectively for references. In particular we will make use of the fact that the derived category of a dga is the homotopy category of a model category. By convention we use the projective model structure on dg-modules where every object is fibrant, and over a ring the cofibrant complexes are precisely the perfect complexes (this is the `q-model structure' of \cite{sixmodels}).

	\part{Homotopical algebra}

	\chapter{DG categories and $A_\infty$-algebras}\label{dcdgcc}

	In this short review chapter we provide some homotopical preliminaries. Nearly all of the material we cover is standard, but we will need some specific facts about dg categories for Chapters 6 and 7, facts about coalgebras and the bar/cobar constructions for Chapters 3 and 4, and facts about minimal models of $A_\infty$-algebras for our computations in Chapter 9.

	\section{DG categories}\label{dgcats}
Viewing derived categories as mere triangulated categories does not quite suffice for some geometric purposes. For example, they lack limits and colimits, and are non-local in nature (see \cite[2.2]{toendglectures} for some good examples). We will see that the derived category of a dga admits extra structure, that of a {differential graded category} (or just {dg category}) fixing some of these problems. Survey articles on dg categories include \cite{toendglectures} and \cite{keller}.

	\begin{defn}
		A ($k$-linear) \textbf{dg category} is a category $\mathcal{C}$ enriched over the monoidal category $(\cat{dgvect}_k,\otimes)$ of dg vector spaces with the usual tensor product. In other words, to every pair of elements $(x,y)\in\mathcal{C}^2$ we assign a chain complex $\dgh_\mathcal{C}(x,y)$, to every triple $(x,y,z)$ we assign a chain map $\mu_{xyz}:\dgh_\mathcal{C}(x,y)\otimes \dgh_\mathcal{C}(y,z) \to \dgh_\mathcal{C}(x,z)$ satisfying associativity, and for every $x \in \mathcal{C}$ we assign a map $\eta_x: k \to \dgh_\mathcal{C}(x,x)$ which is a unit with respect to composition.
	\end{defn}
	Note in particular that for any object $x \in\mathcal{C}$, the complex $\dge_{\mathcal{C}}(x)\coloneqq \dgh_\mathcal{C}(x,x)$ naturally has the structure of a (unital) dga. We will frequently omit the subscript $\mathcal{C}$ if the context is clear.
	\begin{rmk}
		A more usual notation for the enriched hom is $\underline{\hom}$. We will not use this since it risks confusion with the standard notation used for homsets in the stable category of a ring, which we will use later in this thesis.
	\end{rmk}
	\begin{defn}
		A \textbf{dg functor} $F:\mathcal{C}\to\mathcal{D}$ between two dg categories is an enriched functor; i.e.\ a map of objects $\mathcal{C}\to \mathcal{D}$ together with, for every pair $(x,y)\in\mathcal{C}^2$, a map of complexes \linebreak $F_{xy}:\dgh_\mathcal{C}(x,y) \to \dgh_\mathcal{D}(Fx,Fy)$. We require that $F$ satisfies the associativity condition $\mu_{Fx\ Fy \ Fz}\circ (F_{xy}\otimes F_{yz}) = F_{xz}\circ \mu_{xyz}$ and the unitality condition $F_{xx}\circ \eta_x = \eta_{Fx}$.
	\end{defn}
	In particular, a dg functor $F:\mathcal{C}\to\mathcal{D}$ induces dga morphisms $F_{xx}:\dge_{\mathcal{C}}(x)\to\dge_{\mathcal{D}}(Fx)$ for every $x \in \mathcal{C}$.
	\begin{defn}
		Let $\mathcal{C}$ be a dg category. The \textbf{homotopy category} of $\mathcal{C}$ is the $k$-linear category $[\mathcal{C}]$ whose objects are the same as $\mathcal{C}$ and whose hom-spaces are given by\linebreak $\hom_{[\mathcal{C}]}(x,y)\coloneqq H^0(\dgh_\mathcal{C}(x,y))$. Composition is inherited from $\mathcal{C}$. We sometimes write $[x,y]\coloneqq \hom_{[\mathcal{C}]}(x,y)$.
	\end{defn}
	\begin{defn}Let $F:\mathcal{C}\to\mathcal{D}$ be a dg functor.\begin{itemize}
			\item $F$ is \textbf{quasi-fully faithful} if all of its components $F_{xy}$ are quasi-isomorphisms.
			\item $F$ is \textbf{quasi-essentially surjective} if the induced functor $[F]:[\mathcal{C}]\to[\mathcal{D}]$ is essentially surjective.
			\item $F$ is a \textbf{quasi-equivalence} if it is quasi-fully faithful and quasi-essentially surjective.
		\end{itemize}
	\end{defn}
	In a dg category, one may define shifts and mapping cones via the Yoneda embedding into the category of modules. This is equivalent to defining them as representing objects of the appropriate functors; e.g.\ $x[1]$ should represent $\dgh(x,-)[-1]$.
	\begin{defn}
		Say that a dg category is \textbf{pretriangulated} if it contains a zero object and is closed under shifts and mapping cones. 
	\end{defn}If $\mathcal{C}$ is pretriangulated then the homotopy category $[\mathcal{C}]$ is canonically triangulated, with translation functor given by the shift. We list some standard pretriangulated dg categories:
	\begin{defn}\label{dgcatlist}
		If $A$ is a dga, then $D_\mathrm{dg}(A)$ is the dg category of cofibrant dg modules over $A$, and $\cat{per}_\mathrm{dg}(A) \subseteq {D}_\mathrm{dg}(A)$ is the dg subcategory on compact objects. In addition, if $A$ is a $k$-algebra then $D^b_\mathrm{dg}(A)$ denotes the dg category of cofibrant dg $A$-modules with bounded cohomology; these are precisely the bounded above complexes of projective $A$-modules with bounded cohomology.
	\end{defn}
	All of the above dg categories are pretriangulated. In the notation of \cite{toendglectures}, $\cat{per}_\mathrm{dg}(A)$ is $\hat{A}_{\text{pe}}$. One has equivalences of triangulated categories $[D_\mathrm{dg}(A)]\cong D(A)$, ${[D^b_\mathrm{dg}(A)]\cong D^b(A)}$ and $[\cat{per}_\mathrm{dg}(A)]\cong\cat{per}(A)$, via standard arguments about dg model categories. Note that in the dg categories above, $\dgh$ is a model for the derived hom $\R\hom$; we will implicitly use this fact often.
 \p One can invert quasi-equivalences between dg categories: 
\begin{thm}[Tabuada \cite{tabuadamodel}]\label{tabmod}
	The category of all small dg categories admits a (cofibrantly generated) model structure where the weak equivalences are the quasi-equivalences. The fibrations are the objectwise levelwise surjections that lift isomorphisms. Every dg category is fibrant.
	\end{thm}
	See \cite{tabuadamodel} for a more precise description of the model structure. The advantage of this result is that it gives one good control over $\mathrm{Hqe}$, the category of dg categories localised at the quasi-equivalences.

	\section{DG quotients}\label{dgquotsctn}
	Later in this thesis, we will want to take quotients of triangulated categories by triangulated subcategories: for example if $A$ is a (noncommutative) ring, then the singularity category of $A$ is the Verdier quotient $D^b(A)/\cat{per}(A)$. One can also take dg quotients of dg categories; these were first considered by Keller \cite{kellerdgquot} and an explicit construction using ind-categories was given by Drinfeld \cite{drinfeldquotient}, which we recall in this section. We begin with the definition of ind-categories; we define them in terms of pro-categories, which we will use later in Chapter \ref{kd}.
	
	\begin{defn}[e.g.\ {\cite[\S6]{kashschap}}]\label{procats}
		Let $\mathcal{C}$ be a category. A \textbf{pro-object} in $\mathcal{C}$ is a formal cofiltered limit, i.e.\ a diagram $J \to \mathcal{C}$ where $J$ is a small cofiltered category. We denote such a pro-object by $\{C_j\}_{j \in J}$. The category of pro-objects $\cat{pro}\mathcal{C}$ has morphisms $$\hom_{\cat{pro}\mathcal{C}}(\{C_i\}_{i \in I}, \{D_j\}_{j \in J}) \coloneqq  \varprojlim_j \varinjlim_i \hom_{\mathcal{C}}(C_i,D_j).$$
	\end{defn}
	If $\mathcal{C}$ has cofiltered limits, then there is a `realisation' functor $\varprojlim: \cat{pro}\mathcal{C} \to \mathcal{C}$. If $C$ is a constant pro-object, then it is easy to see that one has $\hom_{\cat{pro}\mathcal{C}}(C, \{D_j\}_{j \in J}) \cong  \hom_{\mathcal{C}}(C,\varprojlim_jD_j)$.
	\begin{defn}
		Let $\mathcal{C}$ be a category. The \textbf{ind-category} of $\mathcal{C}$ is $\cat{ind}\mathcal{C}\coloneqq \cat{pro}(\mathcal{C}^\text{op})^\text{op}$. Less abstractly, an object of $\cat{ind}\mathcal{C}$ is a formal filtered colimit $J \to \mathcal{C}$, and the morphisms are $$\hom_{\cat{ind}\mathcal{C}}(\{C_i\}_{i \in I}, \{D_j\}_{j \in J}) \coloneqq  \varprojlim_i \varinjlim_j \hom_{\mathcal{C}}(C_i,D_j)$$
	\end{defn}
	If $\mathcal{C}$ has filtered colimits, then there is a `realisation' functor $\varinjlim: \cat{ind}\mathcal{C} \to \mathcal{C}$. In this situation, if $D \in \mathcal{C}$ is a constant ind-object then one has $\hom_{\cat{ind}\mathcal{C}}(\{C_i\}_{i \in I}, D)\cong \hom_\mathcal{C}( \varinjlim_iC_i,D)$. 
	
	\p Note that if $\mathcal{C}$ is a dg category then so are $\cat{pro}\mathcal{C}$ and $\cat{ind}\mathcal{C}$ in a natural way.
	\begin{defn}[Drinfeld \cite{drinfeldquotient}]
		Let $\mathcal{A}$ be a dg category and $\mathcal{B}\into\mathcal A$ a full dg subcategory. The \textbf{Drinfeld quotient} $\mathcal{A}/\mathcal{B}$ is the subcategory of $\cat{ind}\mathcal{A}$ on those $X$ such that:\begin{enumerate}
			\item $\dgh_{\cat{ind}\mathcal{A}}(\mathcal{B},X)$ is acyclic.
			\item There exists $a \in \mathcal{A}$ and a map $f:a \to X$ with $\mathrm{cone}(f)\in \cat{ind}\mathcal{B}$.
		\end{enumerate}
		Since $\cat{ind}\mathcal{A}$ is a dg category, so is $\mathcal{A}/\mathcal{B}$. The Drinfeld quotient is a model for ``the'' dg quotient:
	\begin{thm}[{\cite[4.02]{tabquot}}]
		Let $\mathcal{A}$ be a dg category and $i:\mathcal{B}\into \mathcal{A}$ a full dg subcategory. Then the quotient $\mathcal{A}/\mathcal{B}$ is the homotopy cofibre of $i$, taken in $\mathrm{Hqe}$.
	\end{thm}
With this in mind, we will use the terms `Drinfeld quotient' and `dg quotient' interchangeably, although the careful reader should keep in mind that the former is merely a model for the latter, which exists only in a homotopical sense. The Drinfeld quotient is a dg enhancement of the Verdier quotient:
\end{defn}
\begin{thm}[{\cite[3.4]{drinfeldquotient}}]\label{drinfeldpretr}
Let $\mathcal{A}$ be a pretriangulated dg category and $\mathcal{B}\into \mathcal{A}$ a full pretriangulated dg subcategory. Then there is a triangle equivalence $[\mathcal{A}/\mathcal{B}]\cong [\mathcal{A}]/[\mathcal{B}]$.
\end{thm}
	\section{$A_\infty$-algebras}

	We collect some material about $A_\infty$-algebras, which can be thought of as `dgas up to homotopy'. Indeed, there is a model structure on the category of $A_\infty$-algebras making it Quillen equivalent to the category of dgas. For the rest of this chapter we work over a field $k$; in all of our applications $k$ will be algebraically closed and characteristic zero but one does not need either of these hypotheses. We will broadly follow the treatment of Keller in \cite{kellerainfty}.
	\begin{defn}
		An $A_\infty$-algebra over $k$ is a graded $k$-vector space $A$ together with, for each $n\geq 1$, a $k$-linear map $m_n: A^{\otimes n} \to A$ of degree $2-n$ satisfying for all $n$ the coherence equations (or the \textbf{Stasheff identities}) $$\mathrm{St}_n:\quad \sum{(-1)^{r+st}}m_{r+1+t}(1^{\otimes r} \otimes m_s \otimes 1^{\otimes t})=0$$ where $1$ indicates the identity map, the sum runs over decompositions $n=r+s+t$, and all tensor products are over $k$. We are following the sign conventions of \cite{getzlerjones}; note that other sign conventions exist in the literature (e.g.\ in \cite{lefevre}).
	\end{defn}
	\begin{rmk}The original motivation for the definition came from Stasheff's work on $A_\infty$-spaces in \cite{stasheff}. If $X$ is a pointed topological space and $\Omega X$ its loop space, then we have a `composition of loops' map $\Omega X \times \Omega X \to \Omega X$. It is not associative, but it is associative up to homotopy. Similarly, one can bracket the product of four loops $a.b.c.d$ in five different ways, and one obtains five homotopies fitting into the Mac Lane pentagon. These homotopies are further linked via higher homotopies; we get an infinite-dimensional polytope $K$ the \textbf{associahedron} with $(n-2)$-dimensional faces $K_n$ corresponding to the homotopies between compositions of $n$ loops. An $A_\infty$\textbf{-space} is a topological space $Y$ together with maps $f_n: K_n \to Y^n$ satisfying the appropriate coherence conditions. For example a loop space is an $A_\infty$-space. If $Y$ is an $A_\infty$-space, then the singular chain complex of $Y$ is an $A_\infty$-algebra.
	\end{rmk}
	For readability, we will often write $a_1 \cdot a_2$ to mean $a_1\otimes a_2$ (multiplication in the tensor algebra). Suppose that $A$ is an $A_\infty$-algebra. Then $\mathrm{St}_1$ simply reads as $m_1^2=0$; in other words $m_1$ is a differential on $A$. Hence we may define the cohomology $HA$. The next identity $\mathrm{St}_2$ tells us that $m_1m_2=m_2(m_1\cdot 1 - 1\cdot m_1)$; in other words $m_2$ is a derivation on $(A,m_1)$. The third identity $\mathrm{St}_3$ yields $$m_2(1\cdot m_2 - m_2 \cdot 1)=m_1m_3+m_3(\sum_{i+j=2}1^{\cdot i}\cdot m_1 \cdot 1^{\cdot j}).$$The left hand side is the associator of $m_2$, and the right hand side is the boundary of the map $m_3$ in the complex $\hom(A^{\otimes 3}, A)$. Hence, $m_2$ is a homotopy associative `multiplication' on $A$. In particular, we obtain:
	\begin{prop}
		Suppose that $A$ is an $A_\infty$-algebra with $m_3=0$. Then $(A,m_1,m_2)$ is a dga. Similarly, if $A$ is any $A_\infty$-algebra, then $(HA,[m_2])$ is a graded algebra. Conversely, if $(A,d,\mu)$ is a dga, then $(A,d,\mu,0,0,0,\cdots)$ is an $A_\infty$-algebra.
	\end{prop}

	Additional signs arise in the above formulas via the Koszul sign rule when one wants to put elements into them.  The following lemma is extremely useful:
	\begin{lem}
		Fix positive integers $n=r+s+t$ and $n$ homogeneous elements $a_1,\ldots,a_n$ in $A$. Then $$(1^{\cdot r} \cdot m_s \cdot 1^{\cdot t})(a_1\cdots a_n) = (-1)^\epsilon a_1\cdots a_r\cdot m_s(a_{r+1}\cdots a_{r+s})\cdot a_{r+s+1}\cdots a_{n}$$where $\epsilon=s\sum_{j=1}^r {|a_j|}$. In particular, if $s$ is even then the na\"ive choice of sign is the correct one.
	\end{lem}
	\begin{proof}The Koszul sign rule gives a power of $|m_s|\sum_{j=1}^r {|a_j|}$, which has the same parity as $\epsilon$.
	\end{proof}

	\begin{defn}
	Let $A$ and $B$ be $A_\infty$-algebras. A \textbf{morphism} is a family of degree $1-n$ linear maps $f_n: A^{\otimes n} \to B$ satisfying the identities $$\sum_{n=r+s+t}(-1)^{r+st}f_{r+1+t}(1^{\otimes r} \otimes m_s \otimes 1^{\otimes t}) = \sum_{i_1+\cdots+ i_r=n}(-1)^{\sigma(i_1,\ldots,i_n)}m_r(f_{i_1}\otimes\cdots\otimes f_{i_r})$$where $\sigma(i_1,\dots,i_n)$ is the sum $\sum_j (r-j)(i_j-1)$ (note that only terms with $r-j$ odd and $i_j$ even will contribute to the sign).
\end{defn}

	In particular, $f_1$ is a chain map. A morphism $f$ is \textbf{strict} if it is a chain map; i.e.\ $f_n=0$ for $n>1$. A morphism $f$ is a \textbf{quasi-isomorphism} if $f_1$ is. One can compose morphisms by setting $(f\circ g)_n=\sum_{i_1+\cdots+ i_r=n}(-1)^{\sigma(i_1,\ldots,i_n)}f_r \circ(g_{i_1}\otimes\cdots\otimes g_{i_r})$.

\begin{defn}
	An $A_\infty$-algebra $A$ is \textbf{strictly unital} if there exists an element $\eta\in A^0$ such that $m_1(\eta)=0$, $m_2(\eta,a)=m_2(a,\eta)=a$, and if $n>2$ then $m_n$ vanishes whenever one of its arguments is $\eta$.
\end{defn}

\begin{defn}
	Let $A$ and $B$ be strictly unital $A_\infty$-algebras. A \textbf{morphism} $f:A \to B$ of strictly unital $A_\infty$-algebras is a morphism $f:A \to B$ of $A_\infty$-algebras such that $f_1$ preserves the unit $\eta$, and if $i>1$ then $f_i$ vanishes whenever one of its arguments is $\eta$.

\end{defn}

			\begin{defn}
			A strictly unital $A_\infty$-algebra $A$ is \textbf{augmented} if there is a morphism of strictly unital $A_\infty$-algebras $A \to k$. In this case the \textbf{augmentation ideal} is $\bar{A}\coloneqq \ker(A \to k)$.
		\end{defn}

	\section{Coalgebras and homotopy theory}
		We give an alternate quick definition of an $A_\infty$-algebra via dg coalgebras. Just like a dga is a monoid in the monoidal category of dg vector spaces over $k$, a \textbf{differential graded coalgebra} (or \textbf{dgc} for short) is a comonoid in this category. More concretely, a dgc is a dg $k$-vector space $(C,d)$ equipped with a comultiplication $\Delta:C \to C \otimes C$ and a counit $\epsilon:C \to k$, satisfying the appropriate coassociativity and counitality identities, and such that $d$ is a coderivation for $\Delta$.  A \textbf{coaugmentation} on a dgc is a section of $\epsilon$; if $C$ is coaugmented then $\bar C\coloneqq  \ker \epsilon$ is the \textbf{coaugmentation coideal}. It is a dgc under the reduced coproduct $\bar{\Delta}x = \Delta x -x\otimes 1 -1\otimes x$, and $C$ is isomorphic as a nonunital dgc to $\bar C \oplus k$. A coaugmented dgc $C$ is \textbf{conilpotent} if every $x \in \bar C$ is annihilated by some suitably high power of $\Delta$.
	\begin{ex}
		If $V$ is a dg vector space, then the tensor algebra ${T}^c(V)\coloneqq k\oplus V \oplus V^{\otimes 2} \oplus\cdots$ is a dg coalgebra when equipped with the \textbf{deconcatenation coproduct} ${T}^c(V) \to {T}^c(V) \otimes {T}^c(V)$ which sends $v_1\cdots v_n$ to $\sum_i v_1\cdots v_i \otimes v_{i+1}\cdots v_n$. The differential is induced from the differential on $V^{\otimes n}$. It is easy to see that ${T}^c(V)$ is conilpotent, since $\Delta^{n+1}(v_1\cdots v_n)=0$.
	\end{ex}
Denote by $\bar{T}^c(W)$ the \textbf{reduced tensor coalgebra}: it is the coaugmentation coideal of the tensor coalgebra. The functor ${T}^c$ is the cofree conilpotent coalgebra functor: if $C$ is conilpotent then $C \to{T}^c(V)$ is determined completely by the composition $l:C \to {T}^c(V) \to V$. In particular, any morphism $f:\bar{T}^c(W) \to \bar{T}^c(V)$ is determined completely by its \textbf{Taylor coefficients} $f_n: W^{\otimes n} \to V$.
	\begin{lem}
		Let $f,g$ be composable coalgebra maps between three reduced tensor coalgebras. Then the Taylor coefficients of the composition $f\circ g$ are given by $$(g\circ f)_n = \sum_{i_1+\cdots+ i_r=n}g_r(f_{i_1}\otimes \cdots \otimes f_{i_r}).$$
	\end{lem}
	Note the similarity with composition of $A_\infty$-algebra maps.
	\begin{defn}Let $C$ be a dg coalgebra. A \textbf{coderivation of degree $p$ on $C$} is a linear degree $p$ endomorphism $\delta$ of $C$ satisfying $(\delta\otimes 1 + 1 \otimes \delta)\circ \Delta = \Delta \circ \delta$.
	\end{defn} 
	The graded space $\mathrm{Coder}(C)$ of all coderivations of $C$ is not closed under composition, but is closed under the commutator bracket. Say that $\delta \in \mathrm{Coder}^1(C)$ is a \textbf{differential} if $\delta^2=0$; in this case $\mathrm{ad}(\delta)$ is a differential on $\mathrm{Coder}(C)$, making $\mathrm{Coder}(C)$ into a dgla. In the special case that $C=\bar{T}^c(V)$, a coderivation is determined by its Taylor coefficients. Coderivations compose similarly to coalgebra morphisms:
	\begin{lem}
		Let $\delta, \delta'$ be coderivations on $\bar{T}^c(V)$. Then the Taylor coefficients of the composition $\delta\circ \delta'$ are given by $$(\delta\circ \delta')_n = \sum_{r+s+t=n}\delta_{r+1+t}(1^{\otimes r} \otimes \delta'_s \otimes 1^{\otimes t}).$$
	\end{lem}
	\begin{thm}
		An $A_\infty$-algebra structure on a graded vector space $A$ is the same thing as a differential $\delta$ on $\bar{T}^c(A[1])$.
	\end{thm}
	\begin{proof}[Proof sketch]Given a coderivation $\delta$ we obtain Taylor coefficients $\delta_n:A[1]^{\otimes n} \to A$ of degree 1; in other words, these are maps $m_n: A^{\otimes n} \to A$ of degree $2-n$. The Stasheff identities are equivalent to $\delta$ being a differential. The sign changes occur in the Stasheff identities because of the need to move elements past the formal suspension symbol $[1]$.
	\end{proof}
	The following proposition can be checked in a similar manner:
	\begin{prop}
		Let $A,A'$ be two $A_\infty$-algebras with associated differentials $\delta, \delta'$. Then an $A_\infty$-morphism $f:A \to A'$ is the same thing as a coalgebra morphism $\bar{T}^c(A[1]) \to \bar{T}^c(A'[1])$ commuting with the differentials.
	\end{prop}
	\begin{defn}
		Let $A,A'$ be $A_\infty$-algebras and $f,g$ a pair of maps $A \to A'$. Let $F,G$ be the associated maps $\bar{T}^c(A[1]) \to \bar{T}^c(A'[1])$. Say that $f$ and $g$ are \textbf{homotopic} if there is a map $H: \bar{T}^c(A[1]) \to \bar{T}^c(A'[1])$ of degree $-1$ with $\Delta H = F \otimes H + H\otimes G$ and $F-G=\partial H$, where $\partial$ is the differential in the $\hom$-complex.
	\end{defn}
	One can unwind this definition into a set of identities on the Taylor coefficients of $H$; this is done in \cite[1.2]{lefevre}. Say that $A,A'$ are \textbf{homotopy equivalent} if there are maps $f:A \to A'$ and $f':A' \to A$ satisfying $f'f\simeq \id_{A}$ and $ff'\simeq \id_{A'}$.
	\begin{thm}[\cite{proute}]
		Homotopy equivalence is an equivalence relation on the category $\cat{Alg}_\infty$ of $A_\infty$-algebras. Moreover, two $A_\infty$-algebras are homotopy equivalent if and only if they are quasi-isomorphic.
	\end{thm}
	The category $\cat{dga}$ of differential graded algebras sits inside the category $\cat{Alg}_\infty$ of $A_\infty$-algebras. It is not a full subcategory: there may be more $A_\infty$-algebra maps than dga maps between two dgas. However, two dgas are dga quasi-isomorphic if and only if they are $A_\infty$-quasi-isomorphic: this is shown in, for example, \cite[1.3.1.3]{lefevre}. Abstractly, this follows from the existence of model structures on both $\cat{dga}$ and $\cat{cndgc}$, the category of conilpotent dg coalgebras, for which the bar and cobar constructions (see \ref{barcobarsctn}) are Quillen equivalences.
	
	\p Including $\cat{dga}\into\cat{Alg}_\infty$ does not create more quasi-isomorphism classes. Indeed, every $A_\infty$-algebra is quasi-isomorphic to a dga: one can take the adjunction quasi-isomorphism \linebreak $A \to \Omega B_\infty A$ induced by the bar and cobar constructions. However, we do get new descriptions of quasi-isomorphism class representatives. One nice such representative is the \textbf{minimal model} of an $A_\infty$-algebra.
	\section{Minimal models}\label{minmods}
	An $A_\infty$-algebra is \textbf{minimal} if $m_1=0$. Every $A_\infty$-algebra admits a minimal model. More precisely:
	\begin{thm}[Kadeishvili \cite{kadeishvili}]\label{kadeish}
		Let $(A,m_1,m_2,\ldots)$ be an $A_\infty$-algebra, and let $HA$ be its cohomology ring. Then there exists the structure of an $A_\infty$-algebra ${\mathscr{H}\kern -1pt A}= (HA,0,[m_2],p_3,p_4,\ldots)$ on $HA$, unique up to $A_\infty$-isomorphism, and an $A_\infty$-algebra morphism ${\mathscr{H}\kern -1pt A} \to A$ lifting the identity of $HA$.
	\end{thm}
	\begin{rmk}
		While the multiplication on $HA$ is induced by $m_2$, we need not have $p_n=[m_n]$ for $n>2$; indeed the $m_n$ need not even be cocycles. For example, if $A$ is a non-formal dga, then ${\mathscr{H}\kern -1pt A}$ must have nontrivial higher multiplications. We also note that ${\mathscr{H}\kern -1pt A} \to A$ is clearly an $A_\infty$-quasi-isomorphism, since it lifts the identity on $HA$. We also remark that the theorem follows from the essentially equivalent \textbf{homotopy transfer theorem}: if $A$ is an $A_\infty$-algebra, and $V$ a homotopy retract of $A$, then $V$ admits the structure of an $A_\infty$-algebra making the retract into an $A_\infty$-quasi-isomorphism (see \cite[9.4]{lodayvallette} for details). The result follows since, over a field, the cohomology of any chain complex is always a homotopy retract as one can choose splittings.
	\end{rmk}
	It is possible to give a explicit description of the higher multiplications $p_n$ appearing in Kadeishvili's theorem: Merkulov did this in \cite{merkulov}. One can define them recursively: suppose for convenience that $A$ is a dga. Choose any section $\sigma:HA \to A$ and let $\pi:A \to HA$ be the projection to $HA$. We will identify $HA$ with its image under $\sigma$. Choose a homotopy $h: \id_A \to \sigma\pi$. Define recursively maps $\lambda_n: (HA)^{\otimes n} \to A$ by $\lambda_2=m_2$, and $$\lambda_n\coloneqq \sum_{s+t=n}(-1)^{s+1}\lambda_2(h\lambda_s\otimes h\lambda_t)$$where we formally interpret $h\lambda_1\coloneqq -\id_A$. Then, $p_n=\pi\circ\lambda_n$. See \cite{markl} for some very explicit formulas (whose sign conventions differ). We remark that there may be many different ways of constructing the $p_n$, but they will all give $A_\infty$-isomorphic algebras.
	
	\begin{defn}
		Let $G$ be an abelian group. An $A_\infty$-algebra $A$ is \textbf{Adams }$G$-\textbf{graded} or just \textbf{Adams graded} if it admits a secondary grading by $G$ such that each higher multiplication map $m_n$ is of degree 0 in the secondary $G$-grading.
	\end{defn}
\begin{rmk}\label{adamsrmk}
	Note that for a dga to be Adams graded, we require only that the differential $d=m_1$ has Adams degree zero. If an $A_\infty$-algebra is Adams graded, then by making appropriate choices one can upgrade Merkulov's construction to give an $A_\infty$-quasi-isomorphism of Adams graded algebras $A \to \mathscr{H}\kern -1pt A$. Moreover, if $A$ is strictly unital, one can choose the morphism to be strictly unital. See \cite[\S2]{lpwzext} for more details.
\end{rmk}
	
	\p One can sometimes compute $A_\infty$-operations on a dga by means of Massey products. In what follows, $\tilde{a}$ means $(-1)^{1+|a|}{a}$, using the same sign conventions as \cite{kraines}.
	\begin{defn}\label{masseys}
		Let $u_1,\ldots, u_r$ be cohomology classes in a dga $A$. Pick representatives $u_i=[a_{i\,i}]$. The \textbf{$r$-fold Massey product} $\langle u_1,\ldots, u_r \rangle$ of the cohomology classes $u_1,\ldots ,u_r$ is defined to be the set of cohomology classes of sums $\tilde{a}_{1\,1}a_{2\,r}+\cdots+\tilde{a}_{1\,r-1}a_{r\,r}$ such that $da_{i\,j}=\tilde{a}_{i\,i}a_{i+1\,j}+\cdots+\tilde{a}_{i\,j-1}a_{j\,j}$ for all $1\leq i \leq j \leq r$ with $(i,j)\neq(1,r)$. This operation is well-defined, in the sense that it depends only on the cohomology classes $u_1,\ldots, u_r$.
	\end{defn}
	We will abuse terminology by referring to elements of $\langle x_1,\ldots, x_r \rangle$ as Massey products. We may also decorate the product $\langle x_1,\ldots, x_r \rangle$ with a subscript $\langle x_1,\ldots, x_r \rangle_r$ to emphasise that it is an $r$-fold product.
	\begin{rmk}
		We remark that $\langle x_1,\ldots, x_r \rangle$ may be empty: for example, in order for $\langle x,y,z\rangle$ to be nonempty, we must have $xy=yz=0$. More generally, for $\langle x_1,\ldots, x_r \rangle$ to be nonempty, we require that each $\langle x_p,\ldots, x_q \rangle$ is nonempty for $0<q-p<n-1$. Most sources define $\langle x,y,z\rangle$ only when it is nonempty, and leave it undefined otherwise.
	\end{rmk}
	The point is that, when Massey products exist, Merkulov's higher multiplications $p_n$ are all Massey products, up to sign:
	\begin{thm}[{\cite[3.1]{lpwzext}}]\label{masseysarehighermults}
		Let $A$ be a dga and let $x_1,\ldots, x_r$ $(r>2)$ be cohomology classes in $HA$, and suppose that  $\langle x_1,\ldots, x_r \rangle$ is nonempty. Give $HA$ an $A_\infty$-algebra structure via Merkulov's construction. Then, up to sign, the higher multiplication $p_r(x_1,\ldots, x_r)$ is a Massey product.
	\end{thm}
	So, if $A$ is a formal dga, then all Massey products (that exist) will vanish. The converse is not true: formality of a dga cannot be checked simply by looking at its Massey products. We will use the existence of Massey products to detect non-formality: the following lemma is computationally useful.
	\begin{lem}\label{masseylemma}Let $A$ be a dga and $u=[a]$ a cohomology class in $A$. Put $b_1\coloneqq a$ and let $r>2$ be an integer. Then the $r$-fold Massey product $\langle u,\ldots, u \rangle_r$ of $u$ with itself $r$ times is the set of cohomology classes of sums $\tilde{b}_1b_{r-1}+\cdots+\tilde{b}_{r-1}b_1$ such that $db_i=\tilde{b}_1b_{i-1}+\cdots+\tilde{b}_{i-1}b_1$ for all $1<i<r$. If $u$ is of odd degree then we may drop the tildes from the $b_i$.
	\end{lem}
	\begin{proof}
		By definition, the Massey product is the set of cohomology classes of sums\linebreak ${\tilde{a}_{1\,1}a_{2\,r}+\cdots+\tilde{a}_{1\,r-1}a_{r\,r}}$ such that $da_{i\,j}=\tilde{a}_{i\,i}a_{i+1\,j}+\cdots+\tilde{a}_{i\,j-1}a_{j\,j}$ for all $1\leq i \leq j \leq r$ with $(i,j)\neq(1,r)$, where we put $a_{i\,i}\coloneqq a$ for all $i$. Inductively, it is easy to see that if $i-j=k-l$ then $a_{i\,j}=a_{k\,l}$. Hence we may define $b_{1+j-i}\coloneqq a_{i\,j}$ and the first claim follows. If $u$ has odd degree, then an easy induction shows that all $b_i$ have odd degree, and the second claim follows.
	\end{proof}

\chapter{Koszul duality}\label{kd}
	 In this chapter we will study the Koszul dual of a dga, and prove a duality result (\ref{kdfin}) for a large class of reasonably finite dgas. Chapter \ref{defmthy} will give a deformation-theoretic interpretation of these results. In this chapter, every dga that we consider will be \textbf{augmented}, meaning that the canonical map $k \to A$ admits a retraction. The \textbf{augmentation ideal} of an augmented dga is $\bar A\coloneqq \ker(A \to k)$. Sending $A$ to $\bar A$ sets up an equivalence between augmented dgas and nonunital dgas. The inverse functor freely appends a unit, and indeed $A$ is isomorphic to $\bar{A}\oplus k$ as an augmented dga. Say that an augmented dga $A$ is \textbf{Artinian local} if $A$ has finite total dimension and $\bar A$ is nilpotent. In other words, Artinian local means `finite-dimensional over $k$, and local with residue field $k$'. Most dgas of interest to us in this section will be concentrated in nonpositive cohomological degrees.
	 
	 \p We remark that all of the results of this chapter remain true in the positive characteristic setting.
	\section{Bar and cobar constructions}\label{barcobarsctn}
	We follow Positselski \cite{positselski}; for other references see Loday--Vallette \cite{lodayvallette} or Lef\`evre-Hasegawa's thesis \cite{lefevre}. 
	\begin{defn}
		Let $A$ be an augmented dga. Put $V\coloneqq \bar A [1]$, the shifted augmentation ideal. Let $d_V$ be the usual differential on the tensor coalgebra $TV$. Let $d_B$ be the \textbf{bar differential}: send $a_1\otimes\cdots \otimes a_n$ to the signed sum over $i$ of the $a_1\otimes\cdots\otimes a_ia_{i+1}\otimes\cdots \otimes a_n$ and extend linearly. The signs come from the Koszul sign rule; see \cite[2.2]{lodayvallette} for a concrete formula. One can check that $d_B$ is a degree 1 map from $V^{\otimes n+1} \to V^{\otimes n}$, and that it intertwines with $d_V$. Hence, one obtains a third and fourth quadrant bicomplex $C$ with rows $V^{\otimes{n}}[-n]$. By construction, the direct sum total complex of $C$ is $TV$, with a new differential $\partial=d_V+d_B$. The \textbf{bar construction} of $A$ is the complex $BA\coloneqq (TV,\partial)$. One can check that the deconcatenation coproduct makes $BA$ into a dgc.
	\end{defn}
Note that the degree 0 elements of $A$ become degree $-1$ elements of $BA$.
\begin{rmk}
	If $A$ is an augmented $A_\infty$-algebra, then one can define the \textbf{$A_\infty$ bar construction} $B_\infty A$, which is a dgc, in an analogous manner (see \cite{lefevre} for a concrete formula). If $A$ is a dga then $B_\infty A = BA$.
	\end{rmk}
	\begin{ex}Let $A$ be the graded algebra $k[\epsilon]/\epsilon^2$, with $\epsilon$ in degree 0. Then $\bar A [1]$ is $k\epsilon$ placed in degree $-1$. Since $\epsilon$ is square zero, the bar differential is identically zero. So $BA$ is the tensor coalgebra $k[\epsilon]$, with $\epsilon$ in degree -1.
	\end{ex}

	\begin{defn}
		Let $C$ be a coaugmented dgc. One can analogously define a \textbf{cobar differential} $d_\Omega$ on the tensor algebra $T(\bar{C}[-1])$ by sending $c_1\otimes\cdots \otimes c_n$ to the signed sum over $i$ of the ${c_1\otimes\cdots \otimes  \bar{\Delta}c_i\otimes\cdots c_n}$, and the \textbf{cobar construction} on $C$ is the dga $\Omega C\coloneqq (T(\bar{C}[-1]),d_C+d_\Omega)$.
	\end{defn}
	
	Bar and cobar are adjoints:
	\begin{thm}[{\cite[2.2.6]{lodayvallette}}]
		If $A$ is an augmented dga and $C$ is a conilpotent dgc, then there is a natural isomorphism $$\hom_{\cat{dga}}(\Omega C, A)\cong \hom_{\cat{dgc}}(C, BA)$$
	\end{thm}
		\begin{lem}
		The bar construction preserves quasi-isomorphisms.
	\end{lem}
	\begin{proof}
		The idea is to filter $BA$ by setting $F_p BA$ to be the elements of the form $a_1\otimes\cdots\otimes a_n$ with $n\leq p$, and look at the associated spectral sequence. A proof for dgas is in \cite[\S6.10]{positselski} and a proof for $A_\infty$-algebras is in \cite[Chapter 1]{lefevre}.
	\end{proof}
\begin{rmk}The cobar construction does not preserve quasi-isomorphisms in general.
	\end{rmk}Bar and cobar give canonical resolutions:
	\begin{thm}[{\cite[2.3.2]{lodayvallette}}]
		Let $A$ be an augmented dga. Then the counit $\Omega B A \to A$ is a quasi-isomorphism. Similarly, the unit is a quasi-isomorphism for coaugmented conilpotent dgcs.
	\end{thm}
	\section{Koszul duality for Artinian local algebras}
	Recall that if $V$ is a complex of vector spaces then $V^*$ denotes its graded linear dual. Let $(C,\Delta,\epsilon)$ be a dgc. Then $\Delta$ dualises to a map $(C\otimes C)^* \to C^*$, and combining this with the natural inclusion $C^* \otimes C^* \to (C \otimes C)^*$ yields a semigroup structure on $C^*$. In fact, $(C^*,\Delta^*,\epsilon^*)$ is a dga. The dual statement is not in general true -- if $A$ is a dga then the multiplication need not dualise to a map $A^* \otimes A^* \to A^*$. However, if $A$ is an Artinian local dga, then it does (because the natural map $A^* \otimes A^* \to (A \otimes A)^*$ is an isomorphism), and indeed $A^*$ is a dgc. If $C$ is coaugmented, then $C^*$ is augmented, and if $C$ is conilpotent, then $\bar{C^*}$ is nilpotent. Similarly, if $A$ is Artinian local then $A^*$ is coaugmented and conilpotent.
	\begin{defn}Let $A$ be an augmented dga. The \textbf{Koszul dual} of $A$ is the dga $A^!\coloneqq (BA)^*$.
	\end{defn}
Loosely, the differential $d(x^*)$ is the signed sum of the products $x_1^*\cdots x_r^*$ such that the $x_i$ satisfy $d(x_1 | \cdots | x_r)=x$, where $d$ is the bar differential. The Koszul dual $A^!$ is a semifree dga, in the sense that the underlying graded algebra is a completed free algebra. In the situations we will be interested in, the underlying graded algebra of $A^!$ will actually be a free algebra in the usual sense. Note that $BA$ is coaugmented, so $A^!$ is again augmented. Because both $B$ and the linear dual preserve quasi-isomorphisms, so does $A \mapsto A^!$. 

\begin{prop}\label{kdisrend}
Let $A$ be an augmented dga, and let $S$ be the $A$-module $ k$ with $A$-action given by augmentation $A \to k$. Then there is a quasi-isomorphism of dgas $A^!\simeq \R\enn_A(S)$.
\end{prop}
\begin{proof}
This is standard and appears as e.g.\ \cite[Lemma 2.6]{kalckyang}. The idea is that by taking the bar resolution of $S$, the bar construction $BA$ becomes a model for the derived tensor product $S\lot_A S$. Taking the linear dual, one obtains the desired statement.
\end{proof}
The key statement about the Koszul dual is the following:
	\begin{prop}\label{artprop}
		Let $A$ be a nonpositive Artinian local dga. Then $A^!$ is naturally isomorphic as a dga to $\Omega(A^*)$.\end{prop}
	\begin{proof}This boils down to the fact that $A \to A^{**}$ is a natural isomorphism. For brevity, we will replace $A$ by its augmentation ideal $\bar A$. The dgc $BA$ is the direct sum total complex of the double complex whose rows are $A^{\otimes n}$. Hence, $A^!$ is the direct product total complex of the double complex with rows $(A^{\otimes n})^*$. However, because $A$ was nonpositive $A^!$ is also the direct sum total complex of this double complex. Because $A$ is nonpositive and locally finite, the natural map $(A^*)^{\otimes n} \to (A^{\otimes n})^*$ is an isomorphism. Hence $A^!$ is the direct sum total complex of the double complex with rows $(A^*)^{\otimes n}$, which -- after checking that the bar differential dualises to the cobar differential -- is precisely the definition of $\Omega (A^*)$.
	\end{proof}
	\begin{cor}\label{kdforart}
		Let $A$ be a nonpositive Artinian local dga. Then $A$ is naturally quasi-isomorphic to $A^{!!}$.
	\end{cor}
	\begin{proof}
		$A^{!!}$ is by definition $(B(A^!))^*$. By \ref{artprop}, $A^!$ is isomorphic to $\Omega(A^*)$. Because $C \to B \Omega C$ is a quasi-isomorphism for conilpotent dgcs, $A^* \to B(A^!)$ is a dgc quasi-isomorphism. Dualising and using exactness of the linear dual gets us a dga quasi-isomorphism $A^{!!}\to A^{**}$. But $A$ is Artinian local, and hence isomorphic to $A^{**}$.
	\end{proof}
	\begin{rmk}Note that we did not use the local hypothesis; we just used that $A$ was nonpositive and finite.
	\end{rmk}
	
	\section{The model structure on pro-Artinian algebras}

	\begin{defn}
		Let $\dga$ be the category of nonpositively cohomologically graded augmented dgas over $k$, and let $\dgart$ be the subcategory on Artinian local dgas. We refer to an object of the procategory $\proart$ as a \textbf{pro-Artinian dga}.
	\end{defn}
	\begin{rmk}
		We caution that in this thesis, ``pro-Artinian'' means ``pro-(Artinian local)''. We will not consider non-local profinite dgas.
	\end{rmk}
	There is a limit functor $\varprojlim: \proart \to \dga$ which sends a cofiltered system to its limit. Moreover, this is right adjoint to the functor $\dga \to \proart$ which sends a dga $A$ to the cofiltered system $\hat A$ of its Artinian local quotients. It is not necessarily the case that a pro-Artinian dga $A$ is isomorphic to $\widehat{\varprojlim A}$. We list some standard results on the structure of $\proart$:
	\begin{prop}[{\cite[6.1.14]{kashschap}}]
		Let $f:A \to B$ be a morphism in $\proart$. Then $f$ is isomorphic to a level map: a collection of maps $\{f_\alpha: A_\alpha \to B_\alpha\}_{\alpha \in I}$ between Artinian local algebras, where $I$ is cofiltered.
	\end{prop}
	\begin{prop}[{\cite[Corollary to 3.1]{grothendieckpro}}]\label{strictpro}
		Every object of $\proart$ is isomorphic to a \textbf{strict} pro-object, i.e.\ one for which the transition maps are surjections.
	\end{prop}
	We aim to show that the limit functor $\varprojlim: \proart \to \dga$ both preserves and reflects the weak equivalences of two model structures, which we start by describing.
	\begin{thm}[\cite{hinichhom}]
		The category $\cat{dga}$ of all dgas is a model category, with weak equivalences the quasi-isomorphisms and fibrations the levelwise surjections.
	\end{thm}
We will regard $\dga$ as a subcategory of $\cat{dga}$ in the obvious way. Note that in $\dga$, every object is fibrant, and the cofibrant objects are precisely the semifree dgas (i.e.\ those that become free graded algebras after forgetting the differential). For example, $\Omega B A \to A$ is a functorial cofibrant resolution of $A$. 
	\begin{thm}\label{proartmodel}
		The category $\cat{pro}(\cat{dgArt}_k)$ is a model category, with weak equivalences those maps $f$ for which each $H^n f: H^nA \to H^nB$ is an isomorphism of profinite $k$-vector spaces, and fibrations those maps $f$ for which $\varprojlim f$ is a levelwise surjection.
	\end{thm}
	\begin{proof}
		The proof of \cite[4.3]{unifying} adapts to the noncommutative case. We remark that this remains true in positive characteristic.
	\end{proof}
We will regard $\proart$ as a subcategory of $\cat{pro}(\cat{dgArt}_k)$ in the obvious way.
	\begin{prop}\label{limreflects}
		The functor $\varprojlim: \proart \to \dga$ both preserves and reflects weak equivalences.
	\end{prop}
	\begin{proof}
		Every vector space is canonically ind-finite \cite[1.3]{unifying}, so that the linear dual provides a contravariant equivalence from $\cat{pro}(\cat{fd-vect}_k)$ to $\cat{vect}_k$. In other words, let $g:U \to V$ be a map of profinite vector spaces. We may take $g$ to be a level map $\{g_\alpha:U_\alpha \to V_\alpha\}_\alpha$. Then $g$ is an isomorphism if and only if $\varinjlim_\alpha g_\alpha^*: \varinjlim_\alpha V_\alpha^* \to \varinjlim_\alpha U_\alpha^*$ is an isomorphism. Dualising again, we obtain a map $\varprojlim_\alpha U_\alpha^{**} \to \varprojlim_\alpha V_\alpha^{**}$, which canonically agrees with $\varprojlim g = \varprojlim_\alpha g_\alpha$ since the $U_\alpha$ and $V_\alpha$ are finite-dimensional. Hence, $g$ is an isomorphism if and only if $\varprojlim g$ is. Similarly, $g$ is a injection or a surjection if and only if $\varprojlim g$ is. Hence, $\varprojlim:\cat{pro}(\cat{fd-vect}_k) \to \cat{vect}_k$ is exact. Let $f$ be a morphism in $\proart$. By definition, $f$ is a weak equivalence if and only if each $H^nf \in \cat{pro}(\cat{fd-vect}_k)$ is an isomorphism. But this is the case if and only if each $\varprojlim H^n f$ is an isomorphism. Because $\varprojlim:\cat{pro}(\cat{fd-vect}_k) \to \cat{vect}_k$ is exact, this is the case if and only if each $H^n\varprojlim f$ is an isomorphism, which is exactly the condition for $\varprojlim f$ to be a weak equivalence.
	\end{proof}
	\begin{rmk}
		Morally, one would like to say that the model structure on $\proart$ is transferred from that of $\dga$ along the right adjoint $\varprojlim$, since one has that a map $f$ of pro-Artinian algebras is a fibration or a weak equivalence precisely when $\varprojlim f$ is. However, this is not strictly the case: the standard argument (as in e.g.\ \cite{cranstransfer}) requires that the left adjoint $A \mapsto \hat A$ preserves small objects. If this is the case, then since the ungraded algebra $k[x]$ is small in $\dga$, the object $\widehat{k[x]}$ is small in $\proart$. However, let $V$ be any infinite-dimensional vector space, and let $V_i$ be its filtered system of finite-dimensional subspaces. Let $k\oplus V_i$ be the square-zero extension (see \ref{dersctn} for a definition). In $\proart$, the colimit of the $k\oplus V_i$ is the square-zero extension $k\oplus V^{**}$. One can check that $\varinjlim_i\hom(\widehat{k[x]},k\oplus V_i)\cong V$, but that $\hom({\widehat{k[x]},k\oplus V^{**}})\cong V^{**}$. Hence, $\widehat{k[x]}$ is not small.
	\end{rmk}
The following lemma will be useful to us later:
	\begin{lem}\label{limisexact}
	The two functors $\varprojlim: \proart \to \dga$ and $\holim: \proart \to \dga$ are quasi-isomorphic.
\end{lem}
\begin{proof}
	Let $P$ be any object of $\proart$. Without loss of generality, by \ref{strictpro} we may assume that $P$ is strict. Let $I$ be the indexing set of $P$. Because every cofiltered set $I$ has a cofinal directed subset $I' \into I$ \cite[Expos\'e 1, 8.1.6]{sga4}, we may without loss of generality assume that $I$ is directed. Hence, $I$ is a Reedy category. It is now easy to see that $P$ is a Reedy fibrant diagram of dgas, and hence $\varprojlim P \simeq \holim P$.
\end{proof}
	
	\section{The model structure on conilpotent coalgebras}\label{coalgmodel}
	\begin{thm}[{\cite[Theorem 9.3b and Theorem 6.10]{positselski}}]\label{bcquillen}
		The category $\cat{cndgc}_k$ of coaugmented conilpotent dgcs admits a model structure where the weak equivalences $f$ are those maps for which $\Omega f$ is a dga quasi-isomorphism, and the cofibrations are the levelwise monomorphisms. Moreover, if $C$ is a conilpotent dgc then the natural map $C \to B\Omega C$ is a fibrant resolution. The pair $(\Omega, B)$ is a Quillen equivalence between $\cat{cndgc}_k$ and $\cat{dga}_k$, the category of unbounded dgas.
	\end{thm}
	Every weak equivalence is a quasi-isomorphism, but the converse is not true \cite[2.4.3]{lodayvallette}. Every object in $\cat{cndgc}_k$ is cofibrant, and the fibrant objects are the semicofree conilpotent coalgebras, i.e.\ those that are tensor coalgebras after forgetting the differential (the tensor coalgebra $T^cV$ is the cofree conilpotent coalgebra on $V$).
	
	\begin{prop}\label{sharpprop}
		The functor $\{A_\alpha\}_\alpha \mapsto\varinjlim_\alpha A_\alpha^*$ is an equivalence $$\proart\xrightarrow{\cong} (\cat{cndgc}^{\geq 0}_k)^\text{op}.$$
	\end{prop}
	\begin{proof}
		Via the linear dual, an Artinian local dga is the same thing as a finite-dimensional coaugmented conilpotent dg coalgebra over $k$. Taking procategories then gives an equivalence between $(\proart)^\text{op}$ and $\cat{ind}(\cat{fd-cndgc}_k^{\geq 0})$. A classical theorem of Sweedler says that a (non-dg) coalgebra is the filtered colimit of its finite-dimensional subcoalgebras. The same remains true for dgcs \cite[1.6]{getzlergoerss}. In particular, the image of the colimit functor $\varinjlim: \cat{ind}(\cat{fd-cndgc}_k^{\geq 0}) \to \cat{dgc}_k^{\geq 0}$ is precisely $\cat{cndgc}_k^{\geq 0}$. So $\varinjlim$ is essentially surjective. By \cite[1.9]{getzlergoerss}, finite-dimensional conilpotent coalgebras are compact: given a finite-dimensional conilpotent dgc $C$, and $D=\{D_\alpha\}_\alpha$ a filtered system of finite-dimensional dgcs, that there is a natural isomorphism $\varinjlim_\alpha\hom(C,D_\alpha)\to \hom(C,\varinjlim_\alpha D_\alpha)$. But then it follows that if $C$ and $D$ are objects of $\cat{ind}(\cat{fd-cndgc}_k^{\geq 0})$, then one has an isomorphism $\hom(C,D)\cong\hom(\varinjlim C, \varinjlim D)$. Moreover, if $A$ and $B$ are finite-dimensional algebras, then one has an isomorphism $\hom(A,B)\cong \hom(B^*,A^*)$. Putting these together we see that $\varinjlim$ is fully faithful and hence an equivalence.
	\end{proof}
	\begin{defn}\label{csharp}
		If $C$ is a nonnegative dgc, let $C^\sharp\in \proart$ denote the levelwise dual of its filtered system of finite-dimensional sub-dgcs.
	\end{defn}
	It is easy to see that $C \mapsto C^\sharp$ is the inverse functor to  $\{A_\alpha\}_\alpha \mapsto\varinjlim_\alpha A_\alpha^*$, and that $C^*$ and $ \varprojlim C^\sharp$ are isomorphic dgas.
	\begin{prop}\label{sharpquillen}
		The equivalence $$(-)^\sharp:(\cat{cndgc}^{\geq 0}_k)^\text{op} \to \proart$$preserves fibrations and weak equivalences.
	\end{prop}
	\begin{proof}
Let $g^\text{op}: D \to C$ be a fibration in $(\cat{cndgc}^{\geq 0}_k)^\text{op}$, i.e.\ $g:C \to D$ is a levelwise injection of dgcs. Hence, $g^*: D^* \to C^*$ is a levelwise surjection of dgas. But $g^*\cong \varprojlim g^\sharp$, so that $g^\sharp$ is a fibration in $\proart$. Similarly, suppose that $g^\text{op}: D \to C$ is a weak equivalence, and in particular a quasi-isomorphism. Then $g: C \to D$ is a quasi-isomorphism. Dualising, $g^*\cong \varprojlim g^\sharp$ is a quasi-isomorphism, and so $g^\sharp$ is a weak equivalence, since, by \ref{limreflects}, $f$ is a weak equivalence of pro-Artinian dgas if and only if $\varprojlim f$ is a quasi-isomorphism.
	\end{proof}
	\section{Koszul duality for pro-Artinian algebras}
	\begin{defn}
		Say that a nonpositive dga $A\in \dga$ is \textbf{good} if \begin{itemize}
			\item $A$ is quasi-isomorphic to $\varprojlim \mathcal{A}$ for some pro-Artinian dga $\mathcal{A}$.
			\item $A$ is cohomologically locally finite.
		\end{itemize}
	\end{defn}
	In the presence of the second condition, the first condition is actually equivalent to requiring simply that $H^0(A)$ is an Artinian local algebra. One direction of this equivalence is clear: if $A$ is quasi-isomorphic to $\varprojlim \mathcal{A}$ and $H^0(A)$ is finite, then it is a finite-dimensional local ring and hence Artinian local. The other direction is a nontrivial result we later prove as \ref{goodchar}, which will require deformation-theoretic methods. We will of course not use this fact until then. We are about to prove a Koszul duality result for the class of good dgas -- we begin by noting an important finiteness property.
	\begin{prop}\label{kdfinite}
		Let $A\in \dga$ be a nonpositive dga. If $A$ is cohomologically locally finite then so are $BA$ and $A^!$.
	\end{prop}
	\begin{proof}Since the linear dual is exact, the statement for $A^!$ is implied by the statement for $BA$. To prove the latter, filter $BA$ by the tensor powers of $A$ to obtain a spectral sequence with $E_1$ page ${H^p(A^{\otimes q}) \Rightarrow H^{p-q}(BA)}$. Since there are only finitely many nonzero $H^p(A^{\otimes q})$ with $p-q$ fixed, and they are all finite-dimensional, $H^{p-q}(BA)$ must also be finite-dimensional.
	\end{proof}
	\begin{rmk}
		One can also prove \ref{kdfinite} by applying the $A_\infty$ bar construction to an $A_\infty$ minimal model for $A$, which yields a locally finite model for $BA$.	
	\end{rmk}
	\begin{thm}\label{kdgood}
		Let $A$ be a dga. If $A$ is good then $A$ is quasi-isomorphic to its double Koszul dual $A^{!!}$.
	\end{thm}
	\begin{proof}
		By assumption there is a pro-Artinian dga $\mathcal{A}$ and a quasi-isomorphism between $A$ and $\varprojlim \mathcal{A}$. Since the bar construction and the Koszul dual preserve quasi-isomorphisms, we may assume that $A=\varprojlim \mathcal{A}$. Moreover, by \ref{strictpro} we may assume that $\mathcal{A}$ is strict. Let $\mathcal{C}\coloneqq\mathcal{A}^*$ be the corresponding ind-conilpotent dgc, and put $C\coloneqq  \varinjlim \mathcal{C}$. It is clear that $\mathcal{A}^! \cong \Omega \mathcal{C}$ as ind-dgas. Taking colimits and using cocontinuity of $\Omega$ we get $\varinjlim \mathcal{A}^! \cong \Omega C$. Hence, $B\varinjlim \mathcal{A}^!$ is weakly equivalent to $C$. Dualising, we see that $(\varinjlim \mathcal{A}^!)^!$ is quasi-isomorphic to $C^* \cong A$. So it is enough to show that $\varinjlim \mathcal{A}^!$ is quasi-isomorphic to $A^!$.
		
		\p Put $\mathcal{D}\coloneqq B\mathcal{A}$; it is a pro-conilpotent dgc that is concentrated in nonpositive degrees. Put $D\coloneqq \varprojlim \mathcal{D}$, where we take the limit in the category of conilpotent coalgebras. Note that this limit exists because the category of conilpotent coalgebras is a coreflective subcategory of the category of all coalgebras \cite[1.3.33]{aneljoyal}, and a coreflective subcategory of a complete category is complete. Since $B$ is continuous we have $D \cong BA$. There is a natural algebra map $\phi:\varinjlim \mathcal{D}^* \to D^*$. Note that $\mathcal{D}^*=\mathcal{A}^!$ and that $D^* = A^!$, so it is enough to show that $\phi$ is a quasi-isomorphism. For $n \in \Z$, consider the induced linear map $$\psi_n:\qquad\varinjlim(H^n(\mathcal{D}^*)) \xrightarrow{\cong}H^n(\varinjlim\mathcal{D}^*) \xrightarrow{H^n\phi} H^n(D^*) \xrightarrow{\cong}H^{-n}(D)^*$$where we have used exactness of filtered colimits and the linear dual, and dualise it to obtain a map $$\chi_n:\qquad H^{-n}(D) \to H^{-n}(D)^{**} \xrightarrow{\psi_n^*} (\varinjlim(H^n(\mathcal{D}^*)))^* \xrightarrow{\cong} \varprojlim H^n(\mathcal{D^*})^{*} \xrightarrow{\cong} \varprojlim H^{-n}(\mathcal{D}^{**})$$where we have used exactness of the linear dual again along with the fact that contravariant Hom sends colimits to limits. By \ref{kdfinite}, $D=BA$ is cohomologically locally finite, which implies that $H^{-n}(D) \to H^{-n}(D)^{**}$ is an isomorphism. Similarly, each level ${\mathcal{D}_\alpha}$ of $\mathcal{D}$ is locally finite, since it is the bar construction on an Artinian local dga. In particular, the natural map $H^{-n}(\mathcal{D}_\alpha) \to H^{-n}(\mathcal{D}_\alpha^{**})$ which sends $[v]$ to $[\mathrm{ev}_v]$ is an isomorphism. Let $[u] \in H^{-n}(D)$; one can compute that $\chi_n([u])=([\mathrm{ev}_{u_\alpha}])_\alpha$, where $u_\alpha$ is the image of $u$ under the natural map $D \to \mathcal{D}_\alpha$. Hence, the composition $H^{-n}(D) \xrightarrow{\chi_n}\varprojlim H^{-n}(\mathcal{D}^{**}) \xrightarrow{\cong} \varprojlim H^{-n}(\mathcal{D})$ of $\chi_n$ with the inverse to the natural isomorphism sends $[u]$ to $[u_\alpha]_\alpha$. But this is precisely the natural map $H^{-n}(D) \to \varprojlim H^{-n} (\mathcal{D})$. Since $B$ preserves surjections, and we chose $\mathcal{A}$ to be strict, $\mathcal{D}=B\mathcal{A}$ is a strict pro-dgc, and in particular satisfies the Mittag-Leffler condition. Hence, the natural map $H^{-n}(D) \to \varprojlim H^{-n} (\mathcal{D})$ is an isomorphism for all $n$. Now it follows that $\chi_n$, $\psi_n$, and $H^n\phi$ are isomorphisms for all $n$. Hence, $\phi$ is a quasi-isomorphism.
	\end{proof}
	\begin{rmk}
		Note that instead of requiring that $A$ itself be cohomologically locally finite, it is enough to require that $BA$ is cohomologically locally finite. If the dgc $BA$ is cohomologically locally finite and admits a minimal model, then $A$ has a resolution $\Omega B A$ with finitely many generators in each level, which can be thought of as a finiteness condition.
	\end{rmk}

	\section{Derivations}\label{dersctn}
	If $A\to B$ is a map of commutative $k$-algebras and $M$ is a $B$-module, then a derivation $A \to M$ is the same as a map of $B$-augmented $k$-algebras $A \to B \oplus M$, where $B \oplus M$ is the square-zero extension of $B$ by $M$. When $A=B$ and the map is the identity, a derivation is the same as a section of the projection $A\oplus M \to A$. In this section, we make the same observation in the noncommutative derived world. We prove a key technical result stating that derived derivations from a `pregood' pro-Artinian dga are the same as derived derivations from its limit (\ref{rder}). Note that when our algebras are noncommutative, one must use bimodules in order to talk about derivations. We broadly follow Tabuada \cite{tabuadaNCAQ}, who is generalising the seminal work of Quillen \cite{quillender} for ungraded commutative algebras. A reference for ungraded noncommutative algebras is Ginzburg \cite{ginzburgnc}.
	\begin{defn}
		Let $B$ be a dga and $M$ a $B$-bimodule. The \textbf{square-zero extension} of $B$ by $M$ is the dga $B \oplus M$ whose underlying dg vector space is $B\oplus M$, with multiplication given by $(b,m).(b',m')=(bb',bm+mb')$. If $A \to B$ is a dga map then a \textbf{derivation} $A \to M$ is a map of $B$-augmented dgas $A \to B\oplus M$, which is equivalently a morphism $A \to B\oplus M$ in the overcategory $\cat{dga}_k / B$. The set of derivations $A \to M$ is $\mathrm{Der}_B(A,M)\coloneqq \hom_{\cat{dga}_k / B}(A, B \oplus M)$.
	\end{defn}
	One can easily check that a derivation $A \to M$ is the same as an $A$-linear map $A \to M$ satisfying the graded Leibniz formula.
	\begin{prop}[{\cite[4.6]{tabuadaNCAQ}}]
		The square-zero extension functor $B$-$\cat{bimod} \to \cat{dga}_k / B$ admits a left adjoint $A\mapsto \Omega(A)_B$, which we refer to as the functor of \textbf{noncommutative K\"ahler differentials}.
	\end{prop}
	\begin{proof}
		An application of Freyd's adjoint functor theorem.
	\end{proof}
	\begin{rmk}
		If $A \to B$ is a morphism of commutative $k$-algebras, then $\Omega(A)_B$ does not agree with the usual commutative K\"ahler differentials. Indeed, one has an isomorphism $\Omega(A)_A\cong \ker(\mu: A \otimes_k A \to A)$, and the commutative K\"ahler differentials are $\Omega(A)_A/\Omega(A)_A^2$. In general one has $\Omega(A)_B\cong B\otimes_A \Omega(A)_A \otimes_A B$ which is the pullback of the bimodule $\Omega(A)_A$ along $A \to B$.
	\end{rmk}
	\begin{cor}
		Let $A \to B$ be a dga map and $M$ a $B$-bimodule. Then the set $\mathrm{Der}_B(A,M)$ is naturally a dg vector space.
	\end{cor}
	\begin{proof}
		The category of dg $B$-bimodules is naturally enriched over $\cat{dgvect}_k$.
	\end{proof}
	The category $\cat{dga}_k / B$ is a model category, with model structure induced from that on $\cat{dga}_k$. The category $B$-$\cat{bimod}\coloneqq B\otimes_k B^{\text{op}}\text{-}\cat{Mod}$ is also a model category in the usual way. It is easy to see that the square-zero extension functor $B\oplus-:B$-$\cat{bimod} \to \cat{dga}_k/B$ is right Quillen. Since every object in $B$-$\cat{bimod}$ is fibrant, $B\oplus-$ is its own right derived functor. However, since not every dga is cofibrant, the noncommutative K\"ahler differentials have a nontrivial left derived functor.
	\begin{defn}
		The \textbf{noncommutative cotangent complex} functor is $\mathbb{L}(-)_B\coloneqq \mathbb{L}\Omega(-)_B$, the total left derived functor of $\Omega(-)_B$.
	\end{defn}
	The model category $B$-$\cat{bimod}$ is a dg model category, in the sense that it is enriched over $\cat{dgvect}_k$ in a way compatible with the model structure (the interested reader should consult Hovey {\cite[4.2.18]{hovey}} for a rigorous definition of enriched model category; in the terminology used there a dg model category is a $\mathrm{Ch}(k)$-model category). In particular, $B$-$\cat{bimod}$ has a well-defined notion of derived hom-complexes, and we may use the Quillen adjunction with $\cat{dga}_k / B$ to define complexes of derived derivations.
	\begin{defn}
		Let $A \to B$ be a dga map and let $M$ be a $B$-bimodule. Let $QA \to A$ be a cofibrant resolution. The space of \textbf{derived derivations} from $A$ to $M$ is the dg vector space $$\R\mathrm{Der}_B(A,M)\coloneqq \mathrm{Der}_B(QA,M)\simeq\dgh_B(\mathbb{L}(A)_B,M)$$where we use the notation $\dgh$ to mean the enriched hom.
	\end{defn}
	Different choices of resolution for $A$ yield quasi-isomorphic spaces of derived derivations. One has an isomorphism $H^0(\R\mathrm{Der}_B(A,M))\cong \hom_{\mathrm{Ho}(\cat{dga}_k / B)}(A, B \oplus M)$.
	\p We mimic the above constructions for pro-Artinian dgas. We will only be interested in the case when the base algebra $B$ is ungraded, which will avoid the need to define bimodules over pro-Artinian dgas in generality (although we remark on how to do this in \ref{probimods}). We give an example of such a situation in \ref{hoprolem}. Suppose that $B$ is an Artinian local $k$-algebra and that $\mathcal{A}$ is a pro-Artinian dga with a map to $B$. If $M$ is a dg $B$-bimodule, then $M$ is naturally a bimodule over $\varprojlim\mathcal{A}$.
	\begin{defn}
		Let $B$ be an Artinian local $k$-algebra and let $\mathcal{A}\to B$ be a pro-Artinian dga with a map to $B$. Let $M$ be a finite dg $B$-bimodule concentrated in nonnegative degrees. Note that $B \oplus M$ is still Artinian local. A \textbf{derivation} $\mathcal{A}\to M$ is a map $\mathcal{A}\to B \oplus M$ in the overcategory ${\proart}/B$. The set of all derivations $\mathcal{A}\to M$ is denoted $\mathrm{Der}_B(\mathcal{A},M)$.
	\end{defn}
	If $\mathcal{A}=\{\mathcal{A}_\alpha \}_\alpha$ with each $\mathcal{A}_\alpha $ Artinian local, then $\mathrm{Der}_B(\mathcal{A},M)\cong \varinjlim_\alpha \mathrm{Der}_B(\mathcal{A}_\alpha, M)$. Hence, $\mathrm{Der}_B(\mathcal{A},M)$ naturally acquires the structure of a dg vector space. We wish to define derived derivations as derivations from a resolution; before we do this we need to check that the definition makes sense.
	\begin{lem}\label{prolemma}
		Let $B$ be an Artinian local $k$-algebra and let $M$ be a finite dg $B$-bimodule concentrated in nonnegative degrees. The functor $\mathrm{Der}_B(-,M):\proart/B \to \cat{dgvect}_k$ preserves weak equivalences between cofibrant objects.
	\end{lem}
	\begin{proof}
		Let $\mathcal{C}$ be the category of finite dg $B$-bimodules concentrated in nonnegative degrees. The square-zero extension functor extends to a functor $B\oplus-:\cat{pro}\mathcal{C}\to \proart/B$. This functor has a left adjoint $\Omega(-)_B$, given by applying the noncommutative K\"ahler differentials functor levelwise. The functor $B\oplus-$ is clearly right Quillen and hence $\Omega(-)_B$ is left Quillen.
	\end{proof}
	\begin{rmk}\label{pronccot}
		Taking this argument seriously leads one to define the \textbf{pro-noncommutative cotangent complex} $\mathbb{L}(\mathcal{A})_B\in \mathrm{Ho}(\cat{pro}(B\text{-}\cat{bimod}))$ of an object $\mathcal{A}\in \proart/B$.
	\end{rmk}
	\begin{defn}
		Let $B$ be an Artinian local $k$-algebra and let $\mathcal{A}\to B$ be a pro-Artinian dga with a map to $B$. Let $Q\mathcal{A} \to \mathcal{A}$ be a cofibrant resolution. Let $M$ be a finite dg $B$-bimodule concentrated in nonnegative degrees. The space of \textbf{derived derivations} from $\mathcal{A}$ to $M$ is the dg vector space $$\R\mathrm{Der}_B(\mathcal{A},M)\coloneqq \mathrm{Der}_B(Q\mathcal{A},M).$$
	\end{defn}
	By \ref{prolemma}, different choices of resolution for $\mathcal A$ yield quasi-isomorphic spaces of derived derivations. One has an isomorphism $H^0(\R\mathrm{Der}_B(\mathcal{A},M))\cong \hom_{\mathrm{Ho}(\proart / B)}(\mathcal{A}, B \oplus M)$. The main technical result of this section is that the two notions of derived derivation match up for good dgas. We will first prove this for $B\cong k$, where the proof is simpler (because the action of $\mathfrak{m}_\mathcal A$ on $M$ is trivial), and then we will adapt the argument to general $B$ via filtering by the action of $\mathfrak{m}_B$ to reduce to the case $B\cong k$. Observe that any pro-Artinian dga $\mathcal{A}$ admits an augmentation $\mathcal{A}\to k$.
	\begin{prop}\label{prerder}
		Let $\mathcal{A}$ be a pro-Artinian dga. Let $M$ be a finite-dimensional dg $k$-vector space concentrated in nonpositive degrees. Assume that $A\coloneqq \varprojlim \mathcal{A}$ is cohomologically locally finite. Then there is a quasi-isomorphism $$\R\mathrm{Der}_{k}(\mathcal{A},M) \simeq \R\mathrm{Der}_{k}(A,M) .$$
	\end{prop}
	\begin{proof}
		The idea is that the cofibrant resolutions agree. For brevity, we will omit the bar notation for (co)augmentation (co)ideals. Consider first the space $\R\mathrm{Der}_{k}(A, M)$. Since $\Omega B A \to A$ is a cofibrant resolution, we have $\R\mathrm{Der}_{k}(A, M)$ quasi-isomorphic to $ \mathrm{Der}_k(\Omega B A, M)$. Now, $\Omega B A$ is freely generated by $BA[-1]$, so $\mathrm{Der}_k(\Omega B A, M) \cong \dgh_k(BA[-1],M)$ as dg vector spaces. Since $M$ is finite-dimensional, $\dgh_k(BA[-1],M)$ is the same as $A^![1]\otimes_k M$ (here is where we are using that $B\cong k$; the underlying graded vector spaces are always isomorphic but in general the differential of the right hand side acquires a twist from the action of $\mathfrak{m}_B$ on $M$). Consider now the space $\R\mathrm{Der}_{k}(\mathcal{A},M)$. We use the equivalence of $\proart$ with conilpotent dgcs, along with the fact that $C \to B \Omega C$ is a fibrant dgc resolution, to see that $\R\mathrm{Der}_{k}(\mathcal{A},M)$ is quasi-isomorphic to $ \mathrm{Der}_k((B\Omega(\varinjlim\mathcal{A}^*))^\sharp,M)$. The dgc $B\Omega(\varinjlim\mathcal{A}^*)$ is cofreely cogenerated by $\Omega(\varinjlim \mathcal{A}^*)[1]$, so that the pro-Artinian algebra $(B\Omega(\varinjlim \mathcal{A}^*))^\sharp$ is freely generated by the profinite vector space $\Omega(\varinjlim \mathcal{A}^*)^\sharp[-1]$. Hence we have isomorphisms of dg vector spaces$${\mathrm{Der}_k((B\Omega(\varinjlim\mathcal{A}^*))^\sharp,M)\cong  \dgh_{\cat{pro}(\cat{fdvect}_k)}(\Omega(\varinjlim \mathcal{A}^*)^\sharp[-1],M)\cong \dgh_k(M^*,\Omega(\varinjlim\mathcal{A}^*)[1])}.$$Again, because $M$ is finite-dimensional and $B\cong k$, this is isomorphic to $\Omega(\varinjlim \mathcal{A}^*)[1]\otimes_k M$. So it suffices to show that $\Omega(\varinjlim \mathcal{A}^*)$ and $A^!$ are quasi-isomorphic as dg vector spaces. This is similar to the proof of \ref{kdgood}: first note that $\Omega(\varinjlim \mathcal{A}^*) \cong \varinjlim \Omega( \mathcal{A}^*)$ because $\Omega$ is cocontinuous, and that $\Omega( \mathcal{A}^*) \cong \mathcal{A}^!$ because each level of $\mathcal{A}$ is Artinian local. Hence, as in \ref{kdgood}, $\Omega(\varinjlim \mathcal{A}^*)$ is quasi-isomorphic to $A^!$, as required. Note that this last fact uses that $BA$ is cohomologically locally finite, which is the only place we use the hypothesis that $A$ is cohomologically locally finite.
	\end{proof}
	Now we will extend the argument to cover all $B$.
	\begin{thm}\label{rder}
		Let $B$ be an Artinian local $k$-algebra and let $\mathcal{A}\to B$ be a pro-Artinian dga with a map to $B$. Let $M$ be a finite-dimensional $B$-bimodule concentrated in nonpositive degrees. Assume that $A\coloneqq \varprojlim \mathcal{A}$ is cohomologically locally finite. Then there is a quasi-isomorphism $$\R\mathrm{Der}_{B}(\mathcal{A},M) \simeq \R\mathrm{Der}_{B}(A,M) .$$
	\end{thm}
	\begin{proof}
		We follow the proof of \ref{prerder}. This time we have to care about twists. We still have a quasi-isomorphism $\R\mathrm{Der}_{k}(A, M)\simeq \dgh_k(BA[-1],M)$. However, the differential on $\dgh_k$ is twisted by the action of $B$ on $M$: explicitly, if $f\in \dgh_k(BA[-1],M)$, then $df$ gains an extra term $\Delta f$, defined by $\Delta f (v)=f(v_{(1)}).v_{(2)}+v_{(1)}.f(v_{(2)})$ where we are using Sweedler notation for the comultiplication $\Delta(v)=v_{(1)}\otimes v_{(2)}$. Across the isomorphism of underlying graded vector spaces $\dgh_k(BA[-1],M)\cong A^![1]\otimes_k M$, the twist $\Delta$ on $\dgh_k(BA[-1],M)$ corresponds to a twist in the differential on $A^![1]\otimes_k M$; let $A^![1]\otimes^\Delta_k M$ denote $A^![1]\otimes_k M$ equipped with this twisted differential. Filtering $M$ by the action of $\mathfrak{m}_B$ gives a finite filtration $F^p= A^![1]\otimes^\Delta_k M.\mathfrak{m}^p_B$ on $A^![1]\otimes^\Delta_k M$. The associated graded pieces are $\mathrm{gr}^p_F\coloneqq F^p/F^{p+1}\cong A^![1]\otimes^\Delta_k\mathrm{gr}^p_M$, where we put $\mathrm{gr}^p_M\coloneqq M.\mathfrak{m}^p_B/M.\mathfrak{m}^{p+1}_B$. One obtains a convergent spectral sequence $( A^![1]\otimes^\Delta_k\mathrm{gr}^p_M)^q \implies H^{p+q}(A^![1]\otimes^\Delta_k M)$. The twist in the differential of $A^![1]\otimes^\Delta_k M$ disappears upon passing to the associated graded pieces and so one has $\mathrm{gr}^p_F\cong A^![1]\otimes_k\mathrm{gr}^p_M$. In other words, the natural map $A^![1]\otimes_k M\to A^![1]\otimes^\Delta_k M$ is an isomorphism on associated graded pieces. By considering the same spectral sequence for $A^![1]\otimes_k M$, we see that the natural map is actually a quasi-isomorphism. Hence we get a quasi-isomorphism $\R\mathrm{Der}_{k}(A, M)\simeq A^![1]\otimes_k M$, as before. The argument to show that $\R\mathrm{Der}_{B}(\mathcal{A},M)\simeq\Omega(\varinjlim \mathcal{A}^*)[1]\otimes_k M$ is similar. The proof that $\Omega(\varinjlim \mathcal{A}^*)[1]\otimes_k M\cong A^![1]\otimes_k M$ is the same as before.
	\end{proof}
	\begin{rmk}
		Continuing on from \ref{pronccot}, the above proof gives a quasi-isomorphism between $\varprojlim\mathbb{L}(\mathcal{A})_B$ and $ \mathbb{L}(A)_B$. In other words, the (pro-)noncommutative cotangent complex functor commutes up to quasi-isomorphism with $\varprojlim$, as long as we assume some finiteness conditions.
	\end{rmk}
	\begin{rmk}\label{probimods}Let $A$ be a pro-Artinian dga. In the spirit of \cite{quillender}, one could define a \textbf{pro-$A$-bimodule} to be a group object in the category of pro-Artinian dgas with a map to $A$. Equivalently, this is an $A$-bimodule in the category of profinite dg vector spaces. In this framework, one can also define pro-derivations, pro-noncommutative K\"ahler differentials, and the pro-noncommutative cotangent complex. If the underlying vector space of a pro-bimodule $M$ is constant, then $M$ is a bimodule over some $A_\alpha$, and hence a bimodule over $A_\beta$ for all $ \beta\to\alpha$. The proofs of \ref{prerder} and \ref{rder} adapt to cover the case when $M$ is a constant pro-$A$-bimodule, and also provide comparisons between the cotangent complexes.
	\end{rmk}
	
	\section{Koszul duality for homotopy pro-Artinian algebras}
	The main result of this part is a characterisation of good dgas (\ref{goodchar}), which can also be thought of as a strictification result. Call a dga $A$ \textbf{homotopy pro-Artinian} if $HA$ is pro-Artinian (in the sense that it is a limit of a pro-Artinian dga). We prove that a certain class of homotopy pro-Artinian dgas, namely, those for which the pro-structure is that of the Postnikov tower, are good. We obtain as a corollary a Koszul duality result for this class of dgas. For a very general approach to some of the ideas of this part, see Lurie's Higher Algebra \cite[7.4]{lurieha}.
	\begin{defn}Let $A \to B$ be a map of dgas. Say that a map $A' \to A$ of dgas is a \textbf{homotopy square-zero extension} over $B$ if there is a $B$-bimodule $M$ such that $A'$ is the homotopy fibre product of a diagram of the form $A \xrightarrow{\delta} B \oplus M \xleftarrow{0} B$ where $\delta$ is a derived derivation and $0$ is the zero derivation.
	\end{defn}
	We will be interested in the case when $B=H^0A$, and $A$ is a good dga. In this case, one can lift the natural map $A \to B$ to a map of pro-Artinian dgas. Observe that since $\varprojlim: \proart \to \dga$ is exact (\ref{limisexact}), we may regard it as the homotopy limit $\holim$; we will use this without further acknowledgement.
	\begin{lem}\label{hoprolem}
		Let $\mathcal{A}\in \proart$. Assume that $H^0\varprojlim \mathcal{A}$ is finite-dimensional. Then there is a map of pro-Artinian dgas $\mathcal{A}\to H^0\varprojlim \mathcal{A}$.
	\end{lem}
	\begin{proof}
		Put $\mathcal{A}=\{\mathcal{A}_\alpha\}_\alpha$ with each $\mathcal{A}_\alpha$ Artinian. For all $\alpha$, there is a structure map $\mathcal{A}\to\mathcal{A}_\alpha$ and hence a map $\mathcal{A}\to H^0(\mathcal{A}_\alpha)$. These maps assemble into an element of the inverse limit $\varprojlim_\alpha\hom_{\proart}(\mathcal{A},H^0\mathcal{A}_\alpha)$. Because $\varprojlim$ is the homotopy limit we have an isomorphism $\varprojlim_\alpha H^0(\mathcal{A}_\alpha)\cong H^0\varprojlim \mathcal{A}$. Because this is Artinian by hypothesis, we get an isomorphism 
		\begin{align*}
		\varprojlim_\alpha\hom_{\proart}(\mathcal{A},H^0\mathcal{A}_\alpha)&\cong \hom_{\proart}(\mathcal{A},\varprojlim_\alpha H^0\mathcal{A}_\alpha)\\&\cong \hom_{\proart}(\mathcal{A},H^0\varprojlim\mathcal{A}).\qedhere
		\end{align*}
	\end{proof}
	Our main examples of homotopy square-zero extensions will be provided by truncations. 
	\begin{defn}
		If $A$ is a dga, set $A_n\coloneqq \tau_{\geq -n}(A)$, the good truncation to degrees above $-n$. Explicitly, we have $$(A_n)^j=\begin{cases} A^j & j>-n \\ \coker(d:A^{-n-1}\to A^n) & j=-n \\ 0 & j<-n \end{cases}$$One has $H^jA_n\cong H^jA$ if $j\geq -n$ and $H^jA_n=0$ if $j<-n$.
	\end{defn}
	\begin{lem}\label{hsqlem}
		Let $A \in \dga$ be a nonpositive dga. Then for every $n\geq 0$, the natural map $A_{n+1} \to A_n$ is a homotopy square-zero extension with base $H^0A$.
	\end{lem}
	\begin{proof}
		There is a homotopy fibre sequence of $A$-bimodules $$H^{-n-1}(A)[-n-1] \to A_{n+1}\to A_n$$indicating that we should take $M$ to be a shift of $H^{-n-1}(A)$. Indeed, this sequence gives a map $A_n \to H^{-n-1}(A)[-n-2]$ in the homotopy category of $A$-bimodules, and so we put $M\coloneqq H^{-n-1}(A)[-n-2]$. Let $C$ be the mapping cone of $H^{-n-1}(A)[-n-1] \to A_{n+1}$ and write $B\coloneqq H^0A$. As in the proof of \cite[1.45]{unifying}, $C$ admits the structure of a dga, quasi-isomorphic to $A_n$, and moreover $A_{n+1}$ is the strict pullback of the diagram $B \to B \oplus M \from C$. The map $C \to B\oplus M$ is a fibration, and hence $A_{n+1}$ is the homotopy pullback.
	\end{proof}
	\begin{lem}\label{hplem}
		Let $A \in \dga$ be a nonpositive dga. Suppose that for some $n$, $A_n$ is good and $H^{-n-1}(A)$ is finite-dimensional. Then $A_{n+1}$ is good.
	\end{lem}
	\begin{proof}
		It is clear that $A_{n+1}$ is cohomologically locally finite. So we just need to prove that $A_{n+1}$ is quasi-isomorphic to something in the image of $\varprojlim$. By \ref{hsqlem}, $A_{n+1}$ is the homotopy pullback of a diagram $J$ of the form $H^0(A)\to H^0(A)\oplus M \from A_n$, where $M$ is a finite module. By \ref{rder} and \ref{hoprolem}, we may view $J$ as a diagram in $\proart$; let $P\in\proart$ be the homotopy pullback. Since $\varprojlim:\proart \to\dga$ is the homotopy limit and homotopy limits commute, we see that $\varprojlim P$ is the homotopy pullback of $J$ considered as a diagram in $\dga$. But by \ref{hsqlem} this is precisely $A_{n+1}$.
	\end{proof}
	\begin{rmk}
		If $A_n$ is Artinian local, then $A_{n+1}$ is homotopy Artinian local, but neither $A_{n+1}$ nor $P$ need be Artinian local.
	\end{rmk}
	\begin{prop}\label{goodchar}
		Let $A \in \dga$ be a nonpositive dga. The following are equivalent:
		\begin{itemize}
			\item $A$ is good.
			\item $A$ is cohomologically locally finite and $H^0(A)$ is local.
		\end{itemize}
	\end{prop}
	\begin{proof}The forward direction is clear. For the backwards direction, assume that $A$ is cohomologically locally finite and that $H^0(A)$ is local. We may also assume that $A$ is cofibrant. Using \ref{hplem} inductively, it is easy to see that each truncation $A_n$ is good. In fact, one can say more: we obtain for each $n$ a pro-Artinian dga $P_n$ together with an isomorphism $A_n \to \varprojlim P_n$ in $\mathrm{Ho}(\dga)$. Using that $A$ is cofibrant, we obtain a dga map $A \to \varprojlim P_n$ making the obvious triangles commute. We have quasi-isomorphisms $A \simeq \holim_n \varprojlim P_n \simeq \varprojlim \holim_n P_n$, because $\varprojlim$ is the homotopy limit. Hence $A$ is quasi-isomorphic to something in the image of $\varprojlim$.
	\end{proof}
	Applying \ref{kdgood} immediately gives us the following:
	\begin{thm}\label{kdfin}
		Let $A \in \dga$ be a cohomologically locally finite dga such that $H^0(A)$ is local. Then $A$ is quasi-isomorphic to its double Koszul dual.
	\end{thm}

\begin{rmk}\label{andreyrmk}
Andrey Lazarev has suggested that the `correct' version of the preceding theorem should be something like the following. Let $A \to k$ be an augmented dga. Then under some mild conditions on $A$, the Koszul double dual $A \to A^{!!}$ is quasi-isomorphic to the Bousfield localisation of the right $A$-module $A$ with respect to the homology theory $M \mapsto \tor_*^A(M,k)$. Moreover, if $A$ is cohomologically locally finite with $H^0(A)$ local, then the Bousfield localisation of $A$ is $A$ again. To prove this, he suggests that one should show that the cobar spectral sequence associated to $A^{!!}$ converges to the cohomology of the localisation of $A$, in a similar manner to the convergence of the $E$-Adams spectral sequence. The relevant computations ought to be similar to those of Bousfield \cite{bousfieldloc} and Dwyer \cite{dwyerexotic}.
	\end{rmk}

\begin{rmk}
	If $A$ is a cohomologically locally finite nonpositively graded augmented dga, then the following are equivalent:
	\begin{enumerate}
		\item $H^0(A)$ is local.
		\item $A$ is quasi-isomorphic to the limit of a system of pro-Artinian dgas.
		\item $A$ is good.
		\item $A$ is quasi-isomorphic to its Koszul double dual.
		\item $A$ admits a `Koszul predual': there exists a dga $B$ and a quasi-isomorphism $A\simeq B^!$.
		\end{enumerate}
	The easiest way to see this is to observe that $$3.\iff 2.\implies1.\implies4.\implies5.\implies2.$$
	\end{rmk}
\begin{rmk}\label{parmk}
	Let $A$ be any augmented nonpositive dga. Let $B^\sharp(A)$ denote the \textbf{continuous Koszul dual}: one takes the bar construction on $A$ and then applies the $(-)^\sharp$ functor of \ref{csharp} to obtain a pro-Artinian dga. If one takes the levelwise Koszul dual to obtain a pro-Artinian dga $\left(B^\sharp(A)\right)^!$, it is clear that $\varprojlim\left(B^\sharp(A)\right)^!$ is quasi-isomorphic to $A$ by applying \ref{kdforart} levelwise. However, it is far from clear that the same applies when we forget that $B^\sharp A$ is pro-Artinian; i.e.\ take $(\varprojlim B^\sharp(A))^!$ instead.
	\end{rmk}
	\section{The pseudo-model category of good dgas}
	Using \ref{goodchar} to write a good dga as an iterated sequence of homotopy square-zero extensions, we will extend the results of \ref{rder} to give an equivalence of pseudo-model categories between good dgas and those pro-Artinian dgas whose limits are good (the `pregood' ones). Because neither of these (pseudo)-model categories are dg model categories, we need to first translate \ref{rder} into the language of simplicial mapping spaces.
	\begin{defn}
		Let $\mathcal{C}$ be a model category, and let $X,Y$ be two objects of $\mathcal{C}$. Following \mbox{{\cite[5.4.9]{hovey}}} write $\R\mathrm{Map}_{\mathcal{C}}({X},{Y}) \in \mathrm{Ho}(\cat{sSet})$ for the derived mapping space from $X$ to $Y$. 
	\end{defn}
	\begin{prop}\label{dugger}
		Let $\mathcal{C}$ be a combinatorial dg model category. Let $X,Y$ be two objects of $\mathcal{C}$. Denote their derived hom-complex by $\R\dgh_{\mathcal{C}}(X,Y)\in D(k)$. Then the quasi-isomorphism type of $\R\dgh_{\mathcal{C}}(X,Y)$ determines the weak homotopy type of the derived mapping space $\R\mathrm{Map}_{\mathcal{C}}({X},{Y})$.
	\end{prop}
	\begin{proof}By results of Dugger \cite{duggerspectra} a combinatorial dg model category is naturally enriched over the category of symmetric spectra. The basic idea is to identify unbounded dg vector spaces as spectrum objects in the category of nonpositive dg vector spaces (via shifting and good truncation), and then apply the Dold--Kan correspondence \cite[III.2]{goerssjardine} levelwise to end up with a spectrum object in simplicial sets. Composition is given by the Alexander--Whitney map \cite[8.5.4]{weibel}. See also Dugger and Shipley \cite{duggershipleyspectra} for an additive version where one ends up with spectrum objects in simplicial abelian groups. In particular, taking the zeroth level of the derived mapping spectra of $\mathcal{C}$ gives an enrichment of $\mathcal{C}$ over simplicial sets. For fibrant-cofibrant objects this enrichment must be weakly equivalent to the usual one. 
	\end{proof}
	\begin{cor}\label{dermapspace}
		Let $B$ be an Artinian local $k$-algebra and let $\mathcal{A}\to B$ be a pro-Artinian dga with a map to $B$. Let $M$ be a finite-dimensional $B$-bimodule concentrated in nonpositive degrees. Assume that $A\coloneqq \varprojlim \mathcal{A}$ is cohomologically locally finite. Then there is a weak equivalence $$\R\mathrm{Map}_{\proart/B}(\mathcal{A},B\oplus M) \simeq \R\mathrm{Map}_{\cat{dga}_k/B}(A,B\oplus M) .$$
	\end{cor}
	\begin{proof}
		Follows from \ref{rder} and \ref{dugger}.
	\end{proof}
	\begin{defn}
		Let $\cat{g}\dga \into \dga$ denote the full subcategory on good dgas. Call a pro-Artinian dga $A$ \textbf{pregood} if $\varprojlim A$ is good, and let $\cat{g}\proart \into \proart$ denote the full subcategory on pregood pro-Artinian dgas. 
	\end{defn}
	Note that a pro-Artinian dga is pregood if and only if for each $n$ the vector space $\varprojlim H^nA$ is finite. The following definition is a slight variant of \cite[4.1.1]{hag1}.
	\begin{defn}
		A \textbf{pseudo-model category} $C$ is a full subcategory of a model category $M$ such that $C$ is closed under weak equivalences and homotopy pullbacks in $M$. 
	\end{defn}
	We will soon show in \ref{pseudomods} that $\varprojlim:\cat{g}\proart \to \cat{g}\dga$ is a Quillen equivalence of pseudo-model categories. The key step in proving this will be to check that the derived mapping spaces agree, which generalises \ref{dermapspace}.
	\begin{prop}\label{dermap}
		Let $\mathcal{A}$, $\mathcal{A}'$ be pregood pro-Artinian dgas. Put $A\coloneqq \varprojlim \mathcal{A}$ and  $A'\coloneqq \varprojlim \mathcal{A}'$. There is a weak equivalence of derived mapping spaces $$\R\mathrm{Map}_{\proart}(\mathcal{A},\mathcal{A}')\simeq\R\mathrm{Map}_{\cat{dga}_k}({A},{A'}).$$
	\end{prop}
	\begin{proof}
		First note that by \ref{dermapspace}, this is the case when $A'$ is a square-zero extension of $k$ by a finite module. Moreover, \ref{hplem} writes $A'$ as an iterated sequence of finite homotopy square-zero extensions, starting with the ungraded Artinian local algebra $H^0(A')$. Since $\R\mathrm{Map}(A,-)$ preserves homotopy limits, it hence suffices to prove the claim in the case when $A'=H^0(A')$. By the same logic, it is enough to prove that any ungraded Artinian local algebra $A'$ is an iterated sequence of finite homotopy square-zero extensions of $k$. The tower $A'/\mathfrak{m}_{A'}^n$ exhibits $A'$ as an iterated sequence of finite (classical) square-zero extensions starting from $k$, so it is enough to prove that if $\pi$ is a square-zero extension of ungraded Artinian local algebras, then $\pi$ is also a homotopy square-zero extension. But this follows from the fact that $\pi$ is surjective, so we have a quasi-isomorphism $\ker (\pi) \simeq \mathrm{cocone}(\pi)$. 
	\end{proof}
	\begin{thm}\label{pseudomods}
		Both $\cat{g}\dga \into \cat{dga}$ and $\cat{g}\proart \into \cat{pro}(\cat{dgArt}_k)$ are pseudo-model categories, and $\varprojlim: \cat{g}\proart \to \cat{g}\dga$ is a Quillen equivalence.
	\end{thm}
	\begin{proof}It is clear that $\cat{g}\dga$ is closed under weak equivalences. Since $\varprojlim$ reflects weak equivalences by \ref{limreflects}, it follows that $\cat{g}\proart$ is also closed under weak equivalences. The closure of both under homotopy pullbacks follows exactly as in \ref{hplem}, using the equivalence of derived mapping spaces from \ref{dermap}. So both are pseudo-model categories. It is easy to see that $\varprojlim$ is right Quillen. By definition, it is also homotopy essentially surjective. It is homotopy fully faithful by \ref{dermap}.
	\end{proof}

\begin{cor}\label{qisocor}
	Let $X$ and $Y$ be two good dgas. Let $\mathcal{X}$ and $\mathcal{Y}$ be pregood pro-Artinian dgas with $\varprojlim \mathcal{X}\simeq X$ and $\varprojlim \mathcal{Y}\simeq Y$. If $X$ and $Y$ are quasi-isomorphic then $\mathcal{X}$ and $\mathcal{Y}$ are weakly equivalent.
	\end{cor}
\begin{proof}
	By \ref{pseudomods}, on the level of homotopy categories one has a functorial bijection $[\mathcal{X}, \mathcal{Y}]\cong [X,Y]$. So isomorphisms on the right hand side give isomorphisms on the left hand side.
	\end{proof}
	
\chapter{Deformation theory}\label{defmthy}

In some sense, commutative formal deformation theory in characteristic zero is about the Koszul duality between the commutative and Lie operads. Indeed, given a commutative deformation problem, one expects it to be `controlled' in some way by a differential graded Lie algebra (\textbf{dgla}). This philosophy is originally due to Deligne, and first appears in print in a paper of Goldman and Millson \cite{goldmanmillson}. Hinich \cite{hinich} viewed deformation problems as equivalent to certain dg coalgebras, which are equivalent to dglas by a Chevalley--Eilenberg type construction, and the correspondence between commutative deformation problems and dglas was made precise by later work of Pridham \cite{unifying} and Lurie \cite{luriedagx}. Correspondingly, in view of the Koszul self-duality of the associative algebra operad, one should expect noncommutative deformation problems to be controlled by noncommutative dgas, and indeed this is true \cite[\S3]{luriedagx}.

\p In this chapter, we will make some of the above statements explicit. We will work primarily with deformation functors valued in simplicial sets. We define the Maurer--Cartan functor and Deligne functor associated to a dga. We study deformations of modules, and obtain some prorepresentability results. We rigidify by considering framed deformations -- a framed deformation of a module $X$ over a dga $A$ is essentially a deformation of $X$ that restricts to the trivial deformation of the underlying vector space of $X$ -- and give a prorepresentability result for framed deformations too. Framings correspond to nonunital dgas, and correspondingly we make use of nonunital dgas throughout.

\p By taking homotopy limits, the formalism of functors valued in $\sset$ allows us to deform modules over pro-Artinian dgas. We set up the theory of prodeformations, and show that our prorepresentability statements give us a universal prodeformation. We finish by tracking the universal prodeformation across quasi-equivalences.

\p Artinian local dgas will be concentrated in nonpositive degrees throughout this chapter. If $\Gamma$ is an Artinian local dga, denote its maximal ideal by $\mathfrak{m}_\Gamma$. The statements we make are not true when one considers Artinian dgas that may have positive degree parts: loosely, elements of positive degree correspond to `stacky' phenomena and introduce extra automorphisms. For example, a general Artinian dga $\Gamma$ may have nontrivial Maurer--Cartan elements. If the results of this chapter remained true in generality, then the $k$-vector space $k$ ought to have as many deformations over $\Gamma$ as there are Maurer--Cartan elements of $\Gamma$.

\p We will mention dglas for motivational purposes only; these are dg vector spaces together with a graded Lie bracket satisfying the graded Leibniz identity with respect to the differential. For more about commutative deformation theory in characteristic zero via dglas, one should refer to the papers of Manetti \cite{manetti,manettidgla}. 

\p We remark that all of the results of this chapter are true when $k$ has positive characteristic: key here is that we are deforming over nonpositive noncommutative dgas, which are Quillen equivalent to simplicial $k$-algebras. In positive characteristic the equivalence between nonpositive cdgas and simplicial commutative algebras breaks down. We note that in positive characteristic, a version of the Lurie--Pridham correspondence for commutative formal moduli problems has recently been proved by Brantner and Mathew \cite{partitionlie}.
	
	\section{The Maurer--Cartan and Deligne functors}
	
	\begin{defn}
		Let $U$ be a (possibly nonunital) dga. The set of \textbf{Maurer--Cartan elements} (or just \textbf{MC elements}) of $U$ is the set $$\mcs(U)\coloneqq \{x \in U^1: \ dx+x^2=0\}.$$
	\end{defn}
	\begin{rmk}
		Note that if $U$ is unital, with differential $d$, then $x \in \mcs(U)$ if and only if the map $u\mapsto d(u)+xu$ is a differential on $U$.
	\end{rmk}
	\begin{rmk}
		A (possibly nonunital) dga $U$ canonically becomes a dgla when equipped with the commutator bracket, and the set of MC elements of the dgla $U$ is the same as the set of MC elements of the dga $U$.
	\end{rmk}

	\begin{defn}
		Let $E$ be a dga. The \textbf{Maurer--Cartan functor} $$\mc(E):\dgart \to \cat{Set}$$ sends an Artinian dga $\Gamma$ to the set $\mc(E)(\Gamma)\coloneqq \mcs(E\otimes \mathfrak{m}_\Gamma)$. Note that this is a functor because a map of Artinian dgas necessarily preserves the maximal ideal.
	\end{defn}
	
	\begin{defn}
		Let $E$ be a dga. The \textbf{gauge group} functor $$\ggr(E):\dgart \to \cat{Grp}$$ sends $\Gamma$ to the set $1+(E\otimes \mathfrak{m}_\Gamma)^0$, which is a group under multiplication.
	\end{defn}
	
	\begin{prop}
		Let $E$ be a nonunital dga and $\Gamma$ an Artinian local dga. Then  $\ggr(E)(\Gamma)$ acts on $\mc(E)(\Gamma)$ via the formula $g.x=gxg^{-1}+gd(g^{-1})$.
	\end{prop}
	\begin{proof}
		This is an easy verification.
	\end{proof}
	\begin{rmk}Regarding $d+x$ as a twisted differential on $E\otimes \mathfrak{m}_\Gamma$, the action of the gauge group is the conjugation action on the space of differentials.
	\end{rmk}

	\begin{rmk}
		If $L$ is a dgla, its gauge group has as elements formal symbols $\exp(a)$ for $a \in L^0$, and multiplication given by the Baker--Campbell--Hausdorff formula \cite{manettidgla}. If $E\otimes \mathfrak{m}_\Gamma$ is made into a dgla using the commutator bracket then its dgla gauge group is isomorphic to the gauge group defined above, via the map that sends each formal exponential $\exp(a)$ to the sum $\sum_n\frac{a^n}{n^!}$, which exists because $\mathfrak{m}_\Gamma$ is nilpotent. Note that the characteristic zero assumption is necessary here. The exponential of the dgla gauge action \cite[V.4]{manetti} is the gauge action described above.
	\end{rmk}

	\begin{defn}
		Let $E$ be a dga. The \textbf{Deligne functor} is the quotient functor $$\del(E)\coloneqq \mc(E)/\ggr(E).$$
	\end{defn}
	Sometimes, $\del(E)$ is called the deformation functor associated to $E$.
	
	\begin{rmk}
		By taking the groupoid quotient rather than the set quotient, one can immediately enhance $\del$ to a groupoid-valued functor. However, we will see that $\del$ has a natural enhancement to a functor $\sdel$ valued in simplicial sets, and we would like the groupoid $\del$ to be the 1-truncation of $\sdel$. However, nontrivial 2-cells in $\sdel$ induce homotopies between gauges, and one has to quotient these out to get the correct fundamental groupoid; see \cite{ELO} or \cite[Proof of 3.2]{manettidgla} in the commutative setting. In the literature, both groupoid-valued functors are referred to as the \textbf{Deligne groupoid}. When deforming along ungraded Artinian algebras the two definitions coincide, so the difference between them only becomes apparent when deforming along genuinely derived objects.
	\end{rmk}
	
	\begin{prop}[{\cite[8.1]{ELO}}]
		If $E$ and $E'$ are quasi-isomorphic dgas then the functors $\del(E)$ and $\del(E')$ are isomorphic.
	\end{prop}
	
	Now we enhance all of our constructions to functors valued in simplicial sets. The loose idea is to take (co)simplicial resolutions to get simplicial mapping spaces, as in e.g.\ \cite[5.4.9]{hovey}. We use the explicit simplicial enhancement of Hinich \cite{hinichhom}, who is generalising the work of Bousfield and Gugenheim \cite{bousfieldgugenheim}.
	
	\begin{defn}
		Let $\pdf$ denote the simplicial cdga of polynomial differential forms on the standard cosimplicial space $\Delta^\bullet$; see e.g.\ \cite[\S1]{bousfieldgugenheim} or  \cite[4.8.1]{hinichhom} for an explicit definition.
	\end{defn}
	\begin{prop}[{Polynomial Poincar\'e Lemma \cite[1.3]{bousfieldgugenheim}}]\label{poincare}
		The simplicial dga $\pdf$ is quasi-isomorphic to the constant simplicial dga $k$.
	\end{prop}
	
	\begin{defn}
		Let $E$ be a dga. The \textbf{simplicial Maurer--Cartan functor} $\smc$ sends $E$ to the simplicial set $\smc(E)\coloneqq \mc(E\otimes \pdf)$.
	\end{defn}
	Unwinding the definitions, we hence have $\smc(E)(\Gamma)=\mcs(E\otimes \pdf \otimes \mathfrak{m}_\Gamma)$. 
	\begin{rmk}
		It is not true that $\mc\cong \pi_0\smc$, because the right-hand side has elements identified by homotopies coming from 1-simplices in $\smc$. All we have is a quotient map $\mc \to \pi_0 \smc$. In fact, $\pi_0 \smc$ is $\del$, which follows by combining \ref{pi0del} and \ref{smcisdel} below.
	\end{rmk}

	\begin{prop}[{\cite[2.18]{jonddefsartin}}]\label{mcinvtqiso}
		If $E$ and $E'$ are quasi-isomorphic dgas then the functors $\smc(E)$ and $\smc(E')$ are weakly equivalent.
		
	\end{prop}
	
	\begin{defn}
		The \textbf{simplicial Deligne functor}	is the homotopy quotient\footnote{See \cite[Chapter V]{goerssjardine} or \cite[1.23]{jonddefsartin} for the definition of homotopy quotients.}$$\sdel(E)\coloneqq [\smc(E)/\ggr(E)].$$ 
	\end{defn}
	
	\begin{lem}\label{pi0del}
		There is an isomorphism	$\del\cong \pi_0\sdel$.
	\end{lem}
	\begin{proof}
		As in \cite[1.27]{jonddefsartin}, this follows by considering the long exact sequence of homotopy groups associated to the fibration $X \to [X/G] \to BG$.
	\end{proof}

	\begin{prop}[{\cite[2.21]{jonddefsartin}}]\label{smcisdel}
		The quotient map $\smc \to \sdel$ is a weak equivalence.
	\end{prop}
	\begin{rmk}
		The idea of the proof is that $\ggr(E \otimes \pdf)$ is contractible and so taking the homotopy quotient does not affect the weak equivalence type of $\smc$.
	\end{rmk}

	\section{Deformations of modules}
	Recall that an underived deformation of an $A$-module $X$ over an Artinian local ring $\Gamma$ is an $A\otimes \Gamma$-module $\tilde X$ that reduces to $X$ modulo $\mathfrak{m}_\Gamma$. A derived deformation is defined similarly:
	\begin{defn}
		Let $A$ be a dga and $X$ an $A$-module. Let $\Gamma$ be an Artinian local dga. A \textbf{derived deformation} of $X$ over $\Gamma$ is a pair $(\tilde X, f)$ where $\tilde X$ is an $A\otimes \Gamma$-module and $f:\tilde X \lot_\Gamma k \to X$ is an isomorphism in $D(A)$. An \textbf{isomorphism} of derived deformations is an isomorphism $\phi:\tilde X_1 \to \tilde X_2$ in $D(A\otimes\Gamma)$ such that $f_1 = f_2 \circ ({\phi \lot_\Gamma k})$.
	\end{defn}
	Deformations are functorial with respect to algebra maps: given a map $\Gamma \to\Gamma'$ of Artinian local dgas, and a derived deformation $\tilde X$ of $X$ over $\Gamma$, then the derived base change $\tilde X \lot_\Gamma \Gamma'$ is a derived deformation of $X$ over $\Gamma'$.
	\begin{defn}
		Let $A$ be a dga and $X$ an $A$-module. The functor $\defm_A(X):\dgart \to \cat{Set}$ sends an Artinian local dga $\Gamma$ to the set $$\defm_A(X)(\Gamma)\coloneqq \frac{\{\text{derived deformations of }X \text{ over }\Gamma\}}{(\text{isomorphism})}.$$We will just write $\defm(X)$ if there is no ambiguity.
	\end{defn}	
	
	We will want to describe our simplicial set valued deformation functor as a homotopy pullback of dg categories; we first describe a recipe for turning dg categories into simplicial sets. The basic idea is to truncate, apply Dold--Kan to each of the morphism spaces, take the `underlying simplicial groupoid', and then take the homotopy coherent nerve.
	\begin{defn}
		A \textbf{simplicial category} is a category enriched in $\sset$. The category of all simplicial categories is denoted $\cat{sCat}$.
		\end{defn}
	\begin{rmk}
		Note that every simplicial category is a simplicial object in $\cat{Cat}$, but not every simplicial object in $\cat{Cat}$ is a simplicial category.
		\end{rmk}
	\begin{defn}
		Let $\mathcal{C}$ be a dg category. Let $\mathcal{C}_{\leq 0}$ denote the associated dg category obtained by taking the good truncation to nonpositive degrees of the morphism complexes. Let $\mathcal{C}_s$ denote the simplicial category obtained from $\mathcal{C}_{\leq 0}$ by applying the Dold--Kan correspondence \cite[III.2]{goerssjardine} to the morphism complexes. Composition is given by the Alexander--Whitney map \cite[8.5.4]{weibel}.
	\end{defn}	
Recall that associated to a simplicial category $\mathcal{C}$, there is an associated category $\pi_0\mathcal{C}$ with the same objects as $\mathcal{C}$ and whose morphism spaces are given by taking $\pi_0$ of the morphism complexes in $\mathcal{C}$. Recall from \cite{bergnermodelscat} that the category of simplicial categories admits a model structure where the weak equivalences are the DK-equivalences: those functors which induce weak equivalences on derived mapping spaces and which induce isomorphisms on $\pi_0$ (cf.\ Tabuada's model structure on dg categories \ref{tabmod}). The fibrant simplicial categories are precisely those enriched in Kan complexes.
	\begin{defn}	
		If $\mathcal{C}$ is a simplicial category, let $c(\mathcal{C})$ denote the subcategory on those morphisms which induce isomorphisms on $\pi_0\mathcal{C}$; in \cite{bergnermodelscat} these are called homotopy equivalences. 
	\end{defn}	

	We think of $c$ as a sort of core functor. Note that $c(\mathcal{C})$ is quasi-equivalent to a simplicial groupoid in the sense of \cite[V.7]{goerssjardine}. 
	\begin{defn}		
		Given a simplicial category $\mathcal{C}$, let $\bar W \mathcal{C}$ denote the right adjoint of Illusie's Dec functor $\bar W$ applied to the nerve of $\mathcal{C}$. See \cite[1.6]{jondmodss} for a concrete definition.
	\end{defn}
	\begin{rmk}\label{diagw}
		In \cite{diagonalw} it is proved that for a bisimplicial set $X$, the canonical morphism $\mathrm{diag}X \to \bar W X$ is a weak equivalence. See \cite[1.7]{jondmodss} for further discussion. If $\mathcal{G}$ is a simplicial groupoid, then $\bar W \mathcal{G}$ is weakly equivalent to the homotopy coherent nerve of $\mathcal{G}$ \cite{hinichnerve}.
	\end{rmk}
\begin{defn}
	Let $\mathcal{C}$ be a dg category. Write $\mathcal{W}(\mathcal C)\coloneqq \bar W(c(\mathcal{C}_s))$. It is clear that we obtain a functor $\mathcal{W}: \cat{dgCat}_k \to \sset$.
	\end{defn}
\begin{rmk}\label{wholims}
		The functor $\mathcal{W}$ is in fact a right derived functor. The truncation functor is right Quillen, and Dold--Kan realisation is right Quillen. Because every dg category is fibrant, it follows that $\mathcal{C}\mapsto \mathcal{C}_s$ is a right derived functor.	Moreover, the composition $\bar W c$ is a right derived functor, because it is weakly equivalent to the derived mapping space functor $\R\mathrm{Map}(\bullet,-)$ from the initial simplicial category. Hence, $\mathcal{W}$ is a composition of right derived functors and so is itself a right derived functor.
	\end{rmk}

	\begin{defn}
		Let $A$ be a dga and let $X$ be an $A$-module. Let $X_\mathrm{dg}$ denote the dga $\R\enn_A(X)$ considered as a one-object dg category. There is an obvious inclusion dg functor $X_\mathrm{dg} \into D_\mathrm{dg}(A)$ whose image is $X$. If $\Gamma\in \dgart$ then there is a `reduce modulo $\mathfrak{m}_\Gamma$' dg functor $D_\mathrm{dg}(A\otimes \Gamma) \to D_\mathrm{dg}(A)$ which sends an $A\otimes\Gamma$-module $M$ to the $A$-module $M\lot_\Gamma k$. Let $\mathrm{dgDef}_A(X)(\Gamma)$ be the homotopy fibre product of dg categories $$\mathrm{dgDef}(X)(\Gamma)\coloneqq  X_\mathrm{dg} \times^h_{D_\mathrm{dg} (A)} D_\mathrm{dg} (A \otimes \Gamma).$$
		Let $\sdefm_A(X)(\Gamma)$ denote the simplicial set $\mathcal{W}\left(\mathrm{dgDef}_A(X)(\Gamma)\right)$. 
	\end{defn}
It is easy to check that $\sdefm_A(X)$ is a functor from $\dgart$ to $\sset$. Observe that the dg category $\mathrm{dgDef}(X)$ is clearly `too big', since it contains many maps which are not isomorphisms after passing to the homotopy category. This is why we take the core; we could just as well have taken the `dg core' of $\mathrm{dgDef}(X)$, converted it to a simplicial category, and applied $\bar W$.

	\begin{rmk}\label{dgffdefs}
		If $A$ and $A'$ are dgas with a fully faithful dg functor $D_\mathrm{dg}(A) \into D_\mathrm{dg}(A')$, then we obtain a weak equivalence $\sdefm_A(X) \xrightarrow{\simeq}\sdefm_{A'}(X)$.
		\end{rmk}
	
	The following is known to the experts in deformation theory:
	\begin{thm}[{see e.g.\ \cite[4.2.6]{dinatalethesis} or \cite[4.6]{jonddefsartin}}]\label{defisdel}
		Let $A$ be a dga and $X$ an $A$-module. Let $E\coloneqq \R\enn_A(X)$ be the endomorphism dga of $X$. Then there is a weak equivalence of $\sset$-valued functors $\sdefm_A(X) \xrightarrow{\simeq}\sdel(E)$.
	\end{thm}
	\begin{proof}We give a sketch proof. The basic idea is that if $X$ is cofibrant then a deformation of $X$ is a deformation of the differential on $X\otimes_k \Gamma$, which is exactly a MC element of $E\otimes \mathfrak{m}_\Gamma$. The proof of \cite[3.13]{jonddefsartin} shows that the simplicial groupoid of deformations of $X$ is the same as the simplicial subgroupoid of $\sdel(E)$ on constant objects. By \cite[\S4.1]{jonddefsartin}, these two are both derived deformation functors and one can compare the induced map on tangent spaces to show that the two functors are weakly equivalent.
	\end{proof}	
	\begin{rmk}
		The statement on $\pi_0$ that $\defm_A(X)\cong \del(E)$ is well-known; see for example Efimov--Lunts--Orlov \cite{ELO} who actually prove an equivalence of fundamental groupoids.
	\end{rmk}
	\begin{rmk}
		We remark that the statement of the theorem makes sense because the Deligne functor $\sdel$ sends dga quasi-isomorphisms to weak equivalences by \ref{mcinvtqiso} and \ref{smcisdel}.	
	\end{rmk}
	
	\begin{cor}\label{defismc}
		Let $A$ be a dga and $X$ an $A$-module. Let $E\coloneqq \R\enn_A(X)$ be the endomorphism dga of $X$. Then the $\sset$-valued functors $\sdefm_A(X)$ and $\smc(E)$ are weakly equivalent.
	\end{cor}
	\begin{proof}
		Combine \ref{defisdel} with \ref{smcisdel}.
	\end{proof}

	\section{Prorepresentability}
	We prove a prorepresentability statement for set-valued functors, and then we enhance this to $\sset$-valued functors. Essentially everything we use here can be found in Loday--Vallette \cite[Chapter 2]{lodayvallette}. We will need to use nonunital dgas in order to get the correct prorepresentability statements; we will later see that the use of nonunital dgas can be avoided if one rigidifies to consider framed deformations.
	\begin{defn}[{see e.g.\ \cite[6.2]{positselski}}]
		Let $E$ be a nonunital dga and let $C$ be a nonunital dgc. Then the complex $\hom_k(C,E)$ of $k$-vector spaces is a nonunital dga under the product given by $fg\coloneqq \mu_E \circ (f\otimes g)\circ \Delta_C$. This dga is the \textbf{convolution algebra}. A Maurer--Cartan element of the convolution algebra is a \textbf{nonunital twisting morphism}; the set of all nonunital twisting morphisms is denoted $\mathrm{Tw}(C,E)$.
	\end{defn}
\begin{rmk}
	In the (co)augmented setting, one should add the condition that twisting morphisms are zero when composed with the augmentation or coaugmentation.
	\end{rmk}
	\begin{lem}\label{twnumc}
		Let $E, Z$ be nonunital dgas, with $Z$ finite-dimensional. Then there is a natural isomorphism $$\mathrm{Tw}(Z^*,E)\cong \mcs(E \otimes Z).$$
	\end{lem}
	\begin{proof}
		There is a standard linear isomorphism $E \otimes Z \to \hom_k(Z^*,E)$, and one can check that this is a map of nonunital dgas after equipping $\hom_k(Z^*,E)$ with the convolution product. Hence the MC elements of both sides agree.
	\end{proof}

\begin{defn}
	Let $E$ be a nonunital dga. The \textbf{nonunital bar construction} is the (coaugmented) dgc $$B_{\mathrm{nu}}(E)\coloneqq B(E\oplus k).$$
	\end{defn}
	We caution that if $E$ is an augmented dga, then $B_{\mathrm{nu}}(E)$ does not agree with $B(E)$ as $B_{\mathrm{nu}}(E)$ will contain elements corresponding to the unit of $E$.
	\begin{defn}
		Let $C$ be a nonunital dgc. The \textbf{nonunital cobar construction} is the (augmented) dga $$\Omega_{\mathrm{nu}}(C)\coloneqq \Omega(C\oplus k).$$
	\end{defn}

	The functor of twisting morphisms is (up to units) representable on either side:
	\begin{thm}[{\cite[2.2.6]{lodayvallette}}]\label{barcobaradj}
		If $E$ is a nonunital dga and $C$ is a noncounital conilpotent dgc, then there are natural isomorphisms $$\hom_{\cat{aug.dga}}(\Omega_{\mathrm{nu}} C, E\oplus k)\cong\mathrm{Tw}(C,E)\cong \hom_{\cat{cndgc}_k}(C\oplus k, B_{\mathrm{nu}}E).$$
	\end{thm}
	We recall from \ref{csharp} that if $C$ is a (counital) conilpotent dgc then $C^\sharp$ denotes the pro-Artinian dga constructed by levelwise dualising the filtered system of finite sub-dgcs of $C$. If $E$ is a nonunital dga, write $B_{\mathrm{nu}}^\sharp E\coloneqq (B_{\mathrm{nu}}E)^\sharp$ for the \textbf{continuous nonunital Koszul dual}.
	
	\begin{prop}\label{mcprorep}
		Let $E$ be a nonunital dga. Then the functor $\mc(E)$ is prorepresented by $B_{\mathrm{nu}}^\sharp E$, in the sense that $\mc(E)$ and $\hom_{\cat{pro}(\cat{dgArt}_k)}(B_{\mathrm{nu}}^\sharp E, -)$ are naturally isomorphic.
	\end{prop}
	\begin{proof}
		Let $\Gamma$ be a nonpositive Artinian local dga. It is easy to see that $\Gamma^*$ is a conilpotent dgc. By \ref{sharpprop}, we have an isomorphism $$\hom_{\cat{cndgc}_k}(\Gamma^*, B_{\mathrm{nu}}E) \cong \hom_{\cat{pro}(\cat{dgArt}_k)}(B_{\mathrm{nu}}^\sharp E, \Gamma^{*\sharp}).$$Because $\Gamma^*$ is a finite-dimensional dgc we have isomorphisms $\Gamma^{*\sharp}\cong \Gamma^{**}\cong \Gamma$ and it follows that we have isomorphisms $$\hom_{\cat{pro}(\cat{dgArt}_k)}(B_{\mathrm{nu}}^\sharp E, \Gamma)\cong \hom_{\cat{cndgc}_k}(\Gamma^*, B_{\mathrm{nu}}E)\cong \mathrm{Tw}(\mathfrak{m}_\Gamma^*,E)$$where the second isomorphism is \ref{barcobaradj}. By \ref{twnumc} we have an isomorphism $$\mathrm{Tw}(\mathfrak{m}_\Gamma^*,E)\cong \mcs(E \otimes \mathfrak{m}_\Gamma)$$and so we are done.
	\end{proof}
	\begin{rmk}\label{nurmk}Suppose now that $E$ is an augmented dga, with augmentation ideal $\bar E$. Then \ref{mcprorep} gives an isomorphism $$\hom_{\cat{pro}(\cat{dgArt}_k)}(B^\sharp E, \Gamma)\cong \mcs(\bar E \otimes \mathfrak{m}_\Gamma).$$However, because $\Gamma$ is nonpositive, the nonunital dga $E \otimes \mathfrak{m}_\Gamma$ has no MC elements of the form $1\otimes \gamma$. So the inclusion $\bar E \otimes \mathfrak{m}_\Gamma \into E \otimes \mathfrak{m}_\Gamma$ induces an isomorphism $$\mcs(\bar E \otimes \mathfrak{m}_\Gamma)\xrightarrow{\cong}\mcs( E \otimes \mathfrak{m}_\Gamma)\eqqcolon \mc(E)(\Gamma).$$So on the level of set-valued functors, $B^\sharp E$ also prorepresents. So why did we bother using nonunital dgas? The answer is that the inclusion $\bar E \otimes \mathfrak{m}_\Gamma \into E \otimes \mathfrak{m}_\Gamma$ does not necessarily induce a weak equivalence after taking $\smc$. We do not in general even get an isomorphism on $\pi_0$ (recall that $\mc$ is not $\pi_0\smc$).
	\end{rmk}

	Now that we have our prorepresentability result, we will enhance it to $\sset$-valued functors. This will not be too hard; we just need to identify the correct simplicial mapping spaces in $\proart$. Note that if $\Gamma$ is an Artinian dga, then $\pdf \otimes \Gamma$ will not be a simplicial Artinian dga, so the answer is not as simple as `replace $\Gamma$ by a simplicial resolution'. However, we are saved by the Quillen equivalence $\Omega \dashv B$:

	\begin{thm}\label{smcisrmap}
		Let $E$ be a nonunital dga. Then there is a weak equivalence of functors $$\smc(E)\simeq \R\mathrm{Map}_{\cat{pro}(\cat{dgArt}_k)}(B_{\mathrm{nu}}^\sharp E, -).$$
		\end{thm}
	\begin{proof}
	If $E$ is a nonunital dga write $E^\bullet$ for the simplicial nonunital dga $E^\bullet\coloneqq E \otimes \pdf$. By definition, we have $$\smc(E)(\Gamma)\coloneqq \mcs(E^\bullet\otimes\mathfrak{m}_\Gamma).$$But applying \ref{barcobaradj} levelwise we see that we have an isomorphism $$\mcs(E^\bullet\otimes\mathfrak{m}_\Gamma)\cong \hom_{\cat{aug.dga}}(\Omega_{\mathrm{nu}}(\mathfrak{m}_\Gamma^*),E^\bullet \oplus \mathrm{const}(k))\cong \hom_{\cat{aug.dga}}(\Omega(\Gamma^*),E^\bullet \oplus \mathrm{const}(k))$$where $\mathrm{const}(k)$ denotes the constant simplicial dga on $k$. But because $\mathrm{const}(k) \to\pdf$ is a levelwise quasi-isomorphism by \ref{poincare}, we have a quasi-isomorphism of simplicial dgas $E^\bullet \oplus \mathrm{const}(k) \simeq (E\oplus k)^\bullet$. Because $\Omega(\Gamma^*)$ is cofibrant and all dgas are fibrant, it follows that we have a weak equivalence $$\hom_{\cat{aug.dga}}(\Omega(\Gamma^*),E^\bullet \oplus \mathrm{const}(k))\simeq \R\mathrm{Map}_{\cat{dga}_k}(\Omega(\Gamma^*),E\oplus k).$$Now because $\Omega$ is part of a Quillen equivalence by \ref{bcquillen} we have $$\R\mathrm{Map}_{\cat{dga}_k}(\Omega (\Gamma^*),E\oplus k)\simeq \R\mathrm{Map}_{\cat{cndgc}_k}(\Gamma^*,B_{\mathrm{nu}}E).$$But because $(-)^\sharp$ preserves weak equivalences and fibrations by \ref{sharpquillen} and is itself an equivalence, we have $$ \R\mathrm{Map}_{\cat{cndgc}_k}(\Gamma^*,B_{\mathrm{nu}}E)\simeq \R\mathrm{Map}_{\cat{pro}(\cat{dgArt}_k)}(B_{\mathrm{nu}}E,\Gamma^{*\sharp}).$$As before we have $\Gamma^{*\sharp}\cong \Gamma$ and hence we are done.
		\end{proof}

	\begin{cor}\label{proreps}
			Let $A$ be a dga and $X$ an $A$-module. Let $E\coloneqq \R\enn_A(X)$ be the endomorphism dga of $X$. Then the $\sset$-valued functors $\sdefm_A(X)$ and $\R\mathrm{Map}_{\cat{pro}(\cat{dgArt}_k)}(B_{\mathrm{nu}}^\sharp E, -)$ are weakly equivalent.
	\end{cor}	
\begin{proof}
	Combine \ref{defismc} with \ref{smcisrmap}.
	\end{proof}

	One might expect that if a pro-Artinian dga $P$ prorepresents a derived deformation functor, then the pro-Artinian algebra $H^0(P)$ prorepresents the associated classical deformation functor obtained by truncation. Indeed, this is the case, at least for the set-valued deformation functor associated to a module over a ring:

\begin{prop}\label{prorepremk}
	Suppose that $A$ is a $k$-algebra and $X$ an $A$-module. Let $\mathrm{clDef}_A(X):\cat{Art}_k \to \cat{Set}$ denote the set-valued functor of noncommutative classical deformations of $X$, in the sense of e.g.\ {\normalfont \cite[2.4]{enhancements}}. Let $\mathrm{Def}_A(X):\cat{dgArt}^{\leq 0}_k \to \cat{Set}$ denote the set-valued functor of noncommutative derived deformations, and suppose that it is prorepresented by a pro-Artinian dga $P$. Then $\mathrm{clDef}_A(X)$ is prorepresented by the pro-Artinian algebra $H^0(P)$.
\end{prop}
\begin{proof}
	Let $\Gamma \in \cat{Art}_k$. By assumption, we have an isomorphism $\mathrm{Def}_A(X)(\Gamma)\cong \hom_{\proart}(P,\Gamma)$, which by the inclusion-truncation adjunction applied levelwise is isomorphic to $\hom_{\cat{pro}(\cat{Art}_k)}(H^0P,\Gamma)$. So we need to prove that $\mathrm{Def}_A(X)(\Gamma)\cong \mathrm{clDef}_A(X)(\Gamma)$. Let $\tilde X$ be an underived deformation of $X$ over $\Gamma$, i.e.\ an $A\otimes\Gamma$-module, flat over $\Gamma$, which reduces to $X$ modulo $\mathfrak{m}_\Gamma$ (equivalently, such that $\tilde{X}\otimes_\Gamma k \cong X$). It is easy to see that $\tilde X \lot_\Gamma k \cong X$ inside the derived category $D(A\otimes \Gamma)$. Hence $\tilde X$ is a derived deformation of $X$. Moreover, if two underived deformations are isomorphic, they are clearly isomorphic as derived deformations, and hence we obtain a map of sets $\Phi:\mathrm{clDef}_A(X)(\Gamma)\to \mathrm{Def}_A(X)(\Gamma)$. It is injective, because $A\otimes\Gamma \text{-}\cat{mod}$ embeds in $D(A\otimes \Gamma)$. Observe that if $\tilde X\in D(A\otimes \Gamma)$ is a derived deformation of $X$ over $\Gamma$, then it must actually be an $A\otimes\Gamma$-module, concentrated in degree zero. Because we have $\tilde X \lot_\Gamma k \simeq X$, we have $\tor^\Gamma_i(\tilde X,k)\cong 0$ for $i>0$. Because $\Gamma$ is local Artinian, this implies Tor-vanishing for all $\Gamma$-modules, and hence $\tilde X$ is a flat $\Gamma$-module. Hence $\tilde X$ is in the image of $\Phi$, and so $\Phi$ is a surjection and thus an isomorphism of sets.
\end{proof}
For a similar proof that the groupoid-valued deformation functors respect truncation, see \cite[2.5]{huakeller}, although we will not need this fact.

\section{Framed deformations}
Let $A$ be a dga and let $X$ be an $A$-module. By \ref{proreps}, the functor of deformations of $X$ is prorepresented by the pro-Artinian algebra $B_{\mathrm{nu}}^\sharp E$. If $E\coloneqq \R\enn_A(X)$ happens to be augmented, does the functor prorepresented by $B^\sharp E$ admit a deformation-theoretic interpretation? In this section we will show that when deforming a one-dimensional module over a ring, one can interpret $\R\mathrm{Map}_{\cat{pro}(\cat{dgArt}_k)}(B^\sharp E, -)$ in terms of rigidified deformations; the data we will need to add to deformations to rigidify will be that of a framing. Via the forgetful functor from $A$-modules to vector spaces, a deformation of the $A$-module $X$ gives rise to a deformation of the vector space $X$; this gives us a natural transformation $\sdefm_A(X) \to \sdefm_k(X)$. Observe that the functor $\sdefm_A(X)$ is pointed by the trivial deformation.
\begin{defn}
	Let $A$ be a dga and let $X$ be an $A$-module. The functor of \textbf{framed deformations} of $X$ is the homotopy fibre $$\frmdef_A(X)\coloneqq \mathrm{hofib}\left(\sdefm_A(X) \to \sdefm_k(X)\right).$$
	\end{defn}
	In other words, one restricts to those deformations of $X$ which are trivial deformations of the underlying dg-vector space.
	\begin{thm}\label{prorepfrm}
	Let $A$ be a $k$-algebra and let $S$ be a one-dimensional $A$-module. Then the derived endomorphism algebra $E\coloneqq \R\enn_A(S)$ is augmented, and there is a weak equivalence $$\frmdef_A(S)\simeq \R\mathrm{Map}_{\cat{pro}(\cat{dgArt}_k)}(B^\sharp E, -).$$
		\end{thm}
	\begin{proof}
		The idea is that $BE \to B_{\mathrm{nu}}E \to B_{\mathrm{nu}}k$ is a homotopy fibre sequence of coalgebras.	It is clear that $E$ is augmented by the map $E \to H^0(E)\cong \ext_A^0(S,S)\cong \enn_A(S)\cong k$. Because $E \to k$ is surjective, there is a weak equivalence of nonunital dgas $\mathrm{hofib}(E \to k)\simeq \bar E$ where $\bar E$ denotes the augmentation ideal. The homotopy fibre sequence $$\bar E \to E \to k$$ gives us a homotopy fibre sequence $$BE \to B_{\mathrm{nu}}E \to B_{\mathrm{nu}}k$$ of coalgebras, because $B$ is right Quillen and all dgas are fibrant. If $\Gamma$ is an Artinian dga then applying $\R\mathrm{Map}_{\cat{cndgc}_k}(\Gamma^*, -)$ to this homotopy fibre sequence gets us a homotopy fibre sequence $$\R\mathrm{Map}_{\cat{cndgc}_k}(\Gamma^*, BE) \to \R\mathrm{Map}_{\cat{cndgc}_k}(\Gamma^*, B_{\mathrm{nu}}E) \to \R\mathrm{Map}_{\cat{cndgc}_k}(\Gamma^*, B_{\mathrm{nu}}k)$$of simplicial sets. 	
		Because $S$ is one-dimensional, there is a weak equivalence $$\sdefm_k(S)(\Gamma)\simeq \R\mathrm{Map}_{\cat{cndgc}_k}(\Gamma^*,B_{\mathrm{nu}}k)$$by the proof of \ref{smcisrmap}. Moreover, this fits into a commutative diagram $$\begin{tikzcd}
		\R\mathrm{Map}_{\cat{cndgc}_k}(\Gamma^*,B_{\mathrm{nu}}E) \ar[r]\ar[d,"\simeq"]& \R\mathrm{Map}_{\cat{cndgc}_k}(\Gamma^*,B_{\mathrm{nu}}k)\ar[d,"\simeq"] \\
		\sdefm_A(S)(\Gamma) \ar[r]& \sdefm_k(S)(\Gamma)
		\end{tikzcd}$$where the vertical maps are weak equivalences, the upper horizontal map is induced by $E \to k$, and the lower horizontal map is the forgetful map. It follows that the homotopy fibres of the rows are weakly equivalent, which, using the proof of \ref{smcisrmap} again, is precisely the claim.
		\end{proof}
	\begin{rmk}
		It is not too hard to show that $B_{\mathrm{nu}}k$ is the dga $\frac{k[\epsilon]}{\epsilon^2}$ with zero differential and where $\epsilon$ has degree one.
		\end{rmk}
		
	\begin{rmk}
		One can let $A$ be a dga at the expense of assuming the extra condition that $\ext^0_A(k,k)$ is an augmented algebra. If one is willing to develop Koszul duality for dgas augmented over matrix algebras, then the condition that $S$ is one-dimensional can be dropped. One should also be able to carry out a similar analysis using \textbf{pointed} deformations, as in \cite{laudalpt} or \cite{kawamatapointed}, where the base ring is no longer $k$ but $k^n$. When $S$ is the direct sum of a finite semisimple collection of perfect $A$-modules, this is done in \cite{ELO2}; see also \cite{huakeller}. Removing the perfect hypothesis ought to be possible; one would have to repeat our Koszul duality arguments in the pointed setting.
		\end{rmk}

	\begin{rmk}
Suppose that $\tilde S$ is a deformation of $S$ over some $\Gamma$. Define a \textbf{framing} of $\tilde S$ to be a quasi-isomorphism $\nu:U(\tilde S) \to S\lot_k\Gamma$, where $U:D(A\otimes \Gamma) \to D(k \otimes \Gamma)$ is the forgetful functor. Define a \textbf{framed deformation} of $S$ to be a pair $(\tilde S, \nu)$ consisting of a deformation and a framing. Define a \textbf{framed isomorphism} $F:(X,\nu_X) \to (Y,\nu_Y)$ to be an isomorphism $F:X \to Y$ of deformations satisfying $\nu_X=\nu_Y\circ UF$. Then one can show that $$\pi_0(\frmdef_A(S))\cong \frac{\{\text{framed deformations of }S \}}{(\text{framed isomorphism})}.$$One method of proof is as follows. Let $\Pi_1$ denote the fundamental groupoid functor. Then$$\pi_0(\frmdef_A(S)) = \pi_0
 \mathrm{hofib}\left( \sdefm_A(X) \to \sdefm_k(X) \right)\cong \pi_0 \mathrm{hofib}\left(\Pi_1\sdefm_A(X) \to \Pi_1\sdefm_k(X)\right)$$where the right-hand homotopy fibre is taken in the model category of groupoids (see e.g.\ Strickland \cite[\S6]{stricklandgpds} for facts about the homotopy theory of groupoids). But one can compute the groupoid $\mathrm{hofib}\left(\Pi_1\sdefm_A(X) \to \Pi_1\sdefm_k(X)\right)$ explicitly (it suffices to do this calculation using the na\"ive Deligne groupoid functors), and its $\pi_0$ is isomorphic to the set of framed deformations modulo framed isomorphisms. These computations are done explicitly in \cite{dqdefm}. 
		\end{rmk}

\section{Prodeformations}
	\begin{defn}
		Let $F:\dgart \to \sset$ be any functor. Denote by $\hat F$ the functor $\proart \to \sset$ which sends an inverse system $\{\Gamma_\alpha\}_\alpha$ to the homotopy limit $\holim_\alpha F(\Gamma_\alpha)$. Call $\hat F$ the \textbf{pro-completion} of $F$.
	\end{defn}
	\begin{defn}
		Let $A$ be a dga and let $X$ be an $A$-module. The functor of \textbf{prodeformations of $X$} is the functor $\prodef_A(X)$. A \textbf{prodeformation} of $X$ is an element of the set $\pi_0\widehat{\sdefm}_A(X)$.
		\end{defn}
	\begin{rmk}\label{cdefsareprodefs}
	If $\Gamma=\{\Gamma_\alpha\}_\alpha$ is pro-Artinian then there is a weak equivalence
		$$\prodef_A(X)(\Gamma)\simeq \mathcal{W}\left(X_\mathrm{dg} \times^h_{D_\mathrm{dg}(A)} \holim_\alpha D_\mathrm{dg}(A\otimes \Gamma_\alpha)\right)$$given by passing the homotopy limit through $\mathcal{W}$ using \ref{wholims} and commuting homotopy limits with homotopy pullbacks. There is an obvious system of maps of dgas $\varprojlim \Gamma \to \Gamma_\alpha$ and this gives a dg functor $$D_\mathrm{dg}(A\otimes \varprojlim \Gamma) \to \holim_\alpha D_\mathrm{dg}(A\otimes \Gamma_\alpha).$$We hence obtain a map of simplicial sets $$\mathcal{W}\left(X_\mathrm{dg} \times^h_{D_\mathrm{dg}(A)} D_\mathrm{dg}(A\otimes \varprojlim \Gamma)\right) \to \prodef_A(X)(\Gamma).$$Although $\varprojlim \Gamma$ may not be Artinian, it is still augmented, so one can use it as a base for deformations. Hence a deformation of $X$ over $\varprojlim \Gamma$ gives a prodeformation of $X$ over $\Gamma$. For example, if $\Gamma = \{\frac{k[x]}{x^n}\}_n$ so that $\varprojlim\Gamma \cong k\llbracket x \rrbracket$, we see that a `complete local deformation' over $k\llbracket x \rrbracket$ gives us a prodeformation over $\Gamma$.
		\end{rmk}
			\begin{defn}
		Let $A$ be a dga and let $X$ be an $A$-module. The functor of \textbf{framed prodeformations of $X$} is the functor $\profrmdef_A(X)$. A \textbf{framed prodeformation} of $X$ is an element of the set $\pi_0\profrmdef_A(X)$.
	\end{defn}
	\begin{lem}\label{prohofiblem}
		Let $A$ be a dga and let $X$ be an $A$-module. Then there is a weak equivalence $$\profrmdef_A(X)\simeq \mathrm{hofib}\left(\prodef_A(X) \to \prodef_k(X)\right).$$
		\end{lem}
	\begin{proof}
		Commute the homotopy limit in the definition of $\profrmdef_A(X)$ through the homotopy fibre in the definition of $\frmdef_A(X)$.
		\end{proof}
We are about to give a representability statement for prodeformations; before we do so we prove a subsidiary lemma.
\begin{lem}\label{holimpro}
Let $\Gamma=\{\Gamma_\alpha\}_\alpha$ be a pro-Artinian dga. Then there is a weak equivalence
$$\R\mathrm{Map}_{\cat{pro}(\cat{dgArt}_k)}(-,\Gamma)\simeq \holim_\alpha\R\mathrm{Map}_{\cat{pro}(\cat{dgArt}_k)}(-,\Gamma_\alpha).$$
	\end{lem}
\begin{proof}
	Let $\Gamma'$ be the filtered diagram of Artinian dgas that defines $\Gamma$, but regarded as a filtered diagram in ${\cat{pro}(\cat{dgArt}_k)}$. It is easy to see directly from the definition that we have $\varprojlim \Gamma' \cong \Gamma$. Moreover, in the exact same manner as the proof of \ref{limisexact} we see that $\holim \Gamma' \simeq \Gamma$ as well, because homotopy limits are just limits. Now use that $\R\mathrm{Map}$ commutes with homotopy limits in the second variable.
	\end{proof}

	\begin{prop}\label{prodefrep}
		Let $A$ be a dga and let $X$ be an $A$-module. Let $E\coloneqq \R\enn_A(X)$ be the endomorphism dga of $X$. Let $\Gamma\in \proart$. Then there is a canonical weak equivalence
		$$\widehat{\sdefm}_A(X)(\Gamma)\cong \R\mathrm{Map}_{\cat{pro}(\cat{dgArt}_k)}(B_{\mathrm{nu}}^\sharp E,\Gamma).$$
	\end{prop}
\begin{proof}
	Put $\Gamma=\{\Gamma_\alpha\}_\alpha$. We have
	\begin{align*}
	\widehat{\sdefm}_A(X)(\Gamma)\coloneqq& \holim_\alpha\sdefm_A(X)(\Gamma_\alpha) &\text{by definition}\\
	\simeq& \holim_\alpha\R\mathrm{Map}_{\cat{pro}(\cat{dgArt}_k)}(B_{\mathrm{nu}}^\sharp E, \Gamma_\alpha) &\text{by \ref{proreps} levelwise}\\
	\simeq& \R\mathrm{Map}_{\cat{pro}(\cat{dgArt}_k)}(B_{\mathrm{nu}}^\sharp E, \Gamma) &\text{by \ref{holimpro}}
	\end{align*}as required.
	\end{proof}
One can prove an analogous version for framed prodeformations:
	\begin{prop}\label{frmprodefrep}
	Let $A$ be a $k$-algebra and let $S$ be a one-dimensional $A$-module. Let $E\coloneqq \R\enn_A(S)$ be the endomorphism dga of $S$. Then $E$ is augmented. Let $\Gamma\in \proart$. Then there is a canonical weak equivalence
	$$\profrmdef_A(S)(\Gamma)\cong \R\mathrm{Map}_{\cat{pro}(\cat{dgArt}_k)}(B^\sharp E,\Gamma).$$
\end{prop}
\begin{proof}Similarly to the previous proof, apply \ref{prorepfrm} levelwise and then use \ref{holimpro}.
\end{proof}
Now that we are able to work with prodeformations on the level of prorepresenting objects, we immediately obtain a universal prodeformation:
\begin{defn}
	Let $A$ be a dga and let $X$ be an $A$-module. Let $E\coloneqq \R\enn_A(X)$ be the endomorphism dga of $X$. The \textbf{universal prodeformation} of $X$ is the prodeformation of $X$ over $B_{\mathrm{nu}}^\sharp E$ corresponding to the element $\id_{B_{\mathrm{nu}}^\sharp E} \in \pi_0\R\mathrm{Map}_{\cat{pro}(\cat{dgArt}_k)}(B_{\mathrm{nu}}^\sharp E,B_{\mathrm{nu}}^\sharp E)$ across the weak equivalence of \ref{prodefrep}.
	\end{defn}
\begin{defn}
	Let $A$ be a $k$-algebra and let $S$ be a one-dimensional $A$-module. Let $E\coloneqq \R\enn_A(X)$ be the endomorphism dga of $X$, which is augmented. The \textbf{universal framed prodeformation} of $X$ is the framed prodeformation of $X$ over $B^\sharp E$ corresponding to the element $\id_{B^\sharp E} \in \pi_0\R\mathrm{Map}_{\cat{pro}(\cat{dgArt}_k)}(B^\sharp E,B^\sharp E)$ across the weak equivalence of \ref{frmprodefrep}.
\end{defn}

We finish this section with some computations. We will see that morally, the universal prodeformation is the algebra $E^!$ itself, regarded as an $A$-$B^\sharp E$-bimodule.

\begin{prop}\label{semiuprodef}
	Let $A$ be a noetherian ring, let $S$ be a one-dimensional $A$-module, and let $E\coloneqq \R\enn_A(S)$ be its derived endomorphism dga. Then $E$ is augmented. Let $P_n$ be the $n^\text{th}$ level of the Postnikov tower of the Koszul dual $E^!$. Then $P_n$ is Artinian, and the framed deformation of $S$ corresponding to the natural map $E^! \onto P_n$ is the $A$-$P_n$-bimodule $P_n$.
\end{prop}
\begin{proof}
	
	Choose a resolution $\tilde S \to S$ and replace $E$ by the quasi-isomorphic $\enn_A(\tilde S)$. The dga $E$ is augmented because $\enn_A(S)\cong k$. Because $A$ is noetherian, $E$ is cohomologically locally finite, and it follows from \ref{kdfinite} that $E^!$ is cohomologically locally finite. In particular, $P_n$ is Artinian. The identity map $B^\sharp E \to B^\sharp E$ gives the universal twisting cochain $\pi:BE \to E$, and in particular a twisting cochain $P_n^* \to E$, which we also denote by $\pi$. By \cite[\S6.5]{positselski}, twisting the differential on $\hom_k(BE,\tilde S)$ gives us a quasi-isomorphism $\hom^\pi_k(BE,\tilde S)\simeq \R\hom_E(k,\tilde S)$ which by \ref{kdisrend} is quasi-isomorphic to $E^!$. Truncating this quasi-isomorphism gives us a bimodule quasi-isomorphism $\hom^\pi_k(P_n^*,\tilde S)\simeq P_n$ (note that the $P_n^*$ form a Postnikov cotower for $BE$). Because $P_n$ is Artinian, the left-hand side is isomorphic to the twist $\tilde S \otimes_\pi P_n$. But this twist is precisely the deformation corresponding to the map $E^! \to P_n$.
\end{proof}

\begin{rmk}
	The above proof works more generally: in fact the proposition is true when we replace $E^! \to P_n$ by any map $E^! \to \Gamma$ with $\Gamma$ Artinian. Regarding $BE$ as an ind-coalgebra lets us regard $\hom^\pi_k(BE,\tilde S)$ as a pro-object, and using that $\varprojlim$ reflects weak equivalences we obtain a weak equivalence $\hom^\pi_k(BE,\tilde S) \simeq B^\sharp E$. Base changing this along $B^\sharp E \to \Gamma$ (the base change is the complex that computes $\mathrm{Cotor}^{BE}(\Gamma^*,\tilde S ^*)$; see \cite[\S2.5]{positselski}) we see that $\tilde S \otimes_\pi \Gamma \simeq \Gamma$. The advantage of using the Postnikov tower in the proof above is that it strictifies weak equivalences of pro-Artinian algebras: we are able to replace our abstract weak equivalence $\hom^\pi_k(BE,\tilde S) \simeq B^\sharp E$ by a level map that is levelwise a quasi-isomorphism.
\end{rmk}

\begin{rmk}
	One has an isomorphism of pro-objects $\hom_k(BE,E)\cong E \otimes B^\sharp E$. Taking the limit, one gets an isomorphism $$\hom_k(BE,E)\cong \varprojlim (E \otimes B^\sharp E) \eqqcolon E \widehat{\otimes} E^!$$between the convolution algebra and the completed tensor product. Loosely, this isomorphism sends the universal twisting cochain $\pi$ to the possibly infinite sum $$\sum_e e\otimes e^* \in \mcs(E \widehat{\otimes} E^!)$$where we sum over a basis of $E$ and where $e^*$ denotes the corresponding element of the dual basis. We can twist the differentials by $\pi$ to obtain an isomorphism $$\hom^\pi_k(BE,E)\cong \varprojlim( E \otimes_\pi B^\sharp E) \eqqcolon E \widehat{\otimes}_\pi E^!.$$As before, we can twist the differential on $\tilde S \widehat{\otimes} E^!$, and it is possible to show that $\tilde S \widehat{\otimes}_\pi E^!\simeq E^!$ as bimodules. In some sense, this is a computation of the universal prodeformation.
\end{rmk}

	\section{Deformations and quasi-equivalences}
	Let $A$ and $B$ be two dgas. Suppose that $F:D(A) \to D(B)$ is a derived equivalence given by tensoring with an $A$-$B$-bimodule. In this section, we analyse how universal prodeformations behave under $F$. Let $X$ be an $A$-module and put $Y\coloneqq F(X)$.  Observe that we can enhance $F$ to a quasi-equivalence $F:D_\mathrm{dg}(A) \to D_\mathrm{dg}(B)$ of dg categories.

	\begin{lem}\label{fdefwe}
		$F$ induces a weak equivalence
		$$F:\sdefm_A(X) \xrightarrow{\simeq} \sdefm_B(Y).$$
		\end{lem}
	\begin{proof}For each Artinian dga $\Gamma$, the quasi-equivalence $F$ induces a commutative diagram of quasi-equivalences of dg categories
		$$\begin{tikzcd}
		X_\mathrm{dg} \ar[r,hook]\ar[d,"F"] & D_\mathrm{dg}(A)\ar[d,"F"] & D_\mathrm{dg}(A \otimes \Gamma)\ar[d,"F"] \ar[l]\\
		Y_\mathrm{dg} \ar[r,hook]& D_\mathrm{dg}(B) & D_\mathrm{dg}(B \otimes \Gamma) \ar[l]
		\end{tikzcd}$$where in the third column we have used that $F$ is tensoring by a bimodule. So we get a quasi-equivalence $$F:\left(X_\mathrm{dg} \times^h_{D_\mathrm{dg}(A)} D_\mathrm{dg}(A \otimes \Gamma)\right) \xrightarrow{\simeq} \left(Y_\mathrm{dg} \times^h_{D_\mathrm{dg}(B)} D_\mathrm{dg}(B \otimes \Gamma)\right)$$between the homotopy pullbacks and hence a weak equivalence $F:\sdefm_A(X) \xrightarrow{\simeq} \sdefm_B(Y)$.
		\end{proof}
	\begin{lem}\label{fprodefwe}
		$F$ induces a weak equivalence $F:\widehat{\sdefm}_A(X) \xrightarrow{\simeq}	\widehat{\sdefm}_B(Y)$.
		\end{lem}
	\begin{proof}
		Take homotopy limits of \ref{fdefwe}.
		\end{proof}
	
	Put $E\coloneqq \R\enn_A(X)$ and $E'\coloneqq \R\enn_B(Y)$. We get a quasi-isomorphism $E \to E'$ induced by $F$ and hence a weak equivalence $\prodef_B(Y)(B^\sharp E') \to \prodef_B(Y)(B^\sharp E)$. Combined with the weak equivalence of \ref{fdefwe}, we get a weak equivalence $\prodef_A(X)(B^\sharp E) \simeq \prodef_B(Y)(B^\sharp E')$ and it is easy to see that this preserves the universal prodeformation.
	
	\p We repeat our analysis in the framed setting.
\begin{lem}\label{frmprodefwe}
	$F$ induces a weak equivalence $F:\profrmdef_A(X) \xrightarrow{\simeq}	\profrmdef_B(Y)$.
	\end{lem}
\begin{proof}
	The weak equivalences of \ref{fprodefwe} commute with the forgetful functors and using \ref{prohofiblem} it follows that we get an induced weak equivalence between framed prodeformations.
	\end{proof}
\begin{rmk}
	We could also have proved this by checking that we get a weak equivalence $F:\frmdef_A(X) \xrightarrow{\simeq} \frmdef_B(Y)$ and taking homotopy limits.
	\end{rmk}
In exactly the same manner as before, one can check that universal framed prodeformations are preserved by $F$. Now we wish to use this to prove the following, which will be extremely useful to us in Chapter 10:

\begin{thm}\label{qeupdf}
	Let $A$ and $B$ be noetherian rings and let $S_A$ and $S_B$ be one-dimensional $A$ and $B$-modules respectively. Suppose that there is a derived equivalence $F:D(A) \to D(B)$ given by tensoring with a complex of bimodules, satisfying $F(S_A)\simeq S_B$. Put $U_A\coloneqq \R\enn_A(S_A)^!$ and $U_B\coloneqq \R\enn_B(S_B)^!$. Then there is a quasi-isomorphism of $B$-modules $F(U_A)\simeq U_B$.
	\end{thm}
\begin{proof}
	Morally, this is true because $U_A$ is the universal framed prodeformation of $S_A$, and similarly for $U_B$, and $F$ sends universal prodeformations to universal prodeformations. Because $A$ and $B$ are noetherian it follows that $U_A$ and $U_B$ are nonpositive cohomologically locally finite dgas, and we have a dga quasi-isomorphism $U_B \to U_A$ provided by $F$. Take $\{U_A^n\}_n$ to be the Postnikov tower of $U_A$. It is a pro-Artinian dga, because $U_A$ was cohomologically locally finite, and we have $\varprojlim_n U_A^n \simeq U_A$. Similarly, let $\{U_B^n\}_n$ be the Postnikov tower of $U_B$. The quasi-isomorphism $U_B \to U_A$ gives a level map $\{U_B^n\to U_A^n\}_n$ which is a quasi-isomorphism at each level. For a fixed $n$ we have a zigzag of quasi-equivalences of dg categories
	$$ \{X\}_\mathrm{dg} \times^h_{D_\mathrm{dg}(A)} D_\mathrm{dg}(A \otimes U^n_A) \xrightarrow{F} 
	{\{Y\}_\mathrm{dg} \times^h_{D_\mathrm{dg}(B)} D_\mathrm{dg}(B \otimes U^n_A)} \xleftarrow{-\lot_{U_B^n}U^n_A} 
	{\{Y\}_\mathrm{dg} \times^h_{D_\mathrm{dg}(B)} D_\mathrm{dg}(B \otimes U^n_B)}
	$$and these fit into a zigzag of morphisms of pro-objects in dg categories. By \ref{frmprodefrep}  and \ref{semiuprodef}, at each level $n$ the universal prodeformation corresponds to the elements $$U_A^n \mapsto F(U_A^n) \simeq U_B^n \lot_{U_B^n} U_A^n \mapsfrom U_B^n$$ where the quasi-isomorphism in the middle is both $B$-linear and $U_A^n$-linear. Because $U_B^n$ and $U_A^n$ are quasi-isomorphic, it follows by a spectral sequence argument that for a bounded above $B$-$U_B^n$-bimodule $X$, the natural map $X \to X \lot_{U_B^n}U_A^n$ is a quasi-isomorphism of $B$-modules. In particular the natural map $U_B^n \to U_B^n \lot_{U_B^n} U_A^n \simeq F(U_A^n)$ is a $B$-linear quasi-isomorphism. This is compatible with the inverse system, in the sense that for each $n\geq m$ we get commutative diagrams $$\begin{tikzcd} U_B^n\ar[r]\ar[d] & U_B^m \ar[d]\\
	F(U_A^n)\ar[r] & F(U_A^m)
	\end{tikzcd}
	$$of $B$-linear maps, where the vertical maps are quasi-isomorphisms. Hence we get a quasi-isomorphism of $B$-modules $\holim_n F(U_A^n) \simeq \holim_n U_B^n$. We see that $\holim_n U_B^n$ is exactly $U_B$, because the transition maps in the Postnikov tower are surjective. Because $F$ is a quasi-equivalence, we get a quasi-isomorphism of $B$-modules $F(\holim_n U_A^n)\simeq \holim_n F(U_A^n)$. But as before, $\holim_n U_A^n\simeq U_A$. So it follows that $F(U_A)\simeq U_B$, as required.

	\end{proof}
\begin{rmk}
	In some sense, this is a computation of the universal framed prodeformation. What we have done is to take the pro-object in dg categories defining the universal framed prodeformation, picked out the pieces of the universal prodeformation, forgotten some of the module structure, and taken the homotopy limit of those pieces. Whereas to find the universal framed prodeformation, we take the homotopy limit of the system of dg categories and look at the object corresponding to the pieces we had. To forget some of the structure we look at the functor $\holim_n D(A\otimes U_A^n) \to D(A)^{\N}$ given by applying the forgetful functor to the inverse system defining the universal prodeformation. We can then take the limit to get an object of $D(A)$. However, it is not clear that this functor $\holim_n D(A\otimes U_A^n) \to D(A)$ behaves well with respect to weak equivalences.
	\end{rmk}

	\part{The derived quotient and the dg singularity category}

		\chapter{The derived quotient}
		In this chapter, we introduce our main object of study, the derived quotient of Braun--Chuang--Lazarev \cite{bcl}.  We will mostly be interested in derived quotients of ungraded algebras by idempotents. The derived quotient is a natural object to study, and has been investigated before by a number of authors: for example it appears in Kalck and Yang's work \cite{kalckyang,kalckyang2} on relative singularity categories, de Thanhoffer de V\"olcsey and Van den Bergh's paper \cite{dtdvvdb} on stable categories, and Hua and Zhou's paper \cite{huazhou} on the noncommutative Mather--Yau theorem. Our study of the derived quotient will unify some of the aspects of all of the above work.
		\p
		We remark that sections 5.1, 5.2, 5.4, and 5.5 are valid over any field $k$. Section 5.6 is valid over any algebraically closed field.
	\section{Derived localisation}\label{derloc}
	The derived quotient is a special case of a general construction -- the derived localisation. Let $A$ be any dga over $k$ (the construction works over any commutative base ring). Let $S \subseteq H(A)$ be any collection of homogeneous cohomology classes. Braun, Chuang and Lazarev define the \textbf{derived localisation} of $A$ at $S$, denoted by $\dloc$, to be the dga universal with respect to homotopy inverting elements of $S$:
	\begin{defn}[{\cite[\S3]{bcl}}]
		Let $Q A \to A$ be a cofibrant replacement of $A$. The \textbf{derived under category} $A \downarrow^\mathbb{L} \cat{dga}$ is the homotopy category of the under category $Q A \downarrow \cat{dga}$ of dgas under $Q A$. A $QA$-algebra $f:Q A \to Y$ is \textbf{$S$-inverting} if for all $s \in S$ the cohomology class $f(s)$ is invertible in $HY$. The \textbf{derived localisation} $\dloc$ is the initial object in the full subcategory of $S$-inverting objects of $A \downarrow^\mathbb{L} \cat{dga}$.
	\end{defn} 
	\begin{prop}[{\cite[3.10, 3.4, and 3.5]{bcl}}]
		The derived localisation exists, is unique up to unique isomorphism in the derived under category, and is quasi-isomorphism invariant.
	\end{prop}
In particular, the derived localisation $\dloc$ comes with a canonical map from $A$ (in the derived under category) making it into an $A$-bimodule, unique up to $A$-bimodule quasi-isomorphism. In what follows, we will refer to $\dloc$ as an $A$-bimodule, with the assumption that this always refers to the canonical bimodule structure induced from the dga map $A \to \dloc$.
	\begin{rmk}
		The derived localisation is the homotopy pushout of the span$$A \from k\langle S \rangle \to k\langle S, S^{-1}\rangle.$$
	\end{rmk}

\begin{defn}A map $A \to B$ of dgas is a \textbf{homological epimorphism} if the multiplication map $B\lot_A B \to B$ is a quasi-isomorphism of $B$-modules.
\end{defn}
\begin{prop}\label{homepi}
Let $A$ be a dga and $S \subseteq H(A)$ be any collection of homogeneous cohomology classes. Then the canonical localisation map $A \to \dloc$ is a homological epimorphism.
\end{prop}
\begin{proof}
The map is a homotopy epimorphism by \cite[3.17]{bcl}, which by \cite[4.4]{clhomepi} is the same as a homological epimorphism.
	\end{proof}

	\begin{defn}[{\cite[4.2 and 7.1]{bcl}}]
		Let $X$ be an $A$-module. Say that $X$ is \textbf{$S$-local} if, for all $s\in S$, the map $s: X \to X$ is a quasi-isomorphism. Say that $X$ is \textbf{$S$-torsion} if $\R\hom_A(X,Y)$ is acyclic for all $S$-local modules $Y$. Let $D(A)_{S\text{-loc}}$ be the full subcategory of $D(A)$ on the $S$-local modules, and let $D(A)_{S\text{-tor}}$ be the full subcategory on the $S$-torsion modules.
	\end{defn}
	Similarly as for algebras, one defines the notion of the derived localisation $\mathbb{L}_S(X)$ of an $A$-module $X$. It is not too hard to prove the following:
	\begin{thm}[{\cite[4.14 and 4.15]{bcl}}]
		Localisation of modules is smashing, in the sense that $X \to X \lot_A \dloc$ is the derived localisation of $X$. Moreover, restriction of scalars gives an equivalence of $D(\dloc)$ with $D(A)_{S\text{-loc}}$.
	\end{thm}
\begin{rmk}
	If $A\to B$ is a dga map then the three statements 
	\begin{itemize}
		\item $A \to B$ is a homological epimorphism.
		\item $A \to B$ induces an embedding $D(B)\to D(A)$.
		\item $-\lot_AB$ is a smashing localisation on $D(A)$.
		\end{itemize}
	are all equivalent \cite{phomepis}.
	\end{rmk}

	One defines a \textbf{colocalisation functor} pointwise by setting $\mathbb{L}^S(X)\coloneqq \mathrm{cocone}(X \to \mathbb{L}_S(X))$. An easy argument shows that $\mathbb{L}^S(X)$ is $S$-torsion. 
	\begin{defn}\label{colocdga}The \textbf{colocalisation} of $A$ along $S$ is the dga $$\mathbb{L}^S(A)\coloneqq \R\enn_A\left(\oplus_{s \in S}\,\mathrm{cone}(A \xrightarrow{s} A)\right).$$
	\end{defn}Note that the dga $\mathbb{L}^S(A)$ may differ from the colocalisation of the $A$-module $A$. If $S$ is a finite set, then the dga $\mathbb{L}^S(A)$ is a compact $A$-module, and we get the analogous:
	\begin{thm}[{\cite[7.6]{bcl}}]
		Let $S$ be a finite set. Then $D(\mathbb{L}^S(A))$ and $D(A)_{S\text{-tor}}$ are equivalent.
	\end{thm}
	Neeman--Thomason--Trobaugh--Yao localisation gives the following:
	\begin{thm}[{\cite[7.3]{bcl}}]\label{ntty}
		Let $S$ be finite. Then there is a sequence of dg categories $$\per\mathbb{L}^S(A) \to \per A \to \per \dloc$$which is exact up to direct summands.
	\end{thm}
	\begin{rmk}[{\cite[7.9]{bcl}}]\label{bclrcl}
		The localisation and colocalisation functors fit into a recollement $$\begin{tikzcd}[column sep=huge]
		D(A)_{S\text{-loc}} \ar[r]& D(A)\ar[l,bend left=25]\ar[l,bend right=25]\ar[r] & D(A)_{S\text{-tor}}\ar[l,bend left=25]\ar[l,bend right=25]
		\end{tikzcd}$$We will see a concrete special case of this in \ref{recoll}.
	\end{rmk}

	\begin{defn}[{\cite[9.1 and 9.2]{bcl}}]
		Let $A$ be a dga and let $e$ be an idempotent in $H^0(A)$. The \textbf{derived quotient} $\dq$ is the derived localisation $\mathbb{L}_{1-e}A$.
	\end{defn}
	Clearly, $\dq$ comes with a natural quotient map from $A$. One can write down an explicit model for $\dq$, at least when $k$ is a field.
	\begin{prop}\label{drinfeld}
		Let $A$ be a dga over $k$, and let $e\in H^0(A)$ be an idempotent. Then the derived quotient $\dq$ is quasi-isomorphic as an $A$-dga to the dga $$B\coloneqq \frac{A\langle h \rangle}{(he=eh=h)}\;, \quad d(h)=e$$with $h$ in degree -1.
	\end{prop}
	\begin{proof}
		This is essentially \cite[9.6]{bcl}; because $k$ is a field, $A$ is flat (and in particular left proper) over $k$. The quotient map $A \to B$ is the obvious one.
	\end{proof}
\begin{rmk}This specific model for $\dq$ is an incarnation of the Drinfeld quotient: see \cite[9.7]{bcl} for details.
	\end{rmk}
\begin{rmk}
	In particular, this is a concrete model for $\dq$ as an $A$-bimodule.
	\end{rmk}
	\section{The cohomology of the derived quotient}
	Let $A$ be a dga and let $e\in A$ be an idempotent. Write $R$ for the cornering\footnote{We take this terminology from \cite{cikcorner}; the motivating example is to take one of the obvious nontrivial idempotents in $M_2(k)$ to obtain a subalgebra (isomorphic to $k$) on matrices with entries concentrated in one corner.} $eAe$. We will investigate the cohomology of the derived quotient $Q\coloneqq \dq$.
	\begin{defn}\label{celldef}
The \textbf{cellularisation functor}, denoted by $\cell : D(A) \to D(A)$, is the functor that sends $M$ to $Me\lot_R eA$.
		\end{defn}
	\begin{rmk}
		The name and the notation for $\cell$ comes from Dwyer and Greenlees \cite{dgcompletetorsion}.
		\end{rmk}
	In particular, the cellularisation of $A$ itself is the bimodule $Ae\lot_R eA$. Note that this admits an $A$-bilinear multiplication map $\mu: \cell A \to A$ which has image the submodule $AeA\into A$.
		\begin{prop}\label{dqexact}
		There is an exact triangle of $A$-bimodules $\cell A \xrightarrow{\mu} A \to Q \to $.
	\end{prop}
	\begin{proof}Forget the algebra structure on $Q$ and view it as an $A$-bimodule; recall that the localisation map $A \to Q$ is the derived localisation of the $A$-module $A$. Observe that the localisation map $A \to Q$ is also the localisation of the $A$-module $A$ at the perfect module $Ae$, in the sense of \cite{dgcompletetorsion}. Thus by \cite[4.8]{dgcompletetorsion}, the homotopy fibre of $A \to Q$ is the cellularisation of $Q$.
		\end{proof}
When $A$ is an ungraded algebra, then one can write down a much more explicit proof using the `Drinfeld quotient' model for $Q$. We do this below, since we will need to use some facts about this explicit model later.
\begin{lem}\label{drinfeldmodel}
Suppose that $A$ is an ungraded algebra with an idempotent $e\in A$. Put $R\coloneqq eAe$ the cornering. Let $B$ be the dga of \ref{drinfeld} quasi-isomorphic to $Q$. Then:
\begin{enumerate}
	\item Let $n>0$ be an integer. There is an $A$-bilinear isomorphism $$B^{-n}\cong Ae \otimes R\otimes\cdots \otimes R \otimes eA$$where the tensor products are taken over $k$ and there are $n$ of them.
	\item Let $n>0$. The differential $B^{-n} \to B^{-n+1}$ is the Hochschild differential, which sends $$x_0\otimes \cdots \otimes x_n\mapsto\sum_{i=0}^{n-1} (-1)^i x_0\otimes\cdots \otimes x_ix_{i+1}\otimes \cdots \otimes x_n.$$
		\item Let $n,m>0$ and let $a\in B^0=A$. Let $x=x_0\otimes \cdots \otimes x_n \in B^{-n}$ and $y=y_0\otimes \cdots \otimes y_m\in B^{-m}$. Then we have
		\begin{align*}
		xy &= x_0\otimes \cdots \otimes x_n y_0\otimes \cdots \otimes y_m
		\\ ax&=ax_0\otimes \cdots \otimes x_n
		\\ xa&=x_0\otimes \cdots \otimes x_na.
		\end{align*}
\end{enumerate}	
	\end{lem}
\begin{proof}
	For the first claim, observe that a generic element of $B^{-n}$ looks like a path $x_0h\cdots hx_n$ where $x_0=x_0e$, $x_n=ex_n$, and $x_j=ex_je$ for $0<j<n$. Replacing occurrences of $h$ with tensor product symbols gives the claimed isomorphism. For the second claim, because $h$ has degree $-1$ we must have $$d(x_0h\cdots hx_n)=\sum_i (-1^i)x_0h\cdots hx_id(h)x_{i+1}h \cdots hx_n$$ but because $d(h)=e$ we have $x_id(h)x_{i+1}=x_ix_{i+1}$. The third claim is clear from the definition of $B$.
	\end{proof}

	\begin{proof}[Alternate proof of \ref{dqexact} when $A$ is ungraded]Consider the shifted bimodule truncation $$T:=(\tau_{\leq-1}B)[-1] \simeq \cdots \to Ae\otimes R\otimes R\otimes eA\to Ae\otimes R\otimes eA\to Ae\otimes eA$$with $Ae\otimes eA$ in degree zero. By \ref{drinfeldmodel}(2) we see that this truncation is exactly the complex that computes the relative Tor groups $\tor^{R/k}(Ae,eA)$ \cite[8.7.5]{weibel}. Since $k$ is a field, the relative Tor groups are the same as the absolute Tor groups, and hence $T$ is quasi-isomorphic to $\cell A$. Because we have $B\simeq \mathrm{cone}(T \xrightarrow{\mu} A)$ we are done.
	\end{proof}
	\begin{rmk}
		The exact triangle $\cell A \to A \to Q \to $ also appears in \cite[\S 7]{kalckyang2}.
	\end{rmk}
	The following is immediately obtained by considering the long exact sequence associated to the exact triangle $\cell A \xrightarrow{\mu} A \to Q \to $.
	\begin{cor}\label{derquotcohom}
		Let $A$ be an algebra over a field $k$, and let $e\in A$ be an idempotent.
		Then the derived quotient $\dq$ is a nonpositive dga with cohomology spaces 
		$$H^j(\dq)\cong\begin{cases}
		0 & j>0
		\\ A/AeA & j=0
		\\ \ker(Ae\otimes_{R}eA \to A) & j=-1
		\\ \tor^{R}_{-j-1}(Ae,eA) & j<-1
		\end{cases}$$
	\end{cor}
	\begin{rmk}
		The ideal $AeA$ is said to be \textbf{stratifying} if the map $Ae \lot_{R} eA \to AeA$ is a quasi-isomorphism. It is easy to see that $AeA$ is stratifying if and only if $H^0:\dq \to A/AeA$ is a quasi-isomorphism.
	\end{rmk}
	\section{Derived quotients of path algebras}
	When $A$ is the path algebra of a quiver with relations, and $e$ is the idempotent corresponding to a set of vertices, then one can interpret the cohomology groups $H^j(\dq)$ in terms of the geometry of the quiver, at least for small $j$. We think of the modules $H^j(\dq)$ as being a (co)homology theory (with coefficients in $k$) for quivers with relations and specified vertices. We will be especially interested in $H^{-1}(\dq)$, since the description of \ref{derquotcohom} is not particularly explicit.
	\begin{defn}Let $(Q,I,V)$ be a triple consisting of a finite quiver $Q$ with a finite set of relations $I$ and a collection of vertices $V$. Let $A=kQ/(I)$ be the path algebra over $k$, and let $e$ be the idempotent of $A$ corresponding to $V$. Write $H_j(Q,I,V;k)\coloneqq H^{-j}(\dq)$. Note the reindexing, so that $H_*(Q,I,V;k)$ is concentrated in nonnegative degrees.
	\end{defn}
	Note that we really consider $I$ and the ideal $(I)$ to be different; in particular $(I)$ is usually not finite. Note that each $H_j(Q,I,V;k)$ is a module over the $k$-algebra $H_0(Q,I,V;k)\cong A/AeA$. It is clear that $H_0(Q,I,V;k)\cong A/AeA$ is the algebra on those paths in the quiver that do not pass through $V$. Dually, $R=eAe$ is the algebra on those paths starting and ending in $V$. To analyse $H_1(Q,I,V;k)$ we need to introduce some new terminology.
	\begin{defn}
		A \textbf{marked relation} $m$ for the triple $(Q,I,V)$ is a formal sum $m=\sum_iu_i\vert v_i$ with each $u_i \in Ae$ and $v_i \in eA$, such that the composition $\sum_iu_iv_i$ is a relation from $I$. We think of the vertical bar as the `marking'.
	\end{defn}
	\begin{prop}The module $H_1(Q,I,V;k)$ is spanned over $A/AeA$ by the marked relations.
	\end{prop}
	\begin{proof}
		We know that $H_1(Q,I,V;k)$ is the middle cohomology of the complex $$Ae\otimes_k R \otimes_k eA \xrightarrow{d} Ae \otimes_k eA \xrightarrow{\mu} A$$where $d$ is the Hochschild differential and $\mu$ is the multiplication. If we write a vertical bar instead of the tensor product symbol, it is immediate that each $A$-bilinear combination of marked relations is a $(-1)$-cocycle in this complex. The $(-1)$-cocycles are all of two forms: firstly, those $x$ such that $\mu (x)$ is zero in $kQ$, and secondly, those $x$ such that $\mu (x)\in (I)$. If $\mu (x)=0$ in $kQ$, then $x$ must just be of the form $x=\sum_i(p^i_1\vert p^i_2 p^i_3 - p^i_1p^i_2 \vert p^i_3)$ and it is easy to see that $x$ must be a coboundary, since $d(p\vert q \vert r)=pq\vert r -p\vert qr$. So $H_1(Q,I,V;k)$ is spanned by those $x$ such that $\mu (x)\in (I)$. But this means that $\mu(x)=\sum_i a_i r_i b_i$ where each $r_i$ is a relation, and each $a_i,b_i$ is in $A$. But then we see that $x=\sum_i a_i' m_i b_i'$, where each $m_i$ is a marked relation. So $H_1(Q,I,V;k)$ is spanned over $A$ by the marked relations. Pick a 1-cocycle $amb$, where $a,b \in A$ and $m$ is a marked relation. If either $a$ or $b$ are in $AeA$, then $amb$ is in fact a coboundary: for example if $a=uev$ then $amb=d(u\vert vmb)$. In other words, $H_1(Q,I,V;k)$ is actually spanned over $A/AeA$ by the marked relations. 
	\end{proof}
	\begin{cor}\label{h1finite}
		Let $Q$ be a finite quiver and $I$ a finite set of relations in $kQ$. Set $A\coloneqq kQ/(I)$. Pick a set of vertices $V \subseteq Q$ and let $e\in A$ be the corresponding idempotent. Then if $A/AeA$ is finite-dimensional, so is $H_1(Q,I,V;k)$.
	\end{cor}
	\begin{proof}
		There are a finite number of relations and hence a finite number of marked relations: since each relation is of finite length, there are only finitely many ways to mark it. This shows that $H_1(Q,I,V;k)$ is finite over the finite-dimensional algebra $A/AeA$, and hence finite-dimensional. Note that we can get an explicit upper bound for the dimension: write relation $i$ as a sum of monomials $q_i^j$, each of length $\ell_i^j$. It is easy to see that there are homotopies $\vert uv \simeq u\vert v \simeq uv \vert$ in $(\dq)^{-1}$, and so each monomial $q_i^j$ has at most $\max(1,\ell_i^j -1)$ markings that are not homotopic. Put $\ell_i\coloneqq \prod_j\max(1,\ell_i^j -1)$; then relation $i$ has at most $\ell_i$ markings, because we can mark each monomial individually. Put $\ell\coloneqq \sum_i \ell _i$, so that there are at most $\ell$ marked relations spanning $H_1(Q,I,V;k)$. Hence if $A/AeA$ has dimension $d$, an upper bound for the dimension of $H_1(Q,I,V;k)$ is $d^2\ell$.
	\end{proof}
	
	\begin{rmk}
		One can use similar ideas to show that $H_2(Q,I,V;k)$ is spanned by cocycles of the form $u\vert v \vert w$, where $uv=v$ and $vw=w$. For if $\sum^n_iu_iv_i\vert w_i=\sum^n_iu_i\vert v_iw_i$, and all of the $u_i,v_i,w_i$ are monomials, then there exists some permutation $\sigma$ such that $u_iv_i=u_{\sigma i}$ and $w_i=v_{\sigma i}w_{\sigma i}$. Restricting to the orbits of $\sigma$, we may assume that $\sigma$ is a cycle, and write $u_iv_i=u_{i+1}$ and $w_{i-1}=v_iw_i$, where the subscripts are taken modulo $ n$. One then shows by induction on $r$ that $\sum_i^ru_i\vert v_i \vert w_i$ is homotopic to $u_1\vert v_1v_2\cdots v_{r-1}v_r\vert w_r$, and the claim follows upon taking $u=u_1, v=v_1\cdots v_n$, and $w=w_n$.
	\end{rmk}
	\begin{ex}\label{quiv1} Let $Q$ be the quiver $$
		
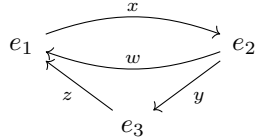
\begin{tikzcd}
		e_1 \arrow[rr,bend left=20,"x"]  && e_2 \arrow[ll,bend left=20,"w"']\ar[ld,"y"] \\ & e_3 \ar[lu,"z"]&
		\end{tikzcd}$$ with relations $w=yz$ and $xyz=yzx=zxy=0$. Let $e=e_1+e_2$, so that the corresponding set of vertices $V$ is $\{e_1,e_2\}$. It is not hard to compute that $\mathrm{dim}_k(R)=4$, $\mathrm{dim}_k(A)=9$, and $\mathrm{dim}_k(A/AeA)=1$. Moreover, $R$ is not a left or a right ideal of $A$, since $x \in R$ but neither $xy$ nor $zx$ are in $R$. We remark that $R$ need not be the path algebra of the `full subquiver' $Q_V$ on $V$, since relations outside of $V$ can influence $R$: observe that $xw$ is zero in $R$, but nonzero in $kQ_V$. We compute that our bound for the dimension $n$ of $H_1(Q,I,V;k)$ is $7$. One can check that there are at most $5$ (homotopy classes of) nontrivial marked relations, namely $\vert w -y\vert z$, $x\vert yz$, $yz\vert x$, $z\vert xy$, and $zx\vert y$. So a better estimate for $n$ is $5$. But this is still too large, as $x\vert yz$ and $yz\vert x$ are both coboundaries, and $z\vert xy \simeq zx \vert y$. So $n$ is at most 2. We see that $w -yz$ and $zxy$ are relations in $I$ that cannot be `seen' from $V$: they start and finish outside of $V$, but pass through $V$ (where we can mark them), and moreover do not contain any `subrelations' lying entirely in $V$. In fact, one can check using \ref{derquotcohom} that $n$ is precisely 2, and $H_1(Q,I,V;k)$ has basis $\{\vert w -y\vert z, z\vert xy\}$.
	\end{ex}

	\section{Recollements}
	Loosely speaking, a recollement (see \cite{bbd} or \cite{jorgensen} for a definition) between three triangulated categories $(\mathcal{T}',\mathcal{T},\mathcal{T}'')$ is a collection of functors describing how to glue $\mathcal{T}$ from a subcategory $\mathcal{T}'$ and a quotient category $\mathcal{T}''$. One can think of a recollement as a short exact sequence $\mathcal{T}' \to \mathcal{T} \to \mathcal{T}''$ where both maps admit left and right adjoints.
	\begin{thm}[{cf.\ \cite[2.10]{kalckyang} and \cite[9.5]{bcl}}]\label{recoll}
		Let $A$ be an algebra over $k$, and let $e\in A$ be an idempotent. Write $Q\coloneqq \dq$ and $R\coloneqq eAe$. Let $Q_A$ denote the $Q\text{-}A$-bimodule $Q$, let ${}_AQ$ denote the $A\text{-}Q$-bimodule $Q$, and let ${}_AQ_A$ denote the $A$-bimodule $Q$. Put \begin{align*}
		i^*\coloneqq -\lot_A {}_AQ, &\quad j_!\coloneqq  -\lot_{R} eA
		\\ i_*=\R\hom_{Q}({}_AQ,-), &\quad j^!\coloneqq \R\hom_A(eA,-)
		\\ i_!\coloneqq \lot_{Q}Q_A, & \quad j^*\coloneqq -\lot_A Ae
		\\ i^! \coloneqq \R\hom_{A}(Q_A,-), & \quad j_*\coloneqq \R\hom_{R}(Ae,-)
		\end{align*}
		Then the diagram of unbounded derived categories
		$$\begin{tikzcd}[column sep=huge]
		D(Q) \ar[r,"i_*=i_!"]& D(A)\ar[l,bend left=25,"i^!"']\ar[l,bend right=25,"i^*"']\ar[r,"j^!=j^*"] & D(R)\ar[l,bend left=25,"j_*"']\ar[l,bend right=25,"j_!"']
		\end{tikzcd}$$
		is a recollement diagram.
	\end{thm}
	\begin{proof}We give a rather direct proof. It is clear that $(i^*,i_*=i_!,i^!)$ and $(j_!, j^!=j^*,j_*)$ are adjoint triples, and that $i_*=i_!$ is fully faithful. Fullness and faithfulness of $j_!$ and $j_*$ follow from \cite[2.10]{kalckyang}. The composition $j^*i_*$ is tensoring by the $Q\text{-}R$-bimodule $Q.e$, which is acyclic since $Q$ is $e$-killing in the sense of \cite[\S9]{bcl}. The only thing left to show is the existence of the two required classes of distinguished triangles. First observe that 
		\begin{align*}& i_!i^!\cong \R\hom_A({}_AQ_A,-)
		\\ & j_*j^* \cong \R\hom_{R}(Ae,\R\hom_A(eA,-))\cong \R\hom_A(\cell A,-)
		\\ & j_!j^! \cong -\lot_A \cell A
		\\ &  i_*i^* \cong -\lot_A {}_AQ_A\end{align*}
		Now, recall from \ref{dqexact} the existence of the distinguished triangle of $A$-bimodules $$\cell A \xrightarrow{\mu} A \xrightarrow{l} {}_AQ_A \to$$
		Taking any $X$ in $D(A)$ and applying $\R\hom_A(-,X)$ to this triangle, we obtain a distinguished triangle of the form $i_!i^!X \to X \to j_*j^* X \to$. Similarly, applying $X\lot_A-$, we obtain a distinguished triangle of the form $j_!j^!X \to X \to i_*i^* X \to$.
	\end{proof}
	\begin{rmk}This recollement is given in \cite[9.5]{bcl}, although they are not explicit with their functors. If $AeA$ is stratifying, this recovers a recollement constructed by Cline, Parshall, and Scott \cite{cps}. See e.g.\ \cite{cikcorner} for the analogous recollement on the level of abelian categories.
	\end{rmk}

	\begin{prop}\label{Rcoloc}
		In the above setup, $D(R)$ is equivalent to the derived category of $(1-e)$-torsion modules.
	\end{prop}
	\begin{proof}
		Recollements are determined completely by fixing one half (e.g.\ \cite[Remark 2.4]{kalck}). Now the result follows from the existence of the recollement of \ref{bclrcl}. More concretely, one can check that the colocalisation $\mathbb{L}^{1-e}A$ is $R$: because $A\cong eA\oplus(1-e)A$ as right $A$-modules, we have $\mathrm{cone}(A \xrightarrow{1-e} A)\simeq eA$, and we know that $\R\enn_A(eA)\simeq \enn_A(eA)\cong R$ because $eA$ is a projective $A$-module.
	\end{proof}
	We show that $\dq$ is a relatively compact $A$-module; before we do this we first introduce some notation.
	\begin{defn}Let $\mathcal X$ be a subclass of objects of a triangulated category $\mathcal{T}$. Then $\thick_\mathcal{T} \mathcal X $ denotes the smallest triangulated subcategory of $\mathcal{T}$ containing $\mathcal{X}$ and closed under taking direct summands. Similarly, $\langle \mathcal X \rangle_\mathcal{T}$ denotes the smallest triangulated subcategory of $\mathcal{T}$ containing $\mathcal{X}$, and closed under taking direct summands and all existing set-indexed coproducts. We will often drop the subscripts if $\mathcal T$ is clear. If $\mathcal X$ consists of a single object $X$, we will write $\thick X$ and $\langle X \rangle$.
	\end{defn}
	\begin{ex}
		Let $A$ be a dga. Then $\langle A \rangle_{D(A)}\cong D(A)$, whereas $\thick_{D(A)}(A) \cong \per A$.
	\end{ex}
	\begin{defn}
		Let $\mathcal{T}$ be a triangulated category and let $X$ be an object of $\mathcal{T}$. Say that $X$ is \textbf{relatively compact} (or \textbf{self compact}) in $\mathcal{T}$ if it is compact as an object of $\langle X \rangle_{\mathcal{T}}$.
	\end{defn}
	\begin{prop}
		The right $A$-module $\dq$ is relatively compact in $D(A)$.
	\end{prop}
	\begin{proof}
		The embedding $i_*$ is a left adjoint and so respects coproducts. Hence $i_*(\dq)$ is relatively compact in $D(A)$ by \cite[1.7]{jorgensen}. The essential idea is that $\dq$ is compact in $D(\dq)$, the functor $i_*$ is an embedding, and $\langle i_*(\dq) \rangle \subseteq \im i_*$.
	\end{proof}
	In situations when $\dq$ is not a compact $A$-module (e.g.\ when it has nontrivial cohomology in infinitely many degrees), this gives interesting examples of relatively compact objects that are not compact.
	\begin{defn}
		Let $D(A)_{A/AeA}$ denote the full subcategory of $D(A)$ on those modules $M$ with each $H^j(M)$ a module over $A/AeA$.
	\end{defn}
	\begin{prop}\label{cohomsupport}There is a natural triangle equivalence $D(\dq)\cong D(A)_{A/AeA}$.
	\end{prop}
	\begin{proof}Follows from the proof of \cite[2.10]{kalckyang}, along with the fact that recollements are determined completely by fixing one half.
	\end{proof}
	\begin{prop}\label{semiorthog}
		The derived category $D(A)$ admits a semi-orthogonal decomposition $$D(A)\cong \langle D(A)_{A/AeA}, \langle eA \rangle \rangle = \langle \im i_*, \im j_! \rangle $$
	\end{prop}
	\begin{proof}
		This is an easy consequence of \cite[3.6]{jorgensen}.
	\end{proof}
	We finish with a couple of facts about t-structures; see \cite{bbd} for a definition. In particular, given t-structures on the outer pieces of a recollement diagram, one can glue them to a new t-structure on the central piece \cite[1.4.10]{bbd}. 
	\begin{thm}\label{tstrs}
		The category $D(\dq)$ admits a t-structure $\tau$ with aisles $$\tau^{\leq 0}=\{X: H^i(X)=0 \text{ for } i>0\}\qquad \text{and} \qquad \tau^{\geq 0}=\{X: H^i(X)=0 \text{ for } i<0\}.$$Moreover, $H^0$ is an equivalence from the heart of $\tau$ to $\cat{Mod}$-$A/AeA$. Furthermore, gluing $\tau$ to the natural t-structure on $D(R)$ via the recollement diagram of \ref{recoll} gives the natural t-structure on $D(A)$.
	\end{thm}
	\begin{proof}
		The first two sentences are precisely the content of \cite[2.1(a)]{kalckyang}. The last assertion holds because gluing of t-structures is unique, and restricting the natural t-structure on $D(A)$ clearly gives $\tau$ along with the natural t-structure on $D(R)$.
	\end{proof}

		\section{Hochschild theory}
		We collect some facts about the Hochschild theory of the derived quotient. The most important is that taking quotients preserves Hochschild (co)homology complexes:
		\begin{prop}[{\cite[6.2]{bcl}}]\label{hochprop}
			Let $A$ be a dga, $e\in H^0(A)$ an idempotent, and $Q\coloneqq \dq$ the derived quotient. Let $M$ be a $Q$-module. Then there are quasi-isomorphisms $$Q \lot_{Q^e} M \simeq A \lot_{A^e} M \quad\text{and}\quad \R\hom_{Q^e}(Q,M) \simeq \R\hom_{A^e}(A,M) $$and hence isomorphisms $$HH_*(Q,M)\cong HH_*(A,M) \quad\text{and}\quad HH^*(Q,M)\cong HH^*(A,M).$$
		\end{prop}
		Hochschild homology is functorial with respect to recollement:
		\begin{prop}[{\cite[3.1]{kellerhoch}}]
			Let $A$ be an algebra over $k$ and $e\in A$ an idempotent. Put $Q\coloneqq \dq$ the derived quotient and $R\coloneqq eAe$ the cornering. Then there is an exact triangle in $D(k)$ $$
			Q\lot_{Q^e} Q \to A \lot_{A^e} A \to R \lot_{R^e} R \to .$$
		\end{prop}
		Unfortunately, Hochschild cohomology does not behave so nicely.
		\begin{lem}Let $A$ be a algebra over $k$ and $e\in A$ an idempotent. Put $Q\coloneqq \dq$ the derived quotient and $R\coloneqq eAe$ the cornering. Then there are exact triangles in $D(k)$\begin{align*}
			\R\hom_{A^e}(Q,A) \to \R\hom_{A^e}(A,A) \to \R\hom_{R^e}(R,R) \to \\
			\R\hom_{A^e}(A,\cell A) \to \R\hom_{A^e}(A,A) \to \R\hom_{Q^e}(Q,Q) \to .
			\end{align*}
		\end{lem}
		\begin{proof}Recall that $\mathrm{cocone}(A \to Q)$ is quasi-isomorphic as an $A$-bimodule to $\cell A$ by \ref{dqexact}. Consider the diagram $$\begin{tikzcd} \phantom{}& \phantom{}& \phantom{}& \phantom{}\\
			\R\hom_{A^e}(\cell A,\cell A) \ar[r]\ar[u] & \R\hom_{A^e}(\cell A,A) \ar[r]\ar[u] & \R\hom_{A^e}(\cell A,Q) \ar[r]\ar[u] & \phantom{} \\
			\R\hom_{A^e}(A,\cell A) \ar[r]\ar[u] & \R\hom_{A^e}(A,A) \ar[r]\ar[u] & \R\hom_{A^e}(A,Q) \ar[r]\ar[u] & \phantom{} \\
			\R\hom_{A^e}(Q,\cell A) \ar[r]\ar[u] & \R\hom_{A^e}(Q,A) \ar[r]\ar[u] & \R\hom_{A^e}(Q,Q) \ar[r]\ar[u,"f"] & \phantom{} \end{tikzcd}$$ whose rows and columns are exact triangles. The first triangle can be seen as the middle column, once we make the observation that $$\R\hom_{A^e}(\cell A,A)\simeq \R\hom_{A^{\mathrm{op}}\otimes R}(Ae,\R\hom_A(eA,A))\simeq\R\hom_{A^{\mathrm{op}}\otimes R}(Ae,Ae)\simeq \R\hom_{R^e}(R,R).$$Now, \ref{hochprop} tells us both that $f$ is a quasi-isomorphism, and moreover that both source and target are quasi-isomorphic to $\R\hom_{Q^e}(Q,Q)$. The second triangle is now visible as the middle row.
		\end{proof}
		\begin{rmk}
			These are two of the three triangles obtained by applying \cite[Theorem 4]{han} to the standard recollement $(D(Q),D(A), D(R))$. The third is $$\R\hom_{A^e}(Q,\cell A) \to \R\hom_{A^e}(A,A) \to \R\hom_{Q^e}(Q,Q) \oplus \R\hom_{R^e}(R,R) \to.$$
		\end{rmk}
			\section{Deformation theory}\label{dqdefm}
		We give the derived quotient a deformation-theoretic interpretation. The following proposition generalises an argument of Kalck and Yang given in the proof of \cite[5.5]{kalckyang}.
		\begin{prop}\label{rendkd}
			Suppose that $A$ is a $k$-algebra and $e\in A$ an idempotent. Suppose that $A/AeA$ is an Artinian local $k$-algebra, and let $S$ be the quotient of $A/AeA$ by its radical. Suppose furthermore that $\dq$ is cohomologically locally finite. Then the dga $\R\enn_A(S)$ is augmented, and its Koszul dual is quasi-isomorphic to $\dq$.
		\end{prop}
		\begin{proof}
			Because $A/AeA$ is local, $S$ is a one-dimensional $A$-module. The augmentation on $\R\enn_A(S)$ is given by the natural map to $\enn_A(S)\cong k$. Since $S$ is naturally a module over $\dq$, and $D(\dq) \to D(A)$ is fully faithful, we have $\R\enn_A(S)\simeq \R\enn_{\dq}(S)$. Note that $\dq \to A/AeA \to S$ is also an augmentation of dgas. Hence we have a quasi-isomorphism $(\dq)^!\simeq \R\enn_{\dq}(S)$ by \ref{kdisrend}. Taking Koszul duals gives a quasi-isomorphism $(\dq)^{!!}\simeq \R\enn_A(S)^!$. It now suffices to prove that $(\dq)^{!!}$ is quasi-isomorphic to $\dq$, which follows from an application of \ref{kdfin}.
		\end{proof}
		\begin{rmk}
			This result can be viewed as saying that the derived category $D(\dq)$ is triangle equivalent to its formal completion along $S$ in the sense of \cite{efimov}. If $S$ is perfect over $A$ then one can prove this more directly using results of \cite[\S4]{efimov}.
		\end{rmk}
		\begin{rmk}
			A pointed version of this is proved in \cite{huakeller}, under the additional assumption that $A$ has finite global dimension.
		\end{rmk}
	\begin{rmk}
		This theorem is valid over any field $k$.
		\end{rmk}
		The main application of this theorem for us will be computational; the point is that one can compute $\R\enn_A(S)^!$ using some fairly standard methods. Letting $P\coloneqq B^\sharp\left(\dq^!\right)$ denote the continuous Koszul dual of $\dq^!$, we hence have a natural quasi-isomorphism $\varprojlim P\simeq \dq$. By \ref{prorepfrm}, $P$ prorepresents the functor of framed deformations of $S$:

		\begin{thm}\label{maindefmthm}Let $A$ be a $k$-algebra and $e\in A$ an idempotent. Suppose that $A/AeA$ is a local algebra and that $\dq$ is cohomologically locally finite. Let $S$ be $A/AeA$ modulo its radical, regarded as a right $A$-module. Then $\dq$ is quasi-isomorphic to the Koszul dual $\R\enn_A(S)^!\simeq \left(\dq\right)^{!!}$ which, regarded as a pro-Artinian dga, prorepresents the deformation functor $\frmdef_A(S)$; i.e.\ there is a weak equivalence of $\sset$-valued functors $$\frmdef_A(S)\simeq \R\mathrm{Map}_{\cat{pro}(\cat{dgArt}_k)}(B^\sharp\R\enn_A(S) , -).$$
		\end{thm}
	\begin{proof}
		The claim that $\dq\simeq \R\enn_A(S)^!$ is \ref{rendkd}, and the claim that $\dq^{!!}\simeq\dq$ follows from \ref{kdfin}. The prorepresentability statement is \ref{prorepfrm}.
		\end{proof}

Although the derived quotient $\dq$ does not strictly prorepresent the functor of framed deformations, since it is a dga defined up to quasi-isomorphism and not a pro-Artinian dga defined up to weak equivalence, it at least determines the weak equivalence class of the deformation functor:

	\begin{prop}\label{dqqiclassdeterminesdefms}
		With the setup as in \ref{maindefmthm}, the quasi-isomorphism class of $\dq$ determines the weak equivalence class of the deformation functor $\frmdef_A(S)$.
		\end{prop} 
	\begin{proof}
 By \ref{qisocor}, if $Q$ is a pro-Artinian dga with $\varprojlim Q$ quasi-isomorphic to $\dq\simeq \varprojlim B^\sharp\R\enn_A(S)$, then $Q$ is weakly equivalent to $B^\sharp\R\enn_A(S)$. Weakly equivalent pro-Artinian algebras give weakly equivalent derived mapping spaces.
		\end{proof}

	Restricting our attention to underived deformations allows us to prove the following analogue of \cite[3.9]{contsdefs}:
	
		\begin{prop}\label{subdefmthm}
			Let $A$ be a $k$-algebra and $e\in A$ an idempotent. Suppose that $A/AeA$ is a local algebra and that $\dq$ is cohomologically locally finite. Let $S$ be $A/AeA$ modulo its radical, regarded as a right $A$-module. Then the set-valued functor of classical noncommutative deformations of $S$ is represented by the Artinian local algebra $A/AeA$.
		\end{prop}
	
\begin{proof}
	By \ref{nurmk} combined with the prorepresentability statements \ref{proreps} and \ref{prorepfrm} we can consider either framed or unframed deformations; the set-valued deformation functor does not see the difference (because every deformation admits a framing). So by \ref{maindefmthm} we know that $B^\sharp\left(\R\enn_A(S)\right)$ prorepresents the set-valued functor of derived deformations. By \ref{prorepremk}, the pro-Artinian algebra $H^0\left(B^\sharp\left(\R\enn_A(S)\right)\right)$ hence prorepresents the functor of classical deformations.
	But because $\varprojlim$ is the homotopy limit by \ref{limisexact}, we have $$A/AeA \cong H^0(\dq)\cong H^0(\varprojlim B^\sharp\left(\R\enn_A(S)\right))\cong \varprojlim H^0(B^\sharp\left(\R\enn_A(S)\right)).$$where the first isomorphism follows from \ref{derquotcohom} and the second isomorphism is an application of \ref{rendkd}. Hence the pro-Artinian algebra $ H^0(B^\sharp\left(\R\enn_A(S)\right))$ is actually Artinian, and so must be isomorphic to its limit $A/AeA$.
	\end{proof}

\begin{rmk}
	We remark that the results of this section remain true in positive characteristic.
	\end{rmk}

	\chapter{Singularity categories}
	Before we can go any further, we must introduce another key concept in this thesis, that of the {singularity category}. First studied by Buchweitz \cite{buchweitz} for noncommutative rings, and then by Orlov \cite{orlovtri} for schemes, this is a measure of how singular a geometric object is. We study singularity categories and their dg enhancements, exploring some different models -- the stable category and the category of matrix factorisations -- before finishing with some recovery theorems.
	
	\section{Some singularity theory}
	We gather together a few results from the world of singularity theory; for good references, see \cite{bikr} in the algebraic setting and \cite{singsintro} in the holomorphic setting. We will focus on isolated hypersurface singularities. Throughout this section, let $S\coloneqq  k\llbracket x_1,\ldots, x_n \rrbracket$ be the complete local ring of $k^n$ at the origin, and let $\sigma \in \mathfrak{m}_S$ be nonzero. We call the quotient $R\coloneqq S/\sigma$ a \textbf{hypersurface singularity}. 
	\begin{defn}
		Let $\sigma \in \mathfrak{m}_S$ be nonzero. The \textbf{Jacobian ideal} (or the \textbf{Milnor ideal}) $J_\sigma$ of $\sigma$ is the ideal of $S$ generated by the partial derivatives $\frac{\partial \sigma}{\partial x_i}$ for $i=1,\ldots,n$. The \textbf{Milnor algebra} $M_\sigma$ is the algebra $S/J_\sigma$ and the \textbf{Milnor number} $\mu_\sigma$ is the dimension (over $k$) of the Milnor algebra. The \textbf{Tjurina algebra} of $\sigma$ is the quotient $$T_\sigma\coloneqq \frac{S}{(\sigma, J_\sigma)}.$$The \textbf{Tjurina number} $\tau_\sigma$ is the dimension of the Tjurina algebra.
		\end{defn}
	
	\begin{defn}
		Say that the singularity $R\coloneqq S/\sigma$ is \textbf{isolated} if the Milnor number $\mu_\sigma$ is finite.
		\end{defn}
	\begin{rmk}
		In the holomorphic setting, this is equivalent to the usual definition: the singular locus is an isolated point \cite[2.3]{singsintro}.
		\end{rmk}

	\begin{rmk}
		Clearly for an isolated singularity the Tjurina number is less than or equal to the Milnor number. If $\sigma$ is \textbf{quasi-homogeneous}, meaning that there are weights $w_i$ making $\sigma(x_1^{w_1},\ldots,x_n^{w_n})$ into a homogenous polynomial, then the Tjurina algebra is isomorphic to the Milnor algebra; this follows by showing that $\sigma$ is already in the Jacobian ideal, which is an easy adaptation of the standard proof of Euler's theorem on homogenous functions.
		\end{rmk}
	The Mather--Yau theorem states that one can recover an isolated hypersurface singularity from its Tjurina algebra, as long as the dimension is fixed:
	\begin{thm}[\cite{matheryau,gpmather}]
		Let $\sigma_1$ and $\sigma_2$ be nonzero elements of the maximal ideal of $S$. Put $R_i\coloneqq S/\sigma_i$ and assume that the $R_i$ are isolated. Then $R_1\cong R_2$ if and only if $T_{\sigma_1}\cong T_{\sigma_2}$.
		\end{thm}
	\begin{rmk}
		The theorem was first proved in the holomorphic setting by Mather and Yau in the case where $k=\C$. Greuel and Pham prove an algebraic version when $k$ is algebraically closed, although the theorem as we state it is only true in characteristic zero.
		\end{rmk}		
	
	\section{Gorenstein rings}
In this section we collect some standard facts about Gorenstein rings, which are rings satisfying a certain homological condition. We begin with some remarks on homological dimension theory; for a detailed reference, including proofs, see Lam \cite{lamlect}. Then we broadly follow Bass \cite{bassgorenstein} and Matsumura \cite{matsumura} for the material on Gorenstein rings.
	
\begin{defn}
Let $R$ be a ring and $M$ be an $R$-module. The \textbf{injective dimension} of $M$ is the minimal length of an injective resolution of $M$ (note that this may be infinite). The \textbf{projective dimension} of $M$ is the minimal length of a projective resolution of $M$. The \textbf{global dimension} of $R$ is the supremum of the projective dimensions of all finitely generated $R$-modules.
\end{defn}
\begin{rmk}
Strictly, we should distinguish between left global dimension and right global dimension. However, we will only ever be interested in right global dimension.
\end{rmk}
\begin{rmk}
If $R$ is noetherian, then projective dimension agrees with \textbf{flat dimension}, where one takes resolutions by flat rather than projective modules. The corresponding global invariant is called the \textbf{Tor-dimension} or \textbf{weak dimension}.
\end{rmk}
\begin{prop}[{\cite[5.45]{lamlect}}]
The global dimension of $R$ is the supremum of the injective dimensions of all finitely generated $R$-modules.
\end{prop}
Recall that if $R$ is a commutative noetherian local ring, with maximal ideal $\mathfrak{m}$ and residue field $K$, then its \textbf{Zariski cotangent space} is the $K$-vector space $T^*_{\mathfrak{m}}R\coloneqq\mathfrak{m}/\mathfrak{m}^2$. Say that $R$ is \textbf{regular} if its Krull dimension is equal to $\dim_K(T^*_{\mathfrak{m}}R)$. Call a scheme $X$ \textbf{regular} if its local rings are all regular; in particular a commutative noetherian ring is regular if and only if its localisation at every prime ideal is a regular local ring. Over a field of characteristic zero, a scheme is smooth if and only if it is regular and locally of finite type; in particular a variety is smooth if and only if it is regular.
\begin{thm}[Auslander--Buchsbaum--Serre; see {\cite[5.84]{lamlect}}]
Let $R$ be a commutative noetherian local ring. Then $R$ has finite global dimension if and only if it is regular. In this case, the global dimension of $R$ is equal to its Krull dimension.
\end{thm}
\begin{thm}[global Auslander--Buchsbaum--Serre; see {\cite[5.95]{lamlect}}]
Let $R$ be a commutative noetherian ring. If $R$ has finite global dimension then it is regular. The converse is true if $R$ has finite Krull dimension.
\end{thm}
\begin{rmk}\label{nagatarmk} The converse of the global Auslander-Buchsbaum-Serre theorem is not true if one omits the Krull dimension hypothesis. Indeed, Nagata's example \cite{nagata} provides a counterexample: there exists a commutative regular noetherian domain $R$, of infinite Krull dimension, whose localisations $R_p$ at every maximal ideal are regular local rings of finite, but arbitrarily large, Krull dimension (and hence, by Auslander-Buchsbaum-Serre, global dimension). Because injective resolutions localise, $R$ cannot have finite global dimension. See \cite[5.96]{lamlect} for an in-depth discussion.
\end{rmk}
	\begin{defn}
		Let $R$ be a noncommutative two-sided noetherian ring. Say that $R$ is \textbf{Gorenstein} (or \textbf{Iwanaga-Gorenstein}) if it has finite injective dimension over itself as both a left module and a right module.
	\end{defn}
	\begin{rmk}
		In this setting, the left injective dimension of $R$ must necessarily agree with the right injective dimension \cite[Lemma A]{zaks}. In general, $R$ might have infinite injective dimension over itself on one side and finite injective dimension on the other.
	\end{rmk}
	\begin{ex}
		A ring of finite global dimension is clearly Gorenstein. In particular, a commutative regular noetherian ring that is either local or of finite Krull dimension is Gorenstein.
	\end{ex}
	We give some useful structure theorems for commutative Gorenstein rings. We begin with the local case:
	\begin{prop}[{\cite[18.1]{matsumura}}]
		Let $R$ be a commutative noetherian local ring of Krull dimension $n$. Then $R$ is Gorenstein if and only if it has injective dimension $n$ as an $R$-module.
	\end{prop}
	When $R$ has finite Krull dimension, being Gorenstein is a local property:	
	\begin{prop}[\cite{bassgorenstein}]
		Let $R$ be a commutative noetherian ring of Krull dimension $n$. The following are equivalent:
		\begin{enumerate}
			\item[\emph{1.}] $R$ is Gorenstein.
			\item[\emph{2.}] $R$ has injective dimension $n$ as an $R$-module.
			\item[\emph{3.}] $R_p$ is Gorenstein for every prime ideal $p\subseteq R$.
			\item[\emph{4.}] $R_p$ is Gorenstein for every maximal ideal $p\subseteq R$.
		\end{enumerate}
	\end{prop}
	\begin{rmk}
		The above result is false if one drops the hypothesis that the Krull dimension of $R$ is finite, for exactly the same reasons as those of \ref{nagatarmk}.
	\end{rmk}
	
	Local complete intersections are Gorenstein:
	\begin{prop}[e.g.\ {\cite[21.19]{eisenbud}}]\label{hypsgor}
		Let $S$ be a commutative noetherian regular local ring and $I \subseteq S$ an ideal generated by a regular sequence. Then $R\coloneqq S/I$ is Gorenstein.
	\end{prop}
	In particular, if $S$ is a commutative noetherian regular local ring and $\sigma \in S$ is a non-zerodivisor then the hypersurface singularity $S/\sigma$ is Gorenstein.

	\section{Triangulated and dg singularity categories}
	In this section we introduce the singularity category of a noncommutative ring, and enhance it to a dg category. We also introduce Kalck and Yang's relative singularity category and its natural dg enhancement. 
	
	\p Let $R$ be any ring. Observe that the triangulated category $\per R$ of perfect complexes of $R$-modules embeds into the triangulated category $D^b(R)$.
	\begin{defn}
	 Let $R$ be a ring. The \textbf{singularity category} of $R$ is the Verdier quotient $D_\mathrm{sg}R\coloneqq D^b(R)/ \per R$.
		\end{defn}
	\begin{rmk}
		Strictly, for noncommutative rings one should distinguish between the left and right singularity categories. However, we will always work with right modules.
		\end{rmk}
	\begin{prop}
		Let $R$ be a ring of finite global dimension. Then $D_\mathrm{sg}R$ vanishes.
		\end{prop}
\begin{proof}
Take a bounded complex $X\ = \ X_p \to \cdots \to X_q$ and write it as an iterated cone of maps between modules. Because $R$ has finite global dimension, each of these modules is quasi-isomorphic to a perfect complex. Taking mapping cones preserves perfect complexes, and hence $X$ must itself be quasi-isomorphic to a perfect complex. Hence $X$ maps to zero in the Verdier quotient.
	\end{proof}
\begin{rmk}
	The converse is not true; one needs `uniform' vanishing of $D_\mathrm{sg}R$ to conclude that $R$ has finite global dimension (for example each $R$-module may have a finite projective resolution, but the minimal length of such resolutions may be unbounded).
	\end{rmk}
\begin{cor}
	Let $R$ be a commutative noetherian regular ring of finite Krull dimension. Then $D_\mathrm{sg}R$ vanishes.
	\end{cor}
For technical reasons, we will need to know that various singularity categories we use are idempotent complete; we recall the notion below.
\begin{defn}
	Let $\mathcal T$ be a triangulated category. A \textbf{projector} in $\mathcal T$ is a morphism\linebreak $\pi:X \to X$ in $\mathcal T$ with $\pi^2=\pi$. Say that a projector $\pi:X \to X$ \textbf{splits} if $X$ admits a direct summand $X'$ such that $\pi$ is the composition $X \to X' \to X$. Say that $\mathcal T$ is \textbf{idempotent complete} if every projector in $\mathcal T$ splits.
\end{defn}
\begin{prop}[\cite{bsidempotent}]
	If $\mathcal T$ is a triangulated category, then there exists an idempotent complete triangulated category ${\mathcal T}^\omega$ and a fully faithful triangle functor ${\mathcal T}\to{\mathcal T}^\omega$, universal among functors from ${\mathcal T}$ into idempotent complete triangulated categories. Call ${\mathcal T}^\omega$ the \textbf{idempotent completion} or the \textbf{Karoubi envelope} of ${\mathcal T}$. The assignment ${\mathcal T}\mapsto{\mathcal T}^\omega$ is functorial.
\end{prop}

\begin{prop}[{\cite[5.5]{kalckyang2}}]\label{sgidemref}
	Let $R$ be a Gorenstein ring. If $R$ is a finitely generated module over a commutative complete local noetherian $k$-algebra, then $D_\mathrm{sg}(R)$ is idempotent complete.
	\end{prop}
\begin{rmk}
The second condition is satisfied for example when $R$ is a finite-dimensional $k$-algebra, or when $R$ is itself a commutative complete local noetherian $k$-algebra.
	\end{rmk}
We now introduce a relative version of the singularity category, due to Kalck and Yang. Let $A$ be an algebra over $k$, and let $e\in A$ be an idempotent. Write $R$ for the cornering $eAe$. Note that by \ref{recoll}, the functor $j_!= -\lot_{R} eA$ embeds $D(R)$ into $D(A)$. In fact, since $D(R)=\langle R\rangle$, we have $j_{!}D(R)=\langle eA \rangle$. Similarly, restricting to compact objects shows that $j_{!}\per R = \thick (eA) \subseteq \per A$. 
\begin{defn}[Kalck--Yang \cite{kalckyang}]\label{relsingcatdefn}
	Let $A$ be an algebra over $k$, and let $e\in A$ be an idempotent. Write $R$ for the cornering $eAe$. The \textbf{relative singularity category} is the Verdier quotient $$\Delta_R(A) \coloneqq  \frac{D^b(A)}{j_! \per R} \cong \frac{D^b(A)}{\thick (eA)}.$$
\end{defn}
In \cite{kalckyang2}, this is referred to as the \textbf{singularity category of $A$ relative to $e$}. We immediately turn to dg singularity categories:
		\begin{defn}\label{dgsing}
		Let $A$ be a $k$-algebra. The \textbf{dg singularity category} of $A$ is the Drinfeld quotient ${D}^{\mathrm{dg}}_\mathrm{sg}(A)\coloneqq D^b_{\mathrm{dg}}(A) / \cat{per}_\mathrm{dg}(A)$. If $e\in A$ is an idempotent, write $R$ for the cornering $eAe$ and $j_!$ for the functor $-\otimes_R^\mathbb{L}eA:D(R)\to D(A)$. It is easy to see that $j_!$ admits a dg enhancement. The \textbf{dg relative singularity category} is the Drinfeld quotient
		$$\Delta^{\mathrm{dg}}_R(A) \coloneqq  \frac{D_{\mathrm{dg}}^b(A)}{j_! \per_{\mathrm{dg}} R} \cong \frac{D_{\mathrm{dg}}^b(A)}{\thick (eA)}.$$
	\end{defn}
	By \ref{drinfeldpretr}, we have  $[{D}^{\mathrm{dg}}_\mathrm{sg}(A)]\cong D_\mathrm{sg}(A)$ and $[{\Delta}^{\mathrm{dg}}_R(A)]\cong {\Delta}_R(A)$. The following easy lemma is useful, since we will want to quotient by perfect complexes:
	\begin{lem}\label{drinfeldlem}
		Let $\mathcal{A}$ be a pretriangulated dg category and $\mathcal{B}$ a full pretriangulated dg subcategory such that every $b\in \mathcal{B}$ is compact. Let $X\in \cat{ind}\mathcal{A}$. Then for all $b \in \mathcal{B}$, there is an isomorphism $\dgh_{\cat{ind}\mathcal{A}}(b,X)\cong \dgh_\mathcal{A}(b,\varinjlim X)$. In particular, if $\varinjlim X\cong 0$ then $\dgh_{\cat{ind}\mathcal{A}}(\mathcal{B},X)$ is acyclic. The converse is true if $[\mathcal{B}]$ contains a generator of $[\mathcal{A}]$.
	\end{lem}
\begin{proof}
For the first statement, we have $$\dgh_{\cat{ind}\mathcal{A}}(b,X)\coloneqq \varinjlim \dgh_{\mathcal A}(b,X)\cong \dgh_\mathcal{A}(b,\varinjlim X)$$where the first isomorphism is by definition and the second isomorphism is because $b$ is compact. For the second statement, if $\varinjlim X\cong 0$ then we have $\dgh_{\cat{ind}\mathcal{A}}(b,X)\cong \dgh_{\cat{ind}\mathcal{A}}(b,0)\simeq 0$ for all $b\in \mathcal{B}$. For the third statement, suppose that $[\mathcal{B}]$ contains a generator $g$ of $[\mathcal{A}]$, and let $\tilde g$ be any lift of $g$ to $\mathcal B$. If $\dgh_{\mathcal{A}}(\tilde{g},\varinjlim X)$ is acyclic then we must have $\hom_{[\mathcal{A}]}(g,[\varinjlim X])\cong 0$. So $[\varinjlim X]\cong 0$ because $g$ was a generator, and hence $\varinjlim X\cong 0$.
	\end{proof}

\begin{lem}The objects of ${D}^{\mathrm{dg}}_\mathrm{sg}(A)$ are precisely those ind-objects $X \in \cat{ind}D^b_{\mathrm{dg}}(A)$ such that $\varinjlim X$ is acyclic and there is an $M \in D^b_{\mathrm{dg}}(A)$ with a map $M \to X$ with ind-perfect cone.
	\end{lem}
\begin{proof}
	By definition of the Drinfeld quotient, the objects of ${D}^{\mathrm{dg}}_\mathrm{sg}(A)$ are precisely those ind-objects $X \in \cat{ind}D^b_{\mathrm{dg}}(A)$ such that:
	\begin{itemize}
	\item $\dgh_{\cat{ind}{D}^{\mathrm{dg}}_\mathrm{sg}(A)}(\per_{\mathrm{dg}}(A),X)$ is acyclic.
	\item There exists $M \in D^b_{\mathrm{dg}}(A)$ and a map $f:M \to X$ with $\mathrm{cone}(f)\in \cat{ind}(\per_{\mathrm{dg}}(A))$.
	\end{itemize}
Because $\per(A)$ consists of the compact objects of $D^b(A)$ and contains a generator of $D^b(A)$, namely $A$ itself, we may apply \ref{drinfeldlem} to conclude that $\dgh_{\cat{ind}{D}^{\mathrm{dg}}_\mathrm{sg}(A)}(\per_{\mathrm{dg}}(A),X)$ is acyclic if and only if $\varinjlim X$ is acyclic.
	\end{proof}

	\section{The stable category}
When $R$ is a Gorenstein ring, then its singularity category has a more algebraic interpretation as a certain stable category of modules. In this section, we follow Buchweitz's seminal unpublished manuscript \cite{buchweitz}.	
\begin{defn}
	Let $R$ be a Gorenstein ring. If $M$ is an $R$-module, write $M^\vee$ for the $R$-linear dual $\hom_R(M,R)$. A finitely generated $R$-module $M$ is \textbf{maximal Cohen--Macaulay} or just \textbf{MCM} if the natural map $\R\hom_R(M,R)\to M^\vee$ is a quasi-isomorphism.
\end{defn}
	\begin{rmk}
			An equivalent characterisation of MCM $R$-modules is that they are those modules $M$ for which $\ext_R^j(M,R)$ vanishes whenever $j> 0$.
		\end{rmk}
	\begin{rmk}
		In some places in this thesis we will need the more general concept of a Cohen-Macaulay (or just CM) module; all that will really concern us is that MCM modules are CM. See \cite{yoshino} for a reference.
		\end{rmk}
	\begin{defn}Let $R$ be a Gorenstein ring and $M,N$ be two MCM $R$-modules. Say that a pair of maps $f,g:M\to N$ are \textbf{stably equivalent} if their difference $f-g$ factors through a projective module.
		\end{defn}
	\begin{lem}
		Let $R$ be a Gorenstein ring and $M,N$ be two MCM $R$-modules. Stable equivalence is an equivalence relation on the set $\hom_R(M,N)$.
		\end{lem}
	\begin{defn}
		Let $R$ be a Gorenstein ring and $M,N$ be two MCM $R$-modules. Denote the set of stable equivalence classes of maps $M \to N$ by $\underline{\hom}_R(M,N)$. We refer to such an equivalence class as a \textbf{stable map}.
		\end{defn}
	\begin{defn}
		Let $R$ be a Gorenstein ring. The \textbf{stable category} of $R$-modules is the category $\stab R$ whose objects are the MCM $R$-modules and whose morphisms are the stable maps. Composition is inherited from $\cat{mod}\text{-}R$.
		\end{defn}
	
	\begin{defn}\label{weaksyz}
		Let $R$ be a Gorenstein ring. For each $R$-module $X$, choose a surjection $f:R^n \onto X$ and set $\Omega X\coloneqq \ker f$. We refer to $\Omega$ as a \textbf{syzygy} of $X$.
		\end{defn}
	\begin{prop}
		The assignment $X \mapsto \Omega X$ is a well-defined endofunctor of $\stab R$.
		\end{prop}
	\begin{rmk}
		In particular, the ambiguities in the definition of syzygies are resolved upon passing to the stable category: syzygies are really only defined up to projective modules, but projective modules go to zero in $ \stab R$. Moreover, the syzygy of a MCM module is again MCM; this is not hard to see by continuing $R^n\onto X$ to a free resolution $F$ of $X$, truncating and shifting to get a resolution $(\tau_{\leq -1}F)[-1]$ of $\Omega X$, dualising, and using that $F^\vee$ has cohomology only in degree zero to see that $(\tau_{\leq -1}F)^\vee[1] \simeq \R\hom_R(\Omega X, R)$ has cohomology only in degree zero.
		\end{rmk}
\begin{prop}
	Let $R$ be a Gorenstein ring. Then the syzygy functor $\Omega$ is an autoequivalence of $\stab R$. Its inverse $\Omega^{-1}$ makes $\stab R$ into a triangulated category.
	\end{prop}
\begin{rmk}
	Later on, we will primarily be interested in situations where $\Omega\cong \Omega^{-1}$, so the use of the inverse syzygy functor instead of the syzygy functor itself is unimportant to us.
	\end{rmk}
A famous theorem of Buchweitz tells us that the stable category we have just defined is the same as the singularity category:
\begin{thm}Let $R$ be a Gorenstein $k$-algebra. The categories $D_{\mathrm{sg}}(R)$ and $\stab R$ are triangle equivalent, via the map that sends a MCM module $M$ to the object $M \in D^b(R)$.
\end{thm}
	
Hence, we can regard the dg singularity category $D^\mathrm{dg}_{\mathrm{sg}}(R)$ as a dg enhancement of $\stab R$. 
\begin{defn}
	Let $R$ be a Gorenstein $k$-algebra and let $M,N$ be elements of $D^\mathrm{dg}_{\mathrm{sg}}(R)$. Write $\R\underline{\hom}_R(M,N)$ for the complex $\dgh_{D^\mathrm{dg}_{\mathrm{sg}}(R)}(M,N)$ and write $\R\underline{\enn}_R(M)$ for the dga $\dge_{D^\mathrm{dg}_{\mathrm{sg}}(R)}(M)$.
\end{defn}We denote Ext groups in the singularity category by $\underline{\ext}$. Note that $\underline\hom$ coincides with $\underline\ext^0$, and that $\underline{\ext}^j(M,N)\cong H^j\R\underline{\hom}_R(M,N)$. In order to investigate the stable Ext groups, we recall the notion of the complete resolution of a MCM $R$-module -- the construction works for arbitrary complexes in $D^b(\cat{mod}\text{-}R)$.
\begin{defn}[{\cite[5.6.2]{buchweitz}}]
	Let $R$ be a Gorenstein $k$-algebra and let $M$ be any MCM $R$-module. Let $P$ be a projective resolution of $M$, and let $Q$ be a projective resolution of $M^\vee$. Dualising and using that $(-)^\vee$ is an exact functor on MCM modules and on projectives gives us a projective coresolution $M \to Q^\vee$. The \textbf{complete resolution} of $M$ is the (acyclic) complex $\mathbf{CR}(M)\coloneqq \mathrm{cocone}(P \to Q^\vee)$. So in nonpositive degrees, $\mathbf{CR}(M)$ agrees with $P$, and in positive degrees, $\mathbf{CR}(M)$ agrees with $Q^\vee[-1]$.
\end{defn}
\begin{prop}[{\cite[6.1.2.ii]{buchweitz}}]\label{buchcohom}Let $R$ be a Gorenstein $k$-algebra and let $M,N$ be MCM $R$-modules. Then $$\underline{\ext}_R^j(M,N) \cong H^j\hom_R(\mathbf{CR}(M),N) .$$
\end{prop}
\begin{cor}
	Let $R$ be a Gorenstein $k$-algebra and let $M,N$ be MCM $R$-modules.\begin{enumerate}
		\item[\emph{1.}] If $j>0$ then $\underline{\ext}_R^j(M,N)\cong\ext_R^j(M,N)$. 
		\item[\emph{2.}] If $j<-1$ then $\underline{\ext}_R^j(M,N)\cong\tor^R_{-j-1}(N,M^\vee)$.
		\end{enumerate}
	\end{cor}
\begin{proof}
Let $P\to M$ and $Q \to M^\vee$ be projective resolutions. If $j>0$ we have $$H^j\hom_R(\mathbf{CR}(M),N)\cong H^j\hom_R(P,N)\cong\ext^j_R(M,N)$$whereas if $j<-1$ we have
$$H^j\hom_R(\mathbf{CR}(M),N)\cong H^j\hom_R(Q^\vee[-1],N)\cong H^j(N\lot_R M^\vee[1]) \cong \tor^R_{-j-1}(N,M^\vee)$$ where we use \cite[6.2.1.ii]{buchweitz} for the quasi-isomorphism $\R\hom_R(Q^\vee,N) \simeq N\lot_R M^\vee$.
	\end{proof}
Finally, we recall AR duality, which will assist us in some computations later:
	\begin{prop}[Auslander--Reiten duality \cite{auslander}]\label{arduality}
		Let $R$ be a commutative complete local	Gorenstein isolated singularity of Krull dimension $d$. Let $M,N$ be MCM $R$-modules. Then we have $$\underline{\hom}_R(M,N) \cong {\ext}_R^1(N,\Omega^{2-d} M)^*$$where the notation $(-)^*$ denotes the linear dual.
	\end{prop}

	\section{Hypersurfaces and periodicity}\label{hypper}
When $R$ is a commutative complete local hypersurface singularity, it is well-known that the syzygy functor is 2-periodic:
	\begin{thm}[Eisenbud {\cite[6.1(\textit{ii})]{eisenbudper}}]\label{ebudper}
		Let $R$ be a commutative complete local hypersurface singularity over $k$. Then $\Omega^2\cong \id$ as an endofunctor of $\stab R$.
		\end{thm}
	\begin{rmk}
		Eisenbud proves that minimal free resolutions of a MCM module without free summands are 2-periodic. Adding free summands if necessary, one can show that MCM modules always admit 2-periodic resolutions. One can decompose these into short exact sequences to see that the syzygies of $M$ are 2-periodic as claimed.
		\end{rmk}
	
	AR duality \ref{arduality} immediately gives:
		\begin{prop}\label{arduality2}
			Let $R$ be a commutative complete local isolated hypersurface singularity over $k$. Let $M,N$ be MCM $R$-modules. If the Krull dimension of $R$ is even, then one has an isomorphism $$\underline{\hom}_R(M,N) \cong {\ext}_R^1(N,M)^*.$$If the Krull dimension of $R$ is odd, then one has an isomorphism $$\underline{\hom}_R(M,N) \cong {\ext}_R^1(\Omega N,M)^*.$$
	\end{prop}
We can immediately deduce that in the odd-dimensional case, the stable endomorphism algebra is a symmetric algebra:
\begin{prop}
	Let $R$ be a commutative complete local isolated hypersurface singularity of odd Krull dimension over $k$. Let $M$ be a MCM $R$-module. Then the stable endomorphism algebra $\Lambda\coloneqq\underline{\enn}_R(M)$ is a symmetric algebra; i.e.\ there is an isomorphism of $\Lambda$-bimodules $\Lambda\cong \Lambda^*$ between $\Lambda$ and its linear dual $\Lambda^*$.
	\end{prop}
\begin{proof}
This is essentially \cite[7.1]{bikr}; see also \cite[3.3]{jennyf}. Let $N$ be another MCM $R$-module and put $\Gamma\coloneqq \underline{\enn}_R(N)$. Because $R$ has odd Krull dimension, \ref{arduality2} tells us that we have a functorial isomorphism $$\underline{\hom}_R(M,N) \cong {\ext}_R^1(\Omega N,M)^*$$of $\Lambda$-$\Gamma$-bimodules. But stable Ext agrees with usual Ext in positive degrees, so we have functorial isomorphisms $${\ext}_R^1(\Omega N,M) \cong \underline{\ext}_R^1(\Omega N,M)\cong \underline{\ext}_R^0(N,M)\cong \underline{\hom}_R(N,M)$$of $\Gamma$-$\Lambda$-bimodules, because $\Omega$ is the shift. Hence we get a functorial isomorphism $$\underline{\hom}_R(M,N) \cong \underline{\hom}_R(N,M)^*$$of $\Lambda$-$\Gamma$-bimodules. Now put $N=M$.
\end{proof}

	\begin{defn}\label{rigiddefn}
		Call $M \in \stab R$ \textbf{rigid} if $\ext_R^1(M,M)\cong 0$.
	\end{defn}
The following is clear:
\begin{lem}
	Let $R$ be a commutative complete local isolated hypersurface singularity over $k$ of even Krull dimension. If $M$ is a MCM module then it is rigid if and only if it is projective.
	\end{lem}
We will show that 2-periodicity in the singularity category is detected by the derived stable hom-complexes and the derived stable endomorphism algebras. As a warm-up, we will show that this periodicity appears in the stable Ext-algebras.
\begin{lem}
Let $R$ be a Gorenstein $k$-algebra satisfying $\Omega^2\cong\id$. Then there are functorial isomorphisms $\underline{\ext}_R^j(M,N)\cong \underline{\ext}_R^{j-2}(M,N)$ for all MCM $R$-modules $M$ and $N$.
	\end{lem}
\begin{proof}
	There are quasi-isomorphisms $$\R\underline{\hom}_R(M,N)\simeq \R\underline{\hom}_R(M,\Omega^{-2}N)\simeq \R\underline{\hom}_R(M,N)[-2]$$where the first exists by assumption and the second exists since $\Omega^{-1}$ is the shift functor of $\stab R$. Now take cohomology.
	\end{proof}
\begin{cor}\label{perext}
Let $R$ be a Gorenstein $k$-algebra satisfying $\Omega^2\cong\id$, and let $M,N$ be MCM $R$-modules. Then for any integers $i,j$ with $i> -j/2$, there are functorial isomorphisms $\underline{\ext}_R^{j}(M,N)\cong \ext_R^{j+2i}(M,N)$. In particular, if $j<0$ then one has an isomorphism $\underline{\ext}_R^{j}(M,N)\cong {\ext}_R^{-j}(M,N)$. 
	\end{cor}
\begin{proof}
Periodicity in the stable Ext groups gives isomorphisms $\underline{\ext}_R^{j}(M,N)\cong \underline{\ext}_R^{j+2i}(M,N)$. By assumption, $j+2i>0$ so that $\underline{\ext}_R^{j+2i}(M,N)$ agrees with the usual Ext group. The second assertion follows from taking $i=-j$.
	\end{proof}
Recall that by definition each MCM $R$-module $M$ comes with a syzygy exact sequence $$0\to \Omega M \to R^a \to M \to 0$$and one in particular has exact sequences of the form $$0\to \Omega^{i+1} M \to R^{a_i}\to \Omega^{i}M \to 0$$for all $i\geq 0$. One can stitch these together into a finite-rank free resolution of $M$. In particular, if $\Omega^2\cong \id$ then one can stitch them together into a $2$-periodic free resolution, and the endomorphism algebra of such a resolution detects the periodicity:
\begin{defn}
Let $R$ be a Gorenstein $k$-algebra satisfying $\Omega^2\cong\id$, and let $M$ be a MCM $R$-module. Let $\tilde M$ be a $2$-periodic free resolution of $M$. A \textbf{periodicity witness} for $\tilde M$ is a central cocycle $\theta \in \enn^{2}_R( \tilde M)$ whose components $\theta_i:\tilde M _{i-2} \to \tilde M _ i$ for $i \geq 0$ are identity maps, up to sign. 
	\end{defn}
It is clear from the above discussion that periodicity witnesses exist. Because $\enn_R(\tilde M)$ is a model for the derived endomorphism algebra $\R\enn_R(M)$, a periodicity witness hence defines an element of $\ext^2_R(M,M)$. However, note that having a periodicity witness is not a homotopy invariant concept: an element of $\ext^2_R(M,M)$ always lifts to a cocycle in any model for $\R\enn_R(M)$, but need not lift to a central one whose components are identities. We will return to this later. Witnessing elements allow us to explicitly produce a periodic model for the derived stable endomorphism algebra:
	\begin{prop}\label{periodicend}
	Let $R$ be a Gorenstein $k$-algebra satisfying $\Omega^2\cong\id$, and let $M$ be a MCM $R$-module. Let $\tilde M$ be a $2$-periodic free resolution of $M$, with periodicity witness $\theta$. Then there is a quasi-isomorphism of dgas $$\R\underline{\enn}_R(M) \simeq \enn_{R}(\tilde M)[\theta^{-1}] .$$
	\end{prop}
	\begin{proof}We will use \ref{drinfeldlem}. Let $V_n$ be $\tilde M [2n]$, that is, $\tilde M$ shifted $2n$ places to the left. We see that the $V_n$ fit into a direct system with transition maps given by $\theta$. It is not hard to see that $\varinjlim_n V_n$ is acyclic. Projection $\tilde M \to V_n$ defines a map in $\cat{ind}(D^b(R))$ whose cone is clearly ind-perfect, since $V_n$ differs from $\tilde M$ by only finitely many terms. In other words, we have computed $$\R \underline{\enn}_R(M)\simeq \varprojlim_m\varinjlim_n\hom_{ R}(V_m, V_n)$$Temporarily write $E$ for $\enn_{R}(\tilde M)$, so that $\hom_{ R}(V_m, V_n)\cong E[2(n-m)]$. Now, the direct limit $\varinjlim_n E[2(n-m)]$ is exactly the colimit of $E[-2m]\xrightarrow{\theta}E[-2m]\xrightarrow{\theta}E[-2m]\xrightarrow{\theta}\cdots$, which is exactly $E[-2m][\theta^{-1}]$. This dga is $2$-periodic, and in particular $E[-2m][\theta^{-1}] \xrightarrow{\theta} E[-2(m+1)][\theta^{-1}]$ is the identity map. Hence $\varprojlim_m E[-2m][\theta^{-1}]$ is just $E[\theta^{-1}]$, as required.
	\end{proof}
We can state a similar result for the derived stable hom-complexes. Morally, one gets these by periodicising the unstable derived hom-complexes:
\begin{prop}\label{periodichom}
	Let $R$ be a Gorenstein $k$-algebra satisfying $\Omega^2\cong\id$, and let $M,N$ be MCM $R$-modules. Then the derived stable hom-complex $\R\underline{\hom}_R(M,N)$ admits a $2$-periodic model.
	\end{prop}
\begin{proof}
	As before, let $\tilde M$ be a periodic resolution for $M$ and write $M_n\coloneqq \tilde M [2n]$. Similarly let $\tilde N$ be a periodic resolution for $N$ and write $N_n\coloneqq \tilde N [2n]$. Then as before one has a quasi-isomorphism $$\R\underline{\hom}_R(M,N)\simeq \varprojlim_m \varinjlim_n E[2n][-2m]$$where we write $E\coloneqq \hom_R(\tilde M, \tilde N)$, which is a model for $\R\hom_R(M,N)$. The inner colimit $E'\coloneqq \varinjlim_n E[2n]$ is a periodic complex, and the transition maps in the limit $\varprojlim_m E'[-2m]$ all preserve this periodicity, and so the limit is periodic.
\end{proof}

	\begin{rmk}
	One might want to consider the seemingly more general case when $\Omega^p\cong \id$ for some $p \geq 1$. But if $R$ is a commutative Gorenstein local ring with residue field $k$ satisfying $\Omega^p\cong \id$ for some $p$ then, following the proof of \cite[5.10(4)$\implies$(1)]{crollperiodic}, the $R$-module $k$ is eventually periodic and has bounded Betti numbers. Hence $R$ must be a hypersurface singularity by Gulliksen \cite[Cor. 1]{gulliksen}, and in particular one can take $p=2$.
\end{rmk}

		\section{Matrix factorisations}
In this section, let $S\coloneqq  k\llbracket x_1,\ldots, x_n \rrbracket$ be the complete local ring of $k^n$ at the origin, and let $\sigma \in \mathfrak{m}_S$ be nonzero. Let $R\coloneqq S/\sigma$ be the quotient, and assume that $R$ is an isolated hypersurface singularity. We have seen in the last section that the dg category $D^\mathrm{dg}_{\mathrm{sg}}(R)$ is in some sense 2-periodic, as the stable derived hom-complexes admit 2-periodic models. In fact, the morphisms can be rectified to 2-periodic morphisms, in the sense that the dg category $D^\mathrm{dg}_{\mathrm{sg}}(R)$ is equivalent to a 2-periodic dg category, the category of \textbf{matrix factorisations}. Our main reference for this section will be Dyckerhoff \cite{dyck}.

\begin{defn}
	A \textbf{2-periodic dg category} is a category enriched over chain complexes of $k[u,u^{-1}]$-modules, where $u$ has degree 2.
\end{defn} 
Note that every $\Z/2$-graded complex $C^0 \substack{\rightarrow\\[-1em] \leftarrow } C^1$ admits an unwinding to an unbounded $\Z$-graded complex $\cdots \to C^1 \to C^0 \to C^1 \to\cdots$ over $k[u,u^{-1}]$, and every such complex is the unwinding of a $\Z/2$-graded complex. In other words, 2-periodic dg categories are the same thing as $\Z/2$-graded dg categories, and we will switch between the two models when convenient.

\begin{defn}
	The dg category of \textbf{matrix factorisations} over $R$ is a $2$-periodic dg category $\mathbf{MF}(S,\sigma)$ with: \begin{itemize}
		\item Objects pairs $(X,d)$ where $X$ is a free $\Z/2$-graded finitely generated $S$-module, and $d$ is an odd degree map with $d^2=\sigma$;
		\item Morphism complexes given by the unwinding of the natural $\Z/2$-graded morphism complexes.
	\end{itemize}
\end{defn}
\begin{rmk}
	The fact that $\mathbf{MF}(S,\sigma)$ is a dg category abstractly follows from the fact that $\mathbf{MF}(S,\sigma)$ is actually the category of cofibrant objects for a model structure on the category of dg modules over a certain curved $\Z/2$-graded dga; see Positselski \cite[\S3.11]{positselski} or Becker \linebreak \cite[\S3.2]{becker}.
\end{rmk}	

Starting from a matrix factorisation $(X,d)$, the $R$-module $\psi(X)\coloneqq \coker(X^1 \to X^0)$ is actually MCM, and in fact this assignment extends to an equivalence of triangulated categories \mbox{$\psi:[\mathbf{MF}(S,\sigma)] \to \stab(R)$} (see e.g.\ \cite{orlovtri}). Even better, this lifts to a quasi-equivalence of dg categories:
\begin{thm}[\cite{dyck, becker, motivicsingcat}]\label{dgmfs}
	The equivalence $\psi:[\mathbf{MF}(S,\sigma)] \to D_\mathrm{sg}(R)$ of homotopy categories lifts to a quasi-equivalence $\Psi:\mathbf{MF}(S,\sigma) \to D^\mathrm{dg}_\mathrm{sg}(R)$ of dg categories.
\end{thm}
\begin{proof}
Since we will not use this result we provide only a sketch proof. Given a matrix factorisation $X$, the idea is to take its `left unwinding mod $\sigma$' $$\tilde X\coloneqq \cdots\to {X}^1\otimes_R S\to {X}^0\otimes_R S\to {X}^1\otimes_R S\to {X}^0\otimes_R S $$which resolves $\coker(X)$. One defines $\Psi(X)$ to be the ind-dg $R$-module $\Psi(X)=\{X_n\}_{n \in \N}$ with $X_n={\tilde X}[2n]$ and transition maps the obvious ones. There is a projection $\tilde X \to \Psi(X)$ with ind-perfect cone and hence $\Psi(X)$ represents $\coker(X)$ in the dg singularity category $D^\mathrm{dg}_\mathrm{sg}(R)$. One checks that $\Psi$ extends to a dg functor. It is clear that $\Psi$ is quasi-essentially surjective since given $M \in \stab R$, one can find a matrix factorisation $X$ with $\coker(X)\cong M$ by taking a 2-periodic resolution. The proof of \cite[4.2]{dyck} shows that $\Psi$ is quasi-fully faithful.
	\end{proof}

We finish this chapter with some discussions on Hochschild cohomology, Morita theory, and recovery theorems. For more on the Hochschild theory and Morita theory of dg categories, see Keller \cite{keller} or To\"en \cite{toendglectures}.
\begin{defn}
	Let $\mathcal T$ be a dg category. The \textbf{Hochschild complex} $HC(\mathcal T)$ of $\mathcal T$ is the endomorphism dga of the identity functor $\mathcal T$. The \textbf{Hochschild cohomology} of $\mathcal T$ is the graded algebra $HH^*(\mathcal T)\coloneqq H^*(HC(\mathcal T))$.
	\end{defn}
We remark that this definition applies both to usual $\Z$-graded dg categories as well as $\Z/2$-graded dg categories.
\begin{prop}[{\cite[5.2]{keller}}]
	Hochschild cohomology is invariant under Morita equivalences. In particular, Hochschild cohomology is invariant under quasi-equivalences. 
	\end{prop}
\begin{thm}[Dyckerhoff \cite{dyck}]
The Hochschild cohomology of the $\Z/2$-graded dg category $\mathbf{MF}(S,\sigma)$ is the Milnor algebra of $\sigma$, concentrated in even degree.
	\end{thm}
\begin{cor}
	Suppose that $\sigma$ is quasi-homogeneous. Then the Morita equivalence type of $\mathbf{MF}(S,\sigma)$, considered as a $\Z/2$-graded dg category, together with the integer $n$, recovers the ring $R$.
	\end{cor}
\begin{proof}
	Because $\sigma$ is quasi-homogenous, the Milnor algebra of $\sigma$ is isomorphic to the Tjurina algebra, which recovers $R$ by the algebraic Mather--Yau theorem.
	\end{proof}
Note that the Hochschild cohomology of $\mathbf{MF}(S,\sigma)$, considered as a $\Z/2$-graded dg category, may differ considerably from the Hochschild cohomology when one considers $\mathbf{MF}(S,\sigma)$ as just a $\Z$-graded dg category. However, an analogue of Dyckerhoff's theorem still holds in the $\Z$-graded world:
\begin{thm}[Hua--Keller \cite{huakeller}]
	The zeroth Hochschild cohomology of the $\Z$-graded dg category $\mathbf{MF}(S,\sigma)$ is the Tjurina algebra of $\sigma$.
\end{thm}
The proof relies on the machinery of singular Hochschild cohomology developed in \cite{kellersing} to reduce to a computation of the Hochschild cohomology of the ring $R$, which was done in \cite{buenosaires}. Using Hua and Keller's result, one can deduce another recovery theorem:
\begin{cor}
	The Morita equivalence type of $\mathbf{MF}(S,\sigma)$, considered as a $\Z$-graded dg category, together with the integer $n$, recovers the ring $R$.
\end{cor}
For future reference, we will state a weaker version of the above Corollary in language that will be more useful to us:
\begin{thm}\label{recoverythm}
Let $R=k\llbracket x_1,\ldots, x_n \rrbracket / \sigma$ and $R'=k\llbracket x_1,\ldots, x_n \rrbracket / \sigma'$ be isolated hypersurface singularities. If the dg singularity categories $D^\mathrm{dg}_\mathrm{sg}(R)$ and $D^\mathrm{dg}_\mathrm{sg}(R')$ are quasi-equivalent, then $R$ and $R'$ are isomorphic.
\end{thm}

\chapter{Noncommutative partial resolutions}\label{chsl}
In this technical chapter we focus on noncommutative partial resolutions, which are certain rings of the form $A=\enn_R(R\oplus M)$. These naturally come with idempotents $e=\id_R$, and we study properties of the associated derived quotient $\dq$. We will assume that the cornering $R$ is Gorenstein, which will allow us to use Buchweitz's machinery of the stable category. We will specialise to the case when $R$ is a commutative complete local hypersurface singularity, where 2-periodicity in the singularity category will give us periodicity in $\dq$ (\ref{etaex}), which will allow us to identify the cohomology algebra of $\dq$ explicitly. We will apply the recovery result \ref{recoverythm} to prove that in certain situations, the quasi-isomorphism class of $\dq$ recovers the geometry of $R$ (\ref{recov}).
			
	\section{The singularity functor} In this section we recall some results of Kalck and Yang \cite{kalckyang,kalckyang2} on relative singularity categories, as seen from the perspective of the derived quotient. We introduce a key technical object, the singularity functor, which links derived quotients to singularity categories. Let $A$ be a right noetherian $k$-algebra with an idempotent $e$, and write $R\coloneqq eAe$ for the cornering. Recall from \ref{recoll} the existence of the recollement $D(\dq)\recol D(A) \recol D(R)$, and recall from \ref{relsingcatdefn} the definition of the relative singularity category $\Delta_R(A)\coloneqq D^b(A)/\thick(eA)$.  The map $j^*: D(A) \to D(R)$ sends $\thick(eA)$ into $\per R$, and hence defines a map $j^*:\Delta_R(A) \to D_\mathrm{sg}R$. In fact, $j^*$ is onto, which follows from \cite[3.3]{kalckyang}. We are about to identify its kernel.
	\begin{defn}\label{dfgdefn}
		Write $ D_\mathrm{fg}(\dq)$ for the subcategory of $D(\dq)$ on those modules whose total cohomology is finitely generated over $A/AeA$. Similarly, write $\per_\mathrm{fg}(\dq)$ for the subcategory of $\per(\dq)$ on those modules whose total cohomology is finitely generated over $A/AeA$.
	\end{defn}
	\begin{lem}
The kernel of the map $j^*:\Delta_R(A) \to D_\mathrm{sg}R$ is precisely $D_\mathrm{fg}(\dq)$.
	\end{lem}
	
	\begin{proof}
		The proof of \cite[6.13]{kalckyang} shows that $\ker j^* \cong \thick_{D(A)}(\cat{mod}\text{-}A/AeA)$, so it suffices to show that $\thick_{D(A)}(\cat{mod}\text{-}A/AeA)\cong D_\mathrm{fg}(\dq)$. But this can be shown to hold via a modification of the proof of \cite[2.12]{kalckyang}.
	\end{proof}
	\begin{rmk}
		If $A/AeA$ is a finite-dimensional algebra, let $\mathcal{S}$ be the set of one-dimensional $A/AeA$-modules corresponding to a set of primitive orthogonal idempotents for $A/AeA$. Then $D_\mathrm{fg}(\dq)\cong \thick(\mathcal{S})$. Because each simple in $\mathcal{S}$ need not be perfect over $\dq$, the category $\per_\mathrm{fg}(\dq)$ may be smaller than $D_\mathrm{fg}(\dq)$. If each simple is perfect over $A$, or if $\dq$ is homologically smooth, then we have an equivalence $D_\mathrm{fg}(\dq)\cong \per_\mathrm{fg}(\dq)$. 
	\end{rmk}
	 When the singularity category is idempotent complete, Kalck and Yang observed that there is a triangle functor ${\Sigma: \per(\dq) \to D_{\mathrm{sg}}(R)}$, sending $\dq$ to the right $R$-module $Ae$. We establish this with a series of results. Recall that when $\mathcal T$ is a triangulated category, ${\mathcal T}^\omega$ denotes the idempotent completion of $\mathcal T$.
\begin{lem}\label{flem}
There is a triangle functor $F:\per(\dq)\to \Delta_R(A)^\omega$ which sends $\dq$ to the object $A$.
	\end{lem}
\begin{proof}
	As in \cite[2.12]{kalckyang} (which is an application of Neeman--Thomason--Trobaugh--Yao localisation; cf.\ \ref{ntty}), the map $i^*$ gives a triangle equivalence $$i^*:\left(\frac{\per A}{j_!\per R}\right)^\omega \xrightarrow{\cong} \per (\dq).$$The inclusion $\per A \into D^b(A)$ gives a map $G:\per A / j_!\per R \to \Delta_R(A)$, which is a triangle equivalence if $A$ has finite right global dimension. The composition $$F:\per(\dq) \xrightarrow{(i^*)^{-1}} \left(\frac{\per A}{j_!\per R}\right)^\omega \xrightarrow{G^\omega} \Delta_R(A)^\omega$$ is easily seen to send $\dq$ to $A$. 
	\end{proof}
\begin{lem}\label{preontolem}
	Suppose that $A$ is of finite right global dimension. Then the map $F$ of \ref{flem} is a triangle equivalence $\per(\dq)\to \Delta_R(A)^\omega$.
\end{lem}	 
\begin{proof}
	When $A$ has finite global dimension then $F$ is a composition of triangle equivalences and hence a triangle equivalence.
	\end{proof}

	\begin{prop}[cf.\ {\cite[6.6]{kalckyang2}}]\label{kymap}
		Suppose that $D_\mathrm{sg}(R)$ is idempotent complete. Then there is a map of triangulated categories $\Sigma:\per(\dq) \to D_\mathrm{sg}(R)$, sending $\dq$ to $Ae$. Moreover $\Sigma$ has image $\thick_{D_\mathrm{sg}(R)}(Ae)$ and kernel $\per_\mathrm{fg}(\dq)$.
	\end{prop}
	\begin{proof}		
We already have a map $j^*:\Delta_R(A)\to D_\mathrm{sg}(R)$, with kernel $D_\mathrm{fg}(\dq)$. Let $\Sigma$ be the composition
$$\Sigma: \per(\dq) \xrightarrow{F} \Delta_R(A)^\omega \xrightarrow{(j^*)^\omega} D_{\mathrm{sg}}(R)$$ of the functor $F$ of \ref{flem} and the idempotent completion of $j^*$. It is easy to see that $\Sigma$ sends $\dq$ to $Ae$, and since $A$ generates $\per A / j_!\per R$, then $\Sigma $ has image $\thick(Ae)$. The kernel of $\Sigma$ is the intersection of $\per (\dq)$ and $D_\mathrm{fg}(\dq)$, which is precisely $\per_\mathrm{fg}(\dq)$.
	\end{proof}
	For future reference, it will be convenient to give $\Sigma$ a name.
	\begin{defn}
		We refer to the triangle functor $\Sigma$ of \ref{kymap} as the \textbf{singularity functor}.
	\end{defn}
\begin{prop}\label{kymapbetter}
	Suppose that $D_\mathrm{sg}(R)$ is idempotent complete. Then the triangle functor $\Sigma$ induces a triangle equivalence $${\bar{\Sigma}}:\frac{\per(\dq)}{\per_{\mathrm{fg}}(\dq)} \to \thick_{D_\mathrm{sg}(R)}(Ae).$$
	\end{prop}
\begin{proof}
	Follows immediately from \ref{kymap}.
	\end{proof}
When $A$ is smooth, one can do better:
\begin{lem}\label{ontolem}
	Suppose that $A$ is of finite right global dimension and $D_\mathrm{sg}(R)$ is idempotent complete. Then the singularity functor $\Sigma$ is onto.
	\end{lem}
\begin{proof}
	The singularity functor is the equivalence of \ref{preontolem} followed by the surjective triangle functor $j^*:\Delta_R(A)\to D_\mathrm{sg}(R)$.
	\end{proof}
It is unclear to the author if the converse of the above statement is true:
\begin{conj}\label{mgenconj}
	With notation as above, if $D_\mathrm{sg}(R)$  is idempotent complete then $M$ generates the singularity category $D_\mathrm{sg}(R)$ if and only if $A$ has finite global dimension.
	\end{conj}
	\begin{prop}[{\cite[1.2]{kalckyang2}}]\label{dsgsmooth}Suppose that $A$ is of finite right global dimension and $D_\mathrm{sg}(R)$ is idempotent complete. Then the singularity functor induces a triangle equivalence $${\bar{\Sigma}}:\frac{\per(\dq)}{D_{\mathrm{fg}}(\dq)} \to D_\mathrm{sg}(R).$$
	\end{prop}
	\begin{proof}
		The singularity functor is surjective by \ref{ontolem}. Since $A$ has finite right global dimension, every finitely generated $\dq$-module is perfect over $A$. Since $i^*$ respects compact objects, $D_{\text {fg}}(\dq) \cong i^*i_*D_{\text {fg}}(\dq)\subseteq \per(\dq)$. So $\per_\mathrm{fg}(\dq)=D_\mathrm{fg}(\dq)$.
	\end{proof}
	\begin{rmk}
		This equivalence is essentially the same as that of \cite[5.1.1]{dtdvvdb}.
	\end{rmk}
	\begin{rmk}
		The quotient $\frac{\per(\dq)}{D_{\text {fg}}(\dq)}$ is reminiscent of the cluster category associated to a Ginzburg dga, and indeed \cite{huakeller} refers to it as the \textbf{generalized cluster category}.
	\end{rmk}
When $A$ is not smooth, one can measure the extent to which ${\bar{\Sigma}}$ is not an equivalence. Recall that a sequence $A\to B \to C$ of triangulated categories is \textbf{exact up to direct summands} if $A\into B$ is a thick subcategory, the composition $A\to C$ is zero, and the natural functor $B/A \to C$ is an equivalence after idempotent completion (i.e.\ every object of $C$ is a summand of $B/A$).
\begin{prop}\label{ninelem}
The inclusion $$D_\mathrm{fg}(\dq) \into \Delta_R(A)$$ and the projection $$D^b(A) \onto D_{\mathrm{sg}}(R)$$ induce a sequence$$\frac{D_\mathrm{fg}(\dq) }{\per_\mathrm{fg}(\dq)} \longrightarrow D_\mathrm{sg}(A) \longrightarrow \frac{D_{\mathrm{sg}}(R)}{\thick_{D_\mathrm{sg}(R)}(Ae)}$$which is exact up to direct summands.
	\end{prop}

\begin{proof}
For brevity, put $Q\coloneqq D_{\mathrm{fg}}(\dq) /\per_\mathrm{fg}(\dq)$ and $K\coloneqq  {D_{\mathrm{sg}}(R)}/{\thick(Ae)}$. We have a map $D_{\text fg}(\dq) \into \Delta_R(A) \onto D_\mathrm{sg}(A)$, because $\Delta_R(A)\coloneqq D^b(A)/ j_!\per R$ has a map to $D_\mathrm{sg}(A)\coloneqq  D^b(A) / \per A$, and $\per A$ contains $j_!\per R$. By construction, this kills $\per_\mathrm{fg}(\dq)$, and so we get a map $Q \to D_{\mathrm{sg}}A$. We show that the following diagram commutes:
$$\begin{tikzcd}
Q \ar[r]& D_\mathrm{sg}(A)\ar[r] & K
\\ D_\mathrm{fg}(\dq) \ar[u]\ar[r]& \Delta_R(A) \ar[r,"j^*"]\ar[u]& D_\mathrm{sg}(R)\ar[u]
\\ \per_\mathrm{fg}(\dq)\ar[r]\ar[u]  & \per(\dq)\ar[r,"\Sigma"]\ar[u] & \thick_{D_\mathrm{sg}(R)}(Ae)\ar[u] 
\end{tikzcd}$$Indeed, the lower right hand square commutes by definition of $\Sigma$. The upper right hand square commutes because both maps are $j^*$. The upper left hand square commutes by definition of the map from $Q$. The lower left hand square commutes because the maps are both just $i_*$. Now I claim that all of the columns are exact up to direct summands: this is clear for the left-hand and right-hand ones. The central column is certainly exact at either end, so we just need to check exactness in the middle. But this is not too hard to see: $\per(\dq)$ is triangle equivalent to $\per A / j_! \per R$ up to direct summands. Hence the quotient of $\Delta_R(A)\coloneqq D^b(A)/ j_!\per R$ by $\per(\dq)$ is the quotient of $D^b(A)$ by $\per A$, which is the definition of $D_{\mathrm{sg}}(A)$. We have already seen that the lower two rows are exact, so by the Nine Lemma (applied to the analogous diagram of idempotent completions) it follows that the top row is exact up to direct summands.
\end{proof}

	\begin{rmk}Intuitively, \ref{ninelem} tells us that $\dq$ is not more singular than $A$. Indeed, if $B$ is an unbounded dga then $D_\mathrm{fg}(B)$ should be thought of as  $D^b(B)$, and the quotient $D_\mathrm{fg}(B) /\per_\mathrm{fg}(B)$ should be thought of as the singularity category $D_{\mathrm{sg}}(B)$.
	\end{rmk}
\begin{rmk}Suppose that $A/AeA$ is finite-dimensional. Then the quotient $$\frac{D_\mathrm{fg}(\dq) }{\per_\mathrm{fg}(\dq)}$$ vanishes if and only if every $\dq$-module of finite total cohomological dimension is perfect, which is equivalent to $S\coloneqq (A/AeA)/\mathrm{rad}(A/AeA)$ being perfect as a $\dq$-module. The quotient $$\frac{D_{\mathrm{sg}}(R)}{\thick_{D_\mathrm{sg}(R)}(Ae)}$$vanishes if and only if $Ae$ generates the singularity category of $R$.
\end{rmk}

	\section{The dg singularity functor and the comparison map}
	In this section, suppose that $A$ is a right noetherian $k$-algebra with idempotent $e$. Put $R\coloneqq eAe$ and assume that $D_{\mathrm{sg}}(R)$ is idempotent complete. We show in \ref{Fdg} that the singularity functor $\Sigma:\cat{per}(\dq) \to D_{\mathrm{sg}}(R)$ lifts to a dg functor. The component of the singularity functor at $\dq$ is a dga map from $\dq$ to an endomorphism dga in ${D}^{\mathrm{dg}}_\mathrm{sg}(R)$, and later in \ref{qisolem} we examine the induced map on cohomology.
	\begin{prop}\label{Fdg}The singularity functor $\Sigma:\cat{per}(\dq) \to D_{\mathrm{sg}}(R)$ lifts to a dg functor $ \Sigma_\mathrm{dg}:\cat{per}_\mathrm{dg}(\dq) \to {D}^{\mathrm{dg}}_\mathrm{sg}(R)$, which we refer to as the \textbf{dg singularity functor}.
	\end{prop}
	\begin{proof}
		We simply mimic the proof (\ref{kymap}) from the triangulated setting. Recalling from \ref{kymap} the construction of $\Sigma$ as a composition $\per(\dq) \xrightarrow{\Sigma_1} \Delta_R(A) \xrightarrow{\Sigma_2} {D}_\mathrm{sg}(R)$, we lift the two maps separately to dg functors. To lift $\Sigma_1$, first note that \ref{ntty} and \ref{Rcoloc} provide a homotopy cofibre sequence of dg categories $\per_{\mathrm{dg}}R \to \per_{\mathrm{dg}}A \to \per_{\mathrm{dg}}(\dq)$, in which the first map is $j_!$. There is a homotopy cofibre sequence $\per_{\mathrm{dg}}R \to D_\mathrm{dg}^b(A) \to \Delta^\mathrm{dg}_R(A)$, and we can extend $\id: \per_{\mathrm{dg}}R \to \per_{\mathrm{dg}}R$ and the inclusion $\per_{\mathrm{dg}}A \into D_\mathrm{dg}^b(A)$ into a morphism of homotopy cofibre sequences, which gives a lift of $\Sigma_1$. Lifting $\Sigma_2=j^*$ is similar and uses the sequence $\per_{\mathrm{dg}}R \to D_\mathrm{dg}^b(R) \to {D}^{\mathrm{dg}}_\mathrm{sg}(R)$.
	\end{proof}
	\begin{rmk}\label{dgperfgrmk}
		By \ref{kymap}, the kernel of $\Sigma_\mathrm{dg}$ is a dg enhancement of the triangulated category $\per_\mathrm{fg}(\dq)$.
	\end{rmk}
	Observe that $\Sigma(\dq)\simeq Ae$. Since we can canonically identify $\dq$ with the endomorphism dga $\dge_{\per_{\mathrm{dg}}(\dq)}(\dq)$, the component of $\Sigma_\mathrm{dg}$ at $\dq$ gives a dga map $\dq \to \dge_{D^\mathrm{dg}_\mathrm{sg}(R)}(Ae)$.
	\begin{defn}\label{comparisonmap}
		The \textbf{comparison map} $\Xi:\dq \to \dge_{D^\mathrm{dg}_\mathrm{sg}(R)}(Ae)$ is the component of the dg singularity functor $\Sigma_{\mathrm{dg}}$ at the object $\dq \in {\per_{\mathrm{dg}}(\dq)}$.
	\end{defn}
In other words, the comparison map is the morphism of dgas given by $$\dq \cong \dge_{\dq}(\dq)\xrightarrow{\Sigma} \dge(\Sigma(\dq))\simeq \dge(Ae).$$ In the next section, we will examine the comparison map and show that, under certain homological conditions on $R$ and $M$, the map $\Xi$ is a `quasi-isomorphism in nonpositive degrees'; i.e.\ the truncated map $\Xi:\dq \to \tau_{\leq 0}\dge(Ae)$ is a quasi-isomorphism.

\section{Partial resolutions of Gorenstein rings}
In this section we introduce the concept of a `noncommutative partial resolution' of a commutative Gorenstein ring, which generalises Van den Bergh's notion of an NCCR, and then we will prove an important theorem about the comparison map associated to such resolutions.
\begin{defn}\label{partrsln}
	Let $R$ be a commutative Gorenstein $k$-algebra. A $k$-algebra $A$ is a \textbf{(noncommutative) partial resolution} of $R$ if it is of the form $A\cong\enn_R(R\oplus M)$ for some MCM $R$-module $M$. Note that $A$ is a finitely generated module over $R$, and hence itself a noetherian $k$-algebra. Say that a partial resolution is a \textbf{resolution} if it has finite global dimension.
\end{defn}
Recall that if $M$ is an $R$-module then we write $M^\vee$ for the $R$-linear dual $\hom_R(M,R)$. If $A=\enn_R(R\oplus M)$ is a noncommutative partial resolution of $R$, observe that $e\coloneqq \id_R$ is an idempotent in $A$. One has $eAe\cong R$, $Ae \cong R \oplus M$, and $eA \cong R \oplus M^\vee$; in particular $(Ae)^\vee\cong eA$. Note that $Ae\cong M $ in the singularity category, and indeed we have $A/AeA\cong \underline{\enn}(M)$. 
\p In the sequel, we will refer to the following setup:
\begin{setup}\label{sfsetup}
	Let $R$ be a Gorenstein $k$-algebra such that $D_{\mathrm{sg}}(R)$ is idempotent complete. Fix a MCM $R$-module $M$ and let $A=\enn_R(R\oplus M)$ be the associated partial resolution. Let $e\in A$ be the idempotent $e=\id_R$. 
\end{setup}
Note that, to the above data, one can canonically attach a dga $\dq$. Since ${R}\cong 0$ in the stable category, $\R \underline{\enn}_{R}({R}\oplus M)$ is naturally quasi-isomorphic to $\R \underline{\enn}_{R}(M)$. Hence, the stable derived endomorphism algebra $\R \underline{\enn}_R(M)$ gets the structure of an $A$-module. Clearly, $\R \underline{\enn}_R(M) e$ is acyclic, and so $\R \underline{\enn}_R(M)$ is in fact a module over $\dq$.

\p The main theorem of this section is that in this setup the comparison map $$\Xi:\dq \to \R\underline{\enn}_R(M)$$ is a `quasi-isomorphism in nonnegative degrees'. The strategy will be as follows. Write $Q\coloneqq \dq$ for brevity. We are going to use the `Drinfeld quotient' model $B$ for $Q$ that appears in \ref{drinfeld}, and use this to write down an explicit model for $M$. This will allow us to calculate an explicit model for $\R\underline{\enn}_R(M)$. At this point we will be able to see that $Q$ and the truncation $\tau_{\leq 0}\R\underline{\enn}_R(M)$ are quasi-isomorphic as abstract complexes. We think of the comparison map $\Xi$ as giving an action of $B$ on (our model for) $M$, which we are able to explicitly identify. We are also able to explicitly identify the action of $\tau_{\leq 0}\R\underline{\enn}_R(M)$ on $M$. This allows us to explicitly identify the comparison map and it will be very easy to see that it is a quasi-isomorphism.
\begin{thm}\label{qisolem}Suppose that we are in the situation of Setup \ref{sfsetup}. Then, for all $j\leq 0$, the comparison map $\Xi:\dq \to \R\underline{\enn}_R(M)$ induces isomorphisms $$H^j(\Xi): H^j(\dq) \xrightarrow{\cong} \underline{\ext}_R^j(M,M).$$
\end{thm}

\begin{proof}
We begin by writing down a model for $M$. Let $B$ be the model for $Q$ from \ref{drinfeld}.
We immediately replace $M$ by the isomorphic (in the singularity category!) object $Ae$. Let $\mathrm{B}(Ae)=\cdots \to Ae \otimes_k R \otimes_k R \to Ae\otimes_k R$ be the bar resolution of the $R$-module $Ae$. Letting $\mathrm{Bar}(Ae)$ be the filtered system of perfect submodules of $B(Ae)$, we then have $B(Ae)\cong \varinjlim \mathrm{Bar}(Ae)$. Upon applying $-\otimes_R eA$ to $B(Ae)$ we obtain (one model of) $\cell A$. Noting that $j_!=-\lot_R eA$, we hence see that $\mathrm{Bar}(Ae)\otimes_R eA$ is an object of $\cat{ind}j_!\per_{\mathrm{dg}} R$. Since tensor products commute with filtered colimits, we have$$\varinjlim(\mathrm{Bar}(Ae)\otimes_R eA)\cong  \cdots \to Ae\otimes R\otimes R\otimes eA\to Ae\otimes R\otimes eA\to Ae\otimes eA\simeq \cell A$$where the tensor products are taken over $k$. Put $T\coloneqq \varinjlim(\mathrm{Bar}(Ae)\otimes_R eA)$ and observe that $T$ is isomorphic to the $A$-bimodule $(\tau_{<0}B)[-1]$ that appears in \ref{drinfeldmodel} and the alternate proof of \ref{dqexact}. Note that $\mathrm{Bar}(Ae)\otimes_R eA$ also comes with a multiplication map $\mu$ to $A$ that lifts the multiplication $Ae\otimes_R eA \to A$. Let $C$ be the cone of this map; then $C$ is an ind-bounded module with a map from $A$ whose cone is in $\cat{ind}j_!\per_{\mathrm{dg}} R$. In fact, $\varinjlim C \simeq \dq$ by \ref{dqexact}. Hence, if $P\in j_!\per_{\mathrm{dg}} R=\thick_{\mathrm{dg}}(eA)$, then $\dgh_{A}(P,C) \simeq\R\hom_A(P,\dq)\simeq 0$, since $P$ is compact and we have the semi-orthogonal decomposition of \ref{semiorthog}. Hence $C$ is a representative of the $A$-module $\dq$ in the Drinfeld quotient $\Delta^\mathrm{dg}_R(A)$.

\p Note that the dg functor $j^*: \Delta^\mathrm{dg}_R(A) \to {D}^{\mathrm{dg}}_\mathrm{sg}(R)$ is simply multiplication on the right by $e$. Hence, sending $C$ through this map, we obtain an ind-object $Ce\simeq\mathrm{cone}(\mathrm{Bar}(Ae) \to Ae)$ that represents $\Sigma(\dq)\simeq Ae$ in the dg singularity category $D^\mathrm{dg}_\mathrm{sg}(R)$. As an aside, one can check this directly: since $B(Ae)$ resolves $Ae$, and mapping cones commute with filtered colimits, it is clear that $\varinjlim Ce$ is acyclic, and that $Ce$ admits a map from $Ae \in D^b(R)$ whose cone is the ind-perfect $R$-module $\mathrm{Bar}(Ae)$.

\p Now we will explicitly identify the dga $\R\underline{\enn}_R(Ae)\simeq \dge(Ce)$; to simplify notation, we will frequently omit subscripts. Write $Ce=\{W_\alpha\}_\alpha$, where each $W_\alpha$ is a cone $V_\alpha \to Ae$ with $V_\alpha$ perfect. We compute
\begin{align*}
		\dge(Ce) & \coloneqq \varprojlim_\alpha\varinjlim_\beta \R\hom_R\left(\mathrm{cone}(V_\alpha \to Ae), W_\beta\right) 
		\\ & \simeq \varprojlim_\alpha\varinjlim_\beta \mathrm{cocone}\left( \R\hom_R(Ae, W_\beta) \to\R\hom_R(V_\alpha, W_\beta)\right)
		\\ & \simeq \varprojlim_\alpha \mathrm{cocone}\left(\varinjlim_\beta  \R\hom_R(Ae, W_\beta) \to \varinjlim_\beta\R\hom_R(V_\alpha, W_\beta)\right)
		\\ & \simeq \varprojlim_\alpha \mathrm{cocone}\left(\varinjlim_\beta  \R\hom_R(Ae, W_\beta) \to \R\hom_R(V_\alpha, \varinjlim_\beta W_\beta)\right)&&\text{since } V_\alpha \text{ is compact}
		\\ & \simeq \varinjlim_\beta  \R\hom_R(Ae, W_\beta) &&\text{since } \varinjlim_\beta W_\beta\simeq 0.
		\\ & \simeq  \varinjlim_\beta \mathrm{cone}\left(\R\hom_R(Ae,V_\beta) \to \R\hom_R(Ae,Ae)\right) 
		\\ & \simeq \mathrm{cone}\left(\varinjlim_\beta \R\hom_R(Ae,V_\beta) \to \R\enn_R(Ae)\right). 
		\end{align*}
Fix a $\beta$ and consider $\R\hom_R(Ae, V_\beta)$. Since $V_\beta$ is perfect, and $Ae$ has some finitely generated projective resolution, we can write $\R\hom_R(Ae, V_\beta)\simeq V_\beta \otimes_R \R\hom_R(Ae,R)$. Since $M$ is MCM we have an isomorphism $\R\hom_R(Ae,R)\simeq\hom_R(Ae,R)$, so that $\R\hom_R(Ae, V_\beta)$ is quasi-isomorphic to the tensor product $V_\beta \otimes_R \hom_R(Ae,R)$. The natural isomorphism $eA\otimes_A Ae \to R$ gives an isomorphism $\hom_R(Ae,R)\cong eA$, so we get $\R\hom_R(Ae, V_\beta)\simeq V_\beta \otimes_R eA$. So we have $$\varinjlim_\beta \R\hom_R(Ae,V_\beta)\simeq\varinjlim_\beta(  V_\beta \otimes_R eA)\cong B(Ae)\otimes_R eA\simeq T.$$Hence, we have $\dge(Ce) \cong \mathrm{cone}(T \to \R\enn_R(Ae))$. From the description above, we see that the map $T \to \R\enn_R(Ae)$ is exactly the multiplication map $\mu$. More precisely, $xe \otimes ey \in T^0$ is sent to the derived endomorphism that multiplies by $xey \in AeA$ on the left. Putting $B'\coloneqq \tau_{\leq 0} \dge(Ce)$, we hence have a quasi-isomorphism $B'\simeq \mathrm{cone}(T \xrightarrow{\mu} A)$. 

\p By \ref{dqexact}, one has a quasi-isomorphism $B\simeq \mathrm{cone}(T \xrightarrow{\mu} A)$, so that we can already conclude that $B$ and $B'$ are abstractly quasi-isomorphic as complexes. Now we need to identify the map $\Xi: B \to B'$. We use the explicit description of $B$ given in \ref{drinfeldmodel}. Take a tensor $t=x\otimes r_1 \otimes \cdots \otimes r_i \otimes y\in B$; we remark that we allow $i=0$, in which case the convention is that $t\in A$. The action of $B$ on itself is via concatenating tensors; i.e.\ $$t.\left(w\otimes s_1 \otimes \cdots \otimes s_j \otimes z\right)=x\otimes r_1 \otimes \cdots \otimes r_i \otimes yw\otimes s_1 \otimes \cdots \otimes s_j \otimes z$$(one can check that this also holds when at least one of $i$ or $j$ are zero).

\p We think of the map $\Xi: B\to \dge(Ce)$ as giving an action of $B$ on $Ce$. More precisely, one takes the action of $B$ on itself and sends it through the map $\Sigma$ to obtain an action of $B$ on $Ce$. It is not hard to check this action of $B$ on $Ce$ is via concatenation; more precisely take $c=a\otimes l_1 \otimes \cdots \otimes l_k \otimes b \in Ce \simeq \mathrm{cone}(\mathrm{Bar}(Ae) \to Ae)$; then with notation as before we have $$t.c=x\otimes r_1 \otimes \cdots \otimes r_i \otimes ya\otimes l_1 \otimes \cdots \otimes l_k \otimes b$$(which remains true for $i=0$). Because $Ce$ is an ind-object, when we write this we mean that $c$ is an element of some level $(Ce)_\alpha$, and $t.c$ is an element of some other level $(Ce)_\beta$, and we identify $c$ as an element of $(Ce)_\beta$ along the canonical map $(Ce)_\alpha \to (Ce)_\beta$. With the convention that for $k=0$, the element $a\otimes l_1 \otimes \cdots \otimes l_k \otimes b$ is just an element of $Ae$, the above is also true for $k=0$.

\p I claim that across the identification $B'\simeq \mathrm{cone}(T \xrightarrow{\mu} A)$, the action of $B'$ on $Ce$ is precisely the action described above. Showing this claim will prove the theorem, since we have then factored the dga map $\tau_{\leq 0}\Xi$ into a composition of two quasi-isomorphisms $B \to \mathrm{cone}(T \xrightarrow{\mu} A) \to B'$.

\p We saw that we could write $\tau_{\leq 0}B'$ as a cone of the form $\mathrm{cone}\left(\varinjlim_\beta \R\hom_R(Ae,V_\beta) \to A\right)$ where the left hand part acts on $Ce$ by sending $Ae$ into $\mathrm{Bar}(Ae)=\{V_\beta\}_\beta$ in the obvious manner, and the right hand part acts on $Ce$ by sending $Ae$ into itself by multiplication on the left. It is clear that across the quasi-isomorphism $B \to B'$, the $A$ summand in the cone acts in the correct manner. So we need to check that the action of $\varinjlim\R\hom_R(Ae,V_\beta)$ on $Ce$ corresponds to the concatenation action of $T$ on $Ce$ provided by $\Xi$.

\p But across the quasi-isomorphism $V_\beta \otimes_R eA \simeq \R\hom_R(Ae, V_\beta)$, an element $v\otimes x$ corresponds to the morphism that sends $y \mapsto v\otimes xy$. Taking limits, we see that across the quasi-isomorphism $T \simeq \varinjlim_\beta \R\hom_R(Ae,V_\beta)$, an element $a\otimes\cdots \otimes x$ corresponds to the morphism that sends $y$ to $a\otimes\cdots \otimes xy$. But this is precisely the concatenation action of $T$ on $Ce$.
\end{proof}

\begin{rmk}
One can think of the above theorem as a categorified version of \ref{buchcohom}: firstly, the proof above writes $\dge(Ce)$ as a cone $\cell A \to \R\enn_R(Ae)$ with $\cell A$ in negative degrees. Hence there is a quasi-isomorphism $\tau_{>0}\R\enn_R(Ae) \to \tau_{>0} \dge(Ce)$. Secondly, for $j<-1$, there is an isomorphism $\underline{\ext}_R^j(Ae,Ae)\cong \tor^R_{-j-1}(Ae,eA)$ which one obtains from \ref{derquotcohom}.
	\end{rmk}
	\begin{rmk}Morally, we ought to have $\R\underline{\enn}_R(M) \simeq \hom_R(\mathbf{CR}(M),M)$, but the right hand side does not admit an obvious dga structure. Note that $\mathbf{CR}(M)$ is glued together out of a projective resolution $P$ and a projective coresolution $Q^\vee$ of $M$, and hence we ought to have $$\R\underline{\enn}_R(M) \simeq \mathrm{cone}\left[ \hom_R(Q^\vee, M) \to \hom_R(P,M)\right] \simeq \mathrm{cone}\left[ \cell A\to \R\enn_R(M)\right]$$which we do indeed obtain during the course of the proof of \ref{qisolem}.
	\end{rmk}

	\section{Periodicity in the derived quotient}\label{dqperiod}
	Assume in this part that $R$ is a complete local isolated hypersurface singularity and that $M$ is a MCM $R$-module. Put $A\coloneqq \enn_R(R\oplus M)$ and $e\coloneqq \id_R$. Note that by \ref{sgidemref}, the singularity category of $R$ is idempotent complete, so that we are in the situation of Setup \ref{sfsetup}. Because $R$ is a complete local hypersurface singularity, by \ref{ebudper} the shift functor of $\stab R$ is 2-periodic. The following lemma is useful:
	\begin{lem}\label{dqcohom}
Let $j \in \N$. Then there are isomorphisms

$$H^j(\dq)\cong \begin{cases} 0 & j>0
\\ \underline{\enn}(M) & j=0
\\ \ext_R^{-j}(M,M) & j<0\end{cases}
$$
	 \end{lem}
 \begin{proof}
The only assertion that is not clear is the case $j<0$. But in this case, by \ref{perext} we have isomorphisms $\underline{\ext}^j_R(M,M)\cong \ext^{-j}_R(M,M)$. Hence by \ref{qisolem} we have isomorphisms $H^{j}(\dq)\cong \ext^{-j}_R(M,M)$ for all $j<0$. 
 	\end{proof}
	 The extra structure given by periodicity allows us to have good control over the relationship between $\dq$ and $\R \underline{\enn}_R(M)$.
	 	\begin{defn}
	 	If $W$ is a dga and $w\in H(W)$ is a cohomology class, say that $w$ is \textbf{homotopy central} if it is central in the graded algebra $H(W)$. We abuse terminology by referring to cocycles in $W$ as homotopy central.
	 \end{defn}
Recall from \S\ref{hypper} the existence of an invertible homotopy central cohomology class $\Theta=[\theta]$ in ${\ext}^{2}_R(M,M)$ such that multiplication by $\Theta$ is an isomorphism.

	\begin{thm}\label{etaex}Let $\Xi$ be the comparison map.\hfill
		\begin{enumerate}
			\item[\emph{1.}] There is a degree $-2$ homotopy central class $\eta \in H^{-2}(\dq)$ such that $\Xi(\eta)=\Theta^{-1}$.
			\item[\emph{2.}] Multiplication by $\eta$ induces isomorphisms $H^j(\dq) \to H^{j-2}(\dq)$ for all $j\leq 0$.
			\item[\emph{3.}] The derived localisation of $\dq$ at $\eta$ is quasi-isomorphic to $\R \underline{\enn}_R(M)$.
			\item[\emph{4.}] The comparison map $\Xi: \dq \to \R \underline{\enn}_R(M)$ is the derived localisation map.
		\end{enumerate}
	\end{thm}
	\begin{proof}By \ref{qisolem}, $H^j(\Xi)$ is an isomorphism for $j\leq 0$. The first statement is now clear. The element $\eta$ is homotopy central in $\dq$ because $\Theta$ is homotopy central in $\R \underline{\enn}_R(Ae)$. Since $\Xi$ is a dga map, the following diagram commutes for all $j$:
		$$\begin{tikzcd} H^j(\dq) \ar[r,"\eta"]\ar[d,"\Xi"] & H^{j-2}(\dq)\ar[d,"\Xi"] \\ \underline{\ext}_R^{j}(M,M) \ar[r,"\Theta^{-1}"] & \underline{\ext}_R^{j-2}(M,M)\end{tikzcd} $$
		The vertical maps and the lower horizontal map are isomorphisms for $j\leq 0$, and hence the upper horizontal map must be an isomorphism, which is the second statement. Let $B$ be the derived localisation of $\dq$ at $\eta$. Because $\eta$ is homotopy central, the localisation is flat \cite[5.3]{bcl} and so we have $H(B) \cong H(\dq)[\eta^{-1}]$. In particular, for $j\leq 0$, we have $H^j(B)\cong H^j(\dq)$. The map $\Xi$ is clearly $\eta$-inverting, which gives us a factorisation of $\Xi$ through $\Xi':B \to \R \underline{\enn}_R(M)$. Again, the following diagram commutes for all $i,j$ :
		$$\begin{tikzcd} H^j(B) \ar[r,"\eta^i"]\ar[d,"\Xi'"] & H^{j-2i}(B)\ar[d,"\Xi'"] \\ \underline{\ext}_R^{j}(M,M) \ar[r,"\Theta^{-i}"] & \underline{\ext}_R^{j-2i}(M,M)\end{tikzcd} $$The horizontal maps are always isomorphisms. For a fixed $j$, if one takes a sufficiently large $i$, then the right-hand vertical map is an isomorphism. Hence, the left-hand vertical map must be an isomorphism too. But since $j$ was arbitrary, $\Xi'$ must be a quasi-isomorphism, proving the last two statements.
	\end{proof} 
\begin{rmk}\label{rigidrmk}
	If $M$ is rigid (see \ref{rigiddefn}) then we have $H(\dq)\cong A/AeA[\eta]$, but in general $\dq$ need not be formal.
\end{rmk}
	Left multiplication by $\eta$ is obviously a map $\dq \to \dq$ of right $\dq$-modules. Since $\eta$ is homotopy central, one might expect $\eta$ to be a bimodule map, and in fact this is the case:
	\begin{prop}\label{etahoch}
		The element $\eta$ lifts to an element of $ HH^{-2}(\dq)$, the $-2^\text{nd}$ Hochschild cohomology of $\dq$ with coefficients in itself.
	\end{prop}
	\begin{proof}
		Using \ref{periodicend} and \ref{etaex} gives us a dga $E$, a genuinely central element $\theta^{-1} \in E^{-2}$, and a dga map $\Xi:\dq \to E$ with $\Xi(\eta)=[\theta^{-1}]$. Since $\theta^{-1}$ is central it represents an element of $HH^{-2}(E)$. Because $\Xi$ is the derived localisation map, we have $HH^*(E)\cong HH^*(\dq,E)$ by \ref{hochprop}. Let $C$ be the mapping cone of $\Xi$. Then $C$ is an $\dq$-bimodule, concentrated in positive degrees. We get a long exact sequence in Hochschild cohomology $$\cdots \to HH^n(\dq) \xrightarrow{\Xi} HH^n(\dq,E) \to HH^n(\dq,C)\to \cdots.$$Because $C$ is concentrated in positive degrees, and $\dq$ is concentrated in nonnegative, the cohomology group $HH^n(\dq,C)$ must vanish for $n\leq0$. In particular we get isomorphisms $HH^n(\dq)\cong HH^n(\dq,E)$ for $n\leq0$. Putting this together we have an isomorphism $$HH^{-2}(\dq)\xrightarrow{\Xi} HH^{-2}(\dq,E)\xrightarrow{\Xi} HH^{-2}(E)$$ and it is clear that $\eta$ on the left hand side corresponds to $\theta^{-1}$ on the right.
	\end{proof}
	\begin{rmk}
		Because $\eta$ is a bimodule morphism, $\mathrm{cone}(\eta)$ is naturally an $\dq$-bimodule. Note that $\mathrm{cone}(\eta)$ is also quasi-isomorphic to the $2$-term dga $\tau_{>-2}(\dq)$. This is a quasi-isomorphism of $\dq$-bimodules, because if $Q$ is the standard bimodule resolution of $\dq$ obtained by totalising the bar complex, then the composition $Q \xrightarrow{\eta}\dq \to \tau_{>-2}(\dq)$ is zero for degree reasons.
	\end{rmk}
	\begin{rmk}
		The dga $\dq$ is quasi-isomorphic to the truncation $\tau_{\leq 0}E$, which is a dga over $k[\theta^{-1}]$. Let $H=HH^*_{k[\theta^{-1}]}(\tau_{\leq 0}E)$ be the Hochschild cohomology of the $k[\theta^{-1}]$-dga $\tau_{\leq 0}E$, which is itself a graded $k[\theta^{-1}]$-algebra. One can think of $H$ as a family of algebras over $\A^1$, with general fibre $H[\theta]\cong HH_{k[\theta,\theta^{-1}]}^*(E)$ and special fibre $HH^*(\mathrm{cone}(\eta))$.
	\end{rmk}
	
	\begin{prop}\label{uniqueeta}
		Suppose that $A/AeA$ is an Artinian local ring. Then $\eta$ is characterised up to multiplication by units in $H(\dq)$ as the only non-nilpotent element in $H^{-2}(\dq)$.
	\end{prop}
	\begin{proof}
		Let $y\in H^{-2}(\dq)$ be non-nilpotent. Since $\eta: H^0(\dq) \to H^{-2}(\dq)$ is an isomorphism, we must have $y = \eta x$ for some $x \in H^0(\dq)\cong A/AeA$. Since $\eta$ is homotopy central, we have $y^n=\eta^n x^n$ for all $n \in \N$. Since $y$ is non-nilpotent by assumption, $x$ must also be non-nilpotent. Since $A/AeA$ is Artinian local, $x$ must hence be a unit. Note that because $H(\dq)$ is concentrated in nonpositive degrees, the units of $H(\dq)$ are precisely the units of $A/AeA$.
	\end{proof}
	\begin{rmk}
		If $A/AeA$ is finite-dimensional over $k$, but not necessarily local, then all that can be said is that $x$ is not an element of the Jacobson radical $J(A/AeA)$. 
	\end{rmk}
	\begin{cor}\label{uniqueetaqi}
		Let $N$ be another MCM $R$-module and put $B:=\enn_R(R\oplus N)$. Let $e\in B$ denote the idempotent $\id_R$. Suppose that $\dq$ is quasi-isomorphic to $\dqb$. Suppose that $A/AeA\cong B/BeB$ is Artinian local. Then $\R \underline{\enn}_{R}(M)$ and $\R \underline{\enn}_{R}(N)$ are quasi-isomorphic.
	\end{cor}
	\begin{proof}
		The idea is that the periodicity elements must agree up to units, and this forces the derived localisations to be quasi-isomorphic. Let $\eta_A\in H^{-2}(\dq)$ and $\eta_B\in H^{-2}(\dqb)$ denote the periodicity elements for $\dq$ and $\dqb$ respectively. By assumption, we have a quasi-isomorphism $\dq \to \dqb$; let $\xi \in H^{-2}(\dqb)$ be the image of $\eta_A$ under this quasi-isomorphism. By \ref{uniqueeta}, there is a unit $u\in H^0(\dqb)$ such that $\xi=u.\eta_B$. Because derived localisation is invariant under quasi-isomorphism, we have $\mathbb{L}_{\eta_A}(\dq) \simeq \mathbb{L}_{\xi}(\dqb)$. Observe that if $W$ is a dga, $w \in HW$ any cohomology class, and $v \in HW$ is a unit, then the derived localisations $\mathbb{L}_wW$ and $\mathbb{L}_{vw}W$ are naturally quasi-isomorphic. So we have a chain of quasi-isomorphisms $$\mathbb{L}_{\eta_A}(\dq) \simeq \mathbb{L}_{\xi}(\dqb)=\mathbb{L}_{u\eta_B}(\dqb) \simeq \mathbb{L}_{\eta_B}(\dqb).$$Now the result follows by applying \ref{etaex}(3).
	\end{proof}
	Since $\R \underline{\enn}_{R}(M)$ is quasi-isomorphic to a dga over $k[\theta,\theta^{-1}]$, and $\R \underline{\enn}_{R}(M)$ is morally obtained from $\dq$ by adjoining $\theta^{-1}$, the following conjecture is a natural one to make:
	\begin{conj}\label{perconj}
		If $A/AeA$ is Artinian local then the quasi-isomorphism type of $\dq$ determines the quasi-isomorphism type	of $\R \underline{\enn}_{R}(M)$ as a dga over $k[\theta,\theta^{-1}]$.
	\end{conj}
	
		\begin{rmk}
		Note that $\eta$ is a central element of the cohomology algebra $H(\dq)$, and need not lift to a genuinely central cocycle in $\dq$.
	\end{rmk}

		\begin{rmk}\label{dsgrmk}
The description of \ref{qisolem} shows that, in this situation, one can compute $\dq$ directly from knowledge of the dg singularity category. This also provides a way to produce an explicit model of $\dq$ where $\eta$ is represented by a genuinely central cocycle: first, stitch together the syzygy exact sequences for $M$ into a 2-periodic resolution $\tilde M \to M$. Let $\theta: \tilde M \to \tilde M$ be the degree 2 map whose components are the identity that witnesses this periodicity. Let $E=\enn_R(\tilde M)$, which is a dga. It is easy to see that $\theta$ is a central cocycle in $E$. Since $\R\underline\enn_R(M)$ is quasi-isomorphic to the dga $E[\theta^{-1}]$, and $\eta$ is identified with $\theta^{-1}$ across this quasi-isomorphism, it follows that $\dq$ is quasi-isomorphic to the dga $\tau_{\leq 0}\left(\enn_R(\tilde M)[\theta^{-1}]\right)$, which is naturally a dga over $k[\eta]=k[\theta^{-1}]$.

	\end{rmk}

	\section{Torsion modules}\label{torsmods}
	We keep the setup as in the last section, where $R$ is a complete local hypersurface singularity, $M$ a MCM $R$-module, and $A=\enn_R(R\oplus M)$ the associated noncommutative partial resolution. For brevity, write $Q$ for the derived quotient $\dq$. By \ref{etaex} there exists a special periodicity element $\eta\in H^{-2}Q$ such that the derived localisation of $Q$ at $\eta$ is the derived stable endomorphism algebra $\R \underline{\enn}_R(M)$. Recall from \ref{colocdga} the construction of the \textbf{colocalisation} $\mathbb{L}^\eta(Q)$ of $Q$, and the fact that an $\eta$-torsion $Q$-module is precisely a module over $\mathbb{L}^\eta(Q)$. 
	\begin{defn}
		Let $\per^bQ$ denote the full triangulated subcategory of $\per Q$ on those modules with bounded cohomology.
	\end{defn}
	It is easy to see that $\per^bQ$ is a thick subcategory of the unbounded derived category $D(Q)$.
	\begin{prop}\label{perbisperl}
		The subcategory $\per^bQ$ is exactly $\per\mathbb{L}^\eta(Q)$.
	\end{prop}
	\begin{proof}We show $\per\mathbb{L}^\eta(Q)\subseteq\per^bQ \subseteq \per\mathbb{L}^\eta(Q)$. Since $\per\mathbb{L}^\eta(Q)=\thick_{D(Q)}(\mathbb{L}^\eta(Q))$, and $ \per^bQ$ is a thick subcategory, to show that $\per\mathbb{L}^\eta(Q)\subseteq\per^bQ$ it is enough to check that $\mathbb{L}^\eta(Q)$ is an element of $\per^bQ$. Put $C\coloneqq\mathrm{cone}(Q \xrightarrow{\eta} Q)$. By construction, the colocalisation $\mathbb{L}^\eta(Q)$ is exactly $\R\enn_Q(C)$. Now, $C$ is clearly a perfect $Q$-module. It is bounded because $\eta$ is an isomorphism on cohomology in sufficiently low degree. As a $Q$-module, we have \begin{align*}
		\R\enn_Q(C) &\simeq \R\hom_Q(\mathrm{cone}(\eta),C)
		\\ &\simeq \mathrm{cocone}\left[\R\hom_Q(Q,C) \xrightarrow{\eta^*} \R\hom_Q(Q,C)\right]
		\\ &\simeq \mathrm{cocone}\left[C \xrightarrow{\eta^*} C\right]
		\end{align*}which is clearly perfect and bounded. Hence $\mathbb{L}^\eta(Q) \in \per^bQ$. To show that $\per^bQ\subseteq\per\mathbb{L}^\eta(Q)$, we first show that a bounded module is torsion. Let $X$ be any bounded $Q$-module. Then there exists an $i$ such that $X\eta^i\simeq 0$. Choose a $Q$-cofibrant model $L$ for $\mathbb{L}_\eta(Q)$, so that $\mathbb{L}_\eta(X) \simeq X\otimes_Q L$. Then we have $X\otimes_Q L \cong X\otimes_Q\eta^i\eta^{-i} L\cong X\eta^i\otimes_Q\eta^{-i} L\simeq 0$. Now it is enough to show that a perfect $Q$-module which happens to be torsion is in fact a perfect $\mathbb{L}^\eta(Q)$-module. But this is clear: a perfect $Q$-module is exactly a compact $Q$-module, and hence remains compact in the full subcategory of torsion modules.
	\end{proof}
	\begin{defn}
		Say that a dga $W$ is \textbf{of finite type} if each $H^jW$ is finitely generated over $H^0W$.
	\end{defn}
	In particular, a cohomologically locally finite dga $W$ is of finite type.
	\begin{prop}\label{fingen}
		If $Q$ is of finite type, then $H(Q)$ is a finitely generated algebra over $A/AeA$.
	\end{prop}
	\begin{proof}
		By \ref{etaex}(2), we know that $\eta H^i(Q)\cong H^{i-2}(Q)$ for all $i\leq 0$. In particular, if $j<-2$, then every element in $H^j(Q)$ is a multiple of $\eta$. Hence, $H(Q)$ is generated as an algebra in degrees $0$ through $-2$. Since it is of finite type, we may choose it to have finitely many generators (over $A/AeA$) in each degree.
	\end{proof}
	In particular, if $Q$ is of finite type and $A/AeA$ is a finitely generated algebra, then so is $H(Q)$. In general, $H(Q)$ is generated in degrees $0$ through $-2$, and the only generator in degree $-2$ is $\eta$.
	\begin{thm}\label{fintype}
		Suppose that $Q$ is of finite type. Then $\per\mathbb{L}^\eta(Q)=\per_\mathrm{fg}(Q)$.
	\end{thm}
	\begin{proof}
		By \ref{perbisperl}, we show that $\per_\mathrm{fg}(Q)=\per^bQ$. Note that $\per_\mathrm{fg} Q$ is always a subcategory of $\per^bQ$. Since $Q$ is of finite type, we see that for $X\in\per Q$, each $H^jX$ is also finitely generated over $H^0Q$. So a bounded perfect $Q$-module has total cohomology finitely generated over $H^0Q$.
	\end{proof}
	One might expect that the triangulated categories $\per(\R\underline{\enn}_R(M))$ and $\thick_{D_\mathrm{sg}(R)}(M)$ are equivalent, and indeed this is the case under a finiteness assumption:
	\begin{prop}\label{perleta}
		Suppose that $Q$ is of finite type. Then the triangulated categories $\per\mathbb{L}_\eta(Q)$ and $\thick_{D_\mathrm{sg}(R)}(M)$ are triangle equivalent, via the map that sends $\mathbb{L}_\eta(Q)$ to $M$.
	\end{prop}
	\begin{proof}
		By \ref{kymapbetter} and \ref{fintype}, the singularity functor induces a triangle equivalence  $${\bar{\Sigma}}:\frac{\per(Q)}{\per\mathbb{L}^\eta(Q)} \to \thick_{D_\mathrm{sg}(R)}(M)$$which sends $Q$ to $M$. In particular, $\frac{\per(Q)}{\per\mathbb{L}^\eta(Q)}$ is idempotent complete.  By \ref{ntty}, this quotient is precisely $\per\mathbb{L}_\eta(Q)$, and the quotient map sends $Q$ to $\mathbb{L}_\eta(Q)$.
	\end{proof}
	We can prove a dg version of \ref{perleta}:
	\begin{prop}\label{perletadg}
		Suppose that $Q$ is of finite type. Then the dg categories $\per_{\mathrm{dg}}\mathbb{L}_\eta(Q)$ and $\thick_{D^\mathrm{dg}_\mathrm{sg}(R)}(M)$ are quasi-equivalent, via the map that sends $\mathbb{L}_\eta(Q)$ to $M$.
	\end{prop}
	\begin{proof}
		We follow the proof of \ref{perleta} and upgrade things to the dg setting. By \ref{ntty} we have a quasi-equivalence of dg categories $$
		\per_\mathrm{dg}\mathbb{L}_\eta Q \xrightarrow{\simeq} \frac{\per_\mathrm{dg}(Q)}{\per_\mathrm{dg}\mathbb{L}^\eta(Q)} $$which sends $\mathbb{L}_\eta(Q)$ to $Q$. By \ref{Fdg} and \ref{kymap} we have a quasi-essential surjection $$\Sigma_{\mathrm{dg}}:\per_{\mathrm{dg}}(Q) \onto \thick_{D^\mathrm{dg}_\mathrm{sg}(R)}(M)$$ which sends $Q$ to $M$, and by $\ref{dgperfgrmk}$ this descends to a quasi-equivalence $$\bar{\Sigma}_\mathrm{dg}: \frac{\per_\mathrm{dg}(Q)}{\per^\mathrm{dg}_\mathrm{fg}(Q)}\xrightarrow{\simeq} \thick_{D^\mathrm{dg}_\mathrm{sg}(R)}(M)$$which enhances $\bar\Sigma$. It is easy to see that the proof of \ref{fintype} gives a quasi-equivalence of dg categories $$\per^\mathrm{dg}_\mathrm{fg}(Q)\simeq \per_\mathrm{dg}\mathbb{L}^\eta(Q)$$ compatible with the inclusion into $\per Q$, and it now follows that the composition
		$$\per_\mathrm{dg}\mathbb{L}_\eta Q \xrightarrow{\simeq} \frac{\per_\mathrm{dg}(Q)}{\per_\mathrm{dg}\mathbb{L}^\eta(Q)}  \xrightarrow{\simeq} \frac{\per_\mathrm{dg}(Q)}{\per^\mathrm{dg}_\mathrm{fg}(Q)}\xrightarrow{\simeq} \thick_{D^\mathrm{dg}_\mathrm{sg}(R)}(M)$$is a quasi-equivalence, as required. 
	\end{proof}
	\begin{thm}\label{recovwk}
		Suppose that $\dq$ is of finite type. Then the pair $(\dq,\eta)$ determines the dg category $\thick_{D^\mathrm{dg}_\mathrm{sg}(R)}(M)$ up to quasi-equivalence. If $A/AeA$ is Artinian local, then $\dq$ alone determines $\thick_{D^\mathrm{dg}_\mathrm{sg}(R)}(M)$. If $A$ has finite global dimension, then $\thick_{D^\mathrm{dg}_\mathrm{sg}(R)}(M)\cong D^\mathrm{dg}_\mathrm{sg}(R)$.
	\end{thm}
	\begin{proof}
		The first statement is true since $\thick_{D^\mathrm{dg}_\mathrm{sg}(R)}(M)\cong \per_{\mathrm{dg}}\mathbb{L}_\eta(\dq)$ by \ref{perletadg}. The second follows from \ref{uniqueetaqi}. The third follows from (the dg version of) \ref{dsgsmooth}. \qedhere
	\end{proof}
	\section{A recovery theorem}
 If one attaches a partial resolution $A$ to a ring $R$, then since the quasi-isomorphism type of $\dq$ recovers part of the dg singularity category $D^\mathrm{dg}_\mathrm{sg}(R)$, it can be used to determine $R$:
	\begin{thm}\label{recov}Let $S\coloneqq k\llbracket x_1,\ldots, x_n \rrbracket$ and take $\sigma_1, \sigma_2 \in \mathfrak{m}_S$, both nonzero, such that the associated hypersurface singularities $R_i\coloneqq S/\sigma_i$ are both isolated. Let $M_i$ be MCM $R_i$-modules and let $A_i\coloneqq \enn_{R_i}(R_i\oplus M_i)$ be the associated partial resolution of $R_i$. Assume that $A_i$ has finite global dimension. Put $e_i=\id_{R_i}$, and assume that $A_1/^\mathbb{L}A_1e_1A_1$ is cohomologically locally finite and that $A_1/A_1e_1A_1$ is a local algebra. Then if $A_1/^\mathbb{L}A_1e_1A_1$ and $A_2/^\mathbb{L}A_2e_2A_2$ are quasi-isomorphic, then $R_1$ and $R_2$ are isomorphic.	
	\end{thm}
	\begin{proof}
		By \ref{recovwk}, the quasi-isomorphism class of $A_1/^\mathbb{L}A_1e_1A_1$ recovers (up to quasi-equivalence) the dg category $\thick_{D^\mathrm{dg}_\mathrm{sg}(R_1)}(M_1)$. By \ref{ontolem} this is quasi-equivalent to ${D^\mathrm{dg}_\mathrm{sg}(R_1)}$. Hence ${D^\mathrm{dg}_\mathrm{sg}(R_1)}$ and ${D^\mathrm{dg}_\mathrm{sg}(R_2)}$ are quasi-equivalent. Now apply Hua--Keller's result \ref{recoverythm}.
	\end{proof}
	\begin{rmk}\label{perconjrmk}
		If Conjecture \ref{perconj} is true then $\dq$ recovers the $\Z/2$-graded dg category of matrix factorisations, and one can prove \ref{recov} for quasi-homogeneous singularities by appealing directly to Dyckerhoff's results \cite{dyck} and the formal Mather--Yau theorem \cite{gpmather}.
	\end{rmk}
	\begin{rmk}\label{genrmk}We list a couple of variations on the above, which follow from \ref{recovwk}. One can weaken\footnote{In view of \ref{mgenconj}, it is unclear to the author if this is actually a weaker condition.} the assumption that $A$ has finite global dimension to the assumption that $M$ generates the singularity category; the key property that one needs is essential surjectivity of the singularity functor, which is equivalent to $M$ generating. One can omit the local condition on $A/AeA$, and weaken cohomological local finiteness of $\dq$ to being of finite type, at the cost of replacing $\dq$ by the pair $(\dq,\eta)$.
	\end{rmk}
	\begin{rmk}
		In the above situation, one has isomorphisms of algebras $$T_\sigma\cong HH^0(D^\mathrm{dg}_\mathrm{sg}(R))\cong HH^0(\per_{\mathrm{dg}}(\mathbb{L}_\eta(\dq)))\cong HH^0(\mathbb{L}_\eta(\dq)).$$As a vector space, one has $HH^0(\mathbb{L}_\eta(\dq))\cong HH^0(\dq)$ via the proof of \ref{etahoch}. An application of \ref{hochprop} gives an isomorphism $HH^0(\dq)\cong HH^0(A,\dq)$. In particular one can calculate the Tjurina number of the singularity as $\tau_\sigma=\dim_k HH^0(A,\dq)$.
	\end{rmk}

	\part{Contraction algebras}

	\chapter{The derived contraction algebra}\label{dercondefns}
	In this chapter we finally define the derived contraction algebra associated to the contraction of a rational curve to a point. Our motivation is to mimic the constructions of Donovan and Wemyss from \cite{DWncdf,contsdefs,enhancements}. We give a deformation-theoretic description of the derived contraction algebra (\ref{dcaprorep}), which is a derived analogue of \cite[3.9]{DWncdf}. We show that the derived analogue of the Donovan--Wemyss conjecture is true (\ref{ddwconj}). We use the deformation-theoretic interpretation of the derived contraction algebra to globalise some of our results.
	
	\section{The construction}
	The global setup will be as follows:
	\begin{setup}[Global]\label{globalsetup}
		Let $\pi:X \to X_\con$ be a projective birational morphism between two noetherian normal integral schemes over an algebraically closed field $k$ of characteristic zero. Assume that $\pi$ is crepant, that $\R\kern 2pt \pi_*\mathcal O_X = \mathcal O _{X_\con}$, and that $\pi$ is an isomorphism away from a single closed point $p$ in the base, where $C\coloneqq \pi^{-1}(p)$ is an irreducible rational (possibly non-reduced) curve. Assume in addition that $X_\con$ is a Gorenstein scheme such that $\widehat{(X_\con)}_p$ is an isolated hypersurface singularity.
	\end{setup}
	\begin{rmk}\label{dimrmk}
		The assumptions above actually imply that $X$ is at most 3-dimensional; see \ref{3dref} below. If $X$ is smooth, then $\pi$ is a crepant resolution of an isolated singularity. One does not need the assumption that $X_\con$ is Gorenstein -- or that $p$ is an isolated hypersurface singularity -- to define the derived contraction algebra, but it is hard to prove much about it without them. One ought to be able make the more general assumption that $C$ is a tree of rational curves, but to do much with this definition, one needs first to prove pointed versions of our earlier results. More generally, one should be able to drop the condition that $\pi^{-1}(p)$ is a curve, at the cost of some more assumptions about tilting bundles, as in \cite[2.5]{enhancements}.
	\end{rmk}
	
	Take an affine neighbourhood $\spec R \into X_\con$ of $p$, and replace $\pi$ by the pullback of $\pi$ along the inclusion. This gives the Zariski local setup:
	\begin{setup}[Zariski Local]\label{localsetup}
		Let $\pi:X \to \spec R$ be a projective birational morphism between two noetherian normal integral schemes over an algebraically closed field $k$ of characteristic zero. Moreover, $\pi$ is crepant, $\R\kern 2pt \pi_*\mathcal O_X = R$, and $\pi$ is an isomorphism away from a single closed point $p$ in the base, where $C\coloneqq \pi^{-1}(p)$ is an irreducible rational (possibly non-reduced) curve. Furthermore, $R$ is Gorenstein and $\hat R _p$ is an isolated hypersurface singularity.
	\end{setup}
	\begin{rmk}\label{hmmpsetup}
		Note that this generalises the Crepant Setup 2.9 of \cite{hmmp}.
	\end{rmk}
	Now, in the local setup, by results of Van den Bergh \cite[3.2.8]{vdb} there exists a finite rank tilting bundle $\mathcal V = \mathcal O_X \oplus \mathcal N$ on $X$. Put $\Lambda\coloneqq \enn_X(\mathcal V)$. By our assumptions, $\Lambda$ can be computed on the base: more precisely, \cite[2.5(2)]{enhancements} tells us that $\pi_*:\Lambda \to \enn_R(\pi_*\mathcal V)$ is an isomorphism. Writing $\pi_*\mathcal{V}\cong R\oplus N$ with $N\coloneqq \pi_*\mathcal N$, we can thus write $\Lambda\cong  \enn_R(R \oplus N)$. I claim that $N$ is a MCM $R$-module: to see this, first note that the $R$-module $\Lambda$ is MCM by \cite[\S4.2]{iyamawemyssfactorial}. Then, since $N$ is a summand of $\Lambda$, it must be MCM too. 
	
	\p We would like to define the contraction algebra to be the derived quotient of $\Lambda$ by the idempotent $e=\id_R$. In order for this to behave well, we would like the finite-dimensional algebra $\Lambda_\con\coloneqq \Lambda/\Lambda e \Lambda$ to be local -- unfortunately, this need not happen, for the same reasons as \cite[\S2.4]{DWncdf}. In order to ensure locality, we need to pass to a complete local base, and then through a Morita equivalence. Letting $\hat R$ be the completion of $R$ along the maximal ideal $p$, and letting $\hat X$ be the formal fibre, we obtain the following setup:
	\begin{setup}[Complete local]
		$\pi:\hat X \to \spec \hat R$ is a projective birational morphism between two noetherian normal integral schemes over an algebraically closed field $k$ of characteristic zero, and $\hat R$ is a complete local hypersurface singularity with maximal ideal $p$. Moreover, $\pi$ is crepant, an isomorphism away from $p$, $\R\kern 2pt \pi_*\mathcal O_X = R$, and $C\coloneqq \pi^{-1}(p)$ is an irreducible rational (possibly non-reduced) curve.
	\end{setup}
	The arguments of \cite[\S2.4]{DWncdf} adapt to ensure that $\hat{ \mathcal V} \cong \mathcal O_{\hat X} \oplus \hat{\mathcal N}$ is a tilting bundle on $\hat X$, and that $\hat \Lambda \cong \enn_{\hat R}(\hat R \oplus \hat N)\cong \enn_{\hat X}(\hat{\mathcal V})$. Again, we may apply \cite{iyamawemyssfactorial} to see that $\hat N$ is still MCM over $\hat R$. Now, by \cite[3.2.7 and 3.5.5]{vdb} we may put $\hat{R}\oplus\hat N = \hat{R}^{\oplus a}\oplus M^{\oplus b}$, for some (necessarily MCM) indecomposable $\hat R$-module $M$ and some positive integers $a,b$. It follows that $A\coloneqq \enn_{\hat R}(\hat R \oplus M)$ is the basic algebra Morita equivalent to $\hat \Lambda$.
	\begin{defn}Put $e\coloneqq \id_R$. The \textbf{contraction algebra} $A_\con$ associated to $\pi$ is the stable endomorphism algebra $A/AeA \cong \underline{\enn}_{\hat R}(\hat{R}\oplus M)\cong\underline{\enn}_{\hat R}(M)$. The \textbf{derived contraction algebra} $\dca$ is the derived quotient $\dq$.
	\end{defn}
	From the definition, it is obvious that $H^0(\dca)\cong A_\con$.

	\begin{rmk}\label{3dref}
		Setup \ref{globalsetup} implies that the dimension of $X$ is at most $3$. With notation as above, let $\hat X$ be the formal fibre; we then have a derived equivalence between $\hat X$ and $A\coloneqq\enn_{\hat R}(\hat R \oplus M)$. Crepancy implies that $M$ is a modifying $\hat R$-module in the sense of \ref{modifiyingdefinition} below. Because $M$ is MCM, periodicity in the singularity category along with the fact that stable Ext agrees with usual Ext in positive degrees tells us that there is an isomorphism $\ext^{2}_{\hat R}(M,M)\cong \underline{\enn}_{\hat R}(M)$. If $\hat {R}$ has dimension strictly greater than $3$, then \ref{modext} tells us that a MCM modifying module $M$ must have $\underline{\enn}_{\hat R}(M)\cong 0$, and hence $M$ must be projective. It now follows that $\dca$ is acyclic.
	\end{rmk}

	\section{First properties}
	Keep notation as in the Global Setup \ref{globalsetup}. Let $\hat R$ be the completion of the local ring of $X_\con$ at $p$. We begin by proving that some easy finiteness properties hold for $\dca$.
	\begin{lem}\label{fdlem}
		Let $W$ be a finitely generated $\hat R$-module which is supported at $p$. Then $W$ is finite-dimensional over $k$.
	\end{lem}
	\begin{proof}
		This is standard: some power $n$ of $p$ annihilates $W$, and hence $W$ is a finitely generated module over ${\hat R}/p^n$, which is finite-dimensional over ${\hat R}/p\cong k$.
	\end{proof}
	\begin{lem}
		Let $M$ be the $\hat R$-module defining $A_\con$ and $\dca$. If $q\neq p$ is a prime ideal of $\hat R$, then $M_q$ is projective.
	\end{lem}
	\begin{proof}
		$M$ is the pushforward of a vector bundle along a map that is an isomorphism away from $p$.
	\end{proof}
	\begin{prop}[cf.\ {\cite[2.13(1)]{DWncdf}}]\label{aconlocal}
		The algebra $A_\con$ is an Artinian local algebra.
	\end{prop}
	\begin{proof}If $q\neq p$ is a prime ideal of $\hat R$ then $(A_\con)_q\cong \underline{\enn}_{\hat R _q}(M_q)$, which vanishes because $M_q$ is projective. Hence $A_\con$ is supported at $p$ and hence Artinian. It is local because $M$ was indecomposable.
	\end{proof}
	\begin{prop}\label{fdcohom}
		The dga $\dca$ has finite-dimensional cohomology in each degree.
	\end{prop}
	\begin{proof}
		We already know that $H^0(\dca)$ is finite-dimensional. Let $j<0$. Then we have $H^j(\dca)\cong{\ext}_{\hat R}^{-j}(M,M)$ by \ref{dqcohom}. But ${\ext}_{\hat R}^{-j}(M,M)_q\cong {\ext}_{{\hat R}_q}^{-j}(M_q,M_q)$ which vanishes if $q\neq p$. So ${\ext}_{\hat R}^{-j}(M,M)$ is supported at $p$, and so \ref{fdlem} applies.
	\end{proof}

	Since $\hat R$ is a hypersurface singularity, by \ref{etaex} the dga $\dca$ admits a periodicity element $\eta\in H^{-2}(\dca)$, inducing isomorphisms $\eta:H^j(\dca) \to H^{j-2}(\dca)$. Moreover, $H(\dca)$ is a finitely generated algebra, finite-dimensional in each degree, generated in degrees $0$, $-1$, and $-2$. The only degree $-2$ generator is the periodicity element $\eta$, which is central and torsionfree. If $\dim R$ is even then every $H^j(\dca)$ has the same dimension for $j\leq 0$, by AR duality \ref{arduality}. Our first main theorem is a positive answer to a derived version of the Donovan--Wemyss conjecture \cite[1.4]{DWncdf}:
	\begin{thm}[derived Donovan--Wemyss]\label{ddwconj}
		Let $\pi:X \to X_\con$ and $\pi': X' \to X'_\con$ be two contractions satisfying the conditions of the Global Setup \ref{globalsetup}, contracting curves to points $p$ and $p'$ respectively. Assume in addition that $X$ and $X'$ are smooth and of the same dimension. Let $\dca$ and $\dca'$ be the derived contraction algebras of $\pi$ and $\pi'$ respectively. If $\dca$ and $\dca'$ are quasi-isomorphic, then the completions $\widehat{(X_\con)}_p$ and $\widehat{(X_\con')}_{p'}$ are isomorphic.
	\end{thm}
	\begin{proof}
		A simple application of \ref{recov}.
	\end{proof}

	\section{Deformations of curves}
	We show that the Koszul double dual of the derived contraction algebra prorepresents the functor of derived noncommutative deformations of the exceptional locus of an irreducible contraction, which generalises \cite[3.9]{DWncdf} to the derived setting. The hard work to prove this has already been done in the preceding chapters. Keep notation as in the Global Setup \ref{globalsetup}. Let $\hat R$ be the completion of the local ring of $X_\con$ at $p$. Let $C$ be the exceptional locus, with the reduced scheme structure.
	\begin{lem}
		Across the derived equivalence $D^b(\hat X) \to D^b(A)$, the sheaf $\mathcal{O}_C(-1)$ corresponds to the simple $S\coloneqq A_\con / \mathrm{rad}(A_\con)$.
	\end{lem}
	\begin{proof}
		Denote the image of $\mathcal{O}_C(-1)$ by $S'$. The sheaf $\mathcal{O}_C(-1)$ is simple by \cite[3.5.7]{vdb}. The proof of \cite[2.13(3)]{DWncdf} adapts to show that $S'$ is naturally a module over $A_\con$. Since it is simple, it must be the unique simple module $S$.
	\end{proof}
	With this in mind, we define:
	\begin{defn}
		A \textbf{noncommutative deformation of $C$} is a noncommutative deformation of the $A$-module $S$. A \textbf{derived noncommutative deformation of $C$} is a framed derived noncommutative deformation of the $A$-module $S$.
	\end{defn}
	\begin{thm}\label{dcaprorep}
		The quasi-isomorphism class of $\dca$ determines the $\sset$-valued functor of derived noncommutative deformations of $C$. More precisely, there is a weak equivalence
		$$\sdefm_{\hat X}(C)\simeq \R\mathrm{Map}_{\cat{pro}(\cat{dgArt}_k)}(B^\sharp(\dca^!) , -)$$ where moreover $\varprojlim B^\sharp(\dca^!) \simeq \dca$. The contraction algebra $A_\con$ represents the functor of underived noncommutative deformations of $C$.
	\end{thm}
	\begin{proof}
		The first sentence is \ref{dqqiclassdeterminesdefms}, the second is \ref{maindefmthm}, and the third is \ref{subdefmthm}.
	\end{proof}
\begin{rmk}
	As in the proof of \ref{subdefmthm}, if we work with set-valued functors only then we can use unframed deformations of $S$ rather than framed deformations.
\end{rmk}
\begin{rmk}
	Note that even in the underived setting, our representability results go further than those of Donovan and Wemyss \cite{DWncdf}, in the sense that they also work in dimension two. 
	\end{rmk}
The following theorem is computationally useful:
		\begin{thm}\label{dqiskd}
$\dca$ is the Koszul dual of the derived endomorphism algebra $\R\enn_A(S)$.
\end{thm}
\begin{proof}
	Follows from \ref{rendkd}.
\end{proof}

	\section{Local to global computations}\label{ltgc}
	We discuss further the deformation-theoretic description of the derived contraction algebra, and how we can use this description to compute the derived contraction algebra in differing neighbourhoods of $p$. The derived contraction algebra is defined using a formal neighbourhood of $p$. Can we compute it by using larger neighbourhoods? The first step is to examine how the dga $\R\enn(S)$ changes under completion.

	\begin{prop}\label{globalrend}
		Suppose that we are in the situation of the Zariski local setup \ref{localsetup}. Let $\Lambda=\enn_R(R\oplus N)$ be the associated noncommutative model for $X$. Let $\hat \Lambda$ be the completion at $p$, and let $A$ be the basic algebra Morita equivalent to $\hat \Lambda$. Let $S_A$ (resp. $S_{\Lambda}$) be the simple $A$-module (resp. $\Lambda$-module) corresponding across the derived equivalence to $\mathcal{O}_C(-1)$. Then there is a quasi-isomorphism
		$$\R\enn_A(S_A) \simeq \R\enn_\Lambda(S_\Lambda).$$
	\end{prop}
	\begin{proof}
		This is the proof of \cite[3.9(3)]{contsdefs}: the idea is that derived endomorphisms of finite length modules supported at $p$ behave well under Morita equivalence and completion.
	\end{proof}
	When one can compute the contraction algebra in a Zariski neighbourhood, one can also compute the derived contraction algebra there:
	\begin{prop}
		Suppose that we are in the local setup \ref{localsetup}. Suppose that $\Lambda_\con\coloneqq \Lambda/\Lambda e \Lambda$ is Artinian local. Then there is a quasi-isomorphism $\dca \simeq\Lambda/^\mathbb{L}\Lambda e \Lambda$.
	\end{prop}
	\begin{proof}We have $H^j(\Lambda/^\mathbb{L}\Lambda e \Lambda)\cong \ext^{-j}_R(N,N)$. Exactly as in \ref{fdcohom}, because this Ext group is supported at $p$ it must be finite-dimensional. By hypothesis, $H^0(\Lambda/^\mathbb{L}\Lambda e \Lambda)\cong \Lambda_\con$ is Artinian local. So we may apply \ref{rendkd} to conclude that $\R\enn_\Lambda(S_\Lambda)^!\simeq \Lambda/^\mathbb{L}\Lambda e \Lambda$. But by \ref{globalrend}, we know that $\R\enn_\Lambda(S_\Lambda)$ is quasi-isomorphic to $\R\enn_A(S_A)$. Since the Koszul dual preserves quasi-isomorphisms, $\Lambda/^\mathbb{L}\Lambda e \Lambda$ is quasi-isomorphic to $\R\enn_A(S_A)^! \simeq \dca$.
	\end{proof}
	Now we may globalise:
	\begin{prop}
		Suppose that we are in the situation of Setup \ref{globalsetup}. Let $\Lambda$ and $S_\Lambda$ be as above. Then there is a quasi-isomorphism $$\R\enn_\Lambda(S_\Lambda) \simeq \R\enn_{X}(\mathcal{O}_C(-1))$$
	\end{prop}
	\begin{proof}This is the proof of \cite[3.9(1)]{contsdefs}. Let $i:U \to X$ be the base change of the inclusion $\spec R \into X_\con$; it is an affine map because it is a base change of an affine map. It is also an open embedding, and hence induces an embedding $D(U)\into D(X)$. Hence we get quasi-isomorphisms $$\R\enn_\Lambda(S_\Lambda) \xrightarrow{\simeq} \R\enn_{U}(\mathcal{O}_C(-1)) \xrightarrow{\simeq} \R\enn_{X}(\mathcal{O}_C(-1))$$
		as required. \end{proof}
	We obtain the immediate corollary:
	\begin{thm}In the situation of Setup \ref{globalsetup}, the derived contraction algebra can be computed as the Koszul dual of the dga $\R\enn_{X}(\mathcal{O}_C(-1))$. In particular, it is intrinsic to the contraction $\pi$, and does not depend on any choice of affine neighbourhood of $p$ or tilting bundle.
	\end{thm}

	\begin{rmk}
		We note for computational purposes that $\dca$ is invariant under standard equivalences. Specifically, if $D(A)$ is standard equivalent to $D(A')$ in such a way that $S$ gets sent to $S'$, then one gets a quasi-isomorphism $\R\enn_A(S)\simeq \R\enn_{A'}(S')$, and hence quasi-isomorphisms $\dca\simeq \R\enn_A(S)^!\simeq \R\enn_{A'}(S')^!$. We remark that what one really needs is that $D(A)$ is quasi-equivalent to $D(A')$ as a dg category, but up to homotopy all such quasi-equivalences are standard by work of To\"{e}n \cite[Corollary 4.8]{toenmorita}.
	\end{rmk}

	\section{Ginzburg dgas}
	Hua and Keller \cite{huakeller} show that, for a flopping contraction in a smooth threefold $X$, one may compute its derived contraction algebra as a certain Ginzburg dga associated to an NCCR of $X$. We briefly recall the definition of Ginzburg dgas and review Hua and Keller's result. Ginzburg dgas, as well as Calabi--Yau algebras, were first defined in the seminal paper \cite{ginzburgdga}. We are following the sign conventions of \cite{vdbginzburg}. Note that we deal only with \textit{completed} Ginzburg dgas; see \cite{vdb} or \cite{huakeller} for references.
\begin{defn}
	Let $Q$ be a finite quiver. Let $kQ$ be its path algebra and let $\widehat{kQ}$ be the completion of $kQ$ along the arrows. Let $[\widehat{kQ},\widehat{kQ}]$, denote the completion of the span of all commutators in $\widehat{kQ}$. A \textbf{superpotential} on $Q$ is an element of the cocentre $\frac{\widehat{kQ}}{[\widehat{kQ},\widehat{kQ}]}$, which we can think of as a sort of formal Hochschild homology space.
	\end{defn}	
Note that a superpotential is a (possibly infinite!) linear combination of cycles in $Q$.
\begin{defn}
	Let $Q$ be a quiver and $W$ a superpotential on $Q$. Let $a$ be an arrow in $Q$. Define the \textbf{cyclic derivative} $\partial_aW$ by the formula
	$$\partial_aW\coloneqq \sum_{W=uav}vu$$where we sum over all arrows $u,v$ such that $W=uav$.
	\end{defn}
	
\begin{defn}
	Let $Q$ be a finite quiver and $W$ a superpotential on $Q$. Let $\bar Q$ be the graded quiver with the same vertex set as $Q$, and with three types of arrows: the arrows $a$ from $Q$, all in degree 0, a reversed degree $-1$ arrow $a^*$ for every arrow $a$ in $Q$, and a loop $z_i$ of degree $-2$ at every vertex $i$. The \textbf{Ginzburg dga} associated to the pair $(Q,W)$ is the dga $\Pi(Q,W)$ with underlying graded algebra the path algebra $\widehat{k\bar Q}$, and with differential given by \begin{align*}
	da&=0
	\\da^*&=-\partial_aW
	\\dz_i&=\sum_ae_i[a,a^*]e_i
	\end{align*}
	Note that we may compute $dz_i$ by summing over only the arrows incident to $i$.
	\end{defn}

\begin{defn}
Let $Q$ be a finite quiver and $W$ a superpotential on $Q$. The \textbf{Jacobi algebra} is the algebra $H^0(\Pi(Q,W))$. It can be computed as the path algebra of $\widehat{kQ}$ modulo the relations given by the $\partial_a W$.
\end{defn}
If one thinks of a Jacobi algebra as simply the path algebra of a quiver with relations, one can think of the superpotential as `integrating' the relations. We remark that if $W$ is a polynomial, and one knows in advance that the Jacobi algebra is finite-dimensional, then one can compute it without taking completions.

\p Recall that to compute the contraction algebra, Donovan and Wemyss take a noncommutative algebra $A=\enn_R(R\oplus M)$, where $R$ is a complete local ring, and quotient $A$ by the idempotent $e=\id_R$. Presenting $A$ as the completed path algebra of a quiver $Q$ with relations, one can hence compute $A_\con$ by simply throwing out the vertex corresponding to $R$ and modifying the relations accordingly. Suppose that the quiver admits a superpotential $W$ making $A$ into the associated Jacobi algebra -- this is the case in the smooth CY threefold setting \cite{vdbginzburg, huakeller}. Then one can compute $A_\con$ by taking the Jacobi algebra of the one-vertex quiver $Q'$ obtained by deleting the vertex at $R$ from $Q$, equipped with modified superpotential $W'$ obtained by removing all the arrows that are not loops at $M$. See \ref{lauferrmk} for an example of such a computation. Since we have $H^0(\dca)\cong H^0(\Pi(Q',W'))$, one might wonder if one can compute $\dca$ as a Ginzburg dga, and in this setting Hua and Keller proved that this is indeed the case for threefolds:
\begin{thm}
Let $X \to \spec R$ be a complete local threefold flopping contraction with $X$ smooth. Let $(Q',W')$ be the quiver with superpotential defined above that computes $A_\con$. Then $\dca$ is quasi-isomorphic to the Ginzburg dga $\Pi(Q',W')$.
	\end{thm}
\begin{proof}
	By \cite[4.17]{huakeller}, there is a quasi-isomorphism $$\Pi(Q',W')\simeq \tau_{\leq 0}\R\underline\enn_R(M)$$between the Ginzburg dga and the truncation to nonpositive degrees of the dg endomorphism algebra of the module $M$ considered as an object of the dg singularity category of $R$. But this truncation is quasi-isomorphic to $\dca$ by \ref{qisolem}.
	\end{proof}

\begin{rmk}
	Let $A$ be the Jacobi algebra of a quiver with superpotential $W$ and let $S$ be the direct sum of the simple vertex modules. Using deformation theory, Segal \cite[\S2]{segaldefpt} proves that under some finiteness conditions one can recover $A$ from the Ext-algebra $\ext_A(S,S)$ along with the higher multiplications $m_r$ needed to make it quasi-isomorphic to $\R\enn_A(S)$. More precisely, he identifies $A$ as $H^0(\R\enn_A(S)^!)$, which is analogous to our isomorphism $A/AeA\cong H^0(\dq)$. Under the presence of an additional Calabi--Yau condition, Segal also proves \cite[3.3]{segaldefpt} that one can recover $W$ from the $m_r$: loosely, one uses the CY pairing to obtain a superpotential on the completion of the tensor algebra $T(\ext^1_A(S,S)^*)$, and this presents $A$.
	\end{rmk}

	\chapter{Computations}\label{threecomps}
In this chapter we do some computations. We will compute the derived contraction algebra associated to the $n$-Pagoda flop, which has base $xy-(u+v^n)(u-v^n)$. As a warmup, we will do the case $n=1$, which is the classical Atiyah flop. We will also sketch the computation of the derived contraction algebra associated to the Laufer flop. Then we will move onto the surface setting, and compute the derived contraction algebra associated to a certain family of partial crepant resolutions of Kleinian $A_n$ singularities obtained by taking a slice of a threefold flopping contraction. Along the way we will prove some useful general facts about slicing a threefold flopping contraction by a generic hyperplane, especially with regards to tilting bundles.

\p In order to actually carry out the computations, we will use the interpretation of the derived contraction algebra as a Koszul dual. The basic idea is to identify our noncommutative model $A$, compute the derived endomorphism ring $E\coloneqq \R\enn_A(S)$, and use \ref{dqiskd} to compute $\dca\simeq E^!$. Typically, we aim to present $\dca$ as a minimal $A_\infty$-algebra. When the calculations get harder to do, we will content ourselves with the cofibrant dga model obtained via Koszul duality. We will also show that, in the threefold setting, one can compute the derived contraction algebra as a certain Ginzburg dga of a quiver with superpotential that one obtains by deleting vertices of a quiver representation for a noncommutative model of the contraction, in accordance with Hua and Keller's work \cite{huakeller}.

	\section{General setup for threefolds}
	The setup will be the following variation on that of Wemyss \cite[2.9]{hmmp}:
	\begin{setup}\label{flopsetup}
		$\pi:X \to \spec R$ is a crepant projective birational morphism between two noetherian normal integral threefolds over an algebraically closed field $k$ of characteristic zero. Moreover, $\pi$ is an isomorphism away from a single closed point $p$ in the base, where $C\coloneqq \pi^{-1}(p)$ is an irreducible rational (possibly non-reduced) floppable curve. We also assume that $R$ is Gorenstein and that $X$ has at worst Gorenstein terminal singularities.
	\end{setup}
	\begin{prop}
		Suppose that we are in the situation of Setup \ref{flopsetup}. Then we are in the situation of the Zariski local setup \ref{localsetup}. In particular we may define the derived contraction algebra.
	\end{prop}
	\begin{proof}
		We need to check that the completion $\hat{R}_p$ is an isolated hypersurface singularity and that $\R\pi_*\mathcal{O}_{X}\simeq R$. The first claim follows from the classification of terminal singularities; namely by \cite[5.38]{kollarmori}, $\hat{R}_p$ is an isolated cDV singularity and in particular a hypersurface singularity. The second claim follows from Kawamata vanishing \cite{kawamatavanish}.
	\end{proof}
	\begin{prop}\label{mmoddca}
		Suppose that we are in the Threefold Setup \ref{flopsetup}. If $X$ is $\Q$-factorial (i.e.\ is a minimal model of $R$) then the cohomology of $\dca$ is simply $A_\con[\eta]$.
	\end{prop}
	\begin{proof}
		By \cite[4.10]{hmmp}, the associated MCM $\hat R$-module $M$ used to define $\dca$ is rigid (see \ref{rigiddefn}), and the claim now follows from \ref{rigidrmk}. In fact, the same holds more generally if $X$ is a `partial minimal model' \cite[4.13]{hmmp}.
	\end{proof}
	In general, our strategy will be to compute $\dca$ as a Koszul double dual. We will start with Setup \ref{flopsetup} and write down a noncommutative model $\Lambda=\enn_R(R\oplus N)$ for $X$, presenting $\Lambda$ as the path algebra of a two-vertex quiver with relations. This noncommutative model $\Lambda$ we write down for $X$ will be a \textbf{noncommutative crepant resolution} (\textbf{NCCR}) of $R$ \cite{vdbnccr}. Because all crepant resolutions of threefolds with terminal singularities are derived equivalent \cite[6.6.3]{vdbnccr}, exhibiting a single NCCR will suffice for our calculations. 
	
	\p Letting $S$ be the simple of $\Lambda$ at the vertex corresponding to $N$, we will compute the derived endomorphism algebra $\R\enn_\Lambda(S)$. It is easiest to do this as an $A_\infty$-algebra, by placing higher multiplications on $\ext_\Lambda(S,S)$ and appealing to Kadeishvili's theorem (\ref{kadeish}). Next, we will compute $\dca$ as the $A_\infty$ Koszul dual of $\R\enn_\Lambda(S)$. Again, we will do this via Kadeishvili's theorem, along with the fact that the cohomology of {\dca} can be calculated explicitly in advance. In order to actually do the $A_\infty$ computations, we will either use Merkulov's construction, or put an Adams grading on the resolution of $S$ and note that many higher multiplications become ruled out, which itself appeals to a graded version of Merkulov's construction (see \ref{adamsrmk}). We will often use Massey products to detect non-formality (see \ref{masseys} and \ref{masseylemma}). 
	
	\p Suppose that $\Gamma$ is a quiver (possibly with relations) and $A=k\Gamma$ its path algebra. We denote multiplication in $A$ left-to-right: that is, $ab$ means `follow edge $a$, then edge $b$'. If $i$ is a vertex of $\Gamma$, with associated idempotent $e_i \in A$, then the \textbf{projective at $i$} is the right $A$-module $P_i\coloneqq  e_i A$ consisting of all those paths starting at $i$. We write $P_i^r$ to mean the $r$-fold direct sum $P_i^{\oplus r}$, and $P_iP_j$ to mean the direct sum $P_i\oplus P_j$. Because we will write maps as matrices, the order here is important.

	\section{The Atiyah flop} \label{atiyahflop}
	The Atiyah flop is the simplest example of a flopping contraction, and also the first known \cite{atiyahflop}. The base is the cone ${R=\frac{k[u,v,x,y]}{(uv-xy)}}$. One MCM module is $(u,x)$, and it is well known that this gives an NCCR $\Lambda$ with a presentation as the path algebra of the following quiver with relations:
	$$   
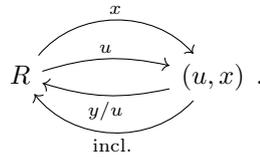
\begin{tikzcd}
	{{R}} \arrow[rr,bend left=15,"u"] \arrow[rr,bend left=50,"x"]  && (u,x) \arrow[ll,bend left=15,"y/u"] \arrow[ll,bend left=50,"\text{incl.}"]
	\end{tikzcd}.$$
	Relabelling, we can write this quiver as\\
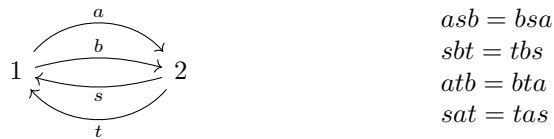
	\begin{wrapfigure}[4]{L}{0.5\textwidth}
		\centering
		\vspace{-15pt}$
		\begin{tikzcd}
		1 \arrow[rr,bend left=15,"b"] \arrow[rr,bend left=50,"a"]  && 2 \arrow[ll,bend left=15,"s"] \arrow[ll,bend left=50,"t"]
		\end{tikzcd}$
	\end{wrapfigure} 
	
	$asb=bsa$
	
	$sbt=tbs$
	
	$atb=bta$
	
	$sat=tas$\newline\newline
	One can check that a resolution for $S$, the simple at 2, is given by 
	$$
	P_2 \xrightarrow{\left(\begin{smallmatrix} b \\ -a \end{smallmatrix}\right)} P_1^2 \xrightarrow{\left(\begin{smallmatrix} bt & at \\ -bs & -as \end{smallmatrix}\right)} P_1^2 \xrightarrow{(s \; t)} P_2$$and it easily follows that the Ext-algebra of $S$ is $k[x]/x^2$, with $x$ placed in degree 3. 
	\begin{lem}The dga $\R\enn_\Lambda(S)$ is formal.
		\end{lem}
	\begin{proof}
		This is not really specific to the situation at hand; indeed we show that any dga whose cohomology algebra is $k[x]/x^2$ with $\vert x \vert=3$ is formal. Use Kadeishvili's theorem to put a strictly unital $A_\infty$-structure on $k[x]/x^2$ making it quasi-isomorphic to $\R\enn_\Lambda(S)$. Suppose that $m_r\neq 0$ for some $r>2$ is a nontrivial higher $A_\infty$-operation for this $A_\infty$-structure. By strict unitality, we must have $m_r(x,\ldots,x)\neq 0$. But the degree of $m_r(x,\ldots,x)$ is $2r+2$, which is even and positive. But this is a contradiction, because $k[x]/x^2$ has no elements of even positive degree. Hence $m_r=0$.
		\end{proof}

So the derived contraction algebra $\dca$ is the Koszul dual of $k[x]/x^2$, which is the same as the tensor algebra on a single element $\eta=x^*$ of degree -2. Hence, $\dca \simeq k[\eta]$. Note that this computation is valid all characteristics.

	\section{Pagoda flops}
	The Pagoda flops, first defined by Reid \cite{reidpagoda}, are natural generalisations of the Atiyah flop. They are all $cA_1$ singularities, and fit into a family parametrised by a natural number $n\geq 1$. The $n=1$ case is the Atiyah flop. The base of the Pagoda flop is $R=\frac{k[u,v,x,y]}{(uv-(x+y^n)(x-y^n))}$. One MCM module for $R$ is $N\coloneqq (u,x+y^n)$. One can write down a quiver presentation for the resulting NCCR $\Lambda=\enn_R(R\oplus N)$: it looks like $$
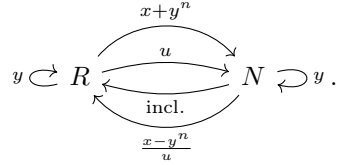
\begin{tikzcd}
	R \arrow[rr,bend left=15,"u"]\arrow[loop left, "y"] \arrow[rr,bend left=50,"x+y^n"]  && N\arrow[loop right, "y"] \arrow[ll,bend left=15,"\text{incl.}"] \arrow[ll,bend left=50,"\frac{x-y^n}{u}"]\end{tikzcd}.$$ 
	Rewriting this abstractly, we obtain the quiver with relations
	
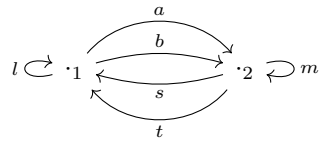
\begin{wrapfigure}[6]{L}{0.5\textwidth}
		\centering $
		\begin{tikzcd}
		\cdot_1 \arrow[rr,bend left=15,"b"]\arrow[loop left, "l"] \arrow[rr,bend left=50,"a"]  && \cdot_2 \arrow[ll,bend left=15,"s"] \arrow[ll,bend left=50,"t"]\arrow[loop right, "m"]
		\end{tikzcd}$
	\end{wrapfigure} 
	
	$la=am$
	
	$lb=bm$
	
	$sl=ms$
	
	$tl=mt$
	
	$as=bt+2l^n$
	
	$sa=tb+2m^n$
	\newline where the simple $S$ at 2 corresponds to the exceptional locus in the resolution. So to compute \dca, we first need to resolve $S$. Assuming $n>1$, then using a tedious basis argument one can write down a four-term resolution $$ \tilde S \coloneqq P_2 \xrightarrow{\left(\begin{smallmatrix}-a\\b\\m\end{smallmatrix}\right)} P_1^2P_2 \xrightarrow{\left(\begin{smallmatrix}s & t & 2m^{n-1}\\ -l & 0 & -a \\ 0 & -l & b\end{smallmatrix}\right)} P_2P_1^2 \xrightarrow{(m,s,t)} P_2$$of $S$. It is now easy to see that the $\ext$-algebra of $S$ is four-dimensional over $k$, with Hilbert series $1+t+t^2+t^3$. In fact, one can check that the $\ext$-algebra is generated by two elements $f_1$ and $f_2$, with $f_i$ placed in degree $i$. Concretely, $f_1$ is represented by $$\begin{tikzcd}[row sep= large, column sep=15ex] & P_2\ar[r]\ar[d,"{\left(\begin{smallmatrix}0\\0\\1\end{smallmatrix}\right)}"] & P_1^2P_2\ar[r]\ar[d,"{-\sthbthm{0}{0}{2m^{n-2}}{1}{0}{0}{0}{1}{0}}"] & P_2P_1^2\ar[r]\ar[d,"{\left(1,0,0\right)}"] & P_2 \\ P_2\ar[r] & P_1^2P_2\ar[r] & P_2P_1^2\ar[r] & P_2 &\end{tikzcd}$$
	while $f_2$ is represented by $$\begin{tikzcd}[row sep= large, column sep=large] & & P_2\ar[r]\ar[d,"{\left(\begin{smallmatrix}1\\0\\0\end{smallmatrix}\right)}"] & P_1^2P_2\ar[r]\ar[d,"{\left(0,0,1\right)}"] & P_2P_1^2\ar[r] & P_2 \\ P_2\ar[r] & P_1^2P_2\ar[r] & P_2P_1^2\ar[r] & P_2 & &\end{tikzcd}$$and $f_3=f_1f_2$ is represented by the identity map between the two copies of $P_2$ at the ends. One can check that $f_1$ and $f_2$ strictly commute, and that $f_2$ genuinely squares to zero (purely for degree reasons). However, one can check that $f_1^2=-2m^{n-2}f_2$, which is merely homotopic to zero (if $n>2$), not identically zero.
	
	\p So if $n=2$ then the $\ext$-algebra is $k[f_1]/f_1^4$, whereas if $n>2$ then it is isomorphic to the algebra $\frac{k[f_1,f_2]}{(f_1^2,\;f_2^2)}$. In general, the derived endomorphism algebra will not be formal; we would like to use Merkulov's construction (or a variant) to work out the higher $A_\infty$ operations on the Ext-algebra. We will grade the resolution $\tilde S$ of $S$ in order to eliminate many of these: the point is that one can apply the Adams graded version of Merkulov's construction (\ref{adamsrmk}) to the Adams graded dga $\enn_\Lambda(\tilde S)\simeq \R\enn_\Lambda(S)$, and since the higher multiplications $m_r$ must have Adams degree $0$, this allows us to conclude that many of them must be zero. 
	
	\p One can put a grading on $\Lambda$, with $l,m$ in degree 2 and $a,b,s,t$ in degree $n$. We will refer to the degree of a homogeneous element of $\Lambda$ in this secondary grading as its \textbf{(Adams) weight}, since we are already using `degree' to refer to maps. Observe that the projectives $P_i$ become weight graded $\Lambda$-modules, concentrated in infinitely many nonnegative degrees. If $M$ is a weight graded $\Lambda$-module, then write $M(d)$ to refer to $M$ shifted by weight $d$; for example the element $m\in P_2(d)$ has weight $2+d$, because $m\in\Lambda$ has weight 2. Observe that if $M$ and $N$ are weight graded modules, and $f:M \to N$ is a map of a given degree, then $f$ will split into a direct product of weight-homogenous components; such a component of weight $d$ is the same thing as a weight graded map $M \to N(d)$. In all of our examples, $f$ will actually be concentrated in a single weight-homogenous component $M \to N(d)$, in which case we say that $f$ has weight $d$.

	\p Recall that $S$ has projective resolution $$\tilde S\coloneqq P_2 \to P_1^2P_2 \to P_2P_1^2 \to P_2.$$
	\begin{lem}
	Suppose that $n \geq 2$. Give $S$ the trivial weight grading as a $\Lambda$-module. Then the projective resolution $\tilde S\to S$ admits a weight grading as follows:
	$$P_2(2n+2) \to P_1(n+2)\oplus P_1(n+2)\oplus P_2(2n) \to P_2(2)\oplus P_1(n)\oplus P_1(n)\to P_2(0).$$In particular, the differential has weight zero.
	\end{lem}
\begin{proof}
	This is a relatively straightforward check. We verify that the rightmost map is of weight zero; verifying this for the other maps in the resolution is entirely analogous. The rightmost map is $P_2(2)\oplus P_1(n)\oplus P_1(n) \xrightarrow{(m,s,t)} P_2(0)$. The first component $P_2(2) \to P_2(0)$ has weight zero because $m$ has weight 2. Similarly, the second component has weight zero because $s$ has weight $n$. Similarly, the third component has weight zero because $t$ has weight $n$. 
	\end{proof}

\begin{lem}
	The dga $\enn_\Lambda(\tilde S)\simeq \R\enn_\Lambda(S)$ admits an Adams grading. Moreover, $f_1$ has weight $2$ and $f_2$ has weight $2n$; i.e.\ the $f_i$ are maps of (bi)graded modules $f_1:\tilde S \to \tilde S(2)[1]$ and $f_2:\tilde S \to \tilde S(2n)[2]$.
\end{lem}
\begin{proof}
The existence of the Adams grading is an entirely formal consequence of the previous lemma, because any $\Lambda$-linear endomorphism of $\tilde S$ will split into a sum of weight-homogenous components. The check for $f_2$ is straightforward; for example the rightmost component sends $P_2(2n)\to P_2(2n)\cong P_2(0)(2n)$ by the identity map. The check for $f_1$ is similarly straightforward, although we spell out the case of the map $P_2(2n)\xrightarrow{-2m^{n-2}} P_2(2)$ in detail. Because $m$ has weight $2$, the element $-2m^{n-2}$ has weight $2n-4$. Hence the map $P_2(2n)\xrightarrow{-2m^{n-2}} P_2(4)$ is a map of weight-graded modules. But this is exactly the claim since $P_2(4)\cong P_2(2)(2)$.
	\end{proof}

	For purely degree reasons, the nontrivial higher multiplications $m_r$ on $\frac{k[f_1,f_2]}{(f_1^2,\;f_2^2)}$ must (up to permutation) all be of the form $m_r(f_1,\ldots,f_1)=\lambda_r f_2$ or $m_r(f_1,\ldots,f_1,f_2)=\mu_{r}f_1f_2$. Since each higher multiplication must have Adams weight zero by \ref{adamsrmk}, the only nonzero coefficients are possibly $\lambda_n$ and $\mu_2$. In particular, we have:
	\begin{prop}\label{ne2formal}
		If $n=2$, then $\R\enn_\Lambda(S)$ is a formal dga.
		\end{prop}
	\begin{proof}
		The point is that one can apply the Adams graded version of Merkulov's construction (\ref{adamsrmk}) to obtain a quasi-isomorphism of Adams weight graded dgas $\R\enn_\Lambda(S) \to \frac{k[f_1,f_2]}{(f_1^2,\;f_2^2)}$, where we equip the right-hand side with $A_\infty$ higher multiplications. But if $n=2$ then the calculation above shows that all $m_r$ with $r>2$ must be zero; and in particular all higher multiplications vanish.
		\end{proof}
	We already know that $\mu_2=1$, since $m_2$ is the multiplication. Can the $\lambda_n$ be nonzero?
	\begin{lem}\label{pagmassey}
		Let $n \geq 2$. Then $-2f_2$ is an element of the Massey product $\langle f_1,\ldots,f_1\rangle_n$.
	\end{lem}
	\begin{proof}
		For $0\leq i \leq n$, let $h_i$ be the degree $1$, weight $2-2i$ map $$\begin{tikzcd}[row sep= large, column sep=15ex] & P_2\ar[r]\ar[d,"0"] & P_1^2P_2\ar[r]\ar[d,"{\sthbthm{0}{0}{-2m^{n-i}}{0}{0}{0}{0}{0}{0}}"] & P_2P_1^2\ar[r]\ar[d,"0"] & P_2 \\ P_2\ar[r] & P_1^2P_2\ar[r] & P_2P_1^2\ar[r] & P_2 & \end{tikzcd}$$One can check that $h_3$ is a homotopy from $f_1^2$ to 0, motivating the definition of $h_i$. One can in fact compute $dh_{i+1}=-2m^{n-i}f_2$ for all $0\leq i \leq n$. Note that the $h_i$ are all orthogonal, in the sense that $h_ih_{i'}=0$ for all $0\leq i,i' \leq n$. Furthermore, one can check that $f_1h_i+h_if_1=-2m^{n-i}f_2=dh_{i+1}$ for all $0\leq i < n$.
		
		\p By \ref{masseylemma}, the $n$-fold Massey product $\langle f_1,\ldots,f_1\rangle_n$ is the set of cohomology classes of sums $f_1b_{n-1}+\cdots+{b}_{n-1}f_1$ such that $db_i=f_1b_{i-1}+\cdots+{b}_{i-1}f_1$ for all $1<i<n$. In particular, set $b_i\coloneqq h_{i+1}$ for $1< i<n$. Then we have $f_1b_{i-1}+\cdots+b_{i-1}f_1=f_1b_{i-1}+b_{i-1}f_1$ for all $1<i<n$ since the middle terms are of the form $h_ih_{i'}$. Hence, we have $f_1b_{i-1}+\cdots+b_{i-1}f_1=db_i$ for all $1<i<n$. So the sum $f_1b_{n-1}+\cdots+b_{n-1}f_1=-2f_2$ is an element of $\langle f_1,\ldots,f_1\rangle_n$.
	\end{proof}
	\begin{cor}
		If $n>2$, then $\R\enn_\Lambda(S)$ is not formal.
	\end{cor}
\begin{proof}By \ref{pagmassey}, the cohomology algebra admits a nontrivial higher Massey product, and hence $\R\enn_\Lambda(S)$ cannot be formal.
	\end{proof}
We sum up our study of the dga $\R\enn_\Lambda(S)$ as $n$ varies:
	\begin{prop}
		If $n=2$, then $\R\enn_\Lambda(S)$ is a formal dga, quasi-isomorphic to the graded algebra $k[f_1]/f_1^4$ with $f_1$ in degree 1. If $n>2$, then $\R\enn_\Lambda(S)$ is quasi-isomorphic to the strictly unital minimal $A_\infty$-algebra with underlying graded algebra $\frac{k[f_1,f_2]}{(f_1^2,\;f_2^2)}$ with $f_i$ placed in degree $i$, and a single nonzero higher multiplication given by $m_n(f_1,\ldots,f_1)=f_2$.
	\end{prop}
	\begin{proof}
		The statement for $n=2$ is \ref{ne2formal}. If $n>2$, then since $\R\enn_\Lambda(S)$ is not formal, there must exist at least one nonzero higher multiplication. But by the above, the only candidate is $m_n(f_1,\ldots,f_1)=\pm 2 f_2$, where we are using that Massey products are higher multiplications up to sign (\ref{masseysarehighermults}). Replacing $f_2$ by $\pm\frac{1}{2}f_2$, we obtain the desired statement.
	\end{proof}
	\begin{rmk}
		A way of stating this that depends less on $n$ is to say that $\R\enn_\Lambda(S)$ is the strictly unital minimal $A_\infty$-$\frac{k[f_2]}{f_2^2}$-algebra generated by a single element $f_1$ subject to the relations $m_r(f_1,\ldots,f_1)=\delta_{r,n}f_2$.
	\end{rmk}
	
	\begin{prop}\label{pagodadga}
		Suppose that $n \geq 2$. As a noncommutative dga, $\dca$ is freely generated by generators $\xi$, $\zeta$, $\theta$, with degrees $0,-1,-2$ and weights $-2,-2n,-(2n+2)$ respectively. The differential is given by $d\theta=[\xi,\zeta]$, $d\zeta=\xi^n$, and $d\xi=0$.
	\end{prop}
	\begin{proof}
		This is the definition of the Koszul dual dga: take the three basis elements $f_1$, $f_2$, $f_3$ of the augmentation ideal of $\R\enn_\Lambda(S)$, dualise (we put $\xi=f_1^*,\zeta=f_2^*,\theta=f_3^*$), and shift. The differential $d(x^*)$ is the signed sum of the products $x_1^*\cdots x_r^*$ such that $d(x_1 | \cdots | x_r)=x$, where $d$ denotes the $A_\infty$ bar differential.
	\end{proof}
\begin{rmk}We note that this identification of the Koszul dual $\R\enn_\Lambda(S)^!$ is valid in all characteristics not equal to $2$.
	\end{rmk}
\begin{rmk}
	As in \cite[3.3]{amquivers}, one can check that the relations in $\Lambda$ come from equipping the quiver with the superpotential $\overline{W}\coloneqq las-lbt-ams+bmt-\frac{2}{n+1}l^{n+1}+\frac{2}{n+1}m^{n+1}$ and taking the associated Jacobi algebra. In order to compute $A_\con$, one simply takes the subquiver $\begin{tikzcd} \cdot_2 \ar[loop right, "m"]\end{tikzcd}$ and equips it with the modified superpotential $W\coloneqq \frac{2}{n+1}m^{n+1}$. One can easily check that the Ginzburg dga associated to $W$ is precisely the dga appearing in \ref{pagodadga}.
\end{rmk}

	Observe that $H^0(\dca)\cong k[\xi]/\xi^n \cong A_\con$, as expected \cite[3.10]{DWncdf}. It is also not too difficult to compute $H^{-1}(\dca)$: the only elements in degree -1 are noncommutative polynomials in $\xi, \zeta$ with exactly one occurrence of $\zeta$. Noting that $\xi$ is a cocycle and $d\theta=[\xi,\zeta]$, we see that such a polynomial is homotopic to one of the form $p=\zeta\sum_i a_i \xi^i$. But $dp=\sum_i a_i \xi^{i+n}$, and this is zero if and only if $p=0$. So $H^{-1}(\dca)\cong 0$ (we could also have deduced this using \ref{mmoddca}). Hence we find that $H(\dca)$ is zero in odd degrees, and $A_\con$ in nonpositive even degrees. We can now prove:
	\begin{lem}
		The algebra map $\dca \to \dca/(\theta, d\theta )$ is a quasi-isomorphism.
	\end{lem}
	\begin{proof}
		It is easy to check that the dga $\dca/(\theta, d\theta )$ is isomorphic to the dga $\frac{k[\xi]\langle\zeta\rangle}{\xi\zeta-\zeta\xi}$ with $d\zeta=\xi^n$. The cohomology algebra of this dga is $\frac{k[\xi]}{\xi^n}[\eta]$, where $\eta=\zeta^2$ is a degree -2 element. In particular, the cohomology algebra is levelwise isomorphic to that of $\dca$, so in order to check that the quotient map $\dca \to \dca/(\theta, d\theta )$ is a quasi-isomorphism we only need to check that it is a quasi-surjection. But one can check that $\zeta^2+\sum_{i=1}^n\xi^{i-1}\theta\xi^{n-i} \in \dca$ is a cocycle, and maps to $\zeta^2$ in the quotient.
	\end{proof}
	\begin{rmk}
		The expression $\zeta^2+\sum_{i=1}^n\xi^{i-1}\theta\xi^{n-i}$ comes from using $d\theta=\xi\zeta-\zeta\xi$ to repeatedly commute $\xi$ with $\zeta$ in $d(\zeta^2)=\xi^n\zeta-\zeta\xi^n$. One can also check that this cocycle is homogenous of weight $4n$.
	\end{rmk}
	\begin{thm}\label{dcapagoda}
		The derived contraction algebra associated to the $n$-Pagoda flop is quasi-isomorphic to the strictly unital minimal $A_\infty$-algebra $\frac{k[\xi]}{\xi^n}[\eta]$ with $\xi$ in degree $0$, $\eta$ in degree $-2$, and higher multiplications given by $$m_r(\eta^{i_1}\xi^{j_1} \otimes \cdots\otimes \eta^{i_r}\xi^{j_r}) = \begin{cases} -(-1)^{\frac{r}{2}}C_{\frac{r}{2}}\eta^{i + \frac{r}{2}-1}\xi^{j -n(r-2)} & r\text{ is even and } n(r-2)\leq j < n(r-1)\\ 0 & \text{otherwise}
		\end{cases}$$where we put $i=i_1 + \cdots + i_r$ and $j=j_1+\cdots+j_r$ and the $C_p$ are the (shifted) Catalan numbers with $C_1=1$, $C_2=1$, $C_3=2$, $C_4=5$, et cetera.
	\end{thm}
	\begin{proof}
		By the above, $\dca$ is quasi-isomorphic to the dga $C\coloneqq \frac{k[\xi]\langle\zeta\rangle}{\xi\zeta-\zeta\xi}$ with $d\zeta=\xi^n$. We know that the cohomology of $C$ is $\frac{k[\xi]}{\xi^n}[\eta]$, where $\eta=\zeta^2$. We use Merkulov's construction to augment $C$ with higher multiplications $m_r$ inducing an $A_\infty$ quasi-isomorphism with $\dca$. One linear section of the projection map $\pi: C \to HC$ is the map $\sigma$ that in odd degrees is zero, and in even degrees sends $\eta^i\xi^j$ to $\eta^i\xi^j$ if $j<n$, and zero otherwise. Composing the projection with the section yields the linear endomorphism of $C$ given by $$\sigma \pi\quad = \quad\begin{tikzcd} \cdots \ar[r,"0"] & k[\xi] \ar[r,"\xi^n"] \ar[d,"0"] & k[\xi] \ar[r,"0"] \ar[d,"\epsilon"] & k[\xi] \ar[r,"\xi^n"] \ar[d,"0"]& k[\xi] \ar[d,"\epsilon"] \\ \cdots \ar[r,"0"] & k[\xi] \ar[r,"\xi^n"] & k[\xi] \ar[r,"0"] & k[\xi] \ar[r,"\xi^n"] & k[\xi]  \end{tikzcd}$$ where $\epsilon(\xi^i)$ is $\xi^i$ if $i<n$, and 0 otherwise. Firstly, we need to construct a homotopy $h: \id_C \to \sigma\pi$. Interpreting a negative power of $\xi$ as 0, one can check that the (periodically extended) homotopy given by $h_1(\xi^i)= \xi^{i-n}$ and $h_2=0$ works: $$\begin{tikzcd} \cdots \ar[r,"0"] & k[\xi] \ar[r,"\xi^n"] \ar[d,"0"] \ar[dl,swap,"h_2"] & k[\xi] \ar[r,"0"] \ar[d,"\epsilon"] \ar[dl,swap,"h_1"] & k[\xi] \ar[r,"\xi^n"] \ar[d,"0"]\ar[dl,swap,"h_2"]& k[\xi] \ar[d,"\epsilon"]\ar[dl,swap,"h_1"] \\ \cdots \ar[r,"0"] & k[\xi] \ar[r,"\xi^n"] & k[\xi] \ar[r,"0"] & k[\xi] \ar[r,"\xi^n"] & k[\xi]  \end{tikzcd}$$In other words, we have $h(\eta^i\xi^j)=\zeta\eta^i\xi^{j-n}$. Now we put inductively $$m_r\coloneqq \sum_{s+t=r}(-1)^{s+1}(s,t)$$where for brevity I write $(s,t)\coloneqq m_2(hm_s\otimes hm_t)$ and $hm_1\coloneqq -\id_A$. Then Merkulov's theorem tells us that $HC$, augmented with the $m_r$, is $A_\infty$-quasi-isomorphic to $C$. First I claim that when $r>1$ is odd, then $m_r$ vanishes: this is clear for degree reasons. From now on we may assume that $r$ is even, and we can conclude that the sum defining $m_r$ reduces to $m_r=-\sum_{s+t=r}(s,t)$. I next claim that the maps $m_r$ are $\eta$-linear, in the sense that $$m_r(\eta^{i_1}\xi^{j_1} \otimes \cdots\otimes \eta^{i_r}\xi^{j_r}) = \eta^{i_1 + \cdots + i_r}m_r(\xi^{j_1}\otimes \cdots\otimes \xi^{j_r})$$holds. This is not hard to see inductively: it is clearly true for $m_2$. Suppose that it is true for all $r'<r$, and suppose $s+t=r$ with $s,t>0$. Then \begin{align*}
		&(s,t)(\eta^{i_1}\xi^{j_1} \otimes \cdots\otimes \eta^{i_r}\xi^{j_r})
		\\&=m_2(hm_{s}(\eta^{i_1}\xi^{j_1} \otimes \cdots\otimes \eta^{i_{s}}\xi^{j_{s}})\otimes hm_{t}(\eta^{i_{r-t+1}}\xi^{j_{r-t+1}} \otimes \cdots\otimes \eta^{i_r}\xi^{j_r}))\\&=m_2(h(\eta^{i_1+\cdots+ i_s}m_{s}(\xi^{j_1} \otimes \cdots\otimes \xi^{j_{s}}))\otimes h(\eta^{i_{r-t+1}+\cdots+ i_r} m_{t}(\xi^{j_{r-t+1}} \otimes \cdots\otimes \xi^{j_r})))\\&=m_2(\eta^{i_1+\cdots+ i_s}hm_{s}(\xi^{j_1} \otimes \cdots\otimes \xi^{j_{s}})\otimes \eta^{i_{r-t+1}+\cdots+ i_r}h m_{t}(\xi^{j_{r-t+1}} \otimes \cdots\otimes \xi^{j_r}))\\&=\eta^{i_1+\cdots+ i_r}m_2(hm_{s}(\xi^{j_1} \otimes \cdots\otimes \xi^{j_{s}})\otimes h m_{t}(\xi^{j_{r-t+1}} \otimes \cdots\otimes \xi^{j_r}))\\&=\eta^{i_1+\cdots+ i_r}(s,t)(\xi^{j_1} \otimes \cdots\otimes \xi^{j_r})
		\end{align*}
		where the second and sixth line follow because all degrees of elements are even, the third line follows from the induction hypothesis, the fourth uses that $h(\eta x)=\eta h(x)$, and the fifth uses centrality of $\eta$ as well as collating powers of $\eta$. The claim now follows by adding all of these up using $m_r=-\sum_{s+t=r}(s,t)$. Now all that needs to be done is determine $m_r(\xi^{j_1} \otimes \cdots\otimes \xi^{j_r})$ when $r>2$ is even. First observe that $h(\xi^i)=\zeta\xi^{i-n}$, where we again interpret a negative power of $\xi$ as zero. Let $C_p$ be the $p^\text{th}$ Catalan number, with indexing starting from $p=1$; so $C_1=1$, $C_2=1$, $C_3=2$, $C_4=5$, et cetera. The important point for us will be that $C_p=\sum_{s+t=p}C_sC_t$ and $C_1=1$, where in the sum we require that $s$ and $t$ are positive integers. For even $r>0$, put $C'_r\coloneqq -(-1)^{\frac{r}{2}}C_{\frac{r}{2}}$. I claim that $m_r(\xi^{j_1} \otimes \cdots\otimes \xi^{j_r})=C'_r\zeta^{r-2}\xi^{j_1+\cdots+j_r -n(r-2)}$. Certainly this holds for $r=2$. Inductively as before, one sees that for $2s+2t=r$, the expression $(2s,2t)(\xi^{j_1} \otimes \cdots\otimes \xi^{j_r})$ is exactly $C'_sC'_t\zeta^{r-2}\xi^{j_1+\cdots+j_r -n(r-2)}$. Hence, since we may compute the sum defining $m_r$ by summing only over even terms, we get \begin{align*}
		m_r(\xi^{j_1} \otimes \cdots\otimes \xi^{j_r})&=-\sum_{2s+2t=r}(2s,2t)
		(\xi^{j_1} \otimes \cdots\otimes \xi^{j_r})\\&=\left(-\sum_{2s+2t=r}C'_{2s}C'_{2t}\right)\zeta^{r-2}\xi^{j_1+\cdots+j_r -n(r-2)}\\&=\left(-\sum_{s+t=\frac{r}{2}}-(-1)^sC_{s}\cdot -(-1^t)C_{t}\right)\zeta^{r-2}\xi^{j_1+\cdots+j_r -n(r-2)}
		\\&=C'_r\zeta^{r-2}\xi^{j_1+\cdots+j_r -n(r-2)}
		\end{align*}where the final line follows by the identity defining the Catalan numbers. Putting everything together, using $\eta=\zeta^2$, and recalling that we interpret a negative power of $\xi$ as zero, we obtain the required identities. Note that $n(r-2)\leq j$ is required to make the exponent of $\xi$ positive, and $j<n(r-1)$ is required to make it less than $n$ (so one could drop this condition if necessary).
	\end{proof}

\begin{rmk}
This application of Merkulov's theorem is valid in all characteristics.
	\end{rmk}
	
	\begin{rmk}
		Note that the higher multiplications are all $\eta$-linear, so another way to state the above is to say that the derived contraction algebra $\dca$ is quasi-isomorphic to the strictly unital minimal $A_\infty$-$k[\eta]$-algebra generated by $\xi$ subject to the relations $\xi^n=0$ and $$m_r(\xi^{j_1} \otimes \cdots\otimes \xi^{j_r}) = \begin{cases} -(-1)^{\frac{r}{2}}C_{\frac{r}{2}}\eta^{ \frac{r}{2}-1}\xi^{j -n(r-2)} & r\text{ is even and } n(r-2)\leq j < n(r-1)\\ 0 & \text{otherwise}\end{cases}.$$
	\end{rmk}

	\section{The Laufer flop}\label{lauferex}
	We sketch a computation of the derived contraction algebra associated to the Laufer flop, which is a $D_4$ flop with (completed) base $\frac{k\llbracket x,y,u,v \rrbracket}{u^2+v^2y - x(x^2+y^3)}$ first appearing in \cite{laufer}. Following \cite{amquivers}, one noncommutative model is the algebra $A$ given by the (completion of) the following quiver with relations:\\ \\
	
	\begin{wrapfigure}[4]{L}{0.5\textwidth}
		\centering
		\vspace{-15pt}$
		\begin{tikzcd}
		1  \arrow[rr,bend left=30,"a"]  && 2  \arrow[ll,bend left=30,"b"] \arrow[in=20, out=70, distance=7mm,"x"] \arrow[out=-20, in=-70, distance=7mm,"y"]
		\end{tikzcd}$
	\end{wrapfigure} 
	
	$ay^2=-aba$
	
	$y^2b=-bab$
	
	$xy=-yx$
	
	$x^2+yba+bay=y^3$
	\\\hfill\\
	One can check that $A$ admits a nontrivial grading with $a,b,y$ of weight 2 and $x$ of weight 3. The simple $S$ at the vertex 2 has a resolution (compatible with the weights) given by $$P_2 \xrightarrow{\left( \begin{smallmatrix} x \\ a \\ y \end{smallmatrix} \right)\cdot} P_2P_1P_2 \xrightarrow{\sthbthm{0}{ab}{ay}{y}{0}{x}{x}{yb}{ba-y^2}\cdot} P_1P_2P_2 \xrightarrow{( b ,  x , y )\cdot} P_2$$and using this one can check that the Ext-algebra of $S$ is of the form $$\ext^*_A(S,S)\cong\frac{k[f,g]}{(f^3,g^2)}$$ where the generators $g$ and $f$ both have degree 1 and weights $2$ and $3$ respectively. This is a six-dimensional algebra, with basis $\{1,g,f, gf, f^2, gf^2\}$. A computation with weights tells us that the only possible nontrivial higher Massey product for $\ext^*_A(S,S)$ is of the form $\langle g,g,g \rangle=\lambda f^2$. Picking an explicit lift $G$ of $g$ and a homotopy $H:G^2 \xrightarrow{\simeq} 0$ allows one to show that $f^2=[GH+HG]\in \langle g,g,g \rangle$. Hence $\R\enn_A(S)$ is not formal, and must have a higher Massey product of the form $\langle g,g,g \rangle=\pm f^2$ by \ref{masseysarehighermults}. Rescaling $f\mapsto \frac{1}{\sqrt{-1}}f$ if necessary, we see that $\R\enn_A(S)$ must be quasi-isomorphic to the strictly unital minimal $A_\infty$-algebra with underlying graded algebra $\ext^*_A(S,S)$ and a single nonzero higher multiplication given by $m_3(g,g,g)=f^2$.
	\p Now we wish to compute the Koszul dual of $\R\enn_A(S)$. Put $x=f^*$, $y=g^*$, $\zeta=(fg)^*$, $\xi=(f^2)^*$ and $\theta=(gf^2)^*$. Then $\R\enn_A(S)^!$ is freely generated by $\{x,y,\zeta,\xi,\theta\}$, and after working out the $A_\infty$ bar differential one arrives at the following theorem:
	
\begin{thm}\label{lauferflop}
	The derived contraction algebra of the Laufer flop is freely generated as a noncommutative dga by elements $x,y,\zeta,\xi,\theta$ in degrees $0,0,-1,-1,-2$ and of weights $-3,-2,-5,-6,-8$ respectively. The differential is defined on generators by $dx=dy=0$, $d\zeta=-(xy+yx)$, $d\xi=y^3-x^2$, and $d\theta = [\xi,y] + [\zeta,x]$.
\end{thm}
	\begin{rmk}
		In particular, one can use the above description to see that $$H^0(\dca)\cong \frac{k\langle x,y\rangle }{(xy+yx, x^2-y^3)}$$the \textbf{quantum cusp}, which recovers the computation of \cite[Example 1.3]{DWncdf}.
	\end{rmk}
\begin{rmk}
	The above computations are valid in all characteristics.
\end{rmk}	
	
	It is unclear to the author how to produce an explicit cocycle representing the periodicity element $\eta \in H^{-2}(\dca)$ in terms of the generators given above. A computer calculation suggests that $\eta$ must have weight $-18$. It may be feasible to use \ref{dsgrmk} to produce a model of $\dca$ where $\eta$ is represented by a genuinely central cocycle.
	
	\begin{rmk}\label{lauferrmk}
		The relations on the path algebra $A$ come from a superpotential \cite[4.4]{amquivers} $$\overline{ W}\coloneqq bay^2+\frac{1}{2}abab+x^2y-\frac{1}{4}y^4.$$Following \cite{DWncdf}, to compute the contraction algebra $A_\con$ one considers the subquiver 
		$$
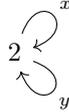
\begin{tikzcd}
		2  \arrow[in=20, out=70, distance=7mm,"x"] \arrow[out=-20, in=-70, distance=7mm,"y"]
		\end{tikzcd}$$equipped with the modified superpotential $W\coloneqq x^2y-\frac{1}{4}y^4$. One can easily see that the Jacobi algebra of this quiver with superpotential is precisely $A_\con$. The Ginzburg dga associated to $W$ has generators $\{x,y,x^*,y^*,z\}$ in degrees $0,0,-1,-1,-2$ respectively, with differential $dx=dy=0$, $dx^*=-(xy+yx)$, $dy^*=y^3-x^2$, and $dz=xx^*-x^*x+yy^*-y^*y$. One can easily see that this Ginzburg dga is isomorphic (not just quasi-isomorphic!) to the dga we obtain in \ref{lauferflop} above.
		
	\end{rmk}

	\section{Slicing}\label{slicingsctn}
	In this section, we will think about slicing flopping contractions by generic hyperplanes to get partial crepant resolutions of surface singularities. We will pay special attention to how tilting bundles and their endomorphism rings behave under slicing; our aim is to prove \ref{cutsthm}. All of the arguments we use in this part were communicated to us by Michael Wemyss; the general idea is to adapt a proof of Ishii and Ueda \cite[8.1]{ishiiueda}.
	
	\p The setup for this part will be a threefold flopping contraction $\pi: X \to \spec R$. Slice it by a generic hyperplane section $g$ to get a pullback diagram of the form $$\begin{tikzcd} Y\ar[d,"\psi"] \ar[r,"j"]& X \ar[d,"\pi"]\\ \spec(R/g) \ar[r,"i"]& \spec R
	\end{tikzcd}$$First note that, by Reid's general elephant principle \cite[1.1, 1.14]{reidpagoda}, $\psi$ is a partial crepant resolution of a Kleinian ADE singularity, and in particular projective and birational.
	
	\begin{lem}\label{cutslem}
		With the setup as above, the following hold:\hfill
		\begin{enumerate}
			\item[\emph{1.}] $g: \mathcal{O}_X \to \mathcal{O}_X $ is an injection.
			\item[\emph{2.}] Let $\mathcal W$ be a vector bundle on $X$. Then there is a short exact sequence $$0\to \mathcal W \xrightarrow{g} \mathcal W \to j_*j^*\mathcal W \to 0$$
			\item[\emph{3.}] Let $\mathcal W$ be a vector bundle on $X$ such that $\R^p\pi_*\mathcal W \simeq 0$ for $p>0$. Putting $W\coloneqq \pi_*\mathcal W$, we have a quasi-isomorphism $\R\pi_*(j_*j^* \mathcal W) \simeq W / g W$.
		\end{enumerate}
	\end{lem}
	\begin{proof}
		For 1., note first that $g$ is a global section of $\mathcal{O}_X$, or equivalently an endomorphism of $\mathcal{O}_X$. Let $\mathcal K$ be its kernel. Because $X$ is normal, $\mathcal K$ is a reflexive sheaf: this is because on a normal integral noetherian scheme reflexive sheaves are characterised by the $S_2$ property, which is closed under taking kernels (e.g.\ \cite[2.10]{schwedediv}). Because $R$ is an integral domain and $\pi_*$ preserves kernels, one has $\pi_*\mathcal K \cong 0$. But $\pi_*$ is a reflexive equivalence \cite[4.2.1]{vdb}, and hence $\mathcal K \cong 0$ as required. For 2., the only thing to check is exactness on the left. But this follows by tensoring the injection $g: \mathcal{O}_X \to \mathcal{O}_X$ with $\mathcal W$. For 3., note that by assumption $\mathcal W$ is $\pi_*$-acyclic, so we may compute the derived pushforward $\R\pi_*(j_*j^* \mathcal W)$ using its $\pi_*$-acyclic resolution $\mathcal W \xrightarrow{g} \mathcal W$. We hence get $\R\pi_*(j_*j^* \mathcal W)\simeq W \xrightarrow{g} W$ where the righthand $W$ is placed in degree zero. By reflexive equivalence again, $W$ is reflexive, and since it is a submodule of a free module we see that $g: W \to W$ is also injective. It follows that $\R\pi_*(j_*j^* \mathcal W)$ has cohomology only in degree zero, where it is $W/gW$.
	\end{proof}
	
	\begin{prop}\label{slicetilt}
		With the setup as above, let $\mathcal W$ be a tilting bundle on $X$. Then $j^*\mathcal W$ is a tilting bundle on $Y$, with endomorphism ring $\enn_Y(j^*\mathcal W)\cong \enn_{R/g}(\psi_*j^*\mathcal W)\cong \enn_{R/g}(i^*\pi_*\mathcal W)$.
	\end{prop}
	\begin{proof}
		First observe that because $\mathcal W$ is a vector bundle, so is $j^*\mathcal W$, and we also have a quasi-isomorphism $j^*\mathcal W \simeq \mathbb{L} j^*\mathcal W$. Because $j$ is a closed immersion (it is the pullback of the closed immersion $i$), we have $\R j_* \simeq j_*$. Now it follows by adjunction that $$\R\hom_Y(j^*\mathcal W,-) \simeq  \R\hom_X(\mathcal W, j_* -).$$For generation, let $\mathcal F \in D(\mathrm{QCoh}(Y))$. By the quasi-isomorphism above, $\R\hom_Y(j^*\mathcal W,\mathcal F)\simeq 0$ if and only if $\R\hom_X(\mathcal W, j_* \mathcal F)\simeq 0$. But $\mathcal W$ generates by assumption, so this is the case if and only if $j_* \mathcal F\simeq 0$, which is the case if and only if $\mathcal F \simeq 0$. Hence, $j^*\mathcal W$ generates. To show Ext vanishing, we first compute $\R\enn_Y(j^*\mathcal W)\simeq \R\hom_X(\mathcal W, j_*j^* \mathcal W)$ as before. Because $\mathcal W$ is a vector bundle, we have $$\R\hom_X(\mathcal W, j_*j^* \mathcal W)\simeq \R\hom_X(\mathcal{O}_X,\mathcal W^*\otimes j_*j^* \mathcal W)\simeq \R\pi_*(\mathcal W^*\otimes j_*j^* \mathcal W).$$Again, because $\mathcal W$ is a vector bundle, $\mathcal W^*\otimes j_*j^* \mathcal W$ is quasi-isomorphic to $\mathcal W ^* \otimes (\mathcal W \xrightarrow{g} \mathcal{W})$ using \ref{cutslem}(2) But $\mathcal W ^* \otimes (\mathcal W \xrightarrow{g} \mathcal{W})\cong (\mathcal W ^* \otimes \mathcal W )\xrightarrow{g} (\mathcal{W}^* \otimes \mathcal{W})$, which, using \ref{cutslem}(2) again, is quasi-isomorphic to $j_*j^*(\mathcal{W}^* \otimes \mathcal{W})$. So we have $\R\enn_Y(j^*\mathcal W)\simeq \R \pi_*(j_*j^*(\mathcal{W}^* \otimes \mathcal{W}))$, which by \ref{cutslem}, 3. (using that higher Exts between $\mathcal W$ vanish) is concentrated in degree zero. So $j^*\mathcal W$ is tilting. Note that this does not tell us about the ring structure on $\R\enn_Y(j^*\mathcal W)$, since we had to pass through adjunctions. 
		
		\p For the statements about endomorphism rings, observe first that we have a ring map\linebreak  $\psi_*:\enn_Y(j^*\mathcal W)\to \enn_{R/g}(\psi_*j^*\mathcal W)$ which is also a map of reflexive $R/g$-modules. Since it is an isomorphism at height one primes, and $R/g$ is normal, it hence must be an isomorphism. It remains to check that $\enn_{R/g}(\psi_*j^*\mathcal W)\cong \enn_{R/g}(i^*\pi_*\mathcal W)$. To prove this we will show that $\psi_*j^*\mathcal W \cong i^*\pi_*\mathcal W$. Proceeding as before, we have \begin{align*}
		\psi_*j^*\mathcal W &\simeq \R\hom_Y(\mathcal{O}_Y, j^*\mathcal W)
		\\ &\simeq \R\hom_Y(j^*\mathcal{O}_X, j^*\mathcal W)
		\\ &\simeq \R\hom_X(\mathcal{O}_X, j_*j^*\mathcal W)
		\\ & \simeq \R\pi_*j_*j^*\mathcal W
		\\&\simeq i^*\pi_*\mathcal W
		\end{align*}where the last isomorphism is \ref{cutslem}, 3.
	\end{proof}
	We would like to say a little more: not just that one can compute the endomorphism ring of a tilting bundle on the base, but also that one can compute the endomorphism ring of $i^*W$ by applying the functor $i^*$ to the endomorphism ring of $W$. This is a little delicate and will require some more hypotheses; we show this in the case that $\mathcal W$ is Van den Bergh's tilting bundle.
	\begin{thm}\label{cutsthm}
		With the setup as above, let $\mathcal V$ be Van den Bergh's tilting bundle on $X$ constructed in \text{\normalfont\cite[3.2.8]{vdb}}. Then  $j^*\mathcal{V}$ is a tilting bundle on $Y$, and one has a ring isomorphism $\enn_Y(j^*\mathcal{V})\cong \enn_R(\pi_*\mathcal{V})/g\enn_R(\pi_*\mathcal{V})$.
	\end{thm}
	\begin{proof}
		Immediately from \ref{slicetilt}, we see that $j^*\mathcal V$ is tilting and has endomorphism ring given by $\enn_Y(j^*\mathcal{V})\cong  \enn_{R/g}(i^*\pi_*\mathcal V)$. Putting $V\coloneqq \pi_*\mathcal V$, it remains only to prove that we have an isomorphism $\enn_{R/g}(i^*V)\cong i^*\enn_R(V)$. By \cite[3.2.10]{vdb}, both $V$ and $\enn_R(V)$ are Cohen--Macaulay $R$-modules, and one moreover has an isomorphism $\enn_X(\mathcal V)\cong\enn_R(V)$. Because $R$ is an isolated singularity, \cite[2.7]{iwmaxmod} now gives $\ext^1_R(V,V)\cong 0$. We now follow the proof of \cite[A.1]{vdb}. Note that because $V$ is Cohen--Macaulay it follows that $i^*V$ is Cohen--Macaulay over $R/g$. Applying $\hom_R(V,-)$ to the exact sequence $$0 \to V \xrightarrow{g} V \to i^* V \to 0$$ gives an exact sequence $$0 \to \enn_R(V) \xrightarrow{g} \enn_R(V) \to \hom_R(V,i^*V)\to 0$$or in other words an isomorphism $\hom_R(V,i^*V)\cong i^*\enn_R(V)$. But there is an isomorphism $\hom_R(V,i^*V) \cong \enn_{R/g}(i^*V)$, and it is not hard to check that the induced linear isomorphism $\enn_{R/g}(i^*V)\cong i^*\enn_R(V)$ is a ring map.
	\end{proof}

	\section{Partial resolutions of $A_n$ singularities}
	In this section we compute the derived contraction algebra associated to a certain 1-curve partial crepant resolution of an $A_n$ singularity, obtained by slicing a 1-curve flop. Note that for a Kleinian singularity we are automatically in the Local Setup \ref{localsetup} so may define the derived contraction algebra.  
	\p Let $\tilde{X} \to \spec\tilde{R}$ be the Atiyah flop. Observe that, for any choice of $n$, one can slice $\tilde{R}$ along the hypersurface $x=y^n$ to obtain a partial crepant resolution $X \to \spec(R)$ of an $A_n$ singularity. Let $\tilde \Lambda$ be the NCCR of $\tilde R$ from \S\ref{atiyahflop}, with quiver presentation\\
	\begin{wrapfigure}[4]{L}{0.5\textwidth}
		\centering
		\vspace{-15pt}$
		\begin{tikzcd}
		1 \arrow[rr,bend left=15,"b"] \arrow[rr,bend left=50,"a"]  && 2 \arrow[ll,bend left=15,"s"] \arrow[ll,bend left=50,"t"]
		\end{tikzcd}$
	\end{wrapfigure} 
	
	$asb=bsa$
	
	$sbt=tbs$
	
	$atb=bta$
	
	$sat=tas$\newline\newline
	By \ref{cutsthm}, it follows that $X$ is derived equivalent to the `sliced NCCR' $\Lambda\coloneqq \tilde \Lambda / (x-y^n)\tilde \Lambda$. Recalling the construction of $\tilde \Lambda$, we had $x=at+ta$ and $y=sb+bs$. Moreover, since we have $sbbs=bssb=0$, we have $y^n=(sb)^n + (bs)^n$. So we need to add the relation $at+ta=(sb)^n + (bs)^n$, which is equivalent to adding the two relations $at=(bs)^n$ and $ta=(sb)^n$. The algebra $\Lambda$ we get is\newline
	\begin{wrapfigure}[4]{L}{0.5\textwidth}
		\centering
		\vspace{-15pt}$
		\begin{tikzcd}
		1 \arrow[rr,bend left=15,"b"] \arrow[rr,bend left=50,"a"]  && 2 \arrow[ll,bend left=15,"s"] \arrow[ll,bend left=50,"t"]
		\end{tikzcd}$
	\end{wrapfigure} 
	
	$asb=bsa$
	
	$sbt=tbs$
	
	$at=(bs)^n$
	
	$ta=(sb)^n$\newline\newline noting that $atb=bta$ and $sat=tas$ follow from the new relations. Observe that $\Lambda$ admits a grading by putting the generators $e_1,e_2$ in degree 0, the generators $b$ and $s$ in degree 1, and the generators $a$ and $t$ in degree $n$. We will refer to the degree of a homogeneous element of $\Lambda$ as its \textbf{weight}, since we are already using `degree' to refer to maps. Let $S$ be the simple module at 2. Our main theorem is the following:
	\begin{thm}\label{mainsurfacethm}
		If $n=1$, then $\dca$ is quasi-isomorphic to the free noncommutative graded algebra $k\langle\zeta\rangle$ on a variable $\zeta$ in degree $-1$. If $n\geq 2$, then $\dca$ is quasi-isomorphic to the strictly unital minimal $A_\infty$-algebra with underlying graded algebra $k[\eta,\zeta]$ where $\eta$ has degree $-2$, $\zeta$ has degree -1, and the only nontrivial higher multiplications are $$m_{n+1}(\eta^{b_1}\zeta,\ldots, \eta^{b_{n+1}}\zeta)=\eta^{b_1+\cdots+b_{n+1} + n}.$$Note in particular that since $\zeta$ has odd degree and is a commutative element, it must square to zero.
	\end{thm}
	\begin{proof}First suppose that $n=1$. Since $S$ is spherical, we see that the cohomology algebra is $\ext^*_\Lambda(S,S)\cong k[x]/x^2$, where $x$ has degree 2. By the same argument we used for the Atiyah flop, $\R\enn_\Lambda(S)$ must be formal, and hence the derived contraction algebra is the noncommutative dga $$\dca=k\langle\zeta\rangle$$ where $\zeta$ has degree $-1$. Note that the periodicity element is $\eta=\zeta^2$. The $n\geq2$ case requires much more work; the rest of this section is devoted to the proof of this case, which appears as \ref{surfdcang2}.
	\end{proof}
	\begin{rmk}
		Note that in the $n=1$ case, the periodicity element $\eta$ is given by $\eta=\zeta^2$. When $n=1$, note that $m_{n+1}(\eta^{b_1}\zeta,\ldots, \eta^{b_{n+1}}\zeta)=\eta^{b_1+\cdots+b_{n+1} + n}$ still holds; hence a more uniform way to state the above is to say that $\dca$ as a strictly unital minimal $A_\infty$-$k[\eta]$-algebra is generated by the single element $\zeta$ subject to the relations $m_{r}(\zeta,\ldots,\zeta)=\delta_{r,n+1}\eta^n$.
	\end{rmk}
	\begin{rmk}
		Unlike in the threefold setting, there does not seem to be a simple way of obtaining $\dca$ as a Ginzburg dga. Indeed, the subquiver at the vertex 2 is a point, and so must have associated Ginzburg dga $k[\eta]$, with $\eta$ in degree -2, independent of what the superpotential is. One would expect this to happen: the derived category of a Ginzburg dga is always 3CY, so one would only expect a Ginzburg dga interpretation of the derived contraction algebra in the threefold setting.
	\end{rmk}
	
	From now on, we assume that $n\geq2$. We begin by noting a few preliminaries. For brevity, put $\beta\coloneqq bs$ and $\sigma\coloneqq sb$. One can check that $S$ has projective resolution $\tilde S$ given by
	$$\cdots \xrightarrow{\stbtm{-\beta^{n-1}}{a}{t}{-\sigma} \cdot} P_1 P_2 \xrightarrow{\stbtm{\beta}{a}{t}{\sigma^{n-1}}\cdot} P_1 P_2 \xrightarrow{\stbtm{-\beta^{n-1}}{a}{t}{-\sigma} \cdot} P_1 P_2 \xrightarrow{\stbtm{bt}{\beta^{n-1}b}{-\beta}{-a}\cdot} P_1^2 \xrightarrow{\left( s \;\; t\right)\cdot}  P_2$$which eventually becomes periodic with period two. Moreover one can check that with the grading conventions on $\Lambda$ from above this admits a secondary Adams grading by weight$$\cdots \xrightarrow{} P_1(3n+2) \oplus P_2(4n) \xrightarrow{} P_1(3n) \oplus P_2(2n+2) \xrightarrow{} P_1(n+2)\oplus P_2(2n) \xrightarrow{} P_1(1)\oplus P_1(n) \xrightarrow{}  P_2(0) $$ and in particular the differential has weight zero.
	
	\p  We see that $$\ext^i_\Lambda(S,S)\cong \begin{cases} 0 & i<0\normalfont\text{  or  } i=1\\
	k & i=0 \normalfont\text{  or  } i>1
	\end{cases}$$spanned by the classes of the projection maps $P_2 \to S$. Define maps $g_0=\id$, and for $k\geq 1$ 
	$$g_{2k}\coloneqq \begin{tikzcd}[sep=large]	S^{2k}\ar[d,swap,"\left( 0 \;\; -1\right)\cdot"]&S^{2k+1}\ar[l,swap,"d_{2k}"]\ar[d,swap,"\stbtm{0}{b}{-1}{0}\cdot"]&S^{2k+2}\ar[l,swap,"d_{2k+1}"]\ar[d,swap,"\id"]&S^{2k+3}\ar[l,swap,"d_{2k+2}"]\ar[d,swap,"\id"]&S^{2k+4}\ar[l,swap,"d_{2k+3}"]\ar[d,swap,"\id"]&\cdots\ar[l,swap,"d_{2k+4}"]\\S^0&S^1\ar[l,"d_0"]&S^2\ar[l,"d_1"]&S^3\ar[l,"d_2"]&S^4\ar[l,"d_3"]&\cdots\ar[l,"d_4"]\end{tikzcd}$$
	$$g_{2k+1}\coloneqq \begin{tikzcd}[sep=large]S^{2k+1}\ar[d,swap,"\left( 0 \;\; -1\right)\cdot"]&S^{2k+2}\ar[l,swap,"d_{2k+1}"]\ar[d,swap,"\stbtm{0}{\beta^{n-2}b}{1}{0}\cdot"]&S^{2k+2}\ar[l,swap,"d_{2k+3}"]\ar[d,swap,"\stbtm{-\beta^{n-2}}{0}{0}{1}\cdot"]&S^{2k+4}\ar[l,swap,"d_{2k+3}"]\ar[d,swap,"\stbtm{-1}{0}{0}{\sigma^{n-2}}\cdot"]&S^{2k+5}\ar[l,swap,"d_{2k+4}"]\ar[d,swap,"\stbtm{-\beta^{n-2}}{0}{0}{1}\cdot"]&\cdots\ar[l,swap,"d_{2k+5}"]\\S^0&S^1\ar[l,"d_0"]&S^2\ar[l,"d_1"]&S^3\ar[l,"d_2"]&S^4\ar[l,"d_3"]&\cdots\ar[l,"d_4"]\end{tikzcd}$$
	Then the $g_k$ span the cohomology algebra $\ext_\Lambda^*(S,S)$ since each (up to sign) lifts the projection maps $P_2 \to S$. Moreover, letting $\phi$ be the degree zero map with $\phi_0=\sigma^{n-2}\cdot$, $\phi_1=\beta^{n-2}\cdot$, and $\phi_j=\tbtm{\beta^{n-2}}{0}{0}{\sigma^{n-2}}\cdot$ for all $j>1$, one can check that the $g_k$ satisfy $$g_ig_j=\begin{cases} g_{i+j} & \text{if }ij \text{ is even} \\ g_{i+j}\phi & \text{ else}\end{cases}$$Put $x\coloneqq [g_2]$ and $y\coloneqq [g_3]$.
	
	\begin{prop}
		Suppose that $n=2$. Then the derived endomorphism algebra $\R\enn_\Lambda(S)$ is formal, with cohomology algebra $\frac{k[x,y]}{(x^3-y^2)}$.  Note that this is a noncommutative dga, because $y$ does not commute with itself.
	\end{prop}
	\begin{proof}
		We see that $\phi=\id$ and hence $\ext_\Lambda^*(S,S)$ is isomorphic to the given algebra where $x$ has degree 2, $y$ has degree 3, and the differential is zero. It is easy to see that $\R\enn_\Lambda(S)$ must be formal, since it is quasi-isomorphic to the subalgebra generated by $\id, g_2$ and $g_3$.
	\end{proof}
	
	\begin{prop}
		Suppose that $n>2$. Then the Ext-algebra $\ext_\Lambda^*(S,S)\cong k[x,y]$ is freely generated as a cdga by $x$ and $y$. Note that $y^2=0$.
	\end{prop}
	\begin{proof}
		One can check that $[\phi]=0$, and the result follows.
	\end{proof}
	
	Now we need to split our argument into cases. We can handle the $n=2$ case already, but part of the argument will be identical for $n>2$, so we defer this for the present moment. We aim first to identify, for $n>2$, the higher $A_\infty$ multiplications on $\ext_\Lambda^*(S,S)$ making it quasi-isomorphic to $\R\enn_\Lambda(S)$, via a Massey product computation. In order to do this, note that the resolution $\tilde S$ of $S$ becomes eventually periodic, with period 2. It will be convenient for us to work in the 2-periodic part of the dga $\R\enn_\Lambda(S)$.
	
	\begin{defn}
		Let $E^\mathrm{ep}$ be the subspace of the dga $\enn_\Lambda(\tilde S)$ consisting of those maps of degree at least 2 which commute with $g_2$. We call such a map an \textbf{eventually periodic} map.
		\end{defn}
	
	\begin{lem}
		An eventually periodic map $f\in E^\mathrm{ep}$ is given by the formula
			$$f= \begin{tikzcd}[sep=large]	S^{j}\ar[d,swap,"f_0"]&S^{j+1}\ar[l,swap,"d_{j}"]\ar[d,swap,"f_1"]&S^{j+2}\ar[l,swap,"d_{j+1}"]\ar[d,swap,"f_2"]&S^{j+3}\ar[l,swap,"d_{j+2}"]\ar[d,swap,"f_3"]&S^{j+4}\ar[l,swap,"d_{j+3}"]\ar[d,swap,"f_2"]&\cdots\ar[l,swap,"d_{j+4}"]\\S^0&S^1\ar[l,"d_0"]&S^2\ar[l,"d_1"]&S^3\ar[l,"d_2"]&S^4\ar[l,"d_3"]&\cdots\ar[l,"d_4"]\end{tikzcd}$$where $f_i=f_{i+2}$ for $i\geq 2$, and $f_0=\left( 0 \;\; -1\right)f_2$ and $f_1=\stbtm{0}{b}{-1}{0}f_3$.
		\end{lem}
	\begin{proof}
		Compute $[f,g_2]=f\circ g_2 - g_2 \circ f$ and set it to zero.
		\end{proof}
	In particular, $f$ is determined by the pair $(f_2, f_3)$, and any such pair of maps defines an eventually periodic map.
	\begin{defn}
		Let $f$ be an eventually periodic map of a given degree. Since $f$ is determined by its components $f_2$ and $f_3$, we use the notation $f_2 \vert f_3$ to specify $f$ uniquely.
		\end{defn}
	\begin{defn}
		Let $f\in \enn_\Lambda(\tilde S)$ be a map of degree $\geq 2$ satisfying $f_i=f_{i+2}$ for $i\geq N$ for some natural $N\geq 2$. The \textbf{periodicisation} of $f$ is the map $f^{\mathrm{ep}} \in E^{\mathrm{ep}}$ of the same degree as $f$ defined by the formula $$f^{\mathrm{ep}}\coloneqq\begin{cases}f_N\vert f_{N+1}&\text{if }N \text{ is even} \\ f_{N+1}\vert f_{N}& \text{if }N \text{ is odd}\end{cases}.$$
		\end{defn}In other words, we go up the resolution until $f$ becomes $2$-periodic, and then extend $f$ back down to an eventually periodic map. In particular if $f\in E^{\mathrm{ep}}$, then $f^{\mathrm{ep}}=f$. Note that $f^{\mathrm{ep}}$ agrees with $f$ in all degrees $\geq N$.
	\begin{lem}
		The complex $E^\mathrm{ep}$ is a nonunital dga, and the inclusion $\iota:E^\mathrm{ep}\into \enn_\Lambda(\tilde S)$ is a dga map that induces isomorphisms on cohomology in degrees $> 2$ and a surjection on cohomology in degrees $\geq 2$.
		\end{lem}
	\begin{proof}
		The fact that $\iota$ is an inclusion of nonunital dgas is not hard to see. One can verify that the $g_j$ are eventually periodic, and since they generate the cohomology of $\ext^*_\Lambda(S,S)$, the map $\iota$ must be a quasi-surjection in degrees $\geq 2$. To see that $\iota$ is a quasi-injection in degrees $>2$, take an $h\in \enn_\Lambda(\tilde S)$ of degree $\geq 2$, and assume that $dh\in E^\mathrm{ep}$. We need to find an $l\in E^\mathrm{ep}$ with $dh=dl$. Because $dh\in E^\mathrm{ep}$, $h$ must be $2$-periodic in high degrees. Since $h^\mathrm{ep}\in E^\mathrm{ep}$, we have $dh^\mathrm{ep}\in E^\mathrm{ep}$, and so $dh-dh^\mathrm{ep}\in E^\mathrm{ep}$. But $h$ agrees with $h^\mathrm{ep}$ in high degrees, and so $dh-dh^\mathrm{ep}$ is zero in high degrees. So $dh-dh^\mathrm{ep}=0$.
	\end{proof}
In particular, any map of degree at least $3$ in $\enn_\Lambda(\tilde S)$  is homotopic to an eventually periodic map. We use this to assist us in our Massey product computation.
	\begin{prop}Suppose that $n>2$. Then the Massey product $\langle y,\ldots, y\rangle_n$ is nontrivial.
	\end{prop}
	\begin{proof}
		This is rather involved notationally but ultimately straightforward. We in fact show that $(-1)^nx^{n+1}$ is an element of $\langle y,\ldots, y\rangle_n$. We are going to proceed by setting $e_1\coloneqq g_3$ and inductively finding $e_i$ such that $de_i=e_1 e_{i-1}+\cdots+e_{i-1}e_1$. Note that we will require $de_2 =g_3^2$. For $2\leq i\leq n$, define a degree $2i-1$ eventually periodic map $\nu_i$ by the formula$$\nu_i\coloneqq \tbtm{\beta^{n-i}}{0}{t}{-\sigma}\mid\tbtm{\beta}{0}{t}{-\sigma^{n-i}}$$
		The $\nu_i$ will satisfy some simple relations, but we will need to keep track of the degrees of our maps. Unfortunately this makes things notationally messy. If $w=w_2\vert w_3$ is an eventually periodic map of degree $j$, then we denote by $w\{l\}$ the eventually periodic map of degree $j+l$ given by the formula$$w\{l\}\coloneqq\begin{cases}w_2\vert w_3&\text{if }l \text{ is even} \\ w_3\vert w_2& \text{if }l \text{ is odd}\end{cases}.$$In other words, $w\{l\}$ is $w$ but viewed as a map of a different degree. One can check that the following hold:
		
					\begin{enumerate}
			\item $d \nu_i = \tbtm{\beta^{n-i+1}}{0}{0}{\sigma^{n-i+1}} \mid \tbtm{\beta^{n-i+1}}{0}{0}{\sigma^{n-i+1}}$.
			\item $\nu_i\nu_j = d(\nu_i\{2j-2\} + \nu_j\{2i-2\})$.
			\item If $i<n$ then $g_3\nu_i+\nu_ig_3 =d( -\nu_{i+1}-\nu_2\{2i-1\})$.
			\item If $i<n$ then $\nu_i \simeq 0$.
		\end{enumerate}
Observe that $d \nu_3=g_3^2$. So we set $e_1\coloneqq g_3$, and we want to inductively find $e_i$ such that $d e_i=e_1 e_{i-1}+\cdots+e_{i-1}e_1$, starting with $e_2=\nu_3$. We prove by induction that for $2\leq i<n$ there exist maps $e_i$ of degree $2i+1$ such that: \begin{enumerate}
				\item$e_i$ is a $\Z$-linear combination of the maps $\nu_{i+1},\ldots,\nu_2\{2i-1\}$, and the coefficient of $\nu_{i+1}$ is $(-1)^i$.
				\item $d e_i= e_1 e_{i-1}+\cdots+e_{i-1}e_1$.
			\end{enumerate}

			The idea of the induction is simple; we just expand out each expression $e_1 e_{i-1}+\cdots+e_{i-1}e_1$ and `integrate term-by-term'. The hard part is keeping track of all the indices. The base case is clear; we may take $e_2\coloneqq \nu_3$ as above. For the induction step, suppose that for all $j<i$, all $e_j$ are defined and have the two properties above. We wish to construct $e_i$. For $j\geq 2$ write $$e_j=\sum_{r=2}^{j+1}\lambda_r^j\nu_r\{a^j_r\}$$with $\lambda_{j+1}^j=(-1)^j$ and $a^j_r=2(j-r)+2$. Then it is clear that for $1<j,k<i$, we have the identity $$e_je_k=\sum_{r=2}^{j+1}\sum_{q=2}^{k+1}\lambda_r^j\lambda_q^k\nu_r\nu_q\{a^j_r+a^k_q\}$$Hence, if we set $$m_{jk}\coloneqq \sum_{r=2}^{j+1}\sum_{q=2}^{k+1}\lambda_r^j\lambda_q^k(\nu_r\{a^j_r+2k\}+\nu_q\{a^k_q+2j\})$$we see that $d m_{jk}= e_je_k$. Observe that $m_{jk}$ is a map of degree $2(j+k)+1$. Moreover we have $$g_3e_{i-1}+e_{i-1}g_3=\sum_{r=2}^{i}\lambda_r^{i-1}(g_3\nu_r+\nu_rg_3)\{a^{i-1}_r\}$$So if we set $$m\coloneqq -\sum_{r=2}^{i}\lambda_r^{i-1}(\nu_{r+1}+\nu_2\{2r-1\})\{a^{i-1}_r\}$$we see that $m$ is a map of degree $2i+1$ with $d m=e_1e_{i-1}+e_{i-1}e_1$. Thus if we set $$e_i\coloneqq m + m_{2(i-2)}+\cdots+m_{(i-2)2}$$ we see that by construction, $e_i$ satisfies condition $2$. Clearly $e_i$ is a is a $\Z$-linear combination of $\nu_{i+1},\ldots,\nu_2\{2i-1\}$. So we just need to check what the coefficient of $\nu_{i+1}$ in $e_i$ is. It is easy to see that this coefficient is $-\lambda^{i-1}_i$ which by the induction hypothesis is $-(-1)^{i-1}=(-1)^i$. Hence $e_i$ satisfies condition $1.$ as well.
\p
			We are almost done. By \ref{masseylemma}, one element of the $n$-fold Massey product $\langle g_3,\ldots,g_3\rangle$ is given by $[e_1 e_{n-1}+\cdots+e_{n-1}e_1]$. So it suffices to prove that $e_1 e_{n-1}+\cdots+e_{n-1}e_1 \not\simeq 0$. We see that $e_je_k\simeq 0$ holds as long as $1<j,k<n$, so that we have $e_1 e_{n-1}+\cdots+e_{n-1}e_1\simeq e_1 e_{n-1}+e_{n-1}e_1$. Observe also that $e_1\nu_j+\nu_je_1\simeq0$ holds if $2\leq j<n$. Hence we see that we have a homotopy $e_1e_{n-1}+e_{n-1}e_1\simeq (-1)^{n-1}(e_1\nu_n+\nu_ne_1)$. It is easy to check that $e_1\nu_n+\nu_ne_1 \simeq -g_{2n+2}$. Hence $e_1 e_{n-1}+\cdots+e_{n-1}e_1 \simeq (-1)^ng_{2n+2}\not\simeq 0$.
\end{proof}
	\begin{cor}
		When $n>2$, $\R\enn_\Lambda(S)$ is not a formal dga.
	\end{cor}
	
	\begin{prop}\label{surfeng2}
		Let $n>2$. Then $\R\enn_\Lambda(S)$ is quasi-isomorphic to the strictly unital minimal $A_\infty$-algebra with underlying graded algebra $k[x,y]$, with $x$ in degree 2 and $y$ in degree 3, with the only nontrivial higher multiplications being $m_n(x^{b_1}y,\ldots,x^{b_l}y)=x^{n+1+b_1+\cdots +b_n}$.
	\end{prop}
	\begin{proof}
		We employ the usual trick of using the Adams grading on the resolution $\tilde S$ to rule out most higher multiplications (see \ref{adamsrmk}). Observe that in the secondary grading on the resolution, $x$ has weight $2n$, and $y$ has weight $2n+2$. Appealing to the graded version of Merkulov's construction, one can consider the higher multiplication $m_{r+l}(x^{a_1},\ldots,x^{a_r},x^{b_1}y,\ldots,x^{b_l}y)$, which must be of degree $2-r+2a+2b+2l$ and weight $2na+2nb+2(n+1)l$, where we write $a=a_1+\cdots+a_r$ and $b=b_1+\cdots+b_r$. Via casework on the parity of $r$, one can see that if $r+l>2$, the only way for this to be nonzero is when we are looking at a product of the form $m_{n}(x^{b_1}y,\ldots,x^{b_n}y)=\lambda x^{1+b+n}$, where $\lambda$ depends on the $b_i$. Consideration of the Stasheff identity $\mathrm{St}_{n+1}$ with inputs of the form $x^{b_1}y\otimes\cdots\otimes x^{b_i}y\otimes x^m \otimes x^{b_{i+1}}y\otimes\cdots\otimes x^{b_n}y$ shows that the higher multiplications $m_n$ are $x$-linear, in the sense that $m_n(x^{b_1}y,\ldots, x^{b_n}y)=x^bm_n(y,\ldots,y)$. So the only higher multiplication of interest is $m_n(y,\ldots,y)=\lambda_0 x^{1+n}$. Because $\R\enn_\Lambda(S)$ is not formal, we must have $\lambda_0\neq 0$, and rescaling if necessary one can fix $\lambda_0=1$.
	\end{proof}
	
	\begin{rmk}\label{surfextrmk}
		Alternately, one can say that $\R\enn_\Lambda(S)$ is quasi-isomorphic to the strictly unital minimal $A_\infty$-$k[x]$-algebra generated by $y$ subject to the relations $m_r(y,\ldots,y)=\delta_{r,n}x^{n+1}$. Note that this also holds for $n=2$.
	\end{rmk}
	
	\begin{prop}\label{surfdcacohom}Let $n\geq 2$.
		We have $H^*(\dca)\cong k[\eta,\zeta]$ where $\eta=x^*$ has degree $-2$ and $\zeta=y^*$ has degree $-1$.
	\end{prop}
	\begin{proof}
		For brevity put $E\coloneqq \R\enn_\Lambda(S)$ and recall that $\dca=E^!$ the Koszul dual. We filter $E$ by powers of $y$, use this to get a filtration on $\dca$, and consider the resulting spectral sequence. Let $W^0E=k[x]$ and let $W^1E=E$. One can check easily that this is a multiplicative filtration. We obtain $\mathrm{gr}_1^WE\cong k[y]$ and $\mathrm{gr}_0^WE\cong k[x]$. The filtration $W$ gives us a filtration on $E^!$, which we also call $W$, with associated graded $\mathrm{gr}^W(E^!)\cong(\mathrm{gr}^WE)^!$. Now, $\mathrm{gr}^WE\cong k[x,y]$ and so $\mathrm{gr}^W(E^!)\cong k[\eta,\zeta]$.
		
		\p Now we consider the spectral sequence $F$ associated to the filtration $W$ on $E^!$. It has $F_0$ page $F_0^{pq}=(\mathrm{gr}_p^W (E^!))^{p+q}\;\implies\;H^{p+q}(E^!)$. Writing out this page, we see that all differentials are trivial and so $F_0=F_\infty$. Hence we have $(\mathrm{gr}_p^W (E^!))^{p+q}=\mathrm{gr}_p^W H^{p+q}(E^!)$, and so forgetting the extra grading we get $H(E^!)\cong\mathrm{gr}^W(E^!)\cong k[\eta,\zeta]$ as required.
	\end{proof}
	
	\begin{rmk}
		Note that this holds for both $n=2$ and $n>2$. To see this in a more unified way, one can use the description of \ref{surfextrmk}.
	\end{rmk}	
	
	\begin{prop}\label{surfdcang2}
		Let $n\geq2$. Then the derived contraction algebra $\dca$ is quasi-isomorphic to the strictly unital minimal $A_\infty$-algebra with underlying graded algebra $ k[\eta,\zeta]$, where $\eta$ has degree $-2$, $\zeta$ has degree $-1$, and the only nontrivial higher multiplication is $$m_{n+1}(\eta^{b_1}\zeta,\ldots, \eta^{b_{n+1}}\zeta)=\eta^{b_1+\cdots+b_{n+1} + n}.$$
		
	\end{prop}
	\begin{proof}
		This is extremely similar to the proof of \ref{surfeng2}. A calculation with degree and weight yields that the only possible nontrivial higher multiplications are of the form	$$m_{n+1}(\eta^{b_1}\zeta,\ldots, \eta^{b_{n+1}}\zeta)=\lambda\eta^{b_1+\cdots+b_{n+1} + n}$$where $\lambda$ depends on the $b_i$. One gets $\eta$-linearity of the higher multiplications by considering the Stasheff identity $\mathrm{St}_{n+2}$. To see that $\dca$ is not formal, use that the Koszul dual of $\dca$ must be $\R\enn_\Lambda(S)$ again, but this does not agree with $k[\eta,\zeta]^!$. Hence we must have an equality $m_{n+1}(\zeta,\ldots, \zeta)=\lambda_0\eta^{ n}$ for some $\lambda_0\neq 0$, and one can choose $\lambda_0=1$.
	\end{proof}

	\begin{rmk}
		Again, this applies for both $n=2$ and $n>2$, and one can equivalently describe $\dca$ as the strictly unital minimal $A_\infty$-$k[\eta]$-algebra generated by $\zeta$ subject to the relations $m_r(\zeta,\ldots,\zeta)=\delta_{r,n+1}\eta^{n}$.
	\end{rmk}
	
	\begin{rmk}
		All of the computations of this section are valid in all characteristics.
		\end{rmk}

	\chapter{The mutation-mutation autoequivalence}\label{mutnauto}
	
	In this chapter, we will study the mutation-mutation autoequivalence, which is a noncommutative generalisation of the Bridgeland--Chen flop-flop autoequivalence. Our main theorem is \ref{mutncontrol}, which is a generalisation of Donovan and Wemyss's result that the contraction algebra of a threefold flop controls the mutation-mutation autoequivalence \cite[5.10]{DWncdf}. In particular, we show that the truncation $A_\mm\coloneqq \tau_{\geq -1}(\dca)$ controls the mutation-mutation autoequivalence in more general settings, via noncommutative twists. This both recovers and extends Donovan and Wemyss's result. In the first couple of sections we will set up the theory. We will do some computations, and show that mutation respects the recollement of \ref{recoll}, which we will use to obtain some results on t-structures analogous to those of Bridgeland \cite{bridgeland}. We will in fact show that the mutation-mutation autoequivalence, when restricted to the derived category $D(\dca)$, is simply the shift [-2] (\ref{mmshift}). In the hypersurface setting, this will be enough since one can use arguments involving the periodicity element $\eta$ to interchange shifts and truncation.
	\section{Singular Calabi--Yau rings and modifying modules}
	Given a reasonable commutative ring $R$ and a reasonable $R$-module $V$, one can consider the ring $A=\enn_R(V)$ as a sort of noncommutative partial resolution of $R$. One would like to be able to `mutate' $A$ into a new ring $A^+=\enn_R(V^+)$, and obtain a derived equivalence between $A$ and $A^+$. In this section we follow Iyama--Reiten \cite{iyamareiten} and Iyama--Wemyss \cite{iwmaxmod} to provide rigorous definitions of `reasonable'.
	\begin{defn}\label{scydefn}
		Let $R$ be a commutative $k$-algebra. Say that $R$ is \textbf{singular Calabi--Yau} (or just \textbf{sCY}) if the three following conditions are satisfied:
		\begin{enumerate}
			\item $R$ is Gorenstein.
			\item $R$ has finite Krull dimension $d$.
			\item $R$ is equicodimensional; i.e.\ all of its maximal ideals have the same height. This is equivalent to specifying that $\dim R_\mathfrak{m} =d$ for all $\mathfrak m \subseteq R$ maximal.
		\end{enumerate}
	\end{defn}
	\begin{rmk}
		This is a special case of Iyama and Reiten's definition \cite[\S3]{iyamareiten} for noncommutative rings; see \cite[3.10]{iyamareiten} for the proof of equivalence. In \cite{iyamareiten} this condition is called $d$-CY$^-$, and in \cite{iwmaxmod} it is called $d$-sCY. This is because $R$ is $d$-sCY if and only if a Calabi--Yau type condition $\hom_{D(R)}(X,Y[d])\cong D_M\hom_{D(R)}(Y,X)$ holds for certain $X,Y \in D^b(R)$, where $D_M$ denotes the Matlis dual.
	\end{rmk}
	A typical example of a sCY ring is a local complete intersection (l.c.i.) domain, or a localisation or completion thereof:
	\begin{lem}\label{lciscy}Let $R=k[x_1,\ldots, x_n]/I$ be a l.c.i.\ domain and $\mathfrak m \subseteq R$ be a maximal ideal. Then all of $R$, $R_{\mathfrak m}$ and $\hat R_{\mathfrak m}$ are sCY.
	\end{lem}
	\begin{proof}
		The ring $R$ is Gorenstein because it is a l.c.i.\ domain, equicodimensional because it is an affine domain \cite[13.4]{eisenbud}, and clearly of finite Krull dimension. Hence $R$ is sCY. The localisation $R_{\mathfrak m}$ is sCY by \cite[3.1(3)]{iyamareiten} and $\hat R_{\mathfrak m}$ is sCY by the proof of \cite[3.1(4)]{iyamareiten}.
	\end{proof}
	\begin{defn}[{\cite[4.1]{iwmaxmod}}]\label{modifiyingdefinition}
		Let $R$ be a sCY ring and $V$ a reflexive $R$-module. Say that $V$ is \textbf{modifying} if $\enn_R(V)$ is a Cohen--Macaulay $R$-module.
	\end{defn}

	\begin{prop}[{\cite[5.12(1)]{iwmaxmod}}]\label{modext}
		Let $R$ be a sCY ring of dimension $d$ with isolated singularities. Let $V$ be a Cohen--Macaulay $R$-module. Then $V$ is modifying if and only if $\ext^i_R(V,V)$ vanishes for all $1\leq i \leq d-2$.
	\end{prop}
	\begin{cor}\label{twoorthreecor}
	Let $R$ be a sCY ring with isolated singularities and let $V$ be a Cohen--Macaulay $R$-module. \begin{enumerate}
		\item[\emph{1.}] If $R$ is a surface, then $V$ is modifying.
		\item[\emph{2.}] If $R$ is a threefold, then $V$ is modifying if and only if it is rigid (i.e.\ $\ext_R^1(V,V)\cong 0$).
	\end{enumerate}
\end{cor}
	\begin{lem}\label{mcmsummod}
		Let $R$ be a sCY ring with isolated singularities. Let $M$ be a modifying $R$-module. If $M$ is MCM then $R\oplus M$ is modifying. 
	\end{lem}
	\begin{proof}
		Let $d$ be the Krull dimension of $R$. Since $R\oplus M$ is CM then by \ref{modext}, we show that the Ext group $\ext^i_R(R\oplus M, R\oplus M)$ vanishes for all $1\leq i \leq d-2$. Because $R$ is projective, for all $i>0$ we have$$\ext^i_R(R\oplus M, R\oplus M)\cong \ext^i_R(M, R)\oplus\ext^i_R(M, M).$$ Because $M$ is MCM, the Ext group $\ext^i_R(M,R)$ vanishes for all $i> 0$. Because $M$ is modifying, by \ref{modext} the Ext group $\ext^i_R(M, M)$ vanishes for all $1\leq i \leq d-2$. So $\ext^i_R(R\oplus M, R\oplus M)$ vanishes for all $1\leq i \leq d-2$, as required.
	\end{proof}

	\section{Mutation}\label{mutsn}
	In this section, we fix a complete local isolated hypersurface singularity $R$ of dimension at least 2. Note that $R$ is a sCY ring by \ref{lciscy}. Moreover, $R$ is normal by Serre's criterion.
	\begin{rmk}
		Everything we discuss in this section will still work \textit{mutatis mutandis} if $R$ is assumed to be any sCY ring of dimension at least 2 with isolated singularities. We choose $R$ to be a complete local hypersurface singularity in order to simplify notation when dealing with syzygies. In the more general case, one needs to distinguish between $\Omega$ and $\Omega^{-1}$.
	\end{rmk}
	Below, the mutation of a modifying module will be its syzygy. Syzygies are normally defined up to free summands; we define them here to preserve the number of free summands of a MCM module.
	\begin{defn}
		Let $M$ be a MCM $R$-module with no free summands. Take a minimal free resolution $\cdots \xrightarrow{d_1} F_1 \xrightarrow{d_0} F_0$ of $M$. The \textbf{syzygy} of $M$ is the module $\Omega M\coloneqq  \ker(d_0)$.
	\end{defn}
\begin{rmk}
	Note that this is a special case of the weaker definition used in \ref{weaksyz}.
	\end{rmk}
	\begin{defn}
		Let $M$ be a MCM $R$-module, and write $M=F\oplus M'$ where $F$ is free and $M'$ has no free summands. Put $\Omega M\coloneqq F\oplus \Omega M'$. 
	\end{defn}
	It is easy to see that for any MCM $R$-module $M$ we have a short exact sequence 
	\begin{equation}\label{syz}
	0 \to \Omega M \to R^m \to M \to 0
	\end{equation}
	for some $m \in \N$ depending on $M$. Moreover, $\Omega M$ has the same number of free summands as $M$. Because $R$ is a hypersurface singularity, $\Omega$ is 2-periodic by \ref{ebudper}. In particular, one also has a short exact sequence
	\begin{equation}\label{cosyz}
	0 \to M \to R^m \to \Omega M \to 0.
	\end{equation}
	\begin{rmk}
		As in \ref{3dref}, one may harmlessly assume that $R$ is a surface or a threefold, in which case \ref{twoorthreecor} gives easily checked criteria for when a general MCM module is modifying.
	\end{rmk}
	Fix a MCM modifying $R$-module $M$ with no free summands. Put $V\coloneqq R\oplus M$ and \linebreak  $A\coloneqq  \enn_R(V)$. By construction, $A$ comes with an idempotent $e=\id_R$ with $eAe\cong R$. By \ref{mcmsummod}, the $R$-module $V$ is modifying. Add copies of $R$ to (\ref{cosyz}) to get a short exact sequence
	\begin{equation}\label{vcosyz}
	0 \to V \to R^l \to \Omega V \to 0.
	\end{equation}
	Apply $\R\hom_R(V,-)$ to (\ref{vcosyz}) and take cohomology to obtain a long exact sequence of $A$-modules 
	\begin{equation}0\to \hom_R(V,V) \to \hom_R(V,R^l) \to \hom_R(V,\Omega V) \to \ext^1_R(V,V) \to \ext^1_R(V, R^l) \to \cdots.
	\end{equation}\label{vlong}Since $M$ is MCM, the $\ext^1_R(V, R^l)$ term vanishes, and we obtain an exact sequence \begin{equation}0\to \hom_R(V,V) \to \hom_R(V,R^l) \to \hom_R(V,\Omega V) \to \ext^1_R(V,V) \to 0 
	\end{equation}\label{vlong2} of right $A$-modules. Set $$T_A\coloneqq \mathrm{coker}\left(\hom_R(V,V) \to \hom_R(V,R^l)\right)\cong \ker\left(\hom_R(V,\Omega V) \to \ext^1_R(V,V)\right).$$
	\begin{rmk}Note that if $M$ was rigid, then so is $V$, and we obtain $T_A\coloneqq \hom_{R}( V, \Omega V)$. In \cite{iwmaxmod}, $\Omega V$ is denoted either $\mu_R^+(V)$ or $\mu_R^-(V)$ (the two agree since $\Omega\cong\Omega^{-1}$).
	\end{rmk}Note that the right $A$-module $T_A$ has a projective summand isomorphic to $\hom_R(V,R)$, and hence the ring $\enn_{A}(T_A)$ has an idempotent $\id_{\hom_R(V,R)}$.
	\begin{thm}[Iyama--Wemyss]\label{mtilt}Put $B\coloneqq \enn_{R}(\Omega V)$ and $e\coloneqq \id_R \in B$. Put\linebreak  $B'\coloneqq \enn_A(T_A)$ and $e'\coloneqq \id_{\hom_R(V,R)} \in B'$.
		\begin{enumerate}
			\item[\emph{1.}] There is an isomorphism of $R$-algebras $B\cong B'$ that restricts to a ring isomorphism $eBe \cong e'B'e'\cong R$.
			\item[\emph{2.}] The map $\upmu_A\coloneqq \R\hom_A(T_A,-):D(A) \to D(B)$ is an equivalence. We call $\upmu_A$ the \textbf{mutation equivalence}.
		\end{enumerate}
	\end{thm}
	\begin{proof}This is essentially \cite[6.8]{iwmaxmod}. Note that because $V$ has a free summand, it is a generator. Iyama and Wemyss prove that $T_A$ is a tilting module, and hence induces a derived equivalence $D^b(A) \to D^b(B)$, but this can be promoted to an equivalence $D(A) \to D(B)$ using Happel's theorem \cite{happel}.
	\end{proof}
	Starting from $B$, one can repeat the above constructions to obtain a tilting $B$-module $T_B$. By the above arguments, one obtains a derived equivalence $$\upmu_B\coloneqq \R\hom_B(T_B,-):D(B) \to D(\enn_R(\Omega^2 V)).$$However, since $\Omega^2 V \cong V$, there is an $R$-linear isomorphism $\enn_R(\Omega^2 V) \cong A$. By composition one gets an autoequivalence $\mm\coloneqq \upmu_B\circ\upmu_A: D(A)\to D(A)$. Writing $_BT_A$ for $T_A$ and $_AT_B$ for $T_B$, the (derived) hom-tensor adjunction gives an isomorphism $\mm\cong \R\hom_A(_AT_B\lot_B{_BT_A},-)$.
	\begin{defn}\label{immdefn}
		We call $\mm: D(A) \to D(A)$ the \textbf{mutation-mutation autoequivalence}. Write $I_\mm\coloneqq _AT_B\lot_B{_BT_A}$, so that $\mm$ is represented by the $A\text{-}A$-bimodule $I_\mm$.
	\end{defn}
	\begin{rmk}
		By construction, $_AT_B$ has a two-term $B$-projective resolution (i.e.\ ${\mathrm{pd}_B(_AT_B)\leq 1}$) so it follows that $I_\mm$ has cohomology only in degrees 0 and -1.
	\end{rmk}
	\begin{rmk}\label{mutnchain}
		In general, not assuming that $R$ is a hypersurface singularity, one obtains an a priori doubly infinite sequence of derived equivalences $$\cdots\xrightarrow{\cong} D(\enn_R(\Omega^i V)) \xrightarrow{\cong} D(\enn_R(\Omega^{i+1} V))\xrightarrow{\cong}\cdots$$
	\end{rmk}
	\begin{rmk}
		If $M$ is rigid then so is $\Omega M$, and it follows that $$I_\mm\simeq \hom_R(V,\Omega V)\lot_B\hom_R(\Omega V, V).$$In fact, in this situation one has $I_\mm\simeq AeA$ by results of Donovan and Wemyss \cite[5.10]{DWncdf} which we review later in \ref{dwbimodmap}.
	\end{rmk}
We finish with some results on the structure of mutation.
\begin{lem}\label{aconmutlemma}
	Let $X\in D(A)$ be an $A/AeA$-module placed in degree zero. Then $\upmu_A(X)$ has cohomology concentrated in degree one, with $H^1(\upmu_A(X))\cong \ext^1_A({_BT_A},X)$. Hence there is a quasi-isomorphism of $B$-modules $\upmu_A(X)\simeq \ext^1_A({_BT_A},X)[-1]$.
\end{lem}
\begin{proof}Because $H^i(\upmu_A(X))\cong \ext^i_A({_BT_A},X)$ by definition, we only need to show that $H(\upmu_A(X))$ is concentrated in degree one. By construction, ${_BT_A}$ comes with an $A$-projective resolution $$0 \to \hom_R(V,V) \to \hom_R(V,R^l) \to {_BT_A} \to 0.$$Note that $\hom_R(V,R)\cong eA$. Because $X$ is an object of the subcategory $D(\dq)$, there are no maps from $eA$ to $X$.
\end{proof}

\begin{lem}\label{shiftlem}
	Let $X \in D(\dq)$. Then for all $q\in \Z$ there are $B$-module isomorphisms $$H^{1+q}(\upmu_A(X))\cong H^1(\upmu_A(H^q(X))).$$
\end{lem}
\begin{proof}
	This follows from the previous lemma together with a spectral sequence argument. Take a $B\text{-}A$-bimodule quasi-isomorphism $P \to {_BT_A}$ that is a projective resolution of right $A$-modules. Consider the double complex of $B$-modules $E_0^{pq}\coloneqq \hom_A(P^{-p}, X^q)$ whose total product complex is $\mathrm{Tot}^{\Pi}(E_0)\cong \upmu_A(X)$. We may regard $E_0$, equipped with the differential of $X$, as the zeroth page of a (cohomological) spectral sequence $E$. Because $P$ is zero in positive degrees, it follows by the discussion after \cite[5.6.1]{weibel} that the spectral sequence $E$ weakly converges to $H^n(\mathrm{Tot}^{\Pi}(E_0)) \cong H^n(\upmu_A(X))$. It is easy to see that $E_2^{pq}\cong H^{p}(\upmu_A(H^q(X)))$. By \ref{aconmutlemma}, this module is zero unless $p=1$, where it is $\ext^1_A(_BT_A, H^q(X))$. In other words, the spectral sequence collapses at the $E_2$ page. A weakly convergent spectral sequence which collapses must converge, and we get the desired isomorphisms.
\end{proof}
\begin{cor}\label{shiftcor}
	If $X \in D(\dq)$ satisfies $H^q(X)\cong 0$, then $H^{1+q}(\upmu_A(X))\cong 0$.
\end{cor}

	\section{Recollements}
	Throughout this section we will use the following setup:
	\begin{setup}\label{mutsetup}
	Let $R$ be a complete local isolated hypersurface singularity of dimension at least 2, $M$ a MCM modifying $R$-module with no free summands, $V\coloneqq R\oplus M$, $A\coloneqq \enn_R(V)$, $B\coloneqq \enn_R(\Omega V)$, and $e=\id_R$ (we use the same notation for $\id_R \in A$ and $\id_R \in B$). 
		\end{setup}We aim to prove that mutation respects the recollement of \ref{recoll}. The following lemma will be useful to us:
\begin{lem}\label{bimodlem}
	Let $_BT_A$ and $_AT_B$ be the tilting modules constructed in \S\ref{mutsn}. Then:
	\begin{itemize}
		\item $_BT_Ae\cong Be$ as $B$-$R$-bimodules.
		\item $_AT_Be\cong Ae$ as $A$-$R$-bimodules.
		\item $e_BT_A\cong eA$ as $R$-$A$-bimodules.
		\item $e_AT_B\cong eB$ as $R$-$B$-bimodules.
		\end{itemize}
	\end{lem}
\begin{proof}
For the first statement, note that given an $A$-module of the form $\hom_R(V,X)$ then one has $\hom_R(V,X)e \cong X$ as $R$-modules. Moreover, if $X$ itself was a left $A$-module then this is an isomorphism of $A$-$R$-bimodules. By definition, $_BT_A$ is the cokernel of the map $\hom_R(V,V) \to \hom_R(V,R^l)$. Because the functor $Y \mapsto Ye $ is exact, we see that we have $_BT_Ae\cong \mathrm{coker}(V \to R^l)\cong \Omega V \cong Be$ as $B\text{-}R$-bimodules. The second statement is completely analogous to the first. For the third statement, observe that \ref{mtilt} along with exactness of multiplication by $e$ gives an isomorphism $e_BT_A\cong \hom_R(V,R)\cong eA$ of $B$-$A$-bimodules. The fourth statement is completely analogous to the third.
	\end{proof}

Let $\upmu_A$ be the mutation equivalence. Recall from \ref{recoll} the existence of the recollement  $D(\dq)\recol D(A)\recol D(R)$.
\begin{defn}
	Let $C$, $C'$ be two dgas. Say that $C$ and $C'$ are \textbf{derived Morita equivalent} if there is a $C'\text{-}C$-bimodule $P$ such that $\R\hom_C(P,-):D(C)\to D(C')$ is a derived equivalence. Note that in this case the inverse is necessarily given by the functor $-\lot_{C'}P$.
\end{defn}
	
	\begin{prop}\label{recolhalf}Put $\upmu_{\mathbb{L}}\coloneqq \R\hom_{\dq}(\dqb\lot_B{_BT_A}\lot_A \dq,-) $. Then the diagram $$\begin{tikzcd}[column sep=huge]
		D(\dq)\ar[swap, dd,"\upmu_{\mathbb{L}} "] \ar[r,"i_*=i_!"]& D(A)\ar[l,bend left=25,"i^!"']\ar[l,bend right=25,"i^*"']\ar[r,"j^!=j^*"]\ar[swap, dd,"\upmu_A"] & D(R)\ar[l,bend left=25,"j_*"']\ar[l,bend right=25,"j_!"']\ar[swap, dd,"\id"]
		\\  & &
		\\ D(\dqb) \ar[r,"i_*=i_!"]& D(B)\ar[l,bend left=25,"i^!"']\ar[l,bend right=25,"i^*"']\ar[r,"j^!=j^*"] & D(R)\ar[l,bend left=25,"j_*"']\ar[l,bend right=25,"j_!"']
		\end{tikzcd}$$is a morphism of recollement diagrams, with vertical maps equivalences. In particular, $\dq$ and $\dqb$ are derived Morita equivalent.
	\end{prop}
	\begin{proof}
	First we check that the three squares on the right-hand side commute. But this is not hard: it follows from \ref{bimodlem} and the fact that $\upmu_A^{-1}\cong -\lot_B{}_BT_A$ that we have isomorphisms $j^*_B \circ \upmu_A \cong j^*_A$, $\upmu_A \circ j_*^A \cong j_*^B$, and $\upmu_A \circ j_!^A \cong j_!^B$. Because morphisms of recollements are determined uniquely by one half (e.g.\ \cite[2.4]{kalck}), there is a unique (up to isomorphism) map $F:D(\dq) \to D(\dqb)$ fitting into a morphism of recollements with the righthand two, and since the righthand two are equivalences, so is $F$. Since the $i_*$ maps are fully faithful, $F$ is determined completely by $i_*F$: if $F'$ is any other functor such that $i_*F' \cong \upmu_A \circ i_*$, then $F'\cong F$. But one can check that the given functor satisfies this condition.
	\end{proof}
	\begin{prop}Put $I^\mathbb{L}_\mm\coloneqq \dq\lot_A{I_\mm}\lot_A \dq$ and $\mm_{\mathbb{L}}\coloneqq \R\hom_{\mathbb{L}}(I^\mathbb{L}_\mm,-) $. Then the diagram $$\begin{tikzcd}[column sep=huge]
		D(\dq)\ar[swap, dd," \mm_{\mathbb{L}} "] \ar[r,"i_*=i_!"]& D(A)\ar[l,bend left=25,"i^!"']\ar[l,bend right=25,"i^*"']\ar[r,"j^!=j^*"]\ar[swap, dd,"\mm"] & D(R)\ar[l,bend left=25,"j_*"']\ar[l,bend right=25,"j_!"']\ar[swap, dd,"\id"]
		\\  & &
		\\ D(\dq) \ar[r,"i_*=i_!"]& D(A)\ar[l,bend left=25,"i^!"']\ar[l,bend right=25,"i^*"']\ar[r,"j^!=j^*"] & D(R)\ar[l,bend left=25,"j_*"']\ar[l,bend right=25,"j_!"']
		\end{tikzcd}$$is a morphism of recollement diagrams, with vertical maps equivalences.
	\end{prop}
\begin{proof}
	One could show this via a proof similar to that of \ref{recolhalf}. We instead show that $\mm_{\mathbb{L}}$ is isomorphic to the composition of the functors $$D(\dq)\xrightarrow{\mu_{\mathbb{L}}} D(\dqb) \xrightarrow{\mu_{\mathbb{L}}} D(\dq)$$obtained by applying \ref{recolhalf} twice. It is easy to show that $\mu_{\mathbb{L}}\circ \mu_{\mathbb{L}}$ is represented by the object $$W\coloneqq\dq\lot_A{_AT_B}\lot_B \dqb\lot_B \dqb\lot_B{_BT_A}\lot_A \dq.$$So we want to show that $W$ is quasi-isomorphic as an $\dq$-bimodule to $I^\mathbb{L}_\mm$, which represents $\mm_{\mathbb{L}}$. It follows by considering the representing objects of both sides of the equation $$\R\hom_A({_BT_A},i_*) \simeq i_*i^!\R\hom_A({_BT_A},i_*)$$that ${_BT_A} \lot _A \dq \simeq \dqb\lot_B {_BT_A}\lot_A \dq$ as $B$-$A$-bimodules, and similarly for $\dq\lot_A{_AT_B}$. Hence, we have quasi-isomorphisms \begin{align*}
	W \simeq &\dq\lot_A{_AT_B}\lot_B {_BT_A}\lot_A \dq \\ \eqqcolon &\dq\lot_AI_{\mm}\lot_A \dq
	\\  \eqqcolon &I^\mathbb{L}_\mm
	\end{align*} as required.
\end{proof}

	We combine our results on recollements with some standard facts about t-structures; see Be{^^c4^^ad}linson--Bernstein--Deligne \cite{bbd} for the definition of a t-structure.
	\begin{prop}[\cite{hkmtstrs, amiotcluster, kelleryangmutn, kalckyang}]\label{tstr}
		Let $Z$ be a nonpositive dga. Then the derived category $D(Z)$ admits a $t$-structure $(D^{\leq 0}(Z), D^{\geq 0}(Z))$ where 
		\begin{align*}
		&D^{\leq 0}(Z)\coloneqq \{X\in D(Z): H^i(X)=0 \text{ \normalfont for } i>0 \}\phantom{.}
		\\ 	&D^{\geq 0}(Z)\coloneqq \{X\in D(Z): H^i(X)=0 \text{ \normalfont for } i<0 \}.
		\end{align*}
		Moreover, the inclusion $\cat{Mod}$-$H^0(Z) \into D(Z)$ is an equivalence onto the heart of this t-structure, with inverse given by taking zeroth cohomology.
		
	\end{prop}
	\begin{rmk}
		When $Z=\dq$ this is the restriction of the standard t-structure on $D(A)$.
	\end{rmk}
	
	\begin{prop}\label{mutntexact}{\normalfont(cf.\ Bridgeland \cite[4.7]{bridgeland}).}
		The shifted mutation functor $$ X \mapsto \upmu_{\mathbb{L}}(X)[1]:\quad D(\dq) \longrightarrow D(\dqb)$$ is a t-exact equivalence.
	\end{prop}
	\begin{proof}For brevity, write the functor under consideration as $G$. By construction, $G$ is an equivalence. By \ref{shiftcor}, if $X$ is concentrated in nonnegative degrees, then so is $G(X)$, and similarly for nonpositive degrees. In other words, $G$ is t-exact.
	\end{proof}
	Taking hearts one arrives at:
	\begin{cor}
		$A/AeA$ and $B/BeB$ are (classically) Morita equivalent.
	\end{cor}
In fact, one can do much better:
\begin{thm}[{Iyama--Wemyss \cite[6.20]{iwmaxmod}}]\label{aconisbcon}
With the setup as above, there is a ring isomorphism $A/AeA \cong B/BeB$.	
\end{thm}
\begin{rmk}
	Note that we use here the fact that $R$ is normal.
	\end{rmk}
Recall that given t-structures on the outer pieces of a recollement diagram, one can glue them to a new t-structure on the central piece \cite[1.4.10]{bbd}.
	\begin{defn}
		 Let $D$ be the t-structure on $D(\dq)$ constructed in \ref{tstr}. Let $\tau_A^p$ be the t-structure on $D(A)$ obtained by gluing $D[-p]$ (i.e.\ $D$ shifted so that the heart is in degree $p$) to the standard t-structure on $D(R)$. In particular, $\tau_A^0$ is the standard t-structure on $D(A)$. Write ${}^p\mathrm{Per}A$ for the heart of $\tau_A^p$, so that e.g.\ ${}^0\mathrm{Per}A\cong\cat{Mod}$-$A$. Call ${}^p\mathrm{Per}A$ the abelian category of \textbf{$p$-perverse sheaves} on $A$. 
	\end{defn}
	\begin{thm}[cf.\ Bridgeland {\cite[4.8]{bridgeland}}]
		Fix a natural number $p$. The mutation functor $\upmu_A:D(A)\to D(B)$ is $t$-exact for the t-structures $\tau^p_A$ and $\tau^{p+1}_B$. Mutation induces a chain of exact equivalences of abelian categories $$\cdots \to {}^p\mathrm{Per}A \to {}^{p+1}\mathrm{Per}B \to {}^{p+2}\mathrm{Per}A \to \cdots$$
	\end{thm}
	\begin{proof}
		$\upmu_A$ is t-exact because it is the gluing of two $t$-exact functors. Similarly, $\upmu_B$ is t-exact, and the chain of \ref{mutnchain} becomes a chain of t-exact equivalences. Passing to hearts gets us the second statement. 
	\end{proof}
\begin{rmk}
	Bridgeland's signs in {\cite[4.8]{bridgeland}} are wrong if one is considering the flop functor, and this error was corrected by Toda in {\cite[Appendix B]{todawidth}}. However, the flop functor is the inverse of our mutation functor \cite[7.18]{DWncdf}, and so our maps go in the `wrong' direction, and so our chain looks like Bridgeland's.
	\end{rmk}

\section{Bimodules and natural transformations}
Assume that we are in the situation of Setup \ref{mutsetup}. We aim to show in this section that there is a natural transformation $\id \to \mm$ compatible with the recollement; we do this by describing an appropriate morphism of representing objects. This section is inspired by a result of Donovan and Wemyss obtained in the threefold setting \cite[5.10]{DWncdf}, and indeed if $M$ is rigid then one can check that our results reduce to theirs. Recall the short exact sequence $$0 \to {_BT_A} \to \hom_R(V,\Omega V) \to \ext^1_R(V,V) \to 0$$of right $A$-modules, which exists by the definition of $_BT_A$. The left hand terms are both $B$-$A$-bimodules, and the left hand map is $B$-linear. Hence, one can put the structure of a $B$-$A$-bimodule on $\ext^1_R(V,V)$ making this short exact sequence into a sequence of $B$-$A$-bimodules. Similarly, $\ext^1_R(\Omega V,\Omega V)$ admits the structure of an $A$-$B$-bimodule making the short exact sequence $$0 \to {_AT_B} \to \hom_R(\Omega V,V) \to \ext^1_R(\Omega V,\Omega V) \to 0$$ into a sequence of $A$-$B$-bimodules. In particular, because $I_\mm\coloneqq {_A}T_B\lot_B{_BT_A}$ by definition, one obtains a map $I_\mm \to \hom_R(\Omega V, V)\lot_B \hom_R(V,\Omega V)$ in the derived category of $A$-bimodules. Composing with $H^0$, one obtains a map $I_\mm \to \hom_R(\Omega V, V)\otimes_B \hom_R(V,\Omega V)$ in the derived category of $A$-bimodules. There is an obvious $A$-bilinear composition map \begin{align*}
\hom_R(\Omega V, V)\otimes_B \hom_R(V,\Omega V) &\to \enn_R(V) = A \\
f \otimes g &\mapsto f \circ g
\end{align*}
and by composition we get a map $I_\mm \to A$. For ease of reference, we give this map a name:
\begin{defn}\label{unitmap}
The \textbf{unit map} is the map $I_\mm \to A$ in the derived category of $A$-bimodules given by the composition
$$I_\mm \to \hom_R(\Omega V, V)\lot_B \hom_R(V,\Omega V) \xrightarrow{H^0} \hom_R(\Omega V, V)\otimes_B \hom_R(V,\Omega V) \to A.$$
	\end{defn}
The idea behind the name is that on representable endofunctors of $D(A)$ this gives us a unit $\id\to\mm$.

\begin{lem}\label{aealem}
	The image of the unit map is contained within the submodule $AeA \into A$.
	\end{lem}

\begin{proof}	
	The composition map $$\hom_R(\Omega V, V)\otimes_B \hom_R(V,\Omega V) \to \enn_R(V)$$ descends to a map $$\underline{\hom}_R(\Omega V, V)\otimes_B \underline{\hom}_R(V,\Omega V) \to \underline{\enn}_R(V)\cong A/AeA.$$	Since $A_\con\cong A/AeA$, it suffices to show that if $f\otimes g$ is an element of the module $_AT_B \otimes_B {}_BT_A$, then their composition is zero in $\underline{\enn}_R(V)$. But because $\Omega $ is the shift functor of $\stab R$, this stable composition map is precisely the stable Ext pairing $$\underline{\ext}_R^{-1}(V,V) \otimes \underline{\ext}_R^{1}(V,V) \to \underline{\ext}_R^0(V,V)$$and because $\Omega^2=\id$, this is the same as the stable Ext pairing $$\underline{\ext}_R^{1}(V,V) \otimes \underline{\ext}_R^{1}(V,V) \to \underline{\ext}_R^2(V,V).$$By definition, there is a short exact sequence $$0 \to {}_BT_A \to \hom_R(V,\Omega V) \to \underline{\ext}^1_R(V,V) \to 0$$where we have used that stable Ext agrees with usual Ext in positive degrees. The right-hand map is the boundary map provided by the Snake Lemma, but this agrees with the projection ${\hom}_R(V,\Omega V)\to \underline{\hom}_R(V,\Omega V)\cong \underline{\ext}_R^{1}(V,V)$. Hence, $_BT_A$ is precisely the kernel of the projection map $\hom_R(V,\Omega V) \to \underline{\ext}^1_R(V,V)$ and hence the composition $$_AT_B \otimes_B {}_BT_A \to \underline{\ext}_R^{1}(V,V) \otimes \underline{\ext}_R^{1}(V,V) \to \underline{\enn}_R(V)$$ is zero, as required.
	\end{proof}

\begin{rmk}\label{dwbimodmap}
	If $V$ is rigid, then the proof of \ref{aealem} adapts to give an isomorphism\linebreak $\hom_R(\Omega V, V)\otimes_B \hom_R(V,\Omega V) \xrightarrow{\cong} AeA$, as in Donovan--Wemyss \cite[\S5]{DWncdf}. In fact, one then obtains a quasi-isomorphism $I_\mm \simeq AeA$, exactly as in \cite[5.10]{DWncdf}. In general, if $V$ is not rigid then the image of the composition map may not lie in $AeA$.
	\end{rmk}

\begin{thm}\label{unitrecol}
	For brevity, put $Q\coloneqq \dq$. The unit map induces a natural transformation $\id_{D(A)}\to\mm$ which descends to a natural transformation $\id_{D(Q)} \to \mm_{\mathbb{L}}$.
	\end{thm}
\begin{proof}
	It is clear that $I_\mm \to A$ induces a natural transformation $\id \to \mm$. Tensoring the unit map $I_\mm \to A$ with ${Q}$ on both sides gives a $Q$-bimodule map $I^\mathbb{L}_\mm \to {Q}\lot_A Q$. By \ref{homepi}, the natural map $Q\lot_A Q \to Q$ is a quasi-isomorphism of $Q$-bimodules, and so one obtains a bimodule map $I^\mathbb{L}_\mm \to Q$. But this gives a natural transformation $\id_{D(Q)} \to \mm_{\mathbb{L}}$ which must agree with the restriction of $\id \to \mm$.
	\end{proof}

	\section{Simple modules and deformations}
	Throughout this section we will use the following setup:
\begin{setup}\label{mutsetupp}
	Assume that we are in the situation of Setup \ref{mutsetup}. Assume furthermore that $A/AeA$ is Artinian local and that the dga $\dq$ is cohomologically locally finite.
	\end{setup}
In the geometric situations that we care about, the hypotheses of Setup \ref{mutsetupp} are always satisfied (\ref{aconlocal}, \ref{fdcohom}). Denote by $S_A$ the one-dimensional $A$-module $A/AeA / \text{rad}(A/AeA)$. Since it is naturally an $A/AeA$-module, we may regard it as an object of $D({\dq})$ concentrated in degree zero. Recalling that $A/AeA\cong B/BeB$ by \ref{aconisbcon}, we denote the analogous one-dimensional $B$-module by $S_B$.
	\begin{lem}\label{msissone}
		There is a quasi-isomorphism of $B$-modules $\upmu_A(S_A)\simeq S_B[-1]$.
	\end{lem}
	\begin{proof}
		A computation using \ref{aconmutlemma} shows that $\upmu_A(S_A)$ is a 1-dimensional object in $D(\dqb)$ concentrated in degree 1. Since it is hence a simple one-dimensional module over $B/BeB$, which is an Artinian local algebra by \ref{aconisbcon}, it must be a copy of $S_B$.
	\end{proof}
\begin{rmk}
	This relies crucially on $B/BeB$ being local.   
	\end{rmk}
	
	\begin{cor}[cf.\ {\cite[5.11]{DWncdf}}]\label{mmshiftssimple}
		$\mm(S_A)\simeq S_A[-2]$.
	\end{cor}
\begin{rmk}\label{dwabrmk}
	One can prove this directly without appeal to \ref{aconisbcon}. The proof of \ref{msissone} shows that $\mm(S_A)$ is a one-dimensional object in $D(\dq)$ concentrated in degree 2. Hence it is a shift of a one-dimensional simple module over $A/AeA$, which we have already assumed to be Artinian local.
	\end{rmk}
	
	By \ref{maindefmthm}, $\dq$ prorepresents the derived noncommutative framed deformations of $S_A$. More accurately, since it is not naturally a pro-Artinian dga it does not prorepresent, but it does at least determine the functor of framed deformations by \ref{dqqiclassdeterminesdefms}. Since one can regard $\dq$ in some sense as the universal prodeformation of $S_A$, which is invariant under derived equivalences by \ref{qeupdf}, one can deduce the following:
	\begin{thm}\label{dictionary}With the setup as above, \hfill
		\begin{enumerate}
			\item[\emph{1.}] The dgas ${\dq}$ and $\dqb$ are quasi-isomorphic over $k$.
			\item[\emph{2.}] The $B$-module $\upmu_A({\dq})$ is quasi-isomorphic to $\dqb[-1]$.
			\item[\emph{3.}]{\normalfont (cf.\ {\cite[5.9(1)]{DWncdf}}).} The $B$-module $\upmu_A(A/AeA)$ is quasi-isomorphic to $B/BeB[-1]$. 
		\end{enumerate}
	\end{thm}
	\begin{proof}
		Since $\upmu_A:D(A) \to D(B)$ is a standard equivalence, it can be enhanced to a quasi-equivalence of dg categories, and it follows that $\R\enn_A(S_A)\simeq\R\enn_B(\upmu_A(S_A))$ as dgas. By \ref{msissone} we hence have $\R\enn_A(S_A)\simeq \R\enn_B(S_B)$. Taking the Koszul dual of both sides and appealing to \ref{rendkd} gives the first claim. The second claim is precisely \ref{qeupdf}: because $\dqb$ is the universal prodeformation of $S_B$, it follows that $\dqb[-1]$ is the universal prodeformation of $S_B[-1]$. Note that this is a nontrivial theorem which requires the deformation-theoretic properties of the algebra $\dq$, and in particular the interpretation of the $A$-module $\dq$ as the universal prodeformation. Applying \ref{shiftlem} with $q=0$ now gets us an isomorphism $H^1(\upmu_A(A/AeA))\cong B/BeB$. But because $\upmu_A(A/AeA)$ is a module placed in degree one by \ref{aconmutlemma}, it follows that we have a quasi-isomorphism $\upmu_A(A/AeA)\simeq H^1(\upmu_A(A/AeA))[-1]$, which is the third claim.
	\end{proof}
	\begin{rmk}
		If one could prove \ref{msissone} without appeal to \ref{aconisbcon}, then the above provides a new proof of \ref{aconisbcon} by simply observing that $$A/AeA\cong H^0({\dq})\cong H^0(\dqb)\cong B/BeB$$follows immediately from 1.
	\end{rmk}
\begin{rmk}\label{mmdqidqt}
	As in \ref{dwabrmk}, one can deduce the existence of quasi-isomorphisms \linebreak $\mm(\dq)\simeq \dq[-2]$ and $\mm(A/AeA)\simeq A/AeA [-2]$ without appeal to \ref{aconisbcon}.
	\end{rmk}
	\begin{rmk}
		By working in the appropriate subcategories, we may promote the quasi-isomorphisms of 2.\ and 3.\ to quasi-isomorphisms of $\dqb$-modules, and that of 3. to an isomorphism of $B/BeB$-modules.
	\end{rmk}

	\section{Singularity categories}
	In this section, we track the unit map $\id \to \mm_{\mathbb{L}}$ across an equivalence to the singularity category, and show that it becomes an isomorphism there. Assume that we are in the setup of \ref{mutsetupp}. For brevity, we put $Q\coloneqq \dq$. Because $A/AeA$ is finite-dimensional, it follows that the category $D_{\mathrm{fg}}(Q)$ (as defined in \ref{dfgdefn}) agrees with the category $D_{\mathrm{fd}}(Q)$ on those modules with finite-dimensional total cohomology, and similarly the subcategory $\per_\mathrm{fg}(Q)$ agrees with the subcategory $\per_\mathrm{fd}(Q)$ of perfect modules with finite-dimensional total cohomology. We will use these facts frequently in the sequel.
	\begin{lem}\label{fdresplem}
		The functor $\mm_{\mathbb{L}}:D(Q)\to D(Q)$ respects $\per({Q})$, $D_\mathrm{fd}({Q})$, and $\per_\mathrm{fd}({Q})$.
	\end{lem}
	\begin{proof}
		Since $\mm_{\mathbb{L}}({Q})$ is perfect by \ref{dictionary}(2), and perfect modules are built out of $Q$ under cones and shifts, it follows that $\mm_{\mathbb{L}}$ preserves all perfect modules. Since $\mm$ sends $S_A$ to a finite-dimensional module by \ref{mmshiftssimple}, and all finite-dimensional modules are built out of $S_A$ under cones and shifts (because $A/AeA$ is Artinian local), it follows that $\mm$ preserves $D_\mathrm{fd}({Q})$ too. The third assertion is now clear.
	\end{proof}
	\begin{defn}
	Write $\mathcal{M}$ for $\thick_{D_{\mathrm{sg}}(R)}(M)$.
\end{defn}
Recall that the singularity functor of \ref{kymap} induces an equivalence $${\bar{\Sigma}}:\per({Q})/\per_\mathrm{fd}({Q}) \xrightarrow{\cong}  \mathcal M.$$
	\begin{defn}
		Let $\mm_{\mathrm{sg}}:\mathcal{M}\to \mathcal{M}$ be the autoequivalence defined by the commutative diagram of equivalences $$\begin{tikzcd} \per({Q})/\per_\mathrm{fd}({Q})\ar[d,"\mm_{\mathbb{L}}"] \ar[r,"{\bar{\Sigma}}"]& \mathcal M \ar[d, "\mm_{\mathrm{sg}}"] \\ \per({Q})/\per_\mathrm{fd}({Q})\ar[r,"{\bar{\Sigma}}"] & \mathcal M 
		\end{tikzcd}$$Observe that one gets a natural transformation $\id_{\mathcal M} \to \mm_{\mathrm{sg}}$.
	\end{defn}
	\begin{rmk}
		One can enhance $\mm_{\mathrm{sg}}$ to a dg functor, although we will not need this fact.
	\end{rmk}
	We finish with some technical observations which will give us control over $\mm_{\mathrm{sg}}$.
	\begin{lem}\label{unitelem}
		Tensoring the unit map $I_\mm \to A$ on the right with $Ae$ gives an $A-R$-bimodule quasi-isomorphism $I_\mm e \xrightarrow{\simeq} Ae$.
	\end{lem}
	\begin{proof}
		Since $_BT_Ae\cong Be$ as bimodules by \ref{bimodlem}, we see that $I_\mm e$ is bimodule quasi-isomorphic to $Ae$. So it suffices to show that the map $$_AT_B \otimes_B {}_BT_A \to \hom_R(\Omega V, V)\otimes_B \hom_R(V,\Omega V) \to A$$ becomes an isomorphism after tensoring with $Ae$. Observing that $\hom_R(V,\Omega V)e\cong \Omega V \cong Be$, we see that the induced map $_BT_Ae \to \hom_R(V,\Omega V)e$ is an isomorphism. Moreover, because $_AT_Be\cong Ae$ by \ref{bimodlem}, and because $\hom_R(\Omega V,V)e\cong Ae$, we see that $_AT_B \otimes_B {}_BT_Ae \to \hom_R(\Omega V, V)\otimes_B \hom_R(V,\Omega V)e$ is an isomorphism. Similarly, we see that the multiplication map $\hom_R(\Omega V, V)\otimes_B \hom_R(V,\Omega V)e \to Ae$ is an isomorphism.
	\end{proof}
	\begin{rmk}
		The same logic shows that the unit map induces bimodule quasi-isomorphisms $eI_\mm \xrightarrow{\simeq} eA$ and $eI_\mm e \xrightarrow{\simeq} R$.
	\end{rmk}
	
	\begin{prop}\label{mmsgisid}
		The natural transformation $\id_{\mathcal M} \to \mm_{\mathrm{sg}}$ is an isomorphism.
	\end{prop}
	\begin{proof}The idea is that all constructions made respect the recollement, which forces $\id \to \mm_{\mathrm{sg}}:{\mathcal M} \to {\mathcal M}$ to agree with the map induced by $\id \to \mm: D(R)\to D(R)$, which is an isomorphism. For the purposes of this proof, we introduce some notation and terminology. If $\psi:F\to G$ is a natural transformation of functors $F,G:\mathcal{C}\to \mathcal{D}$, we write it in diagrammatic form as simply $\mathcal{C}\xrightarrow{\psi} \mathcal{D}$. Given another natural transformation $\psi':F'\to G'$ of functors $F',G':\mathcal{C}'\to \mathcal{D}'$, and functors $c:\mathcal{C}\to \mathcal{C}'$ and $d:\mathcal{D}\to \mathcal{D}'$ with natural transformations $dF\to F'c$ and $dG\to G'c$, say that the diagram of natural transformations $$\begin{tikzcd}
		\mathcal{C}\ar[r,"c"]\ar[d,"\psi"]& \mathcal{C}'\ar[d,"\psi'"]\\
		\mathcal{D}\ar[r,"d"]& \mathcal{D}'
		\end{tikzcd}$$commutes if, for every $f:X \to Y$ in $\mathcal{C}$, the diagram
		$$\begin{tikzcd}
		&F'cX\ar[rr,"F'cf"]\ar[dd,"\psi'_{cX}", near start]&&F'cY\ar[dd,"\psi'_{cY}"]\\
		dFX\ar[ur] \ar[rr,"dFf", near start]\ar[dd,"d\psi_X"]& & dFY\ar[dd,"d\psi_Y", near start]\ar[ur] &\\
		&G'cX\ar[rr,"G'cf", near start]&&G'cY\\
		dGX\ar[ur] \ar[rr,"dGf"]& & dGY\ar[ur] &\\
		\end{tikzcd}
		$$ in $\mathcal{D}'$ commutes. In this proof we will be interested in cases when $\mathcal{C}=\mathcal{D}$, $c=d$, both $F$ and $F'$ are the identity, and the map $dF\to F'c$ is the identity, in which case the cube above reduces to the prism$$\begin{tikzcd}
		cX\ar[dr,"\psi'_{cX}"] \ar[rr,"cf", near start]\ar[dd,"c\psi_X"]& & cY\ar[dd,"c\psi_Y", near start]\ar[dr,"\psi'_{cY}"] &\\
		&G'cX\ar[rr,"G'cf", near start]&&G'cY\\
		cGX\ar[ur] \ar[rr,"cGf"]& & cGY\ar[ur] &\\
		\end{tikzcd}.
		$$

		Because $\mm$ preserves both $\per{(Q)}$ and $\per R$, it follows that it preserves $\per A$ too, and moreover descends to an autoequivalence of $\per A / j_! \per R$. By the definition of ${\bar{\Sigma}}$, one has a commutative diagram of functors $$\begin{tikzcd}
		\per A \ar[d,"i^*"]\ar[r]& \per A / j_! \per R \ar[dl,"i^*"]\ar[d]\ar[dr,".e"]& 
		\\ \per{Q} \ar[r]& \per({Q})/\per_\mathrm{fd}({Q})\ar[r,"{\bar{\Sigma}}"'] & \mathcal{M}
		\end{tikzcd}$$and by \ref{fdresplem} and the definition of $\mm_{\mathrm{sg}}$ one has a commutative diagram of functors $$\begin{tikzcd}
		\per{Q} \ar[r]\ar[d,"\mm_{\mathbb{L}}"]& \per({Q})/\per_\mathrm{fd}({Q})\ar[r,"{\bar{\Sigma}}"']\ar[d,"\mm_{\mathbb{L}}"] & \mathcal M \ar[d,"\mm_{\mathrm{sg}}"]
		\\ \per{Q} \ar[r]& \per({Q})/\per_\mathrm{fd}({Q})\ar[r,"{\bar{\Sigma}}"'] & \mathcal M
		\end{tikzcd}.$$By gluing two copies of the first diagram to the second, one sees that the diagram$$\begin{tikzcd}
		\per A \ar[d,"\mm"]\ar[r]& \per A / j_! \per R\ar[d,"\mm"]\ar[r,".e"] & \mathcal M\ar[d,"\mm_{\mathrm{sg}}"]
		\\ \per A \ar[r]& \per A / j_! \per R \ar[r,".e"] & \mathcal M
		\end{tikzcd}$$ commutes. Let $\psi:\id \to \mm$ denote the natural transformation given by the unit map; by an abuse of notation we will denote the obvious analogues $\id \to \mm_{\mathbb{L}}$, $\id \to \mm_{\mathrm{sg}}$, etc. by the same letter. It is not hard to check that the analogous diagram of natural transformations $$\begin{tikzcd}
		\per A \ar[d,"\psi"]\ar[r]& \per A / j_! \per R\ar[d,"\psi"]\ar[r,".e"] & \mathcal M\ar[d,"\psi"]
		\\ \per A \ar[r]& \per A / j_! \per R \ar[r,".e"] & \mathcal M
		\end{tikzcd}$$ commutes. A similar argument to the above using the commutative diagram of functors $$\begin{tikzcd}
		\per A \ar[rr,".e"]\ar[d]&& D^b(R)\ar[d]
		\\ \per A / j_! \per R\ar[r,".e"] & \mathcal{M}\ar[r,hook] & D_{\mathrm{sg}}(R) 
		\end{tikzcd}$$ shows that the diagram of functors $$\begin{tikzcd}
		D^b(R)\ar[d,"\mm"]\ar[rr,bend left=15] & \mathcal{M}\ar[r,hook]\ar[d,"\mm_{\mathrm{sg}}"]& D_\mathrm{sg}(R)\ar[d,"\mm"]
		\\ D^b(R)\ar[rr,bend right=15] & \mathcal{M}\ar[r,hook]& D_\mathrm{sg}(R)
		\end{tikzcd}$$ commutes, and moreover the analogous diagram of natural transformations
		$$\begin{tikzcd}
		D^b(R)\ar[d,"\psi"]\ar[rr,bend left=15] & \mathcal{M}\ar[r,hook]\ar[d,"\psi"]& D_\mathrm{sg}(R)\ar[d,"\psi"]
		\\ D^b(R)\ar[rr,bend right=15] & \mathcal{M}\ar[r,hook]& D_\mathrm{sg}(R)
		\end{tikzcd}$$ commutes. But the left-hand vertical map is an isomorphism, because it is the restriction of $\id \to \mm$ to $D^b(R)$, which is an isomorphism because $I_{\mm}e \to Ae$ is a quasi-isomorphism of $A$-$R$-bimodules by \ref{unitelem}. So the right-hand vertical map must be an isomorphism, because $D^b(R) \to D_\mathrm{sg}(R)$ is a quotient, and hence the map $\id \to \mm_{\mathrm{sg}}$ is an isomorphism.
	\end{proof}

\section{Periodicity, localisation, and the main theorem}
Assume that we are in the setup of \ref{mutsetupp}. For brevity, we put $Q\coloneqq \dq$. Recall from \ref{etaex} the existence of the periodicity element $\eta \in H^{-2}(Q)$. Let $E\simeq \R\underline{\enn}_R(M)$ be the derived localisation of ${Q}$ at $\eta$.
\begin{prop}\label{dgcd}There is a commutative diagram in the homotopy category of dg categories $$\begin{tikzcd} \per Q\ar[d,"\pi"]\ar[dr, "\Sigma"]\ar[d] \ar[r, "-\lot_Q E"]& \per(E)\ar[d,"\alpha"] \\ \per Q / \per_\mathrm{fd}(Q) \ar[r, "\bar{\Sigma}", swap] & \mathcal{M} \end{tikzcd}$$	where $\pi$ is the standard projection functor, $\Sigma$ is the singularity functor, and $\alpha$ and $\bar{\Sigma}$ are quasi-equivalences.
\end{prop}
\begin{proof}
The bottom left triangle commutes by the definition of $\bar{\Sigma}$, which is an equivalence by \ref{perletadg}. We show that the top right triangle commutes. Recall that we can write $\mathcal{M}\simeq \per \R\underline{\enn}_R(M)$, where $M\cong Ae$ is an object of the dg singularity category $D_\mathrm{sg}^\mathrm{dg}(R)$. Moreover, by the proof of \ref{etaex} we have a commutative triangle in the homotopy category of dgas $$\begin{tikzcd}
Q \ar[dr,"\Xi"] \ar[r]& E \ar[d,"\simeq"] \\
& \R\underline{\enn}_R(M)
\end{tikzcd}$$where $\Xi$ is the comparison map of \ref{comparisonmap}, $Q \to E$ is the derived localisation at the periodicity element $\eta$, and $E \to \R\underline{\enn}_R(M)$ is a quasi-isomorphism. This gives us a diagram in the homotopy category of dg categories
$$\begin{tikzcd}
BQ \ar[dr,"B\Xi"] \ar[r]& BE \ar[d,"\simeq"] \\
& B\R\underline{\enn}_R(M)
\end{tikzcd}$$where $BW$ means the dg category with a single object with endomorphism dga $W$. Note that the rightmost map is a quasi-equivalence. Taking perfect modules now gives us a commutative diagram inside the homotopy category of dg categories $$\begin{tikzcd}
\per Q \ar[dr,"F'"] \ar[r,"F"]& \per E \ar[d,"\simeq"] \\
& \per\left(\R\underline{\enn}_R(M)\right)
\end{tikzcd}$$where the rightmost map is a quasi-equivalence. It remains to prove that the induced maps $F$ and $F'$ are the correct ones. But if $\mathcal{T}$ and $\mathcal{T}'$ are pretriangulated dg categories where $\mathcal{T}$ is generated by a single object $G$, then any dg functor $\mathcal{T}\to\mathcal{T}'$ is determined by its value on $G$: because objects in $\mathcal{T}$ are generated by $G$ under cones and shifts, their hom complexes are all iterated cones of maps between $\dge(G)$. The same clearly applies for the image of $\mathcal{T}$. So given $G' \in \mathcal{T}'$ and a dg functor of one-object dg categories $G \to G'$, this uniquely extends to a dg functor $\mathcal{T}\to \mathcal{T}'$ by tensoring the hom-complexes in $\mathcal{T}$ with the map $\dge(G)\to \dge(G')$ \cite[Exercice 34]{toendglectures}. In particular it follows that the induced map $F:\per Q \to  \per E$ is the tensor product $-\lot_Q E$. Recall the definition of $\Xi$ from \ref{comparisonmap}: it is the component of the dg functor $\Sigma:\per Q \to \mathcal{M}$ at the object $Q$. In particular, if one restricts $\Sigma$ to the one object dg category $BQ\subseteq \per Q$, then one gets the dg functor $\Xi$. So $F'\cong\Sigma$.
\end{proof}
\begin{defn}
	Let $\mm_E:\per E \to \per E$ be the endofunctor defined by $$\mm_E\coloneqq \alpha \mm_{\mathrm{sg}}\alpha^{-1}.$$
\end{defn}
\begin{lem}
	The isomorphism $\id_{\mathcal M} \to \mm_\mathrm{sg}$ induces an isomorphism $\id_{\per E} \to \mm_E$.
\end{lem}
\begin{proof}
	Follows from applying $\alpha^{-1}\circ (-) \circ \alpha$ to \ref{mmsgisid}.
	\end{proof}

\begin{lem}
	The following diagram is commutative: $$\begin{tikzcd} \per Q\ar[d,"\mm_{\mathbb{L}}"] \ar[r, "-\lot_Q E"]& \per E\ar[d,"\mm_{E}"] \\ \per Q \ar[r, "-\lot_Q E"]& \per E \end{tikzcd}$$
\end{lem}
\begin{proof}
	Follows from the definition of $\mm_E$ along with \ref{dgcd}.
\end{proof}

\begin{lem}\label{immltriv}
	Applying $-\lot_{{Q}}E$ to the ${Q}$-bimodule map $I_\mm^\mathbb{L} \to {Q}$ of \ref{unitrecol} gives a ${Q}$-$E$-bimodule quasi-isomorphism $I_\mm^\mathbb{L} \lot_{{Q}} E \to E$.
\end{lem}
\begin{proof}The idea is to look at $\mm_{\mathbb{L}} ^{-1} \to \id$, which becomes a quasi-isomorphism upon inverting $\eta$. Consider the functor $\mm_{\mathbb{L}} ^{-1}: \per Q \to \per Q$ which sends $X$ to $X\lot_Q I^\mathbb{L}_\mm$. It comes with a natural transformation $\mm_{\mathbb{L}} ^{-1}\to \id$, which gives a natural transformation $\mm_E^{-1} \to \id$. By \ref{dgcd} and (the proof of) \ref{mmsgisid}, this natural transformation $\mm_E^{-1} \to \id$ must be an isomorphism. Hence, for all perfect $Q$-modules $X$ the natural map $X \lot_Q I_\mm^\mathbb{L} \lot_Q E \to X \lot_Q E$ is a quasi-isomorphism of $E$-modules. So the natural map $I_\mm^\mathbb{L} \lot_Q E \to E$ must be a quasi-isomorphism of $Q\text{-}E$-bimodules.
\end{proof}

\begin{prop}\label{mmshiftrep}
	As $Q$-bimodules, $I^\mathbb{L}_\mm$ is quasi-isomorphic to $Q[2]$.
\end{prop}
\begin{proof}For brevity, write $I\coloneqq I_\mm^\mathbb{L}$. By \ref{dictionary}(2) applied twice (or \ref{mmdqidqt}), one gets a right $Q$-module quasi-isomorphism $\mm(Q)\simeq Q[-2]$ and hence a right $Q$-module quasi-isomorphism $\mm^{-1}(Q)\simeq Q[2]$. Now it follows that $I$ is quasi-isomorphic to $Q[2]$ as right $Q$-modules. The main difficulty lies in upgrading this to a bimodule quasi-isomorphism. Pick a right quasi-isomorphism $I\xrightarrow{f} Q[2]$ and tensor it with $E$ to get a map $f'$ fitting into a commutative diagram
	$$\begin{tikzcd}Q[2]\ar[r,"g'"] & Q[2]\lot_Q E \\ I\ar[u,"f"]\ar[r,"g"] & I\lot_Q E\ar[u,"f'"]
	\end{tikzcd}$$where $f$ and $f'$ are right $Q$-module quasi-isomorphisms and $g$ and $g'$ are $Q$-bimodule maps. Truncate this diagram to degrees weakly below $-2$ to get a commutative diagram
	$$\begin{tikzcd}Q[2]\ar[r,"v'"] & \tau_{\leq -2}\left(Q[2]\lot_Q E\right) \\ I\ar[u,"u"]\ar[r,"v"] & \tau_{\leq -2}\left(I\lot_Q E\right) \ar[u,"u'"]
	\end{tikzcd}$$where, as before, $u$ and $u'$ are right $Q$-module quasi-isomorphisms and $v$ and $v'$ are $Q$-bimodule maps. After identifying $Q[2]\lot_Q E$ with $E[2]$, we may identify $g'$ with the shifted localisation map $Q[2]\to E[2]$. By \ref{qisolem} and \ref{etaex}(4) the localisation map $Q \to E$ induces a quasi-isomorphism $Q \to \tau_{\leq 0}E$. Hence $v'$ is a quasi-isomorphism. Now it follows that $v$ is a quasi-isomorphism too. So as a bimodule, $I$ is quasi-isomorphic to $\tau_{\leq -2}\left(I\lot_Q E\right)$ and it hence remains to show that $\tau_{\leq -2}\left(I\lot_Q E\right)$ is bimodule quasi-isomorphic to $Q[2]$. From \ref{immltriv}, one has a $Q\text{-}E$-bimodule quasi-isomorphism, and hence a $Q$-bimodule quasi-isomorphism, $I \lot_{{Q}} E \to E$. So it remains to show that $\tau_{\leq -2}E \simeq Q[2]$ as $Q$-bimodules. But one has $\tau_{\leq -2}E \simeq \tau_{\leq -2}Q$, and furthermore the periodicity element $\eta$ gives a $Q$-bimodule quasi-isomorphism $ \tau_{\leq -2}Q \simeq Q[2]$, using \ref{etaex}(2) and \ref{etahoch}.
\end{proof}

\begin{cor}\label{mmshift}
	The autoequivalence $\mm_{\mathbb{L}}:D(Q) \to D(Q)$ is isomorphic to the shift $[-2]$.
\end{cor}
\begin{proof}
	Follows immediately from \ref{mmshiftrep} by looking at representing objects.
\end{proof}

\begin{thm}\label{mutncontrol}
	Let $R$ be a complete local isolated hypersurface singularity of dimension at least 2, $M$ a MCM modifying $R$-module with no free summands, $A\coloneqq \enn_R(R\oplus M)$, and\linebreak $e\coloneqq \id_R\in A$. Then the dga $A_\mm\coloneqq \tau_{\geq -1}(\dq)$ controls the mutation-mutation autoequivalence $\mm: D(A)\to D(A)$, in the sense that $\mm$ is represented by the cocone of the natural map $A \to A_\mm$.
\end{thm}
\begin{proof}
Considering representing objects, we need to show that the sequence $I_\mm \to A \to A_\mm$ of $A$-bimodules extends to an exact triangle. This will be an involved argument making much use of the Nine Lemma, although we have done all of the difficult work already. First observe that $I_\mm \to A$ has image contained in the $A$-bimodule $AeA$ by \ref{aealem} and hence the composition $I_\mm \to A \to A_\mm$ is zero in the derived category. For brevity, put $I\coloneqq I_\mm$ and $L\coloneqq A_\mm$. If $X$ is an $A$-bimodule, we put $X^C\coloneqq X\lot_A\mathrm{Cell}(A)$ and $^CX\coloneqq \mathrm{Cell}(A)\lot_AX$, where $\mathrm{Cell}$ is the cellularisation functor of \ref{celldef}. We also put $X^Q\coloneqq X\lot_A\dq$ and $^QX\coloneqq \dq\lot_AX$. Note that we may write expressions like $^QX^C$ without ambiguity, since there is a canonical quasi-isomorphism $^Q(X^C)\simeq (^QX)^C$. Recall from \ref{dqexact} the distinguished triangle of $A$-bimodules $\mathrm{Cell}(A)\to A \to \dq\to $. Take the sequence $I \to A \to L$ and tensor it on the right with $\mathrm{Cell}(A)\to A \to \dq\to $ to obtain a commutative square
\begin{equation*}
\begin{tikzcd}
I^C \ar[r]\ar[d]& A^C \ar[r]\ar[d]& L^C\ar[d]\\
I \ar[r]\ar[d]& A\ar[r] \ar[d]& L \ar[d]\\
I^Q \ar[r]& \ar[r]A^Q & L^Q \\
\end{tikzcd}
\end{equation*}
with exact columns (we refer to this square as the `middle face'). Take this commutative square and tensor it on the left with $\mathrm{Cell}(A)\to A \to \dq\to $ to obtain a commutative cube with front face 
\begin{equation*}
\begin{tikzcd}
^CI^C \ar[r]\ar[d]& ^CA^C \ar[r]\ar[d]& ^CL^C\ar[d]\\
^CI \ar[r]\ar[d]& ^CA\ar[r] \ar[d]& ^CL \ar[d]\\
^CI^Q \ar[r]& \ar[r]^CA^Q & ^CL^Q \\
\end{tikzcd}
\end{equation*}and with back face
\begin{equation*}
\begin{tikzcd}
^QI^C \ar[r]\ar[d]& ^QA^C \ar[r]\ar[d]& ^QL^C\ar[d]\\
^QI \ar[r]\ar[d]& ^QA\ar[r] \ar[d]& ^QL \ar[d]\\
^QI^Q \ar[r]& \ar[r]^QA^Q & ^QL^Q \\
\end{tikzcd}.
\end{equation*}
Observe that every sequence in this cube which runs vertically and every sequence in this cube which goes `into the page' is exact. We analyse the rows individually and use the Nine Lemma successively to prove that the row $I \to A \to L$ is exact. First we analyse the back face. Because $L$ is a truncation of $\dq$, it is $e$-acyclic, so both $L^C$ and $^CL$ are acyclic, and moreover $^QL^Q$ is quasi-isomorphic to $L$. Because $A\to \dq$ is a homological epimorphism, $^QA^Q$ is quasi-isomorphic to $\dq$. The bottom row of the back face now reads as $^QI^Q \to \dq \to L$, and this is exact because $^QI^Q\to \dq$ is exactly the inclusion $\tau_{\leq -2}(\dq) \to \dq$ by (the proof of) \ref{mmshiftrep}. 
\p Because $Ie \to Ae$ is a bimodule quasi-isomorphism, $I^C \to A^C$ is a quasi-isomorphism. Because $L^C$ is acyclic, it follows that the top row of the middle face is exact. It now follows that the top rows of the front and back faces are also exact. The Nine Lemma applied to the back face now tells us that the middle row of the back face is exact.

\p Consider the bottom row of the front face. Because $^CL$ is acyclic, $^CL^Q$ must be acyclic. Moreover, $^CA^Q$ is acyclic because $e.(\dq)$ is acyclic. By \ref{bimodlem} we have a quasi-isomorphism $eI\simeq eA$, and so $^CI^Q$ must also be acyclic. Hence, the bottom row of the front face is exact, since it consists of acyclic objects.

\p The Nine Lemma applied to the front face now tells us that the middle row of the front face is exact. Because the middle row of the back face is exact, the Nine Lemma applied again to the square formed by the middle rows now tells us that the row $I \to A \to L$ is exact.
\end{proof}

\begin{rmk}\label{twistrmk}
	We remark that $\mm$ can be interpreted as a sort of `noncommutative twist' around $A_\mm$. We follow the proofs in Donovan--Wemyss \cite[\S6.3]{DWncdf}; see also Segal \cite{segaltwists} for background on twists. Let $F:D(A_\mm) \to D(A)$ be restriction of scalars along $A \to A_\mm$. First observe that $F$ has right and left adjoints given by $R\coloneqq \R\hom_A(A_\mm,-)$ and $L\coloneqq -\lot_A A_\mm$ respectively. By the above, we have an exact triangle of $A$-bimodules $I_\mm \to A \to A_\mm \to$. Applying derived hom and tensor respectively gives exact triangles of endofunctors of $D(A)$ of the form \begin{align*}
	& FR\to \id \to \mm\to \\ & \mm^{-1} \to \id \to FL \to
	\end{align*}
	using that $\R\hom_A(A_\mm,-)\simeq FR$ and $-\lot_A A_\mm\simeq FL$.
\end{rmk}

\begin{rmk}
	Let $\Gamma=\Gamma^{-1}\to \Gamma ^0$ be a $[-1,0]$-truncated noncommutative Artinian dga. As in \ref{prorepremk}, the inclusion-truncation adjunction gives an isomorphism between $\hom(\dq,\Gamma)$ and $\hom(A_\mm,\Gamma)$, and we see that $A_\mm$ controls the $[-1,0]$-truncated derived noncommutative framed deformations of the simple module $S$.
\end{rmk}

\begin{prop}\label{immexsq}
The complex $I_\mm$ is a module, which moreover fits into a short exact sequence of $A$-bimodules
$$0\to H^{-1}(\dq) \to I_\mm \to AeA\to 0.$$
\end{prop}	
\begin{proof}
By \ref{mutncontrol}, we have a distinguished triangle of $A$-bimodules $I_\mm \to A \to A_\mm \to$. Because $A$ has cohomology only in degree 0, and $A_\mm$ has cohomology only in degrees $0$ and $-1$, the long exact sequence in cohomology tells us that $I_\mm$ has cohomology only in degree zero. In this degree, the long exact sequence turns into an exact sequence $$0 \to H^{-1}(\dq) \to I_\mm \to A \to A/AeA\to 0$$where the rightmost map is the standard projection. Replacing $A \to A/AeA$ by its kernel $AeA$ gives the desired result.
	\end{proof}
\begin{cor}
The following are equivalent:
\begin{enumerate}
	\item The map $I_\mm \to AeA$ is an isomorphism.
	\item The map $A_\mm \to A/AeA$ is a quasi-isomorphism.
	\item The cohomology group $H^{-1}(\dq)$ vanishes.
	\item The $R$-module $M$ is rigid.
	\item The cohomology algebra of $\dq$ is $H(\dq)\cong A/AeA[\eta]$, where $\eta$ is the periodicity element of \ref{etaex}.
	\end{enumerate}	
\end{cor}
\begin{proof}
By the exact sequence of \ref{immexsq}, the map $I_\mm \to AeA$ is an isomorphism if and only if the cohomology group $H^{-1}(\dq)\cong H^{-1}(A_\mm)$ vanishes, so $(1) \iff (3)$. Since $A_\mm$ has cohomology only in degrees $0$ and $-1$ by definition, we see that $H^{-1}(A_\mm)$ vanishes if and only if the map $A_\mm \to H^0(A_\mm)\cong A/AeA$ is a quasi-isomorphism, so $(2)\iff (3)$. By \ref{dqcohom}, there is an isomorphism $H^{-1}(\dq)\cong \ext^1_R(M,M)$. Since $M$ is rigid if and only if $\ext^1_R(M,M)$ vanishes, we have $(3)\iff (4)$. Finally, $(4)\iff (5)$ is \ref{rigidrmk}.
	\end{proof}

	\begin{rmk}\label{mtnrmk}
		In particular, if $X\to \spec R$ is a minimal model of a three-dimensional terminal singularity, then the module $M$ defining the noncommutative model $A$ is rigid and we have $I_\mm\cong AeA$, which provides a new proof of Donovan--Wemyss's result \cite[5.10]{DWncdf}. If $R$ is a surface, then $M$ is never rigid by AR duality \ref{arduality}, and in particular the contraction algebra $A_\con$ never controls $\mm$ via noncommutative twists.
		\end{rmk}

	\cleardoublepage
	\phantomsection
		\addcontentsline{toc}{chapter}{Bibliography}

	\bibliography{thesisbib}

\end{document}